\documentclass[11pt,a4paper,reqno]{amsart}
\usepackage{amsmath,amsfonts,amssymb,amsthm,mathrsfs}
\usepackage[utf8]{inputenc}
\usepackage[T1]{fontenc}
\usepackage{lmodern}
\usepackage{latexsym}
\usepackage{cancel}
\usepackage{microtype}
\usepackage{caption}
\usepackage{MnSymbol}
\usepackage{xcolor}
\usepackage{enumerate}
\usepackage[normalem]{ulem}
\usepackage{bbm}
\usepackage{marginnote}

\usepackage[a4paper,margin=2.5cm]{geometry}
\usepackage{dsfont}

\usepackage[colorlinks,allcolors=black, breaklinks=true]{hyperref}

\usepackage{empheq}
\usepackage{bbm}
\usepackage{multirow}
\usepackage{mathrsfs}

\usepackage{breakcites}

\makeatletter \def\l@subsection{\@tocline{2}{0pt}{1pc}{5pc}{}} \def\l@subsection{\@tocline{2}{0pt}{2pc}{6pc}{}} \makeatother

\long\def\metanote#1#2{{\color{#1}\
		\ifmmode\hbox\fi{\sffamily\mdseries\upshape [#2]}\ }}

\long\def\TT#1{\metanote{teal}{{\tiny TT} #1}}

\newcommand{\xnot}[1]{x^{\textnormal{not}, #1}}
\newcommand{\xonly}[1]{x^{\textnormal{only}, #1}}
\newcommand{\xnottwo}[2]{{#1}^{\textnormal{not}, #2}}
\newcommand{\xonlytwo}[2]{{#1}^{\textnormal{only}, #2}}

\newcommand{\e}{{\mathbf E}}

\newcommand{\V}[1]{{\mathbf{Var}}\left\{#1\right\}}

\newcommand\cD{\mathscr D}

\newcommand\cG{\mathscr G}

\newcommand\cL{{\mathscr L}}

\newcommand\cN{\mathscr N}

\newcommand\cS{{\mathscr S}}

\newcommand{\bT}{\mathbf{T}}

\providecommand{\ora}[1]{}
\renewcommand{\ora}[1]{\overrightarrow{#1}}
\DeclareRobustCommand{\rSkipTocEntry}[5]{} 

\newtheorem{thm}{Theorem}
\newtheorem{lem}[thm]{Lemma}
\newtheorem{prop}[thm]{Proposition}
\newtheorem{cor}[thm]{Corollary}
\newtheorem{assump}{Assumption}

\newtheorem{dfn}[thm]{Definition}

\newtheorem{remark}[thm]{Remark}

\newcommand{\rM}{\ensuremath{\mathrm{M}}}

\newcommand{\rY}{\ensuremath{\mathrm{Y}}}

\newcommand{\rX}{\ensuremath{\mathrm{X}}}

\newcommand{\rZ}{\ensuremath{\mathrm{Z}}}

\newcommand{\rW}{\ensuremath{\mathrm{W}}}

\newcommand{\rN}{\ensuremath{\mathrm{N}}}

\newcommand{\rP}{\ensuremath{\mathrm{P}}}

\newcommand{\supp}{\ensuremath{\operatorname{supp}}}

\DeclareRobustCommand{\stirling}{\genfrac\{\}{0pt}{}}

\providecommand{\X}{}
\providecommand{\R}{}
\providecommand{\E}{}
\providecommand{\P}{}
\providecommand{\Z}{}

\providecommand{\N}{}
\providecommand{\C}{}

\providecommand{\G}{}

\renewcommand{\P}{\mathbb{P}}
\renewcommand{\E}{\mathbb{E}}
\renewcommand{\R}{\mathbb{R}}
\renewcommand{\X}{\mathbb{X}}
\renewcommand{\Z}{\mathbb{Z}}

\renewcommand{\S}{\mathbb{S}}
\renewcommand{\N}{{\mathbb N}}
\renewcommand{\C}{\mathbb{C}}

\renewcommand{\G}{\mathbb{G}}

\newcommand{\cleq}{\ensuremath{\preccurlyeq}}

\usepackage[french,english]{babel}

\usepackage{lipsum}

\usepackage{graphicx}
\graphicspath{ {./Figures/} }

\usepackage{tikz-cd}
\usepackage{adjustbox}

\usepackage[title, titletoc,toc]{appendix}
\usepackage{esint}
\appto\appendix{\addtocontents{toc}{\protect\setcounter{tocdepth}{2}}}

\counterwithin{equation}{section}

\renewcommand{\epsilon}{\varepsilon}
\renewcommand{\phi}{\varphi}
\title[Fluctuations for a Spatial Logistic Branching Process with Weak Competition]{Non-Equilibrium Fluctuations for a Spatial Logistic Branching Process with Weak Competition}

\author{Thomas Tendron}
\thanks{Department of Statistics, University of Oxford, 24 St Giles’, Oxford, United Kingdom OX1 3LB, \textit{Email:} thomas.tendron@mail.mcgill.ca}

\address{}
\email{}

\date{\today}

\begin{document}
	
	\begin{abstract}
		The spatial logistic branching process is a population dynamics model in which particles move on a lattice according to independent simple symmetric random walks, each particle splits into a random number of individuals at rate one, and pairs of particles at the same location compete at rate $c$. We consider the weak competition regime $c \searrow 0$, corresponding to a local carrying capacity tending to infinity like $c^{-1}$. We show that the hydrodynamic limit of the spatial logistic branching process is given by the Fisher-Kolmogorov-Petrovsky-Piskunov equation. We then prove that its non-equilibrium fluctuations converge to a generalised Ornstein-Uhlenbeck process with deterministic but heterogeneous coefficients. The proofs rely on an adaptation of the method of $v$-functions developed in \cite{Boldrig1992}. An intermediate result of independent interest shows how the tail of the offspring distribution and the precise regime in which $c\searrow 0$ affect the convergence rate of the expected population size of the spatial logistic branching process to the hydrodynamic limit. 
	\end{abstract}
	
	\maketitle
	\tableofcontents
	
	\section{Introduction}\label{intro}
	
	\subsection{Particle system and rescaling}
	
	The spatial logistic branching process is a particle system which models the dynamics of a locally regulated population consisting of colonies that live on the discrete circle $\Z_K = \Z/K\Z$, for some integer $K \geq 2$. We focus on the system up to an arbitrary, but fixed, time horizon $T>0$. The individuals in the population reproduce asexually by multiple fission, migrate on the lattice, and compete for resources within their local colony. In the following definition, we give the precise instantaneous dynamics of the particle system.  
	
	\begin{dfn}\emph{(The spatial logistic branching process)}\label{def:slbp}
		Let $\N=\{1, 2, \cdots\}$ denote the natural numbers, and fix a probability mass function $\rP=(p_\ell)_{\ell\in \N}$ and a constant $c>0$. We consider the particle system with the following instantaneous dynamics.
		\begin{enumerate}
			\item[1.] \textbf{Diffusion}: each particle in the system performs an independent simple symmetric random walk on $\Z_K$.
			\item[2.] \textbf{Branching}: independently of its movement, a particle splits into $\ell+1$ particles at rate $p_\ell$. In particular, branching always increases the number of particles in the system. 
			\item[3.] \textbf{Competition}: each pair of particles at the same location coalesces at rate $c$. 
		\end{enumerate} 
		Let $N_t(z) \geq 0$ denote the number of particles at position $z \in \Z_K$ at time $t\geq 0$. The spatial logistic branching process, or SLBP for short, is the strong Markov process $\rN=((N_t(z))_{z \in \Z_K}:t\geq 0)$. We call $\rP$ its offspring distribution. 
	\end{dfn}
	
	\begin{remark}
		We note that if the process is defined on $\Z$ and $N_0$ is not compactly supported, it is not immediately clear that the SLBP is well-defined for all times. In this article, we will restrict our attention to the SLBP with compactly supported initial data by working on $z \in \Z_K$. 
	\end{remark}
	\noindent We briefly mention two elementary properties of the SLBP. First, since $N_0:\Z_K\to \N_0$ is finite and $\mu<\infty$, a simple comparison with a pure birth process shows that $\rN$ does not explode in finite time. Secondly, because of the local competition, the SLBP does not satisfy the branching property.
	
	We are interested in a regime of weak competition where $c\searrow 0$. More precisely, we scale the process in the following way.
	
	\begin{dfn}\emph{(Weak competition regime)}\label{def:rescaled_slbp}
		Fix an arbitrary $\kappa\geq 0$, and let $\epsilon \in (0, 1)$ denote a scaling parameter. The rescaled SLBP evolves on the discrete circle $\Z_\epsilon = \Z_{K_\epsilon}$, where $K_\epsilon=\lfloor\epsilon^{-1}\rfloor$. Consider $L:(0, 1)\to \N_0$ a function satisfying $\lim_{\epsilon \searrow 0}L(\epsilon)=\infty$ with $L(\epsilon)=o(\epsilon^{-\kappa})$ as $\epsilon \searrow 0$. Given an offspring distribution $\rP=(p_\ell)_{\ell\in \N}$ with finite mean $\mu>1$, we define its truncation $\rP^\epsilon = (p^\epsilon_\ell)_{\ell\in \N}$ by 
		\[
		p_\ell^\epsilon=\begin{cases}c_\epsilon p_\ell, & \ell \leq L(\epsilon)
			\\0,&\ell >L(\epsilon)
			\end{cases}, \quad c_\epsilon = \left(\sum_{\ell=1}^{L(\epsilon)}p_\ell\right)^{-1}. 
		\]
		We consider the rescaling $\rN^\epsilon=((N^\epsilon_t(z))_{z \in \Z_\epsilon}:t\geq 0)$ of the SLBP with the following instantaneous dynamics. 
		\begin{enumerate}
			\item[1.] \textbf{Diffusion}: each particle in the system performs an independent simple symmetric random walk on $\Z_\epsilon$ at the accelerated jump rate $\epsilon^{-2}$.
			\item[2.] \textbf{Branching}: independently of its movement, a particle splits into $\ell+1$ particles at rate $p_\ell^\epsilon$ - branching always increases the number of particles in the system. 
			\item[3.] \textbf{Weak competition}: each pair of particles at the same location coalesces at rate $\epsilon^\kappa$. We call $\kappa$ the competition exponent - the smaller $\kappa$ is, the larger the coalescence rate $\epsilon^\kappa$.
		\end{enumerate}
		We call the process $\rX^\epsilon=((X^\epsilon_t(z))_{z \in \Z_\epsilon}:t\geq 0)$ defined by $X^\epsilon_t(z)=\epsilon^\kappa N^\epsilon_t(z)$, the rescaled SLBP.
	\end{dfn}
	
	\begin{remark}
		In Theorems \ref{thm:LLN} and \ref{thm:CLT_Gaussian} below, we further rescale the lattice so that its mesh size is $\epsilon$ to obtain a diffusive rescaling of the random walks.
	\end{remark} 
	
	\subsection{Main results} The technical assumptions on the initial data $X_0^\epsilon$ can only be stated after a number of definitions, so we postpone the detailed description of the initial data and we first present our results. Let us just mention that the initial data has the following property. For some $\rho_0 : \S^1\to [0, 2\mu]$ and all small $\epsilon >0$, the law of $X_0^\epsilon$ is very close to a product over $z \in \Z_\epsilon$ of independent Poisson distributions with intensities $\epsilon^{-\kappa}\rho_0(K_\epsilon^{-1} z)$ and rescaled by $\epsilon^{\kappa}$. When $\epsilon\searrow 0$, $X_0^\epsilon$ converges in distribution to a product over $z \in \S^1$ of Poisson distributions with intensities $\rho_0(z)$.  

	We define $\rho:[0, \infty)\times \S^1\to [0, 2\mu]$ the unique solution to the Cauchy problem for the Fisher-Kolmogorov-Petrovsky-Piskunov (FKPP) equation:
	\begin{equation}\label{eq:FKPP}
		\begin{cases}
			\partial_t \rho(t, z) = \frac{1}{2}\partial_z^2 \rho(t, z)+\rho(t, z)(\mu-\rho(t, z)/2), & t>0, \quad z \in \S^1,\\
			\lim_{t\searrow 0}\rho(t,z) = \rho_0(z),&\quad z \in \S^1.
		\end{cases}
	\end{equation}
	Our first result shows the convergence in probability of the rescaled SLBP of Definition~\ref{def:rescaled_slbp} to the \textit{hydrodynamic limit} $\rho$. 
	\begin{thm}\emph{(Weak law of large numbers)}\label{thm:LLN}
		Suppose that Assumption \ref{assump:technical_X_0} below holds and that the competition exponent satisfies $\kappa \in [0, 3/8)$. Consider the action of the rescaled SLBP on test functions via
		\[
		X_t^\epsilon(\phi)=\epsilon \sum_{z \in \Z_\epsilon} X_t^\epsilon(z)\phi(\epsilon z), \quad \phi \in C^\infty(\S^1), \quad t\in [0, T].
		\]  
		Let $\rho$ act on test functions by
		\[
		\rho_t(\phi) \coloneqq \int_{\S^1}\rho_t(z)\phi(z)dz,\quad \phi \in C^\infty(\S^1), \quad t\in [0, T].
		\]
		Then, we have $\rX^\epsilon \to \rho$ in probability in the Skorohod space $\cD([0, T], C^\infty(\S^1)')$ as $\epsilon \searrow 0$. In particular, for every $\phi \in C^\infty(\S^1)$, we have $X^\epsilon(\phi) \to \rho(\phi)$ in probability in the Skorohod topology on $\cD([0, T], \R)$ as $\epsilon \searrow 0$.
	\end{thm}
	\begin{remark}
		 In Theorem \ref{thm:LLN}, the precise upper bound $3/8$ for the competition exponent $\kappa$ is an artefact of the proof of the intermediate result Theorem \ref{thm:QLLN}, and it is not indicative of special behaviour at that level of $\kappa$.
	\end{remark}
	
	\noindent In our main result, the central limit theorem, we are interested in characterising the limiting non-equilibrium fluctuations of the rescaled SLBP as $\epsilon \searrow 0$. We introduce the fluctuations field in the next definition.
	\begin{dfn}\emph{(Fluctuations field)}\label{def:fluct_field_def}
		Let $\gamma \geq -1/2$ and suppose that $X_0^\epsilon$, $\epsilon \in(0, 1)$, satisfies the Assumption \ref{assump:technical_X_0} below. For all $\epsilon \in (0, 1)$, we define 
		\[
		Y^\epsilon_t(z)= \epsilon^{\gamma-\kappa}(X_t^\epsilon(z)-\E[X_t^\epsilon(z)]), \quad z \in \Z_\epsilon, \quad t\geq 0.
		\]
		As we did in Theorem \ref{thm:LLN} for $\rX^\epsilon$, we may also view $\rY^\epsilon$ as an element of $\cD([0, T], C^\infty(\S^1)')$ via its action 
		\[
		Y_t^\epsilon(\phi) = \epsilon \sum_{z \in \Z_\epsilon} Y_t^\epsilon(z)\phi(\epsilon z), \quad \phi \in C^\infty(\S^1), \quad t\geq 0.
		\]
	\end{dfn}
	\begin{remark}\label{rmk:centering}
		We may equivalently study the fluctuations process 
		\[
		Y_t^\epsilon(z) = \epsilon^{\gamma-\kappa}(X_t^\epsilon(z)-\rho_t^\epsilon(z)), \quad z \in \Z_\epsilon, \quad t \geq 0,
		\]
		where $\rho^\epsilon$ is a suitable approximation to the solution $\rho$ to the FKPP equation \eqref{eq:FKPP}. The approximation is defined in \eqref{eq:interm_fkpp_diff_form} below, its convergence to $\rho$ is studied in part one of Lemma \ref{lem:FKPP_approx_properties}. In Theorem \ref{thm:QLLN},  we show that $\rho^\epsilon_t(z)-\E[X_t^\epsilon(z)]=O(\epsilon^{1+\kappa})$. To directly study the fluctuations of $X^\epsilon$ around the solution $\rho$ to the FKPP equation, one would need to find the rate of convergence of $\rho^\epsilon$ to $\rho$.
	\end{remark}
	
	In the study of the fluctations of the SLBP, we assume that the offspring distribution has a finite second moment. 
	
	\begin{assump}\emph{(Offspring distribution)}\label{assump:OD}
		Recall that $\rP=(p_\ell)_{\ell \in \N}$ denotes the offspring distribution and that $\mu>1$ is its (finite) mean. Assume that we have $\sigma^2 = \sum_{\ell=1}^\infty (\ell-\mu)^2p_\ell<\infty$.
	\end{assump} 
	
	\noindent Our main result is the following central limit theorem.
	\begin{thm}\emph{(Central limit theorem)}\label{thm:CLT_Gaussian}
		Suppose that Assumption \ref{assump:OD}, and Assumption \ref{assump:technical_X_0} below, both hold. Take $\gamma = (\kappa-1)/2$ with $\kappa \in [0, 3/8)$. Recall that $\rho$ denotes the unique solution to the FKPP equation \eqref{eq:FKPP}. Then $\rY^\epsilon\overset{d}{\to}\rY$ in $\cD([0, T], C^\infty(\S^1)')$ as $\epsilon\to 0$ , where ${\rY\in C([0, T], C^\infty(\S^1)')}$ is the unique solution to the Cauchy problem
		\begin{equation}\label{eq:OU}
		\begin{cases}
			\partial_t Y =\Big(\frac{1}{2}\partial_z^2+\mu-\rho\Big)Y +\sqrt{\rho}\partial_z\dot W^1+\sqrt{(\sigma^2+\mu^2)\rho+\rho^2/2}\dot W^2&\textnormal{on }[0, T]\times \S^1
			\\Y_t\overset{d}{\to} Y_0 \textnormal{ as }t\searrow 0,
		\end{cases},
		\end{equation}
		where $\dot W^1$ and $\dot W^2$ are independent Gaussian space-time white noises on $[0, T]\times \S^1$, and $Y_0$ is a $C^\infty(\S^1)'$-valued centred Gaussian variable with variance
		\[
		\E[Y_0(\phi)^2] = \int_{\S^1} \phi(z)^2 \rho_0(z)dz, \quad \phi \in C^\infty(\S^1). 
		\]
	\end{thm}
	\begin{remark}
		\begin{enumerate}
			\item[1.] We note that the competition exponent $\kappa$ of Definition \ref{def:rescaled_slbp} appears neither in the hydrodynamic limit \eqref{eq:FKPP}, nor in the limiting fluctuations \eqref{eq:OU}. Its contribution becomes clear in a result of independent interest stated in Theorem \ref{thm:QLLN} below. In short, the competition exponent affects the convergence rate of the expected population size of the rescaled SLBP to its hydrodynamic limit. The larger $\kappa$, the weaker the competition, and the faster the convergence.   
			\item[2.] The noise $\partial_z \dot W^1$ results from fluctuations in the motion of particles, while the independent noise $\dot W^2$ captures variations in the local density of particles due to the birth and competition mechanisms. 
			\item[3.] The solution to \eqref{eq:OU} is unique in law and given by the mild solution (see e.g. \cite[Chapter 5]{W1986})
		\end{enumerate}
	\end{remark}
	
	\subsection{A class of correlation functions: the $v$-functions} In establishing the limit theorems for the rescaled SLBP $\rX^\epsilon$, we study the behaviour of some correlation functions whose precise definition is adapted from \cite{Boldrig1992}, and which are known as the $v$-functions. To state our assumptions on the initial data $X_0^\epsilon$ and $Y_0^\epsilon$ of Theorems \ref{thm:LLN} and \ref{thm:CLT_Gaussian}, we must first introduce the correlation functions. We consider the space $\X_\epsilon = (\epsilon^\kappa\N_0)^{\Z_\epsilon}$ of configurations of $\rX^\epsilon$. 
	\begin{dfn}
		We define a strict partial order $\cleq$ on $\X_\epsilon$ as follows. Given two configurations $x, x'\in \X_\epsilon$, we write $x\cleq x'$ if $x(z)\leq x'(z)$ for all $z \in \Z_\epsilon$. If $x, x' \in \X_\epsilon$ are such that $x\cleq x'$, then we write $x\backslash x'$ for the configuration $y\in \X_\epsilon$ which satisfies $y(z) = x(z)-x'(z)$ for all $z \in \Z_\epsilon$.
	\end{dfn}
	\noindent For a configuration $x \in \X_\epsilon$, we define its support 
	\[
	\supp(x) = \{z \in \Z_\epsilon: x(z)>0\},
	\]
	and its cardinality, i.e. the number of particle it contains, by
	\[
	n_x = \epsilon^{-\kappa} \sum_{z \in \Z_\epsilon} x(z). 
	\]
	If we are only interested in the number of particles at a single site $z \in \Z_\epsilon$, then we write ${n_x(z) = \epsilon^{-\kappa} x(z)}$. By abuse of notation, we also think of a configuration $x \in \X_\epsilon$ as an arbitrary but fixed enumeration $z^x_1, \cdots, z_{n_x}^x$ of its $n_x$ particles. For any function $f:\Z_\epsilon\times \X_\epsilon\to \R$, we introduce the notation
	\begin{equation}
		\begin{aligned}\label{def:sum_prod_config_iteration}
		&\sum_{z \in x} f(z, x) \coloneqq \sum_{i=1}^{n_x}f(z^x_i, x),
		\\&\prod_{z \in x}f(z, x) \coloneqq \prod_{i=1}^{n_x}f(z^x_i, x).
		\end{aligned}
	\end{equation}
	A $v$-function is defined as the expectation of certain polynomials known in \cite{Boldrig1992} as the $V$-functions, which we define now. In our work, we introduce a version of the $V$-functions which is rescaled in terms of $\epsilon^\kappa$ to capture the fact that, in the weak competition regime, the local equilibrium population size is of order $\epsilon^{-\kappa}$. For $n, k \in \N_0$, we let 
	\begin{equation}\label{def:ff}
	Q_k(n) \coloneqq 
	\begin{cases}
		n(n-1)\cdots (n-k+1)&1\leq k\leq n\\
		1&k=0\\
		0&k>n
	\end{cases},
	\end{equation}
	denote the falling factorial. Given two configurations $x, x' \in \X_\epsilon$, we let 
	\begin{equation}\label{def:gen_Q_eps}
	Q^\epsilon(x, x') \coloneqq \prod_{z \in \supp(x)} \epsilon^{\kappa n_x(z)}Q_{n_x(z)}(n_{x'}(z)),
	\end{equation}
	with $Q^\epsilon(0, x)=1$ for all $x \in \X_\epsilon$, where we use $0\in \X_\epsilon$ to denote the configuration with no particles. In particular, $Q^\epsilon(x, x')\geq 0$ is non-zero if and only if $x\cleq x'$. 
	\begin{dfn}\emph{($V$-function)}\label{def:V_fcn}
		Fix a function $\rho^\epsilon:\Z_\epsilon\to \R$, and any $\epsilon \in (0, 1)$. Define
		\[
		V^\epsilon(x, x';\rho^\epsilon) \coloneqq \sum_{x_0\cleq x} Q^\epsilon(x_0, x')(-1)^{n_x-n_{x_0}}\prod_{z \in x\backslash x_0}\rho^\epsilon(z), \quad x, x' \in \X_\epsilon.
		\]
		If $n_x=n$ for some $n \in \N_0$, we say that the $V$-function is of order $n$. 
	\end{dfn}
	
	\noindent The configuration $x$ plays the role of an index or a test configuration (notice that the sum and product are taken over sub-configurations of $x$). Informally, we think of the $V$-function as a notion of distance between $x'$ and $\rho^\epsilon$. For instance, if $x$ consists of a single particle at $z \in \Z_\epsilon$, then one computes 
	\begin{equation}\label{eq:V_fcn_one_loc}
	V^\epsilon( x, x';\rho^\epsilon) = x'(z)-\rho^\epsilon(z).
	\end{equation}
	If $x$ consists of two particles, one at $z_1\in \Z_\epsilon$, and one at $z_2\in \Z_\epsilon \backslash \{z_1\}$, then 
	\begin{equation}\label{eq:V_fcn_2_pcs_not_diag}
	V^\epsilon( x, x';\rho^\epsilon) = (x(z_1)-\rho^\epsilon(z_1))(x'(z_2)-\rho^\epsilon(z_2)),
	\end{equation}
	which is also independent of $\epsilon$. The diagonal terms are less intuitive: if $x$ has two particles at the same location $z \in\Z_\epsilon$ and no other particles, then 
	\begin{equation}\label{eq:V_fcn_2_pcs_diag}
	V^\epsilon(x, x';\rho^\epsilon) = (x'(z)-\rho^\epsilon(z))^2-\epsilon^\kappa x'(z).  
	\end{equation}
	Note that this time, the coefficients of the $V$-function depend on $\epsilon$. 
	
	In the following definition, we introduce the correlation functions which are the central object of our study, and are adapted from \cite{Boldrig1992}. 
	\begin{dfn}\emph{($v$-function)}\label{def:scaled_v_fcns}
		Fix a function $\rho^\epsilon:\Z_\epsilon\to \R$, and any $\epsilon \in (0, 1)$. Given an initial measure $\nu^\epsilon$ on the configuration space $\X_\epsilon$, we define the $v$-function $v^\epsilon$ by 
		\[
		v^\epsilon_t(x, \rho^\epsilon|\nu^\epsilon)  \coloneqq \E_{\nu^\epsilon}[V^\epsilon(x, X_t^\epsilon;\rho^\epsilon)], \quad x \in \X_\epsilon, \quad t\geq 0,
		\]
		where $\E_{\nu^\epsilon}$ denotes the expectation with $X_0^\epsilon \sim \nu^\epsilon$. If $n_x=n$ for some $n \in \N_0$, we say that the $v$-function is of order $n$. When working with $v^\epsilon_t(x, \rho^\epsilon|\nu^\epsilon)$, we will sometimes use the notation $v_t^\epsilon(x|\nu^\epsilon)$ for simplicity. Additionally, if $X_0^\epsilon \in \X_\epsilon$ is a deterministic starting configuration for the rescaled SLBP, we will write $v^\epsilon_t(x, \rho^\epsilon|X_0^\epsilon)$ for $v^\epsilon_t(x, \rho^\epsilon|\delta_{X_0^\epsilon})$.
	\end{dfn}
	\noindent As for the $V$-function, the argument $x \in \X_\epsilon$ of a $v$-function is understood as a test configuration. If the test configuration $x$ consists of a single particle at location $z \in \Z_\epsilon$, we have 
	\[
	v^\epsilon_t(x, \rho^\epsilon|\nu^\epsilon) = \E_{\nu^\epsilon}[X^\epsilon_t(z)] - \rho^\epsilon. 
	\]
	We claim that, at a given time $t\geq 0$, the $v$-functions measure how close $X^\epsilon_t$ is to a family of independent Poisson random variables indexed by lattice sites $z \in \Z_\epsilon$. This idea is formalised in the following lemma, whose proof can be found in Appendix \ref{append:Poisson_aux_lem}.
	\begin{lem}\label{lem:law_is_prod}
		Fix a function $\rho:\Z_\epsilon\to \R$, and any $\epsilon \in (0, 1)$. We have $v^\epsilon(x, \rho^\epsilon|\nu^\epsilon)=0$ for all test configurations $x \in \X_\epsilon$ if and only if 
		\begin{equation}\label{eq:law_is_prod}
		\operatorname{Law}(X_t^\epsilon) = \bigotimes_{z \in \Z_\epsilon} \operatorname{Law}(\epsilon^\kappa\operatorname{Poi}(\epsilon^{-\kappa} \rho^\epsilon(z))). 
		\end{equation}
	\end{lem}
	
	In Definition \ref{def:scaled_v_fcns}, we choose $\rho^\epsilon$ close to the solution $\rho$ to \eqref{eq:FKPP}. Specifically, at a time $t\geq 0$, we will take $x'=X_t^\epsilon$ the time-$t$ configuration of the rescaled SLBP, and $\rho^\epsilon=\rho_t^\epsilon$, where $\rho^\epsilon:[0, T]\times \Z_\epsilon\to \R$ is an approximation to the solution of the FKPP equation \eqref{eq:FKPP}. Specifically, for each $\epsilon \in (0, 1)$, given an initial function $\rho_0^\epsilon:\Z_\epsilon \to [0, 2\mu_\epsilon]$, we take $\rho^\epsilon$ to be the unique solution to the problem  
	\begin{equation}\label{eq:interm_fkpp_diff_form}
		\begin{cases}
			\partial_t \rho_t^\epsilon(z) =\frac{1}{2}\Delta^\epsilon_z\rho_t^\epsilon(z)+\rho_t^\epsilon(z)(\mu_\epsilon-\rho_t^\epsilon(z)/2), & z \in \Z_\epsilon\quad t> 0,\\
			\lim_{t\searrow 0}\rho_t^\epsilon(z) = \rho_0^\epsilon(z), & z \in \Z_\epsilon,
		\end{cases}
	\end{equation}
	where $\Delta^\epsilon_z$ denotes the discrete Laplacian which acts on functions $f:\Z_\epsilon \to \R$ by
	\[
		\Delta^\epsilon_z f(z) \coloneqq \epsilon^{-2}(f(z+1)+f(z-1)-2f(z)), \quad z \in \Z_\epsilon.
	\]
	In the following lemma, we recall some elementary properties of $\rho^\epsilon$ whose proofs can be found in Appendix \ref{append:disc_FKPP}.
	\begin{lem}\label{lem:FKPP_approx_properties} The function $\rho^\epsilon$ defined by \eqref{eq:interm_fkpp_diff_form} has the following properties.
		\begin{enumerate}
			\item[1.] Suppose that the initial data $\rho_0$ in \eqref{eq:FKPP} is $C^3(\S^1, [0, 2\mu])$. Let $\widetilde{\rho}_t^{\:\epsilon}(z) \coloneqq \rho_t^\epsilon(K_\epsilon z)$ for all $z =\pi(\widetilde{z})\in \S^1$ with $\widetilde{z} \in [0, 1]\cap K_\epsilon^{-1}\Z$ and $t\geq 0$, where $\pi:\R\to \S^1$ denotes the canonical projection. Extend $\widetilde{\rho}^{\:\epsilon}$ to all of $\S^1$ by $C^3$ interpolation. Then, we have $\widetilde{\rho}^{\:\epsilon}\to \rho$ uniformly on $[0, T]\times \S^1$, as $\epsilon \searrow 0$, where $\rho$ is the unique solution to \eqref{eq:FKPP}.
			\item[2.] We have $\max_{z \in \Z_\epsilon}|\rho^\epsilon_t(z)|\leq \max_{z \in \Z_\epsilon}|\rho^\epsilon_0(z)|$, for all $0\leq t\leq T$ and $\epsilon \in (0, 1)$.
		\end{enumerate}
	\end{lem}
	\begin{remark}
		In part one, a linear interpolation is sufficient for this lemma, but the proof of Theorem \ref{thm:CLT_Gaussian} (specifically Lemmas \ref{lem:a_priori_est} and \ref{lem:conv_phi_eps} in Appendix \ref{append:time_dep_test_fcns}) requires us to take derivatives up to order three to control discrete Laplacians.
	\end{remark}
	
	\subsection{Assumptions on the initial data}
	In the following definition, we specify the general form of the initial data of the rescaled SLBP.
	\begin{dfn}\emph{(Initial configuration of the rescaled SLBP)}\label{def:general_init_data}
		Fix $\rho_0 \in C^3(\S^1, [0, 2\mu])$. By considering the embedding $\Z_\epsilon \hookrightarrow \S^1$ given by $z\mapsto K_\epsilon^{-1}z$, where we recall that $K_\epsilon=\lfloor\epsilon^{-1}\rfloor$, we define a function $\rho_0^\epsilon:\Z_\epsilon \to [0, 2\mu]$ by
		\begin{equation}\label{def:rho_0_eps}
		\rho_0^\epsilon(z) \coloneqq \rho_0(K_\epsilon^{-1}z), \quad z \in \Z_\epsilon.
		\end{equation}
		Then, we assume that $X_0^\epsilon$ the time-$0$ configuration of the rescaled SLBP is distributed according to a probability measure $\nu^\epsilon$ on the configuration space $\X_\epsilon$ such that 
		\[
		\E_{\nu^\epsilon}[X_0^\epsilon(z)] = \rho_0^\epsilon(z), \quad z \in \Z_\epsilon,
		\]
		where the notation $\E_{\nu^\epsilon}$ means that $X_0^\epsilon$ is sampled from $\nu^\epsilon$. Hereafter, we denote $\E=\E_{\nu^\epsilon}$. Additionally, we set the initial data $\rho_0^\epsilon$ in the Cauchy problem \eqref{eq:interm_fkpp_diff_form} for the semi-discrete FKPP equation using \eqref{def:rho_0_eps}.
	\end{dfn}
	
	\begin{remark}\label{rmk:start_X_rho}
		\begin{enumerate}
			\item[1.] For instance, one can take $\rho_0^\epsilon\equiv 2\mu_\epsilon$, and at each site $z\in \Z_\epsilon$ an independent and identically distributed random variable ${X_0^\epsilon(z)\sim \epsilon^\kappa\operatorname{Poi}(\epsilon^{-\kappa}2\mu_\epsilon)}$, where $\operatorname{Poi}(\lambda)$ denotes a Poisson random variable with intensity $\lambda>0$. 
			\item[2.] The choice of $\rho_0^\epsilon$ fixes the initial relationship between the average starting configuration of the rescaled SLBP and the initial data of \eqref{eq:interm_fkpp_diff_form}. 
			\item[3.] Recall $\rho_0$ and $\rho_0^\epsilon$ from Definition \ref{def:general_init_data}. Since $\rho_0 \in C^3(\S^1, [0, 2\mu])$, there is a $C^3$ interpolation of the function $\rho_0^\epsilon(K_\epsilon \cdot)$, $z \in \pi([0, 1]\cap K_\epsilon^{-1}\Z)$, such that its derivatives up to order three are bounded uniformly in $z \in \S^1$ and $\epsilon \in (0, 1)$. In particular, this implies that the spatial derivatives of $\widetilde{\rho}^{\:\epsilon}$ are uniformly bounded in space, time, and $\epsilon$.
		\end{enumerate}
	\end{remark}
	
	\noindent We have all the ingredients to state our assumptions on $X_0^\epsilon$ and $\rho_0^\epsilon$ below.
	
	\begin{assump}\label{assump:technical_X_0}
		Consider $X_0^\epsilon$, $\rho_0^\epsilon$ and $\nu^\epsilon$ from Definition \ref{def:general_init_data}, and fix any $n \in \N$. We assume that the following conditions hold. 
		\begin{enumerate}
			\item[a.] There exists $c=c(n)>0$ such that
			\[
			\max_{\substack{x \in \X_\epsilon
			\\n_x=n}}v^\epsilon_0(x, \rho_0^\epsilon|\nu^\epsilon) \leq c\epsilon^{\lfloor (n+1)/2\rfloor(1+\kappa)},
			\]
			for all $\epsilon \in (0, 1)$. Moreover, $c(n')>c(n)$ for all natural numbers $n'>n$. 
			\item[b.] There exist constants $c_1, c_2>0$ depending only on $\sup_{\epsilon\in (0, 1)}\max_{z \in \Z_\epsilon}|\rho_0^\epsilon(z)|$ such that
			\[
			\E\left[|X_0^\epsilon(z_1)-X_0^\epsilon(z_2)|^2\right]\leq c_1\epsilon^{\kappa}+c_2|\rho_0^\epsilon(z_1)-\rho_0^\epsilon(z_2)|^2,
			\]
			for all $z_1, z_2 \in \Z_\epsilon$ and $\epsilon \in (0, 1)$.
		\end{enumerate}
	\end{assump}
	
	We make the following additional assumptions on the initial data $\widetilde{\rho}_0^{\:\epsilon}$ and $Y_0^\epsilon$ in our central limit theorem. 
	
		\begin{remark}
		We are primarily interested in initial data $X_0^\epsilon$ whose law is the product measure
		\begin{equation}\label{eq:prod_init_law}
		\operatorname{Law}(X_0^\epsilon) = \bigotimes_{z \in \Z_\epsilon} \operatorname{Law}(\epsilon^\kappa \operatorname{Poi}(\epsilon^{-\kappa}\rho_0^\epsilon(z))),
		\end{equation}
		where $\rho_0^\epsilon$ was introduced in Definition \ref{def:general_init_data}. With this assumption, one can check that if $\rho_0^\epsilon\equiv 2\mu_\epsilon$ in Definition \ref{def:general_init_data}, then Assumption~\ref{assump:technical_X_0} holds. As $2\mu =\lim_{\epsilon \searrow 0} 2\mu_\epsilon$ is a stable steady state of the FKPP equation \eqref{eq:FKPP}, this corresponds to studying the equilibrium fluctuations of the rescaled SLBP in Theorem \ref{thm:CLT_Gaussian}. For the non-equilibrium fluctuations, we may choose $\rho_0 \in C^3(\S^1, [0, 2\mu])$ the initial data in \eqref{eq:FKPP} arbitrarily, and choose $\nu^\epsilon$ and $\rho_0^\epsilon$ as described in Definition \ref{def:general_init_data}. Similarly, one easily checks that Assumption \ref{assump:technical_X_0} holds for the product initial law \eqref{eq:prod_init_law}. Assumption \ref{assump:technical_X_0} only allows for vanishing correlations between sites as $\epsilon \searrow 0$. 
	\end{remark}
	
	\section{Overview of the proofs and convergence of the $v$-functions}\label{sec:pf_overview}
	
	In this section, we outline the proof of Theorem \ref{thm:CLT_Gaussian} following the method of compactness-uniqueness (see \cite[Theorem 13.1]{B1999}). We show that the family $\{\rY^\epsilon: \epsilon \in (0, 1)\}$ is tight in $\cD([0, T], C^\infty(\S^1)')$, and that the finite-dimensional distributions of $\rY^\epsilon$ converge as $\epsilon \searrow 0$ to those of $\rY$ the solution to the stochastic partial differential equation \eqref{eq:OU_Br}. The main challenge is the following observation: the limiting fluctuations are characterised by the \textit{linear} equation \eqref{eq:OU_Br}, but \eqref{eq:FKPP} the hydrodynamic equation is \textit{non-linear}. Much of the work consists in justifying rigorously this linearisation. In particular, we must control the non-linear terms of the form $X^\epsilon(z)^2$, $z \in \Z_\epsilon$, in the proof of tightness of $\{\rY^\epsilon:\epsilon \in (0, 1)\}$, and we must show that they can be linearized in a suitable sense when computing the limit of the finite-dimensional distributions of $\rY^\epsilon$ as $\epsilon \searrow 0$. Both of these obstacles are resolved by proving the following \textit{Boltzmann-Gibbs principle}, and deriving some of its consequences. 
	
	Consider for each $\epsilon \in (0, 1)$ a generalised test function 
	\[
	\phi^\epsilon :\{(r, t) \in [0, T]^2: 0\leq r\leq t\}\times \S^1\to \R, \quad ((r, t), z)\mapsto \phi^\epsilon_{r, t}(z),
	\]
	which is assumed to be continuous in $r$ for each fixed $t$ and $z$, and such that $\phi_{r, t}^\epsilon\in C^\infty(\S^1)$ for all $r$ and $t$ with
	\[
	C_{\phi}=C_\phi(T)\coloneqq\sup_{\epsilon \in (0, 1)}\sup_{0\leq r\leq t\leq T} \|\phi_{r, t}^\epsilon\|_\infty <\infty,
	\]
	where $\|\cdot\|_{\infty}$ denotes the supremum norm on $\S^1$. Let $T>0$, $\gamma \in [-1/2, -5/16)$, and $\kappa \in [0, 1+2\gamma]$, and define the non-linear fluctuations field
	\[
	F^\epsilon_r(\phi_{r, t}^\epsilon) \coloneqq \frac{\epsilon^{1+\gamma-\kappa}}{2}\sum_{z \in \Z_\epsilon}\left(X_r^\epsilon(z)(X_r^\epsilon(z)-\epsilon^\kappa)-\E[X_r^\epsilon(z)(X_r^\epsilon(z)-\epsilon^\kappa)]\right)\phi^\epsilon_{r, t}(\epsilon z),
	\]
	for all $t \geq r$.
	\begin{prop}\emph{(Boltzmann-Gibbs principle)}\label{prop:BGP}
		We have 
		\begin{equation}\label{eq:BGP_statement}
			\limsup_{S\to \infty}\limsup_{\epsilon \searrow 0}\sup_{0\leq s\leq t\leq T} \E\left[\Bigg(\frac{1}{\epsilon^2S}\int_s^{s+\epsilon^2S}[F_r^\epsilon(\phi^\epsilon_{r, t})-Y_r^\epsilon(\widetilde{\rho}^{\:\epsilon}_r \phi^\epsilon_{r, t})]dr\Bigg)^2\right]=0.
		\end{equation}
	\end{prop}
	
	\noindent We briefly give some intuition for this result. By Definition \ref{def:rescaled_slbp}, at a given location $z \in \Z_\epsilon$ and at time $t-$, births occur at rate $\epsilon^{-\kappa}X_{t-}^\epsilon(z)$, a particle dies due to competition at rate $\epsilon^{-\kappa}X_{t-}^\epsilon(z)(X_{t-}^\epsilon(z)-\epsilon^\kappa)/2$, and a particle migrates to an adjacent site at rate $\epsilon^{-2-\kappa}X_{t-}^\epsilon(z)$. The time integral in \eqref{eq:BGP_statement} is taken on the slow time scale $S\mapsto \epsilon^2 S$. On that time scale, the rates of birth, competition, and diffusion become $\epsilon^{2-\kappa}X_{t-}^\epsilon(z)$, $\epsilon^{2-\kappa}X_{t-}^\epsilon(z)(X_{t-}^\epsilon(z)-\epsilon^\kappa)/2$, and $\epsilon^{-\kappa}X_{t-}^\epsilon(z)$ respectively. Since $\kappa \in [0, 3/8]$ by assumption in Theorem \ref{thm:CLT_Gaussian}, and $X_{t-}^\epsilon(z)$ is of order one by Theorem \ref{thm:LLN}, the birth-death dynamics are essentially frozen, while we still see significant diffusion of particles. Diffusion conserves the total number of particles and helps the system reach its local equilibrium. The field $F^\epsilon$ captures the fluctuations of the process $(\rX^\epsilon)^2$, while $\rY^\epsilon$ is the fluctuations field of $\rX^\epsilon$. One easily checks that diffusion, which is the only remaining dynamics on the slow time scale, does not preserve particle numbers in $F^\epsilon$. Thus the result says that we can project the fluctuations of the non-conserved quantity $(\rX^\epsilon)^2$ onto those of the conserved quantity $\rX^\epsilon$. It is helpful to recall that $\rX^\epsilon$ and $\widetilde{\rho}^{\:\epsilon}$ both converge to the FKPP equation by Theorem \ref{thm:LLN} and Lemma \ref{lem:FKPP_approx_properties}. This line of reasoning in particular describes the fluctuations of the system when $\kappa=0$ (competition is not weak) and the branching is binary, which was studied in \cite{Boldrig1992}.
		
	 The main ingredient in the proof of Proposition \ref{prop:BGP} is a quantitative convergence result for the $v$-functions (recall Definition \ref{def:scaled_v_fcns}), which is of independent interest as it shows in particular the contributions of the tail of the offspring distribution and of the competition exponent $\kappa$ to the rate of convergence of the \textit{expected} population size of the rescaled SLBP to its hydrodynamic limit $\rho$.
	
	\begin{thm}\emph{(Quantitative convergence of the $v$-functions)}\label{thm:QLLN}
		Suppose that Assumption \ref{assump:technical_X_0} holds, that $\kappa \in [0, 3/8)$, and that the offspring distribution $\rP$ has a finite moment of order $p$, for some $p \in [1, 2]$. If $p\in [1, 2]$ and $n>2$, we have
		\[
		\sup_{t \in [0, T]} \max_{\substack{x \in \X_\epsilon
				\\n_x=n}} |v^\epsilon_t(x|\nu^\epsilon)| = o(\epsilon^{1+\kappa}), \quad \textit{as $\epsilon\searrow 0$}.
		\]
		If $p\in [1, 2)$ and $n \in \{1, 2\}$, we have
		\[
		\sup_{t \in [0, T]} \max_{\substack{x \in \X_\epsilon
				\\n_x=n}} |v^\epsilon_t(x|\nu^\epsilon)| = o(\epsilon^{1+(p-1)\kappa}), \quad \textit{as $\epsilon\searrow 0$}.
		\]
		If $p=2$ and $n \in \{1, 2\}$, there exists $c=c(T, p, n)>0$ such that 
		\[
		\sup_{t \in [0, T]} \max_{\substack{x \in \X_\epsilon
				\\n_x=n}} |v^\epsilon_t(x|\nu^\epsilon)| \leq c\epsilon^{1+\kappa}, \quad \textit{for all $\epsilon\in (0, 1)$}.
		\] 
	\end{thm}
	
	\noindent In Assumption \ref{assump:technical_X_0}.a, we assume that the initial law of the rescaled SLBP is close to a product law over the sites in $\Z_\epsilon$. Then by Lemma \ref{lem:law_is_prod}, we can more generally interpret Theorem \ref{thm:QLLN} as a result on the quantitative propagation of chaos for our population model. The $v$-functions vanish sufficiently fast to imply the Boltzmann-Gibbs principle claimed in Proposition \ref{prop:BGP}. Additionally, Theorem \ref{thm:QLLN} allows us to show that the finite-dimensional distributions of the rescaled SLBP $\rX^\epsilon$ converge to those of $\rho$ the solution to \eqref{eq:FKPP}. Then, we show using suitable martingale problems and Green's function estimates that $\{X^\epsilon:\epsilon \in (0, 1)\}$ is tight in $\cD([0, T], C^\infty(\S^1)')$ to obtain the weak law of large numbers, Theorem \ref{thm:LLN}.
	
	The proofs of Proposition \ref{prop:BGP} and Theorem \ref{thm:QLLN} follow closely those of Proposition~3.1 and Lemma 1 of \cite{Boldrig1992}, respectively. We have adapted the authors' method to our setting, which differs in the following two ways. The first difference concerns the rate of competition. In the work of Boldrighini et al., the competition rate does not depend on $\epsilon$, which means in particular that the local number of particles does not tend to infinity when $\epsilon\searrow 0$. In our weak competition regime, the coalescence rate is of order $\epsilon^\kappa$, and the local number of particles tends to infinity when $\epsilon\searrow 0$ like $\epsilon^{-\kappa}$. Thus, the law of large numbers claimed in Theorem 1 of \cite{Boldrig1992} is only a consequence of the rescaling of the spatial motion, while our law of large numbers (Theorem \ref{thm:LLN}) results from the joint rescaling of space and of the local population size. This observation is the reason that our $v$-functions in Definition \ref{def:scaled_v_fcns} are rescaled in terms of $\epsilon^\kappa$. One recovers the $v$-functions of Boldrighini et al. if the competition exponent is $\kappa=0$.
	
	The second substantial change has to do with the offspring distribution. Boldrighini et al. consider binary branching only, while in the context of the law of large numbers, we allow any offspring distribution with finite mean. To follow the method of proof of Lemma 1 from \cite{Boldrig1992} then requires us to suitably truncate the offspring distribution in terms of $\epsilon$ so as to control its truncated higher order moments, and to adapt arguments which use the symmetry present in Boldrighini et al. where birth and death events respectively add or remove a single particle locally, to our asymmetric local birth-competition process where birth events add any number of offspring while competition still only removes one particle per event. In particular, we adapt a key component of the proof - the fact that the $v$-functions satisfy a hierarchy of ordinary differential equations with respect to time with dynamics that are dual in some sense to the dynamics of the original particle system. See Proposition \ref{prop:v_fcn_eq} below. In \cite{Boldrig1992}, this hierarchy relates the $v$-function of order $n$ to those of orders $k\in \{n-2, n-1, n, n+1\}$ (assuming linear birth rate and quadratic competition rate as in the present paper), while in our case, the hierarchy relates an $n$-th order $v$-function to the $v$-functions of orders $k \in \{1, 2, \cdots, n, n+1\}$. This change reflects the fact that in the hierarchy for general offspring distributions, during a single birth-competition event, at most one particle is added to the system locally, while the entire local population can be whipped out. As a consequence, additional care is required when controlling transition probabilities in Lemma \ref{lem:v_fcn_bd} below. Let us state the hierarchy of equations for the $v$-functions.
	\begin{prop}\label{prop:v_fcn_eq}
		Fix $\epsilon \in (0, 1)$, and recall that $\rho^\epsilon$ denotes the solution to the semi-discrete FKPP equation \eqref{eq:interm_fkpp_diff_form}. Let $\nu^\epsilon$ be a probability measure on $\X_\epsilon$ such that for each $n \in \N$ and all $t \geq 0$, we have
		\[
		\max_{\substack{x \in \X_\epsilon \\ n_x=n}}\E_{\nu^\epsilon}[Q^\epsilon(x, X^\epsilon_t)]<\infty.
		\]
		Then
		\begin{equation}
			\begin{aligned}\label{eq:v_fcn_eq}
			v_t^\epsilon(x, \rho_t^\epsilon|\nu^\epsilon) &= \sum_{\substack{x_1\in \X_\epsilon\\ n_{x_1}=n_x}}\G^\epsilon_t(x, x_1)v_0^\epsilon(x_1, \rho_0^\epsilon | \nu^\epsilon)
			\\&+ \int_0^t ds \sum_{\substack{x_1\in \X_\epsilon\\ n_{x_1}=n_x}} \G^\epsilon_{t-s}(x, x_1)\sum_{z \in \supp(x_1)}\sum_{h=-n_{x_1}(z)}^1 c_h^\epsilon(n_{x_1}(z), \rho_s^\epsilon(z))v_s^\epsilon(x_1^{(z, h)}, \rho_s^\epsilon|\nu^\epsilon),
			\end{aligned}
		\end{equation}
		for all $x \in \X_\epsilon$ and $t\in [0, T]$, for some coefficients $c^\epsilon_h(m, u)$, and where $x_1^{(z, h)}$ denotes the configuration $x_1$ with $h$ particles added at $z$ if $h\geq 0$, and $-h$ particles removed from $z$ if $h<0$. Moreover, there exists a constant $c=c(n)>0$ independent of $\epsilon$ such that
		\begin{equation}\label{ineq:coef_bd}
		|c^\epsilon_h(m, u)|\leq c \max(1, |u|^{m+1}), \quad m \in \{1, \cdots, n\}, \quad u \in \R, \quad h \in \{-m, \cdots, 1\}, \quad \epsilon\in (0, 1).
		\end{equation}
		In the particular case where $m=2$ and $h=-2$ and the offspring distribution $\rP$ has a moment of order $p\in[1, 2]$, there also exists $c>0$ such that 
		\begin{equation}\label{ineq:coeff_bd_2_pcs_removed}
		|c_{-2}^\epsilon(2, u)| \leq c \epsilon^\kappa\max(1, |u|^3)+o(\epsilon^{(p-1)\kappa})|u|,
		\end{equation}
		as $\epsilon \searrow 0$. For $m=2$ and $h=-1$, we have
		\begin{equation}\label{ineq:coeff_bd_2_pcs_1_removed}
		c_{-1}^\epsilon(2, u)=(2\mu_\epsilon+1)\epsilon^\kappa+o(\epsilon^{(p-1)\kappa}),
		\end{equation}
		as $\epsilon \searrow 0$. Finally, if $m=1$ and $h\leq -1$, we have $c_h^\epsilon(1, u)=0$.
	\end{prop}
	\noindent For low order $v$-functions, the coefficients in the hierarchy are small in $\epsilon^\kappa$. This is the main reason why the convergence rates in Theorem \ref{thm:QLLN} have factors involving $\epsilon^\kappa$, and it provides the necessary control for the proof of the Boltzmann-Gibbs principle of Proposition \ref{prop:BGP}. 
	
	\section{Motivation and related works}\label{sec:motivation_related_work}
	
	In this section, we discuss our motivation for the study of the SLBP (Definition~\ref{def:slbp}) and for the weak competition rescaling (Definition~\ref{def:rescaled_slbp}). We mention some closely-related models from the literature. 
	
	\cite{FM2004} study a spatial population process in continuous time and continuous space with the application to the modeling of plant populations in mind. In their model, the plants have an opportunity to move (disperse) at the time of their birth, and they subsequently are motionless. They die at some rate due to intrinsic reasons or due to competition with neighbouring plants. At any given point in time, the system is characterized by the locations of the plants in the closure of an open and connected subset $\chi$ of $\R^d$, for some $d \geq 1$. More specifically, the dynamics are as follows. Each plant at location $x \in \bar \chi$ at time $t\geq 0$
	\begin{itemize}
		\item[1.] dies at rate $\mu(x) \in [0, \infty)$,
		\item[2.] produces a seed at rate $\gamma(x)\in [0, \infty)$, which appears instantaneously as a mature plant at a location $x+z$, where $z$ sampled from a dispersion measure $D(x, dz)$ supported on $\bar \chi-x$, and
		\item[3.] dies due to non-local competition at rate $\alpha(x)\sum_{i=1}^{I(t)}U(x, X^i_t) \in [0, \infty)$, where ${U(x, y)=U(y, x)}$ is a competition kernel, $I(t)$ is the total number of plants alive in the system at time $t$ and $(X^i_t)_{1 \leq i \leq I(t)}$ are the positions of all the plants.
	\end{itemize}
	The authors track the current state of the population via its empirical measure 
	\[
	\nu_t = \sum_{i=1}^{I(t)}\delta_{X^i_t}(dx), \quad t>0.
	\]
	They call this process the \textit{Bolker–Pacala–Dieckmann–Law (BPDL) process}, in reference to the ecology literature where it originated (\cite{LawDieckmann1998} and \cite{BP1999}). In Theorem 5.3, they establish the following hydrodynamic limit. For each $n \in \N$, let $(\nu_t^n)_{t\geq 0}$ denote the BPDL process with the same parameters as above, except for the competition strength $\alpha(x)$ which is scaled as follows: $\alpha_n(x)\coloneqq \alpha(x)/n$. Define the process $\xi_t^n=\frac{1}{n} \nu_t^n$, $t>0$. The scaling of the competition strength is analogous to the weak competition regime that we study in this paper for the SLBP. In \cite[Section 7]{BanMel2015}, the authors suggest the following biological interpretation of this type of rescaling. They interpret the parameter $n$ (or $\epsilon^{-\kappa}$ for the SLBP) as the amount of resources available in the system. If a large amount of resources is available, there is less need for individuals in the population to compete for resources. Additionally, the partition of a fixed amount of resources among individuals in the population may force the biomass of each individual to be inversely proportional to the amount of resources available. Fournier and M\'el\'eard show that $(\xi^n_t)_{t\geq 0}$ converges in distribution in a suitable path space as $n\to \infty$ to the solution to the Cauchy problem for the integro-differential equation 
	\begin{equation}\label{FM2004hydroLim}
			\partial_t \xi_t(x)=\int_{\bar \chi} dy \xi_t(y) \gamma(y) D(y, x-y) -\xi_t(x)\left( \mu(x)+\alpha(x)\int_{\bar \chi} dy \xi_t(y) U(x, y)\right),
	\end{equation}
	where $x \in \bar \chi$ and $t>0$. The BPDL defers from the SLBP in two major ways. First, since the competition is modelled using an interaction kernel summed over the entire population, it is at least from the mathematical point of view \textit{non-local}. This is also reflected in the hydrodynamic limit \eqref{FM2004hydroLim}. Secondly, by \cite[Proposition 3.2]{FM2004}, if at the time zero, there is at most one plant at each location, then this holds true for all future times. By contrast, the interactions in the SLBP are completely local, and the population size on any given site is unbounded. Both of these features are of interest for the modelling of locally-regulated populations where one can reasonably assume at the microscopic level that interactions occur within large colonies of individuals at the same location and individuals can migrate to adjacent colonies. The methods of proof used in \cite{FM2004} do not directly apply to our particle systems as they are suitable for non-local competition dynamics.
	
	In \cite{Boldrig1992}, the authors prove limit theorems for the SLBP with fixed amount of resources ($\kappa=0$), and with binary branching. More precisely, they consider the SLBP $\rN=((N_t(z))_{z \in \Z_K}:t\geq 0)$ in Definition \ref{def:slbp} (without the $1/2$ in the rate of competition). They rescale the spatial motion diffusively by working on the lattice $\Z_\epsilon\coloneqq \Z_{K_\epsilon}$ where $K_\epsilon=\lfloor\epsilon^{-2}\rfloor$, by accelerating the rate of migration by $\epsilon^{-2}$, and by letting the resulting process $\rN^\epsilon=((N^\epsilon_t(z))_{z \in \Z_\epsilon}:t\geq 0)$ act on Schwartz test functions $\phi \in \cS(\R)$ via
	\[
	X_t^\epsilon(\phi)\coloneqq \epsilon \sum_{z \in \Z_\epsilon}N_t^\epsilon(z)\phi(\epsilon z), \quad t>0. 
	\]
	They use a growing discrete circle to obtain equations on all of $\R$ as opposed to $\S^1$ in our case. We find that the basic Green function estimates (Appendix \ref{app:Green_fcn_props}) used to obtain tightness of $\{X^\epsilon:\epsilon \in (0, 1)\}$ in $\cD([0, T], C^\infty(\S^1)')$ for some fixed time horizon $T>0$ do not allow one to work on the growing circle, hence our choice of $K_\epsilon = \lfloor\epsilon^{-1}\rfloor$. Boldrighini et al. sketch a proof that $X^\epsilon$ converges in distribution in $\cD([0, T], \cS'(\R))$ as $\epsilon \searrow 0$ to the unique solution to the FKPP equation
	\begin{equation}\label{eq:Boldrighini_FKPP}
		\begin{cases}
			\partial_t \rho(t, z) = \frac{1}{2}\partial_z^2 \rho(t, z)+\rho(t, z)(1-\rho(t, z)), & t>0, \quad z \in \R,\\
			\lim_{t\searrow 0}\rho(t,z) = \rho_0(z),&\quad z \in \R.
		\end{cases}
	\end{equation}
	They then show that the non-equilibrium fluctuations of $N^\epsilon$, understood as 
	\[
	Y^\epsilon_t(\phi)\coloneqq \epsilon^{1/2}\sum_{z \in \Z_\epsilon} \left(N_t^\epsilon(z)-\E\left[N_t^\epsilon(z)\right]\right)\phi(\epsilon z), \quad \phi \in \cS(\R), \quad t\geq 0, 
	\]
	converge in distribution in $\cD([0, T], \cS(\R)')$ as $\epsilon \searrow 0$ to the solution to the Cauchy problem 
	\begin{equation}\label{eq:Boldrighini_OU}
		\begin{cases}
			\partial_t Y =\Big(\frac{1}{2}\partial_z^2+1-2\rho\Big)Y +\sqrt{\rho}\partial_z\dot W^1+\sqrt{\rho+\rho^2}\dot W^2&\textnormal{on }[0, T]\times \R
			\\Y_t\overset{d}{\to} Y_0 \textnormal{ as }t\searrow 0,
		\end{cases},
	\end{equation}
	where $\dot W^1$ and $\dot W^2$ are independent Gaussian space-time white noises on $[0, T]\times \R$, and $Y_0$ is a $\cS(\R)'$-valued centred Gaussian variable with variance
	\[
	\E[Y_0(\phi)^2] = \int_{\R} \phi(z)^2 \rho_0(z)dz, \quad \phi \in \cS(\R). 
	\]
	The proof is based on a Boltzmann-Gibbs principle analogous to \eqref{eq:BGP_statement} which results from the study of the asymptotics of some $v$-functions. The $v$-functions are defined, for all $x \in \N_0^{\Z_\epsilon}$ and $t\geq 0$, by 
	\[
	v^\epsilon_t(x|\nu^\epsilon)  \coloneqq \E_{\nu^\epsilon}\left[\sum_{x_0\cleq x} \prod_{z \in \supp(x)} Q_{x(z)}(N^\epsilon_t(z))(-1)^{n_x-n_{x_0}}\prod_{z \in x\backslash x_0}\rho_t^\epsilon(z)\right],
	\]
	where $Q_\cdot(\cdot)$ is the falling factorial in \eqref{def:ff},  $\rho^\epsilon$ approximates the solution $\rho$ to \eqref{eq:Boldrighini_FKPP}, and $\nu^\epsilon$ is a suitable initial law (this formula can be recovered from Definition \ref{def:scaled_v_fcns} by taking $\kappa=0$). The authors prove that for all $n>2$, we have
	\begin{equation}\label{eq:v_conv_boldrig}
	\sup_{t \in [0, T]} \max_{\substack{x \in \N_0^{\Z_\epsilon}
			\\n_x=n}} |v^\epsilon_t(x|\nu^\epsilon)| = o(\epsilon), \quad \textit{as $\epsilon\searrow 0$},
	\end{equation}
	and there exists $c=c(T )>0$ such that 
	\begin{equation}\label{eq:v_conv_boldrig2}
	\sup_{t \in [0, T]} \max_{\substack{x \in \N_0^{\Z_\epsilon}
			\\n_x=n}} |v^\epsilon_t(x|\nu^\epsilon)| \leq c\epsilon, \quad \textit{for all $\epsilon\in (0, 1)$},
	\end{equation}
	for $n \in \{1, 2\}$. As we discussed for Theorem \ref{thm:QLLN} in the previous section, this result can be interpreted as the quantitative propagation of chaos for the SLBP with $\kappa=0$. Interestingly, we observe by comparing the hydrodynamic limit \eqref{eq:FKPP} with \eqref{eq:Boldrighini_FKPP}, and the limiting Ornstein-Ulhenbeck process \eqref{eq:OU} with \eqref{eq:Boldrighini_OU} that the parameter $\kappa$ does not play a role in the macroscopic dynamics. The difference between $\kappa=0$ and $\kappa>0$ is only apparent in the time-$t$ measurement of propagation of chaos, that is the convergence rate of the $v$-functions. Indeed, comparing Theorem \ref{thm:QLLN} with \eqref{eq:v_conv_boldrig} and \eqref{eq:v_conv_boldrig2}, we see that assuming $\kappa>0$ (weak competition) and an offspring distribution with a moment $p>1$, the local densities of the SLBP at two distinct sites are guaranteed to become independent faster as $\epsilon \searrow 0$, than if $\kappa=0$. To tie this back to the interpretation of \cite{BanMel2015} recalled above, in a system with more resources and more individuals with lower biomass, the behaviour of individuals at different sites decouples faster. On the qualitative side, we observe that when $\kappa=0$, the limit theorems of \cite{Boldrig1992} are only a consequence of the diffusive rescaling of the random walks. This differs from our results for $\kappa>0$, where birth-competition dynamics are also rescaled. Finally, we note that \cite{Boldrig1992} allow more general polynomial rates of birth and death, as long as the degree of the death rate is strictly larger than that of the birth rate.
	
	\cite[pages 2-4]{DMS2003} briefly discuss a mean-field approximation to the rescaled SLBP. Consider a birth and death process where each particle splits into two particles independently at rate $1$, and pairs of particles coalesce at rate $\chi$. Let $N(t)$ denote the number of particles in the system. Equating the rates of birth and death and solving for $N$, we see that $\overline{N}=1+2\chi^{-1}$ is a natural equilibrium population size, which we interpret as the carrying capacity of the system. Then, to first order, we have that $U(t)=N(t)/\overline{N}$ satisfies the logistic equation $U'(t)=U(t)(1-U(t))$. See for instance section 3.1 of \cite{BanMel2015} for rigorous arguments. To take a scaling limit of large population size, Doering et al, consider the weak competition regime in which $\chi \searrow 0$, in our notation $\chi=\epsilon^\kappa$ and $\epsilon \searrow0$. For this process, \cite{DMS2003} write that they expect the magnitude of the fluctuations to be of order $\sqrt{U+aU^2}$ for some $a>0$. This is confirmed for the spatial process itself as a particular case of our main result, Theorem \ref{thm:CLT_Gaussian}. More generally, our result identifies the constants in terms of the first and second moments of the offspring distribution $\rP$. 
	
	Before closing this section, let use cite other relevant works on scaling limits of regulated populations, and on Boltzmann-Gibbs principles. Local-regulation in the SLBP is endogenous. This contrasts with models such as the $N$-branching Brownian motions ( for instance \cite{M1975}, \cite{BDMM2006}, \cite{BDMM2007}, \cite{BZ2018}, and \cite{BBNP2020}) and the stepping stone models (e.g. \cite{SU1986}, \cite{R1986}, \cite{S1988}, and \cite{Etheridge2004}), in which the global, respectively local, population size is fixed exogenously and constant over time. \cite{MT1995}, \cite{FP2017}, \cite{EVY2014} and \cite{BGSRS2020} have obtained similar scaling limits for various interacting particle systems. \cite{FM2004} also studied convergence of their population model to a superprocess version of a model of \cite{BP1999} introduced by \cite{Etheridge2004}, and which is a spatial analogue of the logistic Feller diffusion. In works on fluctuations of particle systems, e.g. \cite{Boldrig1992}, \cite{ChenFan2016}, \cite{FP2017}, \cite{L2016} and \cite{BGSRS2020}, the reaction term in the stochastic partial differential equation which characterises the limiting fluctuations corresponds to a linearization of the reaction term in the partial differential equation describing the hydrodynamic limit. We also observe this phenomenon in our work: see Theorems \ref{thm:LLN} and \ref{thm:CLT_Gaussian}. To prove this relationship rigorously, one approach consists in showing a Boltzmann-Gibbs principle such as Proposition \ref{prop:BGP}. This principle was originally proved in \cite{BR86} to study the equilibrium fluctuations of the zero range process. Broadly speaking, it states that the fluctuations of local functions of conserved quantities can be suitably projected onto the fluctuations of the conserved quantities. The Boltzmann-Gibbs principle continues to be used in the study of equilibrium fluctuations of particle systems; see e.g. \cite{L2016} and \cite{HJV2017} for recent examples. It was also shown to hold in some non-equilibrium settings: see for instance \cite{Boldrig1992}, \cite{ChenFan2016}, and the present work. In these examples, the intuition for the validity of the principle is similar to our heuristics below Proposition \ref{prop:BGP} in the previous section.
	
	\section{Proof of the law of large numbers}\label{sec:LLN_pf}
		
	In this section, we prove the law of large numbers stated in Theorem \ref{thm:LLN}, assuming the convergence of the $v$-functions claimed in Theorem \ref{thm:QLLN} and proved in Section \ref{sec:QLLN_proof} below. Theorem~\ref{thm:LLN} states that $\rX^\epsilon\overset{d}{\to} \rho$ in $\cD([0, T], C^\infty(\S^1)')$ as $\epsilon \searrow 0$, where $\rho$ solves \eqref{eq:FKPP}:
	\[
	\begin{cases}
		\partial_t \rho(t, z) = \frac{1}{2}\partial_z^2 \rho(t, z)+\rho(t, z)(\mu-\rho(t, z)/2), & t>0, \quad z \in \S^1,\\
		\lim_{t\searrow 0}\rho(t,z) = \rho_0(z),&\quad z \in \S^1.
	\end{cases}
	\]

	By \cite[Theorem 13.1]{B1999} - this result follows from the tightness of $\{\rX^\epsilon: \epsilon \in (0, 1)\}$ in $\cD([0, T], C^\infty(\S^1)')$ and from the convergence $\rX^\epsilon \to \rho$ in finite-dimensional distributions as $\epsilon \searrow 0$. We show the convergence of the finite-dimensional distributions in Proposition \ref{lem:FDD_LLN} below using the asymptotics established in Theorem \ref{thm:QLLN} for the $v$-functions. By Mitoma's Theorem (see Theorem 4.1 in \cite{Mitoma1983}), $\{\rX^\epsilon: \epsilon \in (0, 1)\}$ is tight in the Skorohod J1-topology on $\cD([0, T], C^\infty(\S^1)')$ if, for every $\phi \in C^\infty(\S^1)$, the family of processes $\{\rX^\epsilon(\phi):\epsilon \in (0, 1)\}$ is tight in the Skorohod J1-topology on $\cD([0, T], \R)$. By Aldous' Criterion \cite[Theorem 16.10 and equation (16.32)]{B1999}, the later condition holds if the following two points are true.
	\begin{enumerate}
		\item[1.] For each $t \in [0, T]$, $\{X_t^\epsilon(\phi):\epsilon \in (0, 1)\}$ is tight in $\R$, i.e. for any $\eta \in (0, 1)$, there exists $K=K(t, \eta)>0$ such that 
		\begin{equation}\label{cond:aldous_1}
		\sup_{\epsilon \in (0, 1)}\P(|X_t^\epsilon(\phi)|>K)<\eta.
		\end{equation}
		\item[2.] Given a family $\{\tau_\epsilon: \epsilon\in (0, 1)\}$ of $[0, T]$-valued stopping times and any family $\{\theta_\epsilon: \epsilon\in (0, 1)\}\subset [0, \infty)$ with $\lim_{\epsilon \searrow 0}\theta_\epsilon=0$, we have  
		\begin{equation}\label{cond:aldous_2}
		\limsup_{\epsilon \searrow 0}\E\left[|X_{\tau_\epsilon+\theta_\epsilon}^\epsilon(\phi)-X_{\tau_\epsilon}^\epsilon(\phi)|\right]=0.
		\end{equation}
	\end{enumerate}
	The first condition \eqref{cond:aldous_1} follows immediately from the convergence of the finite-dimensional distributions of $\rX^\epsilon$. The proof of \eqref{cond:aldous_2} utilises martingale problems, Assumptions \ref{assump:technical_X_0}, Green's function estimates shown in Lemma \ref{lem:Green_estimates} in Appendix~\ref{app:Green_fcn_props}, and Lemma \ref{lem:FKPP_approx_properties} to obtain uniform equicontinuity estimates for the expectation in \eqref{cond:aldous_2}. We note that since our limit \eqref{def:fluct_field_def}-\eqref{eq:FKPP} is deterministic, the full convergence will also hold in probability.
	
	\subsection{Convergence of finite-dimensional distributions}
	We establish the following proposition.
	\begin{prop}\label{lem:FDD_LLN}
		Let  $\rho$ denote the solution to the FKPP equation \eqref{eq:FKPP} and define
		\[
		\rho_t(\phi) \coloneqq \int_{\S^1}\rho_t(z)\phi(z)dz,\quad \phi \in C^\infty(\S^1), \quad t\in [0, T].
		\]
		Then, for any $k \in \N$, $t_1, \cdots, t_k \in [0, T]$ and $\phi\in C^\infty(\S^1)$, we have 
		\[
		(X_{t_1}^\epsilon(\phi), \cdots, X_{t_k}^\epsilon(\phi))\overset{d}{\to} (\rho_{t_1}(\phi), \cdots, \rho_{t_k}(\phi)),
		\]
		as $\epsilon \searrow 0$. 
	\end{prop}
	\begin{proof}
	Define
	\[
	\rho_t^\epsilon(\phi)\coloneqq \epsilon \sum_{z \in \Z_\epsilon}\rho_t^\epsilon(z)\phi(\epsilon z), \quad \phi \in C^\infty(\S^1), \quad t \in [0, T],
	\]
	and let $f:\R\to \R$ be a bounded Lipschitz function. By a Taylor expansion of order one, we have
	\[
	\E[f(X_t^\epsilon(\phi))] = f(\rho_t^\epsilon(\phi)) + \E[(X_t^\epsilon(\phi)-\rho_t^\epsilon(\phi))f'(\rho_t^\epsilon(\phi)+\xi)],
	\]
	for some $\xi \in (0, X_t^\epsilon(\phi)-\rho_t^\epsilon(\phi))$. We aim to control the second term using Theorem \ref{thm:QLLN}. By the Cauchy-Schwarz inequality and the fact that $\|f'\|_\infty<\infty$, we obtain 
	\begin{equation}\label{ineq:start_FDD_proof}
		\E[(X_t^\epsilon(\phi)-\rho_t^\epsilon(\phi))f'(\rho_t^\epsilon(\phi)+\xi)]\leq \|f'\|_\infty \E[(X_t^\epsilon(\phi)-\rho_t^\epsilon(\phi))^2]^{1/2}
	\end{equation}
	To use the convergence of the $v$-functions, we need to rewrite the integrand in terms of pointwise differences of $X^\epsilon_t$ and $\rho_t^\epsilon$. By definition, we have 
	\[
	(X_t^\epsilon(\phi)-\rho_t^\epsilon(\phi))^2=\left(\epsilon \sum_{z \in \Z_\epsilon}(X_t^\epsilon(z)-\rho_t^\epsilon(z))\phi(\epsilon z)\right)^2.
	\]
	By the Jensen and Cauchy-Schwarz inequalities,
	\[
	(X_t^\epsilon(\phi)-\rho_t^\epsilon(\phi))^2\leq \left(\epsilon \sum_{z \in \Z_\epsilon}(X_t^\epsilon(z)-\rho_t^\epsilon(z))^2\right)\left(\epsilon \sum_{z \in \Z_\epsilon} \phi(\epsilon z)^2\right).
	\]
	Using \eqref{eq:V_fcn_2_pcs_diag}, we see that 
	\[
	(X_t^\epsilon(z)-\rho_t^\epsilon(z))^2 = V^\epsilon(x, X_t^\epsilon;\rho_t^\epsilon) + 2\epsilon^\kappa X_t^\epsilon(z)\rho_t^\epsilon(z),
	\]
	where $x=2\mathbbm{1}_{\{z\}}\in \X_\epsilon$ denotes the configuration with two particles at $z$. 
	Gathering the above estimates shows that 
	\begin{align*}
		\E[(X_t^\epsilon(\phi)-\rho_t^\epsilon(\phi))^2]&\leq \left(\epsilon\sum_{z\in \Z_\epsilon}\phi(\epsilon z)^2\right) \times \left[\max_{x=2\mathbbm{1}_{\{z\}}\in \X_\epsilon} \left( v_t^\epsilon(x, \rho_t^\epsilon|\nu^\epsilon) + 2\epsilon^\kappa \E[X_t^\epsilon(z)]|\rho_t^\epsilon(z)|\right)\right].
	\end{align*}
	By Theorem \ref{thm:QLLN}, we have 
	\[
	\lim_{\epsilon \searrow 0} \max_{x \in \X_\epsilon: n_x=2} v_t^\epsilon(x, \rho_t^\epsilon|\nu^\epsilon)=0.
	\]
	Since $X_t^\epsilon(z)=Q^\epsilon(x, X_t^\epsilon)$ if $x\in \X_\epsilon$ denotes the configuration with a single particle at $z\in \Z_\epsilon$, then we infer from Proposition \ref{prop:SLBP_moments} that
	\[
	\sup_{\epsilon \in (0,  1)}\max_{z \in \Z_\epsilon}\sup_{t \in [0, T]}\E[X_t^\epsilon(z)]<c(T),
	\]
	for some absolute constant $c(T)>0$. By part two of Lemma \ref{lem:FKPP_approx_properties}, we have $\sup_{\epsilon \in (0, 1)}\sup_{t \in [0, T]}\|\rho_t^\epsilon\|_\infty< \infty$. This shows that 
	\[
	\lim_{\epsilon \searrow 0} \E[(X_t^\epsilon(\phi)-\rho_t^\epsilon(\phi))^2] = 0.
	\]
	Finally, recall from part one of Lemma \ref{lem:FKPP_approx_properties} that as $\epsilon \searrow 0$, we have that $\widetilde{\rho}_t^{\:\epsilon}(z)\coloneqq \rho_t^\epsilon(K_\epsilon z)$, $z \in \S^1$, converges uniformly on $[0, T]\times \S^1$ to the solution $\rho$ of the FKPP equation \eqref{eq:FKPP}. Therefore, writing $\rho_t^\epsilon(z)=\widetilde{\rho}_t^{\:\epsilon}(K_\epsilon^{-1} z)$, we compute 
	\begin{align*}
		\rho_t^\epsilon(\phi)=\epsilon \sum_{z \in \Z_\epsilon} \rho_t^\epsilon(z)\phi(\epsilon z)&=\epsilon \sum_{z \in \Z_\epsilon} \widetilde{\rho}_t^{\:\epsilon}(K_\epsilon^{-1} z)\phi(\epsilon z)
		\\&\overset{\epsilon \searrow 0}{\to} \rho_t(\phi)\coloneqq \int_{\S^1} \rho_t(z)\phi(z)dz,
	\end{align*}
	where we have used that $K_\epsilon^{-1} z \in \S^1$ for $z \in \Z_\epsilon$, and $\epsilon K_\epsilon\to 1$ as $\epsilon \searrow 0$. Thus, by continuity of $f$, we have
	\[
	\lim_{\epsilon\to 0} f(\rho_t^\epsilon(\phi)) = f(\rho_t(\phi)). 
	\]
	Overall, we have shown that 
	\[
	\lim_{\epsilon \searrow 0} \E[f(X_t^\epsilon(\phi))] = f(\rho_t(\phi)),
	\]
	for all $\phi \in C^\infty(\S^1)$ and bounded Lipschitz $f$. One can easily extend this reasoning to a finite number of times $t_1, \cdots, t_k\in [0, T]$ to obtain 
	\[
	\lim_{\epsilon \searrow 0} \E[f(X_{t_1}^\epsilon(\phi), \cdots, X_{t_k}^\epsilon(\phi))] = f(\rho_{t_1}(\phi), \cdots, \rho_{t_k}(\phi)),
	\]
	for all $\phi\in C^\infty(\S^1)$ and $k$-dimensional bounded Lipschitz $f:\R^k\to \R$. Indeed, by a Taylor expansion of order one, we can write
	\begin{align*}
	&\E[f(X_{t_1}^\epsilon(\phi), \cdots, X_{t_k}^\epsilon(\phi))]
	\\&= f(\rho_{t_1}^\epsilon(\phi), \cdots, \rho_{t_k}^\epsilon(\phi))+\sum_{i=1}^k \E\left[(X_{t_i}^\epsilon(\phi)-\rho_{t_i}^\epsilon(\phi)) R_1(X_{t_i^\epsilon}(\phi))\right],
	\end{align*}
	where 
	\[
	|R_1(X_{t_i^\epsilon}(\phi))|\leq \max_{\xi \in \R^k: |\xi|\leq |X_{t_i^\epsilon}(\phi)-\rho_{t_i^\epsilon}(\phi)|} |\nabla f(\xi)|\leq \sup_{\xi \in \R^k}|\nabla f(\xi)|<\infty
	\]
	for each $i \in \{1, \cdots, k\}$. Then we may apply the arguments starting from \eqref{ineq:start_FDD_proof} above for each $i$ separately. 
	\end{proof}
	
	\subsection{Moments and martingale problems}
	
	The aim of this section is to prove the following moment control. 
	
	\begin{prop}\label{prop:SLBP_moments}
		For any $T>0$ and $n \in \N$, there exists $c=c(T, n)>0$ such that 
		\[
		\sup_{\epsilon \in (0, 1)}\sup_{t \in [0, T]} \max_{\substack{x\in \X_\epsilon \\ n_x=n}} \E\left[Q^\epsilon(x, X^\epsilon_t)\right] \leq c. 
		\]
	\end{prop}
	To prove this proposition, we follow the strategy used to prove Proposition 2.1 in \cite{Boldrig1987}. Namely, we apply Dynkin's formula to the test functions $\rX\mapsto Q^\epsilon(x, X^\epsilon_t)$. Then we upper bound the expectation of the birth-competition component of the generator, and we obtain a recursive inequality in $n_x$ which we can bound by iteration.
	
	Before proceeding with this agenda, we compute the generator of the rescaled SLBP $\rX^\epsilon = (\{X^\epsilon_t(z), \; z \in \Z_\epsilon\})_{t\geq 0}$. Let us recall its instantaneous dynamics. Define two polynomials $q_+^\epsilon(y) = \epsilon^{-\kappa}y$ and $q_-^\epsilon(y) = \epsilon^{-\kappa}y(y-\epsilon^\kappa)/2$ for all $y \in \R$. 
	\begin{enumerate}
		\item[1.] Spatial motion: particles perform independent simple symmetric random walks on $\Z_\epsilon$ at rate $\epsilon^{-2}$.
		\item[2.] Local births: independently at every site $z \in \Z_\epsilon$, we have at time $t-$ a transition $X_{t-}^\epsilon(z)\to X_{t-}^\epsilon(z)+\epsilon^\kappa \ell$ at rate $q_+^\epsilon(X_{t-}^\epsilon(z))p_\ell^\epsilon$, for all $\ell\geq 1$. 
		\item[3.] Local competition: independently at every site $z \in \Z_\epsilon$, we have at time $t-$ a transition $X_{t-}^\epsilon(z)\to X_{t-}^\epsilon(z)-\epsilon^\kappa$ at rate $q_-^\epsilon(X_{t-}^\epsilon(z))$.
	\end{enumerate}	
	In view of these dynamics, the generator $\cG^\epsilon$ of $\rX^\epsilon$ can be decomposed as the sum of a spatial motion component $\cG^\epsilon_{RW}$ and a birth-competition component $\cG^\epsilon_{BC}$. Given a test function $f:\X_\epsilon\to \R$, we have
	\begin{equation}\label{def:slbp_gen}
	\cG^\epsilon f(x) = \cG_{RW}^\epsilon f(x)+\cG_{BC}^\epsilon f(x), \quad x \in \X_\epsilon.
	\end{equation}
	where 
	\begin{equation}\label{eq:GRW}
	\cG_{RW}^\epsilon f(x) = \frac{\epsilon^{-2}}{2}\sum_{z \in \Z_\epsilon} \epsilon^{-\kappa}x(z) \Big\{f(x^{z, z+1})+f(x^{z, z-1})-2f(x)\Big\},
	\end{equation}
	with 
	\[
	x^{z, z\pm 1}(z') = \begin{cases} 
		x(z)-\epsilon^\kappa, & z'=z
		\\ x(z)+\epsilon^\kappa, & z'=z\pm 1
		\\ x(z'), & z' \in \Z_\epsilon \backslash \{z, z\pm 1\},
	\end{cases}
	\]
	and 
	\begin{equation}\label{eq:GBC}
	\cG_{BC}^\epsilon f(x) = \sum_{z \in \Z_\epsilon} \Big\{\sum_{\ell \geq 1} p_\ell^\epsilon q_+^\epsilon(x(z))(f(x^{z, \ell})-f(x))+q_-^\epsilon(x(z))(f(x^{z, -1})-f(x))\Big\}.
	\end{equation}
	with 
	\begin{equation}\label{def:updated_config}
	x^{z, \ell}(z') = \begin{cases} 
		x(z)+\ell\epsilon^\kappa, & z'=z
		\\ x(z'), & z' \in \Z_\epsilon \backslash \{z\}
	\end{cases}, \quad \ell \in \N \cup \{-1\}.
	\end{equation}
	
	We introduce the notation $\G^\epsilon_t(x, x')$, $x, x' \in \X_\epsilon$ with $n_x=n_{x'}$, for the probability of transition between configurations $x$ and $x'$ in time $t$, purely from spatial motion. It is clear that
	\begin{equation}\label{def:G}
	\G_t^\epsilon(x, x') = \frac{1}{n_x!} \sum_\sigma \prod_{z \in x} G_t^\epsilon(z, \sigma(z)),
	\end{equation}
	where $\sigma$ runs over the set of bijections $x\to x'$. If $x=x'=0$, we set the sum to zero by convention. Observe that $\G^\epsilon_t$ is symmetric in its arguments. We introduce the discrete Laplacian acting on functions $F:\X_\epsilon\to \R$:
	\begin{equation}
		\begin{aligned}\label{def:disc_Laplacian_config}
			\Delta^\epsilon_z F(x) &\coloneqq \epsilon^{-2}(F(x^{z, z+1})+F(x^{z, z-1})-2F(x)), \quad z \in \Z_\epsilon,
			\\\Delta^\epsilon F(x) &\coloneqq \sum_{z \in x}\Delta^\epsilon_z F(x).
		\end{aligned}
	\end{equation}
	for all $x \in \X_\epsilon$, where the updated configurations $x^{z, z\pm1}$ are defined in \eqref{def:updated_config} and the summation over $z\in x$ is defined in \eqref{def:sum_prod_config_iteration}. We use the same notation for the Laplacian acting on $f:\Z_\epsilon \to \R$:
	\begin{equation}\label{def:disc_Laplacian_sites}
	\Delta^\epsilon_z f(z) = \epsilon^{-2}(f(z+1)+f(z-1)-2f(z)), \quad z \in \Z_\epsilon.
	\end{equation}
	In the following lemma, we derive a martingale problem for the rescaled SLBP.
	\begin{lem}\label{lem:mg_pb_LLN}
		Fix any $\epsilon \in (0, 1)$. For any $x \in \X_\epsilon$, the process $\rM^\epsilon(x) = (M^\epsilon_t(x))_{t\geq 0}$ defined by 
		\begin{equation}\label{eq:true_mg}
		M^\epsilon_t(x)=Q^\epsilon(x, X^\epsilon_t)-\sum_{x'} \G^\epsilon_{t}(x, x')Q^\epsilon(x', X^\epsilon_0)-\int_0^t \sum_{x'} \G^\epsilon_{t-s}(x, x')\cG^\epsilon_{BC}Q^\epsilon(x', \cdot)(X^\epsilon_s) ds, \quad t\geq 0,
		\end{equation}
		where $x'$ ranges over configurations $x'\in \X_\epsilon$ with $n_{x'}=n_x$, is a martingale.
	\end{lem}
	\begin{proof}
		To lighten the notation, we introduce the notation $F(s, X^\epsilon_s) = \sum_{x'} \G^\epsilon_{t-s}(x, x') Q^\epsilon(x', X^\epsilon_s)$. Since the rescaled SLBP is a Markov process, then 
		\begin{equation}\label{eq:mg_pb_LLN}
			F_t(X^\epsilon_t)-F_0(X^\epsilon_0) - \int_0^t \left(\partial_1 + \cG^\epsilon\right)F(s, X^\epsilon_s) ds, \quad t\geq 0,
		\end{equation}
		where $\partial_1$ is the derivative of $F$ in its first coordinate, is a local martingale. We have 
		\begin{equation}\label{eq:partial_1_F}
			\partial_1 F(s, X^\epsilon_s) = \sum_{x'}\partial_s\G^\epsilon_{t-s}(x, x')Q^\epsilon(x', X_s),
		\end{equation}
		where
		\begin{align*}
			\partial_s\G_{t-s}^\epsilon(x, x') &=- \frac{1}{n_x!} \sum_\sigma \sum_{z \in x}\partial_sG_{t-s}^\epsilon(z, \sigma(z))\prod_{z' \in x\backslash \{z\}} G^\epsilon_{t-s}(z', \sigma(z'))
			\\&= -\frac{1}{n_x!}\sum_\sigma \sum_{z \in x} \frac{\Delta^\epsilon_z}{2} G^\epsilon_{t-s}(\cdot, \sigma(z))(z)\prod_{z' \in x\backslash \{z\}} G^\epsilon_{t-s}(z', \sigma(z')).
		\end{align*}
		where $\Delta^\epsilon_z$ is the discrete Laplacian defined in \eqref{def:disc_Laplacian_sites}. Plugging this expression into \eqref{eq:partial_1_F} and using the symmetry of $\G^\epsilon_{t-s}$, we obtain 
		\begin{align*}
			\partial_1 F(s, X^\epsilon_s) &= \sum_{x'}\partial_s\G^\epsilon_{t-s}(x, x')Q^\epsilon(x', X^\epsilon_s)
			\\&= -\sum_{x'} \sum_{z \in x} \frac{\Delta^\epsilon_z}{2} \G^\epsilon_{t-s}(\cdot, x')(x) Q^\epsilon(x', X^\epsilon_s)
			\\&= -\sum_{x'} \sum_{z \in x'} \frac{\Delta^\epsilon_z}{2} \G^\epsilon_{t-s}(\cdot, x)(x') Q^\epsilon(x', X^\epsilon_s), 
		\end{align*}
		where $\Delta_z^\epsilon$ the Laplacian acting on functions of configurations is defined in \eqref{def:disc_Laplacian_config}.
		We now compute $\cG^\epsilon_{RW}F(s, X^\epsilon_s)$ to show that it cancels with $\partial_1F(s, X^\epsilon_s)$. By \eqref{eq:GRW_is_laplacian} in Lemma \ref{lem:localisation} and the symmetry of $\G^\epsilon_{t-s}$, we obtain 
		\begin{align}
			&\notag\cG^\epsilon_{RW} F(s, \cdot)(X^\epsilon_s) 
			\\&\notag= \frac{\epsilon^{-2}}{2}\sum_{z \in \Z_\epsilon}\epsilon^{-\kappa}X^\epsilon_s(z)\sum_{x'} \G_{t-s}^\epsilon(x', x)\left\{Q^\epsilon(x', (X^\epsilon_s)^{z, z+1})+Q^\epsilon(x', (X^\epsilon_s)^{z, z-1})-2Q^\epsilon(x', X^\epsilon_s)\right\}
			\\&=\frac{1}{2}\sum_{x'}\G_{t-s}^\epsilon(x', x)\cG_{RW}^\epsilon Q^\epsilon(x', \cdot)(X^\epsilon_s)\label{eq:RW_gen_is_Laplacian_LLN}
			\\&\notag=\frac{1}{2}\sum_{x'}\sum_{z \in x'}\G_{t-s}^\epsilon(x', x)\Delta_z^\epsilon Q^\epsilon(\cdot, X^\epsilon_s)(x').
		\end{align}
		By a simple change of variables, we obtain 
		\[
		\cG^\epsilon_{RW} F(s, \cdot)(X^\epsilon_s)= \epsilon^{-2}\sum_{x'}\sum_{z \in x'}\Delta_z^\epsilon\G_{t-s}^\epsilon(\cdot, x)(x') Q^\epsilon(x', X^\epsilon_s)=-\partial_1 F(s, X^\epsilon_s). 
		\]
		Going back to the local martingale \eqref{eq:mg_pb_LLN}, we see that only $\cG_{BC}^\epsilon F(s, \cdot)(X^\epsilon_s)$ remains under the integral sign, which proves that the process $\rM^\epsilon_\cdot(x)$ defined in \eqref{eq:true_mg} is a local martingale. Since we are working with a fixed $\epsilon$, the SLBP $\rX^\epsilon$ is dominated by a rescaled pure birth process $\overline{\rX}^\epsilon\coloneqq \epsilon^\kappa\overline{\rN}^\epsilon$, where $\overline{\rN}^\epsilon$ is a pure birth process with transitions $n\mapsto n+k$ at rate $np_k^\epsilon$ and started from $\overline{N}^\epsilon_0 = \sum_{z \in \Z_\epsilon}N_0^\epsilon(z)$, where $\rN^\epsilon$ is from Definition \ref{def:rescaled_slbp}. In particular, we have $\E\left[\sup_{s \in [0, t]}|M_s^\epsilon(x)|\right]<\infty$ and $\E\left[\sup_{s \in [0, t]}|Q^\epsilon(x, X_s^\epsilon)|\right]<\infty$ for every $t\geq 0$. It then easily follows from the Vitali convergence theorem that $\rM^\epsilon(x)$ is a martingale. This concludes the proof of the lemma. 
	\end{proof}
	
	\begin{remark}
		In subsequent uses of the Markov property to obtain local martingales for a fixed $\epsilon$, similar arguments to the ones used at the end of the proof of Lemma \ref{lem:mg_pb_LLN} will show that the local martingales are in fact martingales. We will therefore omit those details.
	\end{remark}
	
	Next, we state a key lemma, which in particular provides an upper bound on the expectation of $\cG^\epsilon_{BC}Q^\epsilon(x, \cdot)(X^\epsilon_s)$. The lemma and its proof are adapted from Proposition 2.1 in \cite{Boldrig1987}.
	
	\begin{lem}\label{lem:bd_BC_gen_LLN}
		We have 
		\begin{align}
			&\notag\cG^\epsilon_{BC} Q^\epsilon(x, \cdot)(X^\epsilon_s) 
			\\&\label{eq:gen_BC_LLN}= \sum_{z \in \supp(x)} \Bigg\{q_+^\epsilon(X^\epsilon_s(z))\sum_{\ell\geq 1}p_\ell^\epsilon \Big[Q^\epsilon_{n_x(z)}(\epsilon^{-\kappa}X^\epsilon_s(z)+ \ell)-Q^\epsilon_{n_x(z)}(\epsilon^{-\kappa}X^\epsilon_s(z))\Big]
			\\&\notag-x(z)q_-^\epsilon(X^\epsilon_s(z))Q^\epsilon_{n_x(z)-1}(\epsilon^{-\kappa}X^\epsilon_s(z)-1)\Bigg\}\prod_{z' \in \supp(x)\backslash \{z\}}Q^\epsilon_{n_x(z')}(\epsilon^{-\kappa}X^\epsilon_s(z')). 
		\end{align}
		Moreover, define
		\begin{align*}
			\phi_k(t)&= \sup_{\epsilon \in (0, 1)}\sup_{s \in [0, t]} \max_{\substack{x\in \X_\epsilon \\ n_x=n}} \E\left[Q^\epsilon(x, X^\epsilon_s)\right], \quad t>0, \quad k \in \N.
		\end{align*}
		Then, for any $k\in \N$, there exists $c=c_k>0$ such that 
		\begin{align*}
			\E\left[\cG^\epsilon_{BC}Q^\epsilon(x, \cdot)(X^\epsilon_t)\right]&\leq c k \phi_{k-1}(t),
		\end{align*}
		for all $t>0$ and $\epsilon >0$.
	\end{lem}
	\begin{proof}
		For configurations $x \in \X_\epsilon$, we have 
		\begin{align*}
			&\cG^\epsilon_{BC}Q^\epsilon(x, \cdot)(X^\epsilon_s)
			\\& = \sum_{z \in \supp(x)} \Bigg\{ q_+^\epsilon(X^\epsilon_s(z)) \sum_{\ell\geq 1}p_\ell^\epsilon \Big[Q^\epsilon_{n_x(z)}(\epsilon^{-\kappa}X^\epsilon_s(z)+ \ell)-Q^\epsilon_{n_x(z)}(\epsilon^{-\kappa}X^\epsilon_s(z))\Big] 
			\\& + q_-^\epsilon(X^\epsilon_s(z))\Big[Q^\epsilon_{n_x(z)}(\epsilon^{-\kappa}X^\epsilon_s(z)-1)-Q^\epsilon_{n_x(z)}(\epsilon^{-\kappa}X^\epsilon_s(z))\Big]\Bigg\}\prod_{z' \in \supp(x)\backslash \{z\}} Q^\epsilon_{n_x(z')}(\epsilon^{-\kappa}X^\epsilon_s(z')). 
		\end{align*}
		For any natural numbers $k\leq n$, we have
		\[
		Q_k(n-1)-Q_k(n) = (n-1)(n-2)\cdots (n-k+1)[n-k-n] = -kQ_{k-1}(n-1).  
		\]
		Applying this observation to the competition term in the generator above, we obtain \eqref{eq:gen_BC_LLN}. Next, we observe that 
		\begin{align*}
			&\epsilon^{-\kappa}\sum_{\ell\geq 1}p_\ell^\epsilon \Big[Q^\epsilon_{n_x(z)}(\epsilon^{-\kappa}X^\epsilon_s(z)+ \ell)-Q^\epsilon_{n_x(z)}(\epsilon^{-\kappa}X^\epsilon_s(z))\Big] \\&=\epsilon^{\kappa (n_x(z)-1)}\mu_\epsilon (\epsilon^{-\kappa}X^\epsilon_s(z))^{n_x(z)-1} + R_\epsilon(\epsilon^{-\kappa}X^\epsilon_s(z)),
		\end{align*}
		where $R_\epsilon$ is a polynomial of degree strictly smaller than $n_x(z)-1$, and which in $\epsilon$, satisfies
		\[
		R_\epsilon(\epsilon^{-\kappa}X^\epsilon_s(z)) = O\left(\max_{j\in \{1, \cdots, n_x(z)-1\}} \epsilon^{\kappa j} \sum_{\ell \geq 1}p_\ell^\epsilon \ell^j\right).
		\]
		Since the offspring distribution is truncated at $L(\epsilon)$, which satisfies 
		\[
		\lim_{\epsilon \searrow 0} \epsilon^\kappa L(\epsilon)=0,
		\] 
		and we have $\sup_{\epsilon >0}\sum_{\ell\geq 1} \ell p_\ell^\epsilon<\infty$ by assumption, we obtain
		\[
		\epsilon^{\kappa j} \sum_{\ell \geq 1}p_\ell^\epsilon \ell^j\leq \epsilon^\kappa (\epsilon^\kappa L(\epsilon))^{j-1}\sum_{\ell\geq 1} \ell p_\ell^\epsilon = o(1),
		\]
		for any choice of $j$. In particular, there exists an absolute constant $c_1>0$ such that 
		\[
		\epsilon^{\kappa (n_x(z)-1)} \Big\{ \mu_\epsilon n^{\kappa (n_x(z)-1)}+R_\epsilon(n)\Big\} \leq c_1  \epsilon^{\kappa (n_x(z)-1)} n^{\kappa (n_x(z)-1)},
		\]
		for all $n \in \N$. Moreover, there exists an absolute constant $c_2>0$ such that 
		\[
		c_1 n^{k-1}\leq c_2 Q_{k-1}(n-1), 
		\]
		for all $k, n \in \N$ with $k \leq n$. Thus 
		\begin{align*}
			&q_+^\epsilon(X^\epsilon_s(x))\sum_{\ell\geq 1}p_\ell^\epsilon \Big[Q^\epsilon_{n_x(z)}(\epsilon^{-\kappa}X^\epsilon_s(z)+ \ell)-Q^\epsilon_{n_x(z)}(\epsilon^{-\kappa}X^\epsilon_s(z))\Big]
			\\&\quad\quad\quad\quad\quad\quad\quad\quad\quad\quad\quad\quad\quad\quad\quad\quad\quad\quad\quad-x(z)q_-^\epsilon(X^\epsilon_s(z))Q^\epsilon_{n_x(z)-1}(\epsilon^{-\kappa}X^\epsilon_s(z)-1)
			\\& \leq \epsilon^\kappa \Big\{c_2q_+^\epsilon(X^\epsilon_s(z))-\epsilon^{-\kappa}x(z)q_-^\epsilon(X^\epsilon_s(z))\Big\}Q^\epsilon_{n_x(z)-1}(\epsilon^{-\kappa}X^\epsilon_s(z)).
		\end{align*}
		We further note that 
		\[
		\epsilon^\kappa \Big\{c_2q_+^\epsilon(X^\epsilon_s(z))-\epsilon^{-\kappa}x(z)q_-^\epsilon(X^\epsilon_s(z))\Big\} = c_2X^\epsilon_s(z)-x(z)X^\epsilon_s(z)(X^\epsilon_s(z)-\epsilon^\kappa)/2
		\]
		is bounded from above by some constant $c_3(x(z))>0$ uniformly in $X^\epsilon_s(z)\geq 0$ and $\epsilon\in (0, 1)$. Hence for any $s>0$, we can upper bound the expectation of \eqref{eq:gen_BC_LLN} as follows:
		\begin{align*}
			&\E\left[\cG^\epsilon_{BC} Q^\epsilon(x, \cdot)(X^\epsilon_s)\right]
			\\&\leq c_3 \E\left[\sum_{z \in \supp(x)} Q^\epsilon_{n_x(z)-1}(\epsilon^{-\kappa}X^\epsilon_s(z)) \prod_{z' \in \supp(x)\backslash \{z\}}Q^\epsilon_{n_x(z')}(\epsilon^{-\kappa}X^\epsilon_s(z'))\right]
			\\& \leq c_3 k \phi_{k-1}(s),
		\end{align*}
		where $c_3 = \max_{z \in \supp(x)}c_3(x(z))$. This proves the lemma. 
	\end{proof}
	We apply the tools above to prove Proposition \ref{prop:SLBP_moments}.
	\begin{proof}[Proof of Proposition \ref{prop:SLBP_moments}]
		By Lemma \ref{lem:mg_pb_LLN} and Tonelli's theorem, we have 
		\begin{align*}
			\E\left[Q^\epsilon(x, X^\epsilon_t)\right]=\sum_{x'} \G^\epsilon_{t}(x, x')\E\left[Q^\epsilon(x', X^\epsilon_0)\right]+\int_0^t \sum_{x'} \G^\epsilon_{t-s}(x, x')\E\left[\cG^\epsilon_{BC}Q^\epsilon(x', \cdot)(X^\epsilon_s)\right] ds,
		\end{align*}
		for all $t\geq 0$. We have $\sum_{x'} \G^\epsilon_{t-s}(x, x')=1$ for all $\epsilon \in (0, 1)$. Thus, by Lemma \ref{lem:bd_BC_gen_LLN} and Assumption~\ref{assump:technical_X_0}.a, there are constants $c_{0, k}, c_{1, k}>0$ such that 
		\[
		\phi_k(t)\leq c_{0, k} +c_{1, k} \int_0^t \phi_{k-1}(s)ds, \quad k \in \N, \quad t\geq 0,
		\]
		with $\phi_0\equiv 1$. It follows by iteration on the right-hand side that there exist functions $h_k:[0, \infty)\to [0, \infty)$, for $k\in \N$, such that $\phi_k(t)\leq h_k(t)<\infty$ for all $t\geq 0$.
	\end{proof}
	
	We record a corollary to Proposition \ref{prop:SLBP_moments} which will be helpful in later sections. 
	
	\begin{cor}\label{cor:SLBP_p_moments}
		For any $T>0$, $n \in \N$, and $p>1$, there exists $c=c(T, n, p)>0$ such that 
		\[
		\sup_{\epsilon \in (0, 1)}\sup_{t \in [0, T]}\max_{\substack{x\in \X_\epsilon\\ n_x=n}} \E\left[Q^\epsilon(x, X^\epsilon_t)^p\right] \leq c. 
		\]
	\end{cor}
	\begin{proof}
		Let $m \in \N$. It is easy to see that $n^m\leq m^mQ_m(n)$ for all $n \geq m$, and therefore 
		\begin{equation}\label{ineq:bound_Q_mom_power_p}
		X_t^\epsilon(z)^m \leq \epsilon^{\kappa m} m^m+m^mQ_m^\epsilon(\epsilon^{-\kappa}X_t^\epsilon(z)),
		\end{equation}
		for all $z \in \Z_\epsilon$, $t\geq 0$ and $\epsilon \in (0, 1)$. We can write the falling factorial as 
		\begin{equation}\label{eq:stirling_ff}
		Q_k(n) = \sum_{i=0}^k s(k, i) n^i,\quad k\leq n,\quad k, n \in \N,
		\end{equation}
		where $s(k, i)$, $0\leq i \leq k$, denote the Stirling numbers of the first kind. Given a configuration $x \in \X_\epsilon$, we obtain 
		\[
		Q^\epsilon(x, X_t^\epsilon) \leq \prod_{z \in \supp(x)}\sum_{i=0}^{n_x(z)}s(n_x(z), i)\epsilon^{\kappa(n_x(z)-i)} [\epsilon^{\kappa i}i^i+i^i Q_i^\epsilon(\epsilon^{-\kappa}X_t^\epsilon(z))].
		\] 
		Using this bound and Jensen's inequality, we see that
		\begin{align*}
		&\E\left[Q^\epsilon(x, X_t^\epsilon)^p\right] 
		\\&\leq c_1(n, p) + n^{p-1} \E\left[ \prod_{ z \in \supp(x)}n_x(z)^{p-1}\sum_{i=0}^{n_x(z)}s(n_x(z), i)^p\epsilon^{p\kappa(n_x(z)-i)}i^{ip} Q_i^\epsilon(\epsilon^{-\kappa}X_t^\epsilon(z))^p\right]
		\end{align*}
		for some constant $c_1(n, p)>0$. Given $K, K_i \in \N$, for each $i \in \{1, \cdots, K\}$, and real numbers $a_{ij}> 0$ for all $j \in \{1, \cdots, K_i\}$, $i \in \{1, \cdots, K\}$, we see by an application of Jensen's inequality for the (concave) log function that
		\[
		\prod_{i=1}^K \sum_{j=1}^{K_i} a_{ij} \leq \left(\frac{1}{K} \sum_{i=1}^K \sum_{j=1}^{K_i} a_{ij}\right)^K.
		\]
		Thus
		\begin{align*}
			&\prod_{ z \in \supp(x)}n_x(z)^{p-1}\sum_{i=0}^{n_x(z)}s(n_x(z), i)^p\epsilon^{p\kappa(n_x(z)-i)}i^{ip} Q_i^\epsilon(\epsilon^{-\kappa}X_t^\epsilon(z))^p
			\\& \leq \left(|\supp(x)|^{-1} \sum_{z \in \supp(z)} \sum_{i=0}^{n_x(z)} s(n_x(z), i)^p\epsilon^{p\kappa(n_x(z)-i)}i^{ip}Q_i^\epsilon(\epsilon^{-\kappa}X_t^\epsilon(z))^p\right)^{|\supp(x)|}.
		\end{align*}
		By Jensen's inequality and the fact that $1\leq |\supp(x)|\leq n$, we have 
		\begin{align*}
			&\left(|\supp(x)|^{-1} \sum_{z \in \supp(z)} \sum_{i=0}^{n_x(z)} s(n_x(z), i)^p\epsilon^{p\kappa(n_x(z)-i)}i^{ip}Q_i^\epsilon(\epsilon^{-\kappa}X_t^\epsilon(z))^p\right)^{|\supp(x)|}
			\\ &\leq |\supp(x)|^{-1} \sum_{z \in \supp(z)} \left(\sum_{i=0}^{n_x(z)} s(n_x(z), i)^p\epsilon^{p\kappa(n_x(z)-i)}i^{ip}Q_i^\epsilon(\epsilon^{-\kappa}X_t^\epsilon(z))^p\right)^{|\supp(x)|}
			\\ &\leq \sum_{z \in \supp(z)} n_x(z)^{n-1} \sum_{i=0}^{n_x(z)} s(n_x(z), i)^{pn}\epsilon^{pn\kappa(n_x(z)-i)}i^{ipn}Q_i^\epsilon(\epsilon^{-\kappa}X_t^\epsilon(z))^{pn}.
		\end{align*}
		Taking expectations, we have 
		\begin{align*}
			&\E\left[\sum_{z \in \supp(z)} n_x(z)^{n-1} \sum_{i=0}^{n_x(z)} s(n_x(z), i)^{pn}\epsilon^{pn\kappa(n_x(z)-i)}i^{ipn}Q_i^\epsilon(\epsilon^{-\kappa}X_t^\epsilon(z))^{pn}\right]
			\\&\leq \sum_{z \in \supp(z)} n_x(z)^{n-1} \sum_{i=0}^{n_x(z)} s(n_x(z), i)^{pn}\epsilon^{pn\kappa(n_x(z)-i)}i^{ipn}\E\left[\sup_{t \in [0, T]}Q_i^\epsilon(\epsilon^{-\kappa}X_t^\epsilon(z))^{pn}\right]
			\\&\leq c_2(n, p) \max_{x \in \X_\epsilon:n_x=n} \max_{z \in \Z_\epsilon} \max_{i\in \{0, \cdots, n_x(z)\}} \E\left[Q_i^\epsilon(\epsilon^{-\kappa}X_t^\epsilon(z))^{pn}\right]. 
		\end{align*}
		for some $c_2(n, p)>0$. Next, we apply \eqref{eq:stirling_ff} followed by Jensen's inequality and \eqref{ineq:bound_Q_mom_power_p} to obtain 
		\[
		\E\left[Q_i^\epsilon(\epsilon^{-\kappa}X_t^\epsilon(z))^{pn}\right] \leq c_3(n, p) + c_4(n, p)  \max_{j\in \{0, \cdots, i\}} \E\left[Q^\epsilon_{jn}(\epsilon^{-\kappa} X_t^\epsilon(z))\right],
		\]
		for some $c_3(n, p), c_4(n, p)>0$. The last expectation is finite uniformly in $\epsilon \in (0, 1)$ by Proposition \ref{prop:SLBP_moments}, which concludes the proof. 
	\end{proof}
	
	\subsection{Uniform equicontinuity and tightness}
	
	The aim of this section is to prove \eqref{cond:aldous_2}, the second condition of Aldous' tightness criterion, which is an immediate consequence of the following uniform equicontinuity estimate. 
	
	\begin{prop}\label{prop:equicty_LLN}
		Fix $\phi \in C^\infty(\S^1)$, a family $\{\tau_\epsilon: \epsilon\in (0, 1)\}$ of $[0, T]$-valued stopping times, and $\{\theta_\epsilon:\epsilon \in (0, 1)\}\subset [0, \infty)$ with $\lim_{\epsilon \searrow 0}\theta_\epsilon=0$. Then, there exist $c=c(T, \phi)>0$ and $a>0$ such that
		\[
		\E\left[|X^\epsilon_{\tau_\epsilon+\theta_\epsilon}(\phi)-X^\epsilon_{\tau_\epsilon}(\phi)|\right]\leq c\left(\epsilon^{\kappa/2}+\max(\theta_\epsilon^{1/2}, \epsilon)+\theta_\epsilon^a\right),
		\]
		for all $\epsilon \in (0, 1)$. 
	\end{prop}
	\begin{proof}
		By Jensen's inequality, we have 
		\begin{equation}\label{ineq:abs_diff}
		\E\left[|X_{\tau_\epsilon +\theta}^\epsilon(\phi)-X_{\tau_\epsilon}^\epsilon(\phi)|\right]\leq \epsilon \sum_{z \in \Z_\epsilon}\E\left[|X_{\tau_\epsilon +\theta}^\epsilon(z)-X_{\tau_\epsilon}^\epsilon(z)|\right]\phi(\epsilon z)
		\end{equation}
		If $x=\epsilon^\kappa\mathbbm{1}_{\{z\}} \in \X_\epsilon$ is the configuration with a single particle at $z \in \Z_\epsilon$, we have
		\[
		X^\epsilon_{t}(z)-X^\epsilon_{s}(z) = Q^\epsilon(x, X^\epsilon_{t})-Q^\epsilon(x, X^\epsilon_{s}).
		\] 
		Thus by Lemma \ref{lem:mg_pb_LLN},
		\begin{align*}
			Q^\epsilon(x, X^\epsilon_{\tau_\epsilon+\theta_\epsilon})-Q^\epsilon(x, X^\epsilon_{\tau_\epsilon})
			&= \sum_{\substack{x' \in \X_\epsilon \\ n_{x'}=1}}[\G^\epsilon_{\tau_\epsilon+\theta_\epsilon}(x, x')-\G^\epsilon_{\tau_\epsilon}(x, x')]Q^\epsilon(x', X^\epsilon_0)
			\\&+\int_{\tau_\epsilon}^{\tau_\epsilon+\theta_\epsilon}\sum_{\substack{x' \in \X_\epsilon \\ n_{x'}=1}}\G^\epsilon_u(x, x')\cG^\epsilon_{BC}Q^\epsilon(x', \cdot)(X_{\tau_\epsilon+\theta_\epsilon-u}^\epsilon)du
			\\&+\int_0^{\tau_\epsilon}\sum_{\substack{x' \in \X_\epsilon \\ n_{x'}=1}}[\G^\epsilon_{\tau_\epsilon+\theta_\epsilon-u}(x, x')-\G^\epsilon_{\tau_\epsilon-u}(x, x')]\cG_{BC}^\epsilon Q^\epsilon(x', \cdot)(X^\epsilon_u)du
			\\&+M^\epsilon_{\tau_\epsilon+\theta_\epsilon}(x)-M^\epsilon_{\tau_\epsilon}(x). 
		\end{align*}
		Let $\overline{\theta}\coloneqq \sup_{\epsilon \in (0, 1)}\theta_\epsilon \in [0, \infty)$. It follows from Jensen's inequality that 
		\begin{align*}
			\E\Bigg[|X_{\tau_\epsilon+\theta_\epsilon}^\epsilon(z)-X_{\tau_\epsilon}^\epsilon(z)|\Bigg]&\leq \E\Bigg[\Bigg|\sum_{\substack{x' \in \X_\epsilon \\ n_{x'}=1}}[\G^\epsilon_{\tau_\epsilon+\theta_\epsilon}(x, x')-\G^\epsilon_{\tau_\epsilon}(x, x')]Q^\epsilon(x', X^\epsilon_0)\Bigg|\Bigg]
			\\&+\E\Bigg[\sup_{\substack{0\leq s\leq t\leq T+\overline{\theta} \\ t-s\leq \theta_\epsilon}}\int_{s}^{t}\sum_{\substack{x' \in \X_\epsilon \\ n_{x'}=1}}\G^\epsilon_{t-u}(x, x')|\cG^\epsilon_{BC}Q^\epsilon(x', \cdot)(X_u^\epsilon)|du\Bigg]
			\\&+\E\Bigg[\sup_{\substack{0\leq s\leq t\leq T+\overline{\theta} \\ t-s\leq \theta_\epsilon}}\int_0^{s}\sum_{\substack{x' \in \X_\epsilon \\ n_{x'}=1}}|\G^\epsilon_{t-u}(x, x')-\G^\epsilon_{s-u}(x, x')||\cG_{BC}^\epsilon Q^\epsilon(x', \cdot)(X^\epsilon_u)|du\Bigg]
			\\&+\E\left[\Big|M^\epsilon_{\tau_\epsilon +\theta_\epsilon}(x)-M^\epsilon_{\tau_\epsilon}(x)|\right].
		\end{align*}
		We control the four terms on the right-hand side in turn. Denote them by $A_1$, $A_2$, $A_3$, and $A_4$, respectively. To bound $A_1$, we adapt a method used in Lemma 27 of \cite{EthLab2014} and apply Assumption \ref{assump:technical_X_0}.b. By the semigroup property of $G^\epsilon$, and the fact that $\G_\cdot^\epsilon(x, x')=G_\cdot^\epsilon(z, z')$ for configurations $x=\epsilon^\kappa\mathbbm{1}_{\{z\}}$ and $x'=\epsilon^\kappa\mathbbm{1}_{\{z'\}}$ with a single particle at $z, z' \in \Z_\epsilon$ respectively, we obtain
		\begin{align*}
			&\sum_{\substack{x' \in \X_\epsilon \\ n_{x'}=1}}[\G^\epsilon_{\tau_\epsilon+\theta_\epsilon}(x, x')-\G^\epsilon_{\tau_\epsilon}(x, x')]Q^\epsilon(x', X^\epsilon_0)
			\\&=\sum_{\substack{x'' \in \X_\epsilon \\ n_{x''}=1}}\sum_{\substack{x' \in \X_\epsilon \\ n_{x'}=1}}\G^\epsilon_{\theta_\epsilon}(x, x'')[\G^\epsilon_{\tau_\epsilon}(x'', x')-\G^\epsilon_{\tau_\epsilon}(x, x')]Q^\epsilon(x', X^\epsilon_0)
			\\&= \sum_{z'' \in \Z_\epsilon} G^\epsilon_{\theta_\epsilon}(z, z'')\left[\sum_{z' \in \Z_\epsilon}G_{\tau_\epsilon}^\epsilon(z'', z')X_0^\epsilon(z')-\sum_{z' \in \Z_\epsilon}G_{\tau_\epsilon}^\epsilon(z, z')X_0^\epsilon(z')\right].
		\end{align*}
		We then apply a change of variables to obtain 
		\[
		\sum_{z' \in \Z_\epsilon}G_{\tau_\epsilon}^\epsilon(z'', z')X_0^\epsilon(z')=\sum_{z' \in \Z_\epsilon}G_{\tau_\epsilon}^\epsilon(0, z')X_0^\epsilon(z'-z''),
		\]
		and similarly for the second inner summation. Hence, by Jensen's inequality, we have 
		\begin{align*}
			A_1&\leq \sum_{z'' \in \Z_\epsilon} G^\epsilon_{\theta_\epsilon}(z, z'')\sum_{z' \in \Z_\epsilon}\E\left[G_{\tau_\epsilon}^\epsilon(0, z')|X_0^\epsilon(z'-z'')-X_0^\epsilon(z'-z)|\right].
		\end{align*}
		Let $h^\epsilon_t(z')\coloneqq \min(1, \epsilon t^{-1/2})e^{-z'\min(1, \epsilon t^{-1/2})}$ where we think of $z' \in \Z$ by abuse of notation, and where $t \in [0, T]$. By \eqref{ineq:G_exp_bd} in Lemma \ref{lem:Green_estimates}, there exists $C=C(T)>0$ such that 
		\begin{equation}\label{ineq:G_exp_bd}
			G^\epsilon_t(0, z')\leq Ch^\epsilon_t(z'),
		\end{equation}
		for all $t \in [0, T]$, $z' \in \Z_\epsilon$, and $\epsilon \in (0, 1)$. It is easy to see that for any $z' \in \Z_\epsilon$, we may find $t_*(z')$, which maximises $h_t^\epsilon(z')$ over $t \in [0, T]$. Moreover, there exists a deterministic natural number $R$ such that for any $z \in \Z$ with $z\geq R$, if $z'=\pi^\epsilon(z)$, then we have $t_*(z')=T$. We also observe that $h^\epsilon_{t_*(z')}(z')$ for $z'$ among the first $R$ elements of $\Z_\epsilon$ is maximised at $z'=0$ with maximum value one. Overall, this implies that with $C=C(T)>0$ the constant from \eqref{ineq:G_exp_bd}, we have the deterministic upper bound
		\begin{equation}\label{ineq:determ_heat_bound}
		G_{\tau_\epsilon}^\epsilon(0, z')\leq C\overline{h}^\epsilon(z'),
		\end{equation}
		where $\overline{h}^\epsilon(z')\coloneqq \mathbbm{1}_{\{0, 1, \cdots, R-1\}\cap \Z_\epsilon}(z') + h^\epsilon_{T}(z') \mathbbm{1}_{\{R, R+1, \cdots, K_\epsilon-1\}\cap \Z_\epsilon}$. Summing over $z' \in \Z_\epsilon$ gives
		\[
		\sum_{z' \in \Z_\epsilon}G_{\tau_\epsilon}^\epsilon(0, z')\leq \sum_{z' \in \Z_\epsilon}\overline{h}^\epsilon(z')\leq CR+C\sum_{\substack{z' \in \Z_\epsilon \\ z'\geq R}} \min(1, \epsilon T^{-1/2})e^{-z'\min(1, \epsilon T^{-1/2})}.
		\]
		The sum on the right-hand side converges as a Riemann sum to $\int_R^\infty T^{-1/2}e^{-z T^{-1/2}} dz < 1$, so the right-hand side is bound uniformly in $z'\in \Z_\epsilon$ and in $\epsilon \in (0, 1)$ by a constant $C'$ which depends only on $T$. Importantly, the right-hand side of \eqref{ineq:determ_heat_bound} is completely deterministic, hence 
		\begin{align*}
		A_1&\leq C\sum_{z'' \in \Z_\epsilon} G^\epsilon_{\theta_\epsilon}(z, z'')\sum_{z' \in \Z_\epsilon}\overline{h}^\epsilon(z')\E\left[|X_0^\epsilon(z'-z'')-X_0^\epsilon(z'-z)|\right].
		\end{align*}
		By the Cauchy-Schwarz inequality, we have
		\[
		\E\left[|X_0^\epsilon(z'-z'')-X_0^\epsilon(z'-z)|\right]\leq \E\left[|X_0^\epsilon(z'-z'')-X_0^\epsilon(z'-z)|^2\right]^{1/2}.
		\]
		By Assumption \ref{assump:technical_X_0}.b, there exist constants $c_1, c_2>0$ depending only on $\sup_{\epsilon\in (0, 1)}\max_{z \in \Z_\epsilon}|\rho_0^\epsilon(z)|$ such that
		\[
		\E\left[|X_0^\epsilon(z_1)-X_0^\epsilon(z_2)|^2\right]\leq c_1\epsilon^{\kappa}+c_2|\rho_0^\epsilon(z_1)-\rho_0^\epsilon(z_2)|^2,
		\]
		for all $z_1, z_2 \in \Z_\epsilon$ and $\epsilon \in (0, 1)$. Recall from \eqref{def:rho_0_eps} in Definition \ref{def:general_init_data} that $\rho_0^\epsilon(\cdot)=\rho_0(K_\epsilon^{-1} \cdot)$, where $\rho_0 \in C^3(\S^1, [0, 2\mu])$, and hence it is Lipschitz-continuous with some constant $L_0>0$ independent of $\epsilon$. It follows that
		\[
		\E\left[|X_0^\epsilon(z'-z'')-X_0^\epsilon(z'-z)|^2\right]\leq c_1\epsilon^\kappa +c_2L_0^2|z-z''|^2,
		\]
		for all $z, z', z'' \in \Z_\epsilon$, where we note that the right-hand side does not depend on $z'$. Then, using the facts that $\sqrt{a_1+a_2}\leq 2^{1/2}(\sqrt{a_1}+\sqrt{a_2})$ for all $a_1, a_2\geq 0$ and $\sum_{z'' \in \Z_\epsilon}G^\epsilon_{\theta_\epsilon}(z, z'')=1$, we bound our first term as follows:
		\[
		A_1\leq C_1 \epsilon^{\kappa/2} +C_2\sum_{z'' \in \Z_\epsilon} G^\epsilon_{\theta_\epsilon}(z, z'')|z''-z|,
		\]
		where $C_1\coloneqq CC'\sqrt{2c_1}$ and $C_2\coloneqq CC'L_0\sqrt{2c_2}$. We have $y e^{-y}\leq 1$ for all $y \geq 0$. Then, by \eqref{ineq:exp_control_G} in Lemma \ref{lem:Green_estimates}, we compute
		\begin{align*}
			A_1&\leq C_1 \epsilon^{\kappa/2} +C_2\sum_{z'' \in \Z_\epsilon} G^\epsilon_{\theta_\epsilon}(z, z'')|z''-z|
			\\&\leq C_1\epsilon^{\kappa/2}+C_2\max(\theta_\epsilon^{1/2}, \epsilon)\sum_{z'' \in \Z_\epsilon} G^\epsilon_{\theta_\epsilon}(z, z'')e^{|z''-z|/\max(\theta_\epsilon^{1/2}, \epsilon)} 
			\\&\leq C_1\epsilon^{\kappa/2}+C_3\max(\theta_\epsilon^{1/2}, \epsilon)
		\end{align*}
		for some $C_3>0$. Next, we control $A_2$. By \eqref{ineq:basic_heat_estimate}, H\"older's inequality with $q\in (1, 2)$ and $\frac{1}{q}+\frac{q-1}{q}=1$, and Jensen's inequality, we have
		\begin{align*}
		&\int_{s}^{t}\sum_{\substack{x' \in \X_\epsilon \\ n_{x'}=1}}\G^\epsilon_{t-u}(x, x')|\cG^\epsilon_{BC}Q^\epsilon(x', \cdot)(X_{u}^\epsilon)|du
		\\&\leq C\int_s^t (t-u)^{-1/2} \epsilon \sum_{\substack{x' \in \X_\epsilon \\ n_{x'}=1}}|\cG^\epsilon_{BC}Q^\epsilon(x', \cdot)(X_{u}^\epsilon)| du
		\\&\leq C\left(\int_s^t (t-u)^{-q/2}du\right)^{1/q}\left(\int_0^T \epsilon \sum_{\substack{x' \in \X_\epsilon \\ n_{x'}=1}}|\cG^\epsilon_{BC}Q^\epsilon(x', \cdot)(X_{u}^\epsilon)|^{\frac{q}{q-1}}du\right)^{\frac{q-1}{q}}.
		\end{align*}
		Observe that $\int_s^t (t-u)^{-q/2}du=(t-s)^{1-q/2}/(1-q/2)$ and $t-s\leq \theta_\epsilon$. Using the concave Jensen inequality for $y\mapsto y^{\frac{q-1}{q}}$, we obtain
		\begin{align*}
		A_2&\leq C\theta_\epsilon^{\frac{2-q}{2q}} \left(\int_0^T \epsilon \sum_{\substack{x' \in \X_\epsilon \\ n_{x'}=1}}\E\left[|\cG^\epsilon_{BC}Q^\epsilon(x', \cdot)(X_{u}^\epsilon)|^{\frac{q}{q-1}}\right]du\right)^{\frac{q-1}{q}}
		\\& \leq C\theta_\epsilon^{\frac{2-q}{2q}} \sup_{u \in [0, T]}\max_{\substack{x' \in \X_\epsilon \\ n_{x'}=1}} \E\left[|\cG^\epsilon_{BC}Q^\epsilon(x', \cdot)(X_{u}^\epsilon)|^{\frac{q}{q-1}}\right]^{\frac{q-1}{q}}.
		\end{align*}
		The latter expectation is bounded by a constant depending only on $T$ and $q$. Indeed, a direct computation of the generator using \eqref{eq:GBC} shows that if $x'\coloneqq \epsilon^\kappa\mathbbm{1}_{\{z'\}}\in \X_\epsilon$ with $z'\in \Z_\epsilon$, then
		\[
		\cG_{BC}^\epsilon Q^\epsilon(x', \cdot)(X_u^\epsilon) = \mu_\epsilon X_u^\epsilon(z')-\frac{1}{2}X_u^\epsilon(z')(X_u^\epsilon(z')-\epsilon^\kappa), \quad u\geq 0.
		\]
		We re-write this expression in terms of the $Q^\epsilon$ defined in \eqref{def:gen_Q_eps} and apply Corollary \ref{cor:SLBP_p_moments} to Proposition \ref{prop:SLBP_moments} to obtain good control on its absolute moment. Let $x''\coloneqq 2\epsilon^\kappa\mathbbm{1}_{\{z'\}} \in \X_\epsilon$ be the configuration with two particles at $z'$. Then, 
		\[
		\cG_{BC}^\epsilon Q^\epsilon(x', \cdot)(X_t^\epsilon) =  \mu_\epsilon Q^\epsilon(x', X_t^\epsilon)-\frac{1}{2}Q^\epsilon(x'', X_t^\epsilon), \quad t\geq 0.  
		\]
		Thus, by the triangle inequality,  Jensen's inequality, and Corollary \ref{cor:SLBP_p_moments}, there exists a constant $C=C(T, q)>0$ such that
		\begin{align}
			\E\left[|\cG_{BC}^\epsilon Q^\epsilon(x', \cdot)(X_u^\epsilon)|^{\frac{q}{q-1}}\right]&\notag\leq 2^{\frac{1}{q-1}} \Bigg( \mu_\epsilon \sup_{u \in[0, T]}\max_{\substack{x' \in \X_\epsilon\\ n_{x'}=1}}\E\left[Q^\epsilon(x', X_u^\epsilon)^{\frac{q}{q-1}}\right]
			\\&+\sup_{u \in[0, T]}\max_{\substack{x'' \in \X_\epsilon\\ n_{x''}=2}}\E\left[Q^\epsilon(x'', X_u^\epsilon)^{\frac{q}{q-1}}\right]\Bigg)<C,\label{ineq:GBC_control}
		\end{align}
		uniformly in $u \in [0, T]$ and $\epsilon \in (0, 1)$. We conclude that $A_2\leq C\theta_\epsilon^{\frac{2-q}{2q}}$ for some $C=C(T, q)>0$. Consider now $A_3$. By \eqref{ineq:heat_time_holder_cts} from Lemma \ref{lem:Green_estimates}, we obtain, for all $a \in (0, 1/2)$, the bound
		\begin{align}
			|\G^\epsilon_{t-u}(x, x')-\G^\epsilon_{s-u}(x, x')|&\notag=\left|G_{t-u}^\epsilon(z, z')-G_{s-u}^\epsilon(z, z')\right| 
			\\&\leq C\epsilon f_u(s, t),\label{ineq:time_control}
		\end{align}
		for all $x, x' \in \X_\epsilon$ with $x=\epsilon^\kappa\mathbbm{1}_{\{z\}}$ and $x'=\epsilon^\kappa\mathbbm{1}_{\{z'\}}$, $a \in (0, 1/2)$, and $\epsilon \in (0, 1)$, and where 
		\[
		f_u(s, t)\coloneqq (s-u)^{-a-1/2}(t-s)^a +(s-u)^{-1/2}\Big((t-s)+(t-s)^{1/2}\Big).
		\]
		By the Cauchy-Schwarz inequality, we have 
		\begin{align*}
			&\sum_{\substack{x' \in \X_\epsilon \\ n_{x'}=1}}|\G^\epsilon_{t-u}(x, x')-\G^\epsilon_{s-u}(x, x')||\cG_{BC}^\epsilon Q^\epsilon(x', \cdot)(X^\epsilon_u)|
			\\&\leq \Bigg(\sum_{\substack{x' \in \X_\epsilon \\ n_{x'}=1}}|\G^\epsilon_{t-u}(x, x')-\G^\epsilon_{s-u}(x, x')|^2\Bigg)^{1/2}h^\epsilon(u)^{1/2},
		\end{align*}
		where
		\[
		h^\epsilon(u)\coloneqq \epsilon \sum_{\substack{x' \in \X_\epsilon \\ n_{x'}=1}}|\cG_{BC}^\epsilon Q^\epsilon(x', \cdot)(X^\epsilon_u)|^2.
		\]
		By \eqref{ineq:time_control}, we have 
		\[
		\Bigg(\sum_{\substack{x' \in \X_\epsilon \\ n_{x'}=1}}|\G^\epsilon_{t-u}(x, x')-\G^\epsilon_{s-u}(x, x')|^2\Bigg)^{1/2}\leq C\epsilon^{1/2}f_u(s, t).
		\]
		Therefore, 
		\begin{align*}
			&\sum_{\substack{x' \in \X_\epsilon \\ n_{x'}=1}}|\G^\epsilon_{t-u}(x, x')-\G^\epsilon_{s-u}(x, x')||\cG_{BC}^\epsilon Q^\epsilon(x', \cdot)(X^\epsilon_u)|
			\\& \leq Cf_u(s, t)h^\epsilon(u)^{1/2}.
		\end{align*}
		We have shown that
		\begin{align*}
			&\int_0^{s}\sum_{\substack{x' \in \X_\epsilon \\ n_{x'}=1}}|\G^\epsilon_{t-u}(x, x')-\G^\epsilon_{s-u}(x, x')||\cG_{BC}^\epsilon Q^\epsilon(x', \cdot)(X^\epsilon_u)|du
			\\&\leq C\int_0^s f_u(s, t)h^\epsilon(u)^{1/2} du.
		\end{align*}
		Next, we take $q \in (1, 2)$ and apply H\"older's inequality with $\frac{1}{q}+\frac{q-1}{q}=1$:
		\begin{equation}\label{ineq:holder_bd}
			\int_0^s f_u(s, t)h(u)^{1/2} du\leq \left(\int_0^s f_u(s, t)^qdu\right)^{1/q} \left(\int_0^s h^\epsilon(u)^{\frac{q}{2(q-1)}}du\right)^{(q-1)/q}.
		\end{equation}
		We bound these two factors in turn. By Jensen's inequality and straightforward calculations, we obtain
		\begin{align*}
			\int_0^s f_u(s, t)^q du \leq 2^{q-1}(t-s)^{aq}\frac{s^{1-q(a+\frac{1}{2})}}{1-q(a+\frac{1}{2})}+2^{q-1}[t-s+(t-s)^{1/2}]\frac{s^{1-q/2}}{1-q/2}.
		\end{align*}
		To avoid diverging terms as $s\searrow 0$, we require that $q<2/(2a+1)$ - observe that this is ok since $2/(2a+1)\in (1, 2)$ for all $a\in (0, 1/2)$. Hence, using that  
		\[
		(a_1+a_2+a_3)^b\leq 3^b\max(a_1, a_2, a_3)^b\leq 3^b(a_1^b+a_2^b+a_3^b),
		\]
		for all $a_1, a_2, a_3, b\geq 0$, we get
		\begin{align*}
		\sup_{\substack{0\leq s\leq t\leq T+\overline{\theta} \\ t-s\leq \theta_\epsilon }}\left(\int_0^s f_u(s, t)^qdu\right)^{1/q} &\leq\sup_{\substack{0\leq s\leq t\leq T+\overline{\theta} \\ t-s\leq \theta_\epsilon }}\left\{ C_1 (t-s)^{a}+C_2 (t-s)^{1/q}+C_3(t-s)^{\frac{1}{2q}}\right\}
		\\&\leq C_1 \theta_\epsilon^{a}+C_2 \theta_\epsilon^{1/q}+C_3\theta_\epsilon^{\frac{1}{2q}},
		\end{align*}
		for some $C_i=C_i(T, a, q)>0$, $i\in \{1, 2, 3\}$. Next, we control the second factor on the right-hand side of \eqref{ineq:holder_bd}. We observe that $(q-1)/q<1$ for all $q \in (1, 2/(2a+1))$. We can therefore apply the concave Jensen inequality and extend the range of integration up to $T+\overline{\theta}$ to obtain 
		\[
		\E\Bigg[\left(\int_0^s h^\epsilon(u)^{\frac{q}{2(q-1)}}du\right)^{(q-1)/q}\Bigg]\leq \E\Bigg[\int_0^{T+\overline{\theta}} h^\epsilon(u)^{\frac{q}{2(q-1)}}du\Bigg]^{(q-1)/q}
		\]
		We further note that $\frac{q}{2(q-1)}>1$ for all $q \in (1, 2/(2a+1))$. So by Jensen's inequality and using Corollary \ref{cor:SLBP_p_moments} as we did in \eqref{ineq:GBC_control} above, we have 
		\[
		\E\Bigg[\int_0^{T+\overline{\theta}} h^\epsilon(u)^{\frac{q}{2(q-1)}}du\Bigg]\leq \epsilon \sum_{\substack{x' \in \X_\epsilon \\ n_{x'}=1}}\int_0^{T+\overline{\theta}}\E[|\cG_{BC}^\epsilon Q^\epsilon(x', \cdot)(X^\epsilon_u)|^{\frac{q}{2(q-1)}}] du\leq C,
		\]
		for some constant $C=C(T, q)>0$, uniformly in $s \in [0, T]$ and $\epsilon \in (0, 1)$. Overall, we get 
		\begin{align*}
			A_3&\leq \sup_{\substack{0\leq s\leq t\leq T \\ t-s\leq \theta_\epsilon }}\left(\int_0^s f_u(s, t)^qdu\right)^{1/q} \E\Bigg[\left(\int_0^{T+\overline{\theta}} h^\epsilon(u)^{\frac{q}{2(q-1)}}du\right)^{(q-1)/q}\Bigg]
			\\& <C\left((t-s)^{a}+ (t-s)^{1/q}+(t-s)^{\frac{1}{2q}}\right)
			\\&\leq C\left(\theta_\epsilon^{a}+\theta_\epsilon^{1/q}+\theta_\epsilon^{\frac{1}{2q}}\right),
		\end{align*}
		for some $C=C(T, a, q)>0$. We now control $A_4$. One can check that the predictable quadratic variation of $M_\cdot^\epsilon(x)$ is given by 
		\[
		\langle M_\cdot^\epsilon(x)\rangle_t -\langle M_\cdot^\epsilon(x)\rangle_r = \int_r^t \Gamma^\epsilon_s(x) ds, \quad 0\leq r\leq t,
		\]
		where 
		\begin{equation}
			\begin{aligned}\label{eq:M_x_quad_var_2}
				\Gamma^\epsilon_s(x) &= \sum_{x'}\G^\epsilon_{T-s}(x, x')\cG^\epsilon_{BC} (Q^\epsilon(x', \cdot)^2)(X^\epsilon_s)
				\\&\quad\quad\quad\quad\quad\quad\quad\quad\quad\quad\quad\notag-2Q^\epsilon(x, X^\epsilon_s)\sum_{x'}\G^\epsilon_{T-s}(x, x')\cG_{BC}^\epsilon Q^\epsilon(x', \cdot)(X^\epsilon_s).
			\end{aligned}
		\end{equation}
		Observe that the right-hand side can be re-written as a linear combination of the terms $Q^\epsilon(x'', X_s^\epsilon)$, where $n_{x''} \in \{0, 1, \cdots, 2\}$. Thus, it follows from Jensen's inequality, Corollary \ref{cor:SLBP_p_moments} and similar arguments to those leading to \eqref{ineq:GBC_control} above that there exists a constant $c=c(T)>0$ with
		\[
		\sup_{s \in [0, T]}\E[\Gamma^\epsilon_s(x)^2]<c,
		\]
		for all $x \in \X_\epsilon$ with $n_x=1$ and $\epsilon \in (0, 1)$. Using the It\^o isometry, the Cauchy-Schwarz inequality, and Tonelli's Theorem, we obtain 
		\begin{align*}
			A_4\leq \E\left[\left(\int_{\tau_\epsilon}^{\tau_\epsilon+\theta_\epsilon} dM_s^\epsilon(x)\right)^2\right]^{1/2}
			&=\E\left[\int_{\tau_\epsilon}^{\tau_\epsilon+\theta_\epsilon}\Gamma^\epsilon_s(x) ds\right]^{1/2}
			\\&\leq \theta_\epsilon^{1/4} \E\left[\int_{\tau_\epsilon}^{\tau_\epsilon+\theta_\epsilon}\Gamma^\epsilon_s(x)^2 ds\right]^{1/4}\leq \theta_\epsilon^{1/4}\left(T \sup_{s \in [0, T+\overline{\theta}]}\E[\Gamma^\epsilon_s(x)^2]\right)^{1/4}.
		\end{align*}
	    This proves that $A_4\leq C\theta_\epsilon^{1/4}$ for some $C=C(T)>0$. Going back to \eqref{ineq:abs_diff}, and using the four bounds just derived and the fact that $\epsilon \sum_{z \in \Z_\epsilon} |\phi(\epsilon z)|\to \int |\phi|$, the proposition follows.
	\end{proof}
	
	\section{Proof of the central limit theorem}
	
	In this section, we prove the central limit theorem claimed in Theorem \ref{thm:CLT_Gaussian}: the non-equilibrium fluctuations field
	\[
	Y_t^\epsilon(\phi) = \epsilon^{1+\gamma-\kappa} \sum_{z \in \Z_\epsilon} \left(X_t^\epsilon(z)-\E\left[X_t^\epsilon(z)\right]\right)\phi(\epsilon z), \quad \phi \in C^\infty(\S^1), \quad t\geq 0,
	\]
	converges in distribution in the Skorohod topology on $\cD([0, T], C^\infty(\S^1)')$ to the unique (in law) solution to the the stochastic partial differential equation 
	\begin{equation}\label{eq:OU_Br}
		\begin{cases}
			\partial_t Y =\Big(\frac{1}{2}\partial_z^2+\mu-\rho\Big)Y +\sqrt{\rho}\partial_z\dot W^1+\sqrt{(\sigma^2+\mu^2)\rho+\rho^2/2}\dot W^2&\textnormal{on }[0, T]\times \S^1,
			\\Y_t\overset{d}{\to} Y_0 \textnormal{ as }t\searrow 0,
		\end{cases}
	\end{equation}
	where $\rho$, the hydrodynamic limit, solves \eqref{eq:FKPP} the Cauchy problem for the FKPP equation, $\dot \rW^i$, $i\in \{1, 2\}$, are independent Gaussian space-time white noises, and $Y_0$ is a $C^\infty(\S^1)'$-valued centred Gaussian variable with variance
	\[
	\E[Y_0(\phi)^2] = \int_{\S^1} \phi(z)^2 \rho_0(z)dz, \quad \phi \in C^\infty(\S^1). 
	\]
	We follow the strategy outlined in Section \ref{sec:pf_overview}.
	
	\subsection{Boltzmann-Gibbs principle and a corollary}
	
	The goal of this section is to prove the Boltzmann-Gibbs principle of Proposition \ref{prop:BGP} and to derive one of its implications which will be useful to show tightness of the fluctuations. In preparation for the application of that corollary in the next section, we work with slightly more general test functions than $\phi \in C^\infty(\S^1)$. We introduce for each $\epsilon \in (0, 1)$ the generalised test functions 
	\[
	\phi^\epsilon :\{(r, t)\in[0, T]^2:0\leq r\leq t\} \times \S^1\to \R, \quad ((r, t), z)\mapsto \phi^\epsilon_{r, t}(z),
	\]
	where for each in $t$ and $z$, the function is assumed to be continuous in $r$. We also assume that $\phi_{r, t}^\epsilon\in C^\infty(\S^1)$ for all $r$ and $t$, and that
	\begin{equation}\label{ineq:assump_time_phi}
		C_{\phi}=C_\phi(T)\coloneqq\sup_{\epsilon \in (0, 1)}\sup_{0\leq r\leq t \leq T} \|\phi_{r, t}^\epsilon\|_\infty <\infty,
	\end{equation}
	where $\|\cdot\|_{\infty}$ denotes the supremum norm on $\S^1$. We briefly recall the claim of Proposition \ref{prop:BGP}. Let $T>0$, $\gamma \in [-1/2, -5/16)$, and $\kappa \in [0, 1+2\gamma]$.  Consider the non-linear fluctuations field
	\begin{equation}\label{def:non_lin_fluct}
	F^\epsilon_r(\phi_{r, t}^\epsilon) \coloneqq \frac{\epsilon^{1+\gamma-\kappa}}{2}\sum_{z \in \Z_\epsilon}\left(X_r^\epsilon(z)(X_r^\epsilon(z)-\epsilon^\kappa)-\E[X_r^\epsilon(z)(X_r^\epsilon(z)-\epsilon^\kappa)]\right)\phi^\epsilon_{r, t}(\epsilon z),
	\end{equation}
	for all $0\leq r\leq t$. We show that 
	\begin{equation}\label{eq:BGP_goal}
		\limsup_{S\to \infty}\limsup_{\epsilon \searrow 0}\sup_{0\leq s< t\leq T} \E\left[\Bigg(\frac{1}{\epsilon^2S}\int_s^{s+\epsilon^2S}[F_r^\epsilon(\phi^\epsilon_{r, t})-Y_r^\epsilon(\widetilde{\rho}^{\:\epsilon}_r \phi^\epsilon_{r, t})]dr\Bigg)^2\right]=0.
	\end{equation}
	To prove this result, we adapt the proof of Proposition 3.1 in \cite{Boldrig1992}. The main new challenge is the fact that the local carrying capacity of our population is of order $O(\epsilon^{-\kappa})$, instead of simply $O(1)$ in the work of Boldrighini et al. where $\kappa=0$. We are however able to control this larger population using Theorem \ref{thm:QLLN}, our version of the quantitative convergence result for the $v$-functions. Indeed, in the weak competition regime with competition exponent $\kappa$, we find that the rate at which the $v$-functions of degree three or higher vanish improves precisely by a factor of $\epsilon^{\kappa}$. For $v$-functions of degree one or two, we use that the second moment of the $V$-functions is still of order $\epsilon^\kappa$.
	
	\begin{proof}[Proof of Proposition \ref{prop:BGP}]
		Throughout this proof, given a natural number $j$ and a site $z \in \Z_\epsilon$, we will denote $x^{j, z}=j\epsilon^\kappa\mathbbm{1}_{\{z\}}\in \X_\epsilon$ the configuration with $j$ particles at $z$. Recalling that 
		\[
		V^\epsilon(x^{2, z}, X^\epsilon_r;\rho_r^\epsilon) = (X^\epsilon_r(z)-\rho_r^\epsilon(z))^2-\epsilon^\kappa X_r^\epsilon(z),
		\]
		we easily show by adding and subtracting $\rho_r^\epsilon(z)^2$ that 
		\begin{equation}\label{eq:f_minus_y}
		F_r^\epsilon(\phi^\epsilon_{r, t})-Y_r^\epsilon(\widetilde{\rho}^{\:\epsilon}_r \phi^\epsilon_{r, t}) = \frac{\epsilon^{1+\gamma-\kappa}}{2}\sum_{z \in \Z_\epsilon} \big\{V^\epsilon(x^{2, z}, X_r^\epsilon;\rho_r^\epsilon)-\E[V^\epsilon(x^{2, z}, X_r^\epsilon;\rho_r^\epsilon)]\big\}\phi^\epsilon_{r, t}(\epsilon z). 
		\end{equation}
		By Theorem \ref{thm:QLLN}, we have $\E[V^\epsilon(x^{2, z}, X_r^\epsilon;\rho_r^\epsilon)]=O(\epsilon^{1+\kappa})$ uniformly in $z\in \Z_\epsilon$ and $r \in [0, t]$. Thus this term gives a negligible contribution to the left-hand side of \eqref{eq:BGP_goal}. It follows that the Boltzmann-Gibbs principle \eqref{eq:BGP_goal} is equivalent to 
		\begin{equation}\label{eq:A_eps_lim}
			\limsup_{S\to \infty}\limsup_{\epsilon \searrow 0}\sup_{0\leq s<t\leq T}A_\epsilon(s, t, S) = 0,
		\end{equation}
		where
		\begin{align*}
			A_\epsilon(s, t, S) &= \left(\frac{1}{2\epsilon^2 S}\right)^2 \int_s^{s+\epsilon^2 S} dr \int_s^{s+\epsilon^2 S}dr'\epsilon^{2+2\gamma-2\kappa} \sum_{z \in \Z_\epsilon}\sum_{z' \in \Z_\epsilon}\phi^\epsilon_{r, t}(\epsilon z)\phi^\epsilon_{r', t}(\epsilon z')
			\\&\quad\quad\quad\quad\quad\quad\quad\quad\quad\quad\quad\quad\quad\quad\quad\quad\quad\quad\times\E[V^\epsilon(x^{2, z}, X_r^\epsilon;\rho_r^\epsilon)V^\epsilon(x^{2, z'}, X_{r'}^\epsilon;\rho_{r'}^\epsilon)].
		\end{align*}
		We change variables in the inner integral using $\widetilde r=r'-r$ and use even symmetry of the region of integration about the diagonal $r=r'$ to pull out a factor of $2$:
		\begin{align*}
			A_\epsilon(s, t, S) &= 2\left(\frac{1}{2\epsilon^2 S}\right)^2 \int_s^{s+\epsilon^2 S} dr \int_0^{s-r+\epsilon^2 S}d\widetilde r\epsilon^{2+2\gamma-2\kappa} \sum_{z \in \Z_\epsilon}\sum_{z' \in \Z_\epsilon} \phi^\epsilon_{r, t}(\epsilon z)\phi^\epsilon_{r+\widetilde r, t}(\epsilon z')
			\\&\quad\quad\quad\quad\quad\quad\quad\quad\quad\quad\quad\quad\quad\quad\quad\quad\quad\quad\times \E[V^\epsilon(x^{2, z}, X_r^\epsilon; \rho_r^\epsilon)V^\epsilon(x^{2, z'}, X_{r+\widetilde r}^\epsilon; \rho_{r+\widetilde r}^\epsilon)].
		\end{align*}
		By the tower law of expectation and the Markov property, we can write
		\begin{align*}
			\E[V^\epsilon(x^{2, z}, X_r^\epsilon; \rho_r^\epsilon)V^\epsilon(x^{2, z'}, X_{r+\widetilde r}^\epsilon; \rho_{r+\widetilde r}^\epsilon)]&=\E[V^\epsilon(x^{2, z}, X_r^\epsilon; \rho_r^\epsilon)\E[V^\epsilon(x^{2, z'}, X_{r+\widetilde r}^\epsilon; \rho_{r+\widetilde r}^\epsilon)|X_r^\epsilon]].
		\end{align*}
		Proceeding as in the proof of Proposition \ref{prop:v_fcn_eq}, we derive the following equation:
		\begin{align*}
			&\E[V^\epsilon(x^{2, z'}, X_{r+\widetilde r}^\epsilon; \rho_{r+\widetilde r}^\epsilon)|X_r^\epsilon]
			\\&= \sum_{\substack{x_1\in \X_\epsilon\\ n_{x_1}=2}}\G^\epsilon_{\widetilde r}(x^{2, z'}, x_1)V^\epsilon(x_1, X_r^\epsilon; \rho^\epsilon_r)
			\\&+ \int_0^{\widetilde r} ds \sum_{\substack{x_1\in \X_\epsilon\\ n_{x_1}=2}} \G^\epsilon_{\widetilde r-s}(x^{2, z'}, x_1)\sum_{z \in \supp(x_1)}\sum_{h=-n_{x_1}(z)}^1 c_h^\epsilon(n_{x_1}(z), \rho_{r+s}^\epsilon(z))\E[V^\epsilon(x_1^{(z, h)}, X_{r+s}^\epsilon; \rho_{r+s}^\epsilon)|X_r^\epsilon].
		\end{align*}
		Denote the integral on the right-hand side by $I_\epsilon(\widetilde r)$. We have 
		\begin{align*}
			&\E[V^\epsilon(x^{2, z}, X_r^\epsilon; \rho_r^\epsilon)I_\epsilon(\widetilde{r})]
			\\&=\int_0^{\widetilde r} ds\Bigg[ \sum_{\substack{x_1\in \X_\epsilon\\ n_{x_1}=2}} \G^\epsilon_{\widetilde r-s}(x^{2, z'}, x_1)
			\\&\quad\quad\quad\quad\times\sum_{z \in \supp(x_1)}\sum_{h=-n_{x_1}(z)}^1 c_h^\epsilon(n_{x_1}(z), \rho_{r+s}^\epsilon(z))\E\left[V^\epsilon(x^{2, z}, X_r^\epsilon; \rho_r^\epsilon)\E[V^\epsilon(x_1^{(z, h)}, X_{r+s}^\epsilon; \rho_{r+s}^\epsilon)|X_r^\epsilon]\right]\Bigg].
		\end{align*}
		Moreover, by the Cauchy-Schwarz inequality
		\begin{align*}
			&\E\left[V^\epsilon(x^{2, z}, X_r^\epsilon; \rho_r^\epsilon)\E[V^\epsilon(x_1^{(z, h)}, X_{r+s}^\epsilon; \rho_{r+s}^\epsilon)|X_r^\epsilon]\right]
			\\&=\E\left[V^\epsilon(x^{2, z}, X_r^\epsilon; \rho_r^\epsilon)V^\epsilon(x_1^{(z, h)}, X_{r+s}^\epsilon; \rho_{r+s}^\epsilon)\right]
			\\&\leq \E\left[V^\epsilon(x^{2, z}, X_r^\epsilon; \rho_r^\epsilon)\right]^{1/2}\E\left[V^\epsilon(x_1^{(z, h)}, X_{r+s}^\epsilon; \rho_{r+s}^\epsilon)\right]^{1/2}
		\end{align*}
		By Definition \ref{def:V_fcn}, Corollary \ref{cor:SLBP_p_moments}, and part two of Lemma \ref{lem:FKPP_approx_properties}, the right-hand side is dominated by a constant which depends only on $T$, uniformly in $r, s \in [0, T]$, $z \in \Z_\epsilon$, $h \geq 1$, and $\epsilon \in (0, 1)$. Similarly, by point two of Lemma \ref{lem:FKPP_approx_properties}, the coefficients $c_h^\epsilon(n_{x_1}(z), \rho_{r+s}^\epsilon(z))$ are uniformly bounded in absolute value by some constant which depends only on $T$. Overall, we get
		\[
		\E[V^\epsilon(x^{2, z}, X_r^\epsilon; \rho_r^\epsilon)I_\epsilon(\widetilde{r})]=O(\widetilde{r})=O(\epsilon^2 S),
		\] 
		and therefore
		\begin{align*}
			&\E[V^\epsilon(x^{2, z}, X_r^\epsilon; \rho_r^\epsilon)V^\epsilon(x^{2, z'}, X_{r+\widetilde r}^\epsilon; \rho_{r+\widetilde r}^\epsilon)]
			\\&= \sum_{\substack{x_1\in \X_\epsilon\\ n_{x_1}=2}}\G^\epsilon_{\widetilde r}(x^{2, z'}, x_1)\E[V^\epsilon(x^{2, z}, X_r^\epsilon; \rho_r^\epsilon)V^\epsilon(x_1, X_r^\epsilon; \rho^\epsilon_r)]+O(\epsilon^2 S).
		\end{align*}
		Using \eqref{ineq:assump_time_phi}, we see that the big-$O$ term contributes to $A_\epsilon(s, t, S)$ the quantity
		\begin{align*}
			&2\left(\frac{1}{2\epsilon^2 S}\right)^2 \int_s^{s+\epsilon^2 S} dr \int_0^{s-r+\epsilon^2 S}d\widetilde r\epsilon^{2+2\gamma-2\kappa} \sum_{z \in \Z_\epsilon}\sum_{z' \in \Z_\epsilon} \phi^\epsilon_{r, t}(\epsilon z)\phi^\epsilon_{r+\widetilde r, t}(\epsilon z')O(\epsilon^{2} S)
			\\&=O(\epsilon^{1-\kappa+1+2\gamma-\kappa}S)C_\phi^2
			\\&=O(\epsilon^{1-\kappa+1+2\gamma-\kappa}S),
		\end{align*}
		which vanishes as $\epsilon\to 0$, and then $S \to \infty$, since $\kappa \in [0, 1+2\gamma]$ with $\kappa <3/8$ because ${\gamma \in [-1/2, -5/16)}$. Thus
		\begin{align*}
			A_\epsilon(s, t, S)& = 2\left(\frac{1}{2\epsilon^2 S}\right)^2 \int_s^{s+\epsilon^2 S} dr \int_0^{s-r+\epsilon^2 S}d\widetilde r\epsilon^{2+2\gamma-2\kappa} \sum_{z \in \Z_\epsilon}\sum_{z' \in \Z_\epsilon} \phi^\epsilon_{r, t}(\epsilon z)\phi^\epsilon_{r+\widetilde r, t}(\epsilon z')
			\\&\quad\quad\quad\quad\quad\quad\quad\quad\quad\times \sum_{\substack{x_1\in \X_\epsilon\\ n_{x_1}=2}}\G^\epsilon_{\widetilde r}(x^{2, z'}, x_1)\E[V^\epsilon(x^{2, z}, X_r^\epsilon; \rho_r^\epsilon)V^\epsilon(x_1, X_r^\epsilon; \rho^\epsilon_r)]+o(1),
		\end{align*}
		as $\epsilon\to 0$, then $S\to \infty$. To control the remaining integral, we distinguish three different cases in terms of $x_1$: $n_{x_1}(z)=2$, $n_{x_1}=0$, and $n_{x_1}(z)=1$. Suppose first that $x_1 = x^{2, z}=2\epsilon^\kappa\mathbbm{1}_{\{z\}}$. Then, we need to control
		\begin{align*}
			A_{0, \epsilon}(s, t, S) &\coloneqq \left(\frac{1}{2\epsilon^2 S}\right)^2 \int_s^{s+\epsilon^2 S} dr \int_0^{s-r+\epsilon^2 S}d\widetilde r\epsilon^{2+2\gamma-2\kappa} \sum_{z \in \Z_\epsilon}\sum_{z' \in \Z_\epsilon} \phi^\epsilon_{r, t}(\epsilon z)\phi^\epsilon_{r+\widetilde r, t}(\epsilon z') 
			\\&\quad\quad\quad\quad\quad\quad\quad\quad\quad\quad\quad\quad\quad\quad\quad\quad\times\G^\epsilon_{\widetilde r}(x^{2, z'}, x^{2, z})\E[V^\epsilon(x^{2, z}, X_{r}^\epsilon; \rho_r^\epsilon)^2].
		\end{align*}
		By the Green function estimate \eqref{ineq:basic_heat_estimate}, we have the bound
		\[
		\G^\epsilon_{\widetilde r}(x^{2, z'}, x^{2, z})\leq C_1 G^\epsilon_{\widetilde r}(z', z)\epsilon/\sqrt{\widetilde{r}},
		\]
		for some $C_1=C_1(T)$. Next, we show that for some $C_2=C_2(T)>0$, 
		\begin{equation}\label{ineq:goal_V_sq}
			\E[V^\epsilon(x^{2, z}, X_{r}^\epsilon; \rho_r^\epsilon)^2]\leq C_2\epsilon^\kappa,
		\end{equation}
		uniformly in $z\in \Z_\epsilon$ and $r\in [0, T]$. To see this, we note that by a direct computation
		\begin{align*}
			V^\epsilon(x^{2, z}, X_r^\epsilon;\rho_r^\epsilon)^2&=V^\epsilon(x^{4, z}, X_r^\epsilon; \rho_r^\epsilon)+\epsilon^\kappa\Bigg[6X_{r}^\epsilon(z)^3+6X_{r}^\epsilon(z)\rho_r^\epsilon(z)^2-12X_{r}^\epsilon(z)^2 \rho_r^\epsilon(z)
			\\&\quad\quad\quad\quad\quad\quad\quad\quad\quad\quad-2X_{r}^\epsilon(z)(X_{r}^\epsilon(z)-\rho_r^\epsilon(z))^2+13\epsilon^\kappa X_{r}^\epsilon(z)^2\Bigg].
		\end{align*}
		We take the expectation of both sides and apply Theorem \ref{thm:QLLN} and part two of Lemma \ref{lem:FKPP_approx_properties}, which yields
		\[
		\E[V^\epsilon(x^{2, z}, X_{r}^\epsilon;\rho_r^\epsilon)^2] = o(\epsilon^{1+\kappa}) + \epsilon^\kappa \E[f(X_{r}^\epsilon(z), \rho_r^\epsilon(z))],
		\]
		where 
		\[
		f(x, \rho)=6x^3+6x\rho^2-12x^2\rho-2x(x-\rho)^2+13\epsilon^\kappa x^2.
		\]
		By Corollary \ref{cor:SLBP_p_moments}, we have 
		\[
		\sup_{\epsilon \in (0, 1)}\sup_{r \in [0, T]}\max_{z \in \Z_\epsilon}|\E\left[f(X_r^\epsilon(z), \rho_r^\epsilon(z))\right]|<\infty,
		\]
		which proves \eqref{ineq:goal_V_sq}. Using the Green function estimate above, the bound \eqref{ineq:goal_V_sq}, the fact that ${\sum_{z\in \Z_\epsilon}G^\epsilon_{\widetilde r}(z', z)=1}$ for all $z'$ and $\widetilde{r}$, and \eqref{ineq:assump_time_phi}, we obtain 
		\begin{align*}
			&\left|\epsilon^{2+2\gamma-2\kappa} \sum_{z \in \Z_\epsilon}\sum_{z' \in \Z_\epsilon} \phi^\epsilon_{r, t}(\epsilon z)\phi^\epsilon_{r+\widetilde r, t}(\epsilon z')\G^\epsilon_{\widetilde r}(x^{2, z'}, x^{2, z})\E[V^\epsilon(x^{2, z}, X_{r}^\epsilon; \rho_r^\epsilon)^2]\right|
			\\& \leq C_1C_2C_\phi^2 \epsilon^{2+2\gamma-\kappa} /\sqrt{\widetilde{r}}.
		\end{align*}
		Consequently, 
		\begin{align*}
			|A_{0, \epsilon}(s, t, S)|&\leq C \left(\frac{1}{2\epsilon^2 S}\right)^2 \int_s^{s+\epsilon^2 S}dr \int_0^{s-r+\epsilon^2 S}d \widetilde{r} \epsilon/\sqrt{\widetilde{r}}
			\\&\leq C \left(\frac{1}{2\epsilon^2 S}\right)^2 \frac{4}{3}(\epsilon^2 S)^{3/2} \epsilon= CS^{-1/2},
		\end{align*}
		where $C\coloneqq C_1C_2C_\phi^2$. Thus $A_{0, \epsilon}(s, t, S)\to 0$ as $\epsilon \to 0$, and then $S\to \infty$. Next, we assume that $x_1(z)=0$, meaning that we have to control
		\begin{align*}
			A_{1, \epsilon}(s, t, S) &\coloneqq \left(\frac{1}{2\epsilon^2 S}\right)^2 \int_s^{s+\epsilon^2 S} dr \int_0^{s-r+\epsilon^2 S}d\widetilde r\epsilon^{2+2\gamma-2\kappa} \sum_{z \in \Z_\epsilon}\sum_{z' \in \Z_\epsilon} \phi^\epsilon_{r, t}(\epsilon z)\phi^\epsilon_{r+\widetilde{r}, t}(\epsilon z') 
			\\&\quad\quad\quad\quad\quad\quad\quad\quad\quad\quad\quad\times\sum_{\substack{x_1\in \X_\epsilon\\ n_{x_1}=2 \\ x_1(z)=0}}\G^\epsilon_{\widetilde r}(x^{2, z'}, x_1)\E[V^\epsilon(x^{2, z}, X_{r}^\epsilon; \rho_r^\epsilon)V^\epsilon(x_1, X_{r}^\epsilon; \rho_r^\epsilon)].
		\end{align*}
		By part one of Lemma \ref{lem:V_fcn_properties_LLN} and Theorem \ref{thm:QLLN}, we have 
		\[
		\E[V^\epsilon(x^{2, z}, X_{r}^\epsilon; \rho_r^\epsilon)V^\epsilon(x_1, X_{r}^\epsilon; \rho_r^\epsilon)]=\E[V^\epsilon(x^{2, z}+x_1, X_{r}^\epsilon; \rho_r^\epsilon)]=o(\epsilon^{1+\kappa}),
		\]
		as $\epsilon \searrow 0$, uniformly in $z \in \Z_\epsilon$, $x_1 \in \X_\epsilon$ with $n_{x_1}=2$, and $r \in [0, T]$. Using \eqref{ineq:assump_time_phi} and the fact that $\sum_{x_1}\G^\epsilon_{\widetilde r}(x^{2, z'}, x_1)<1$, this implies that 
		\begin{align*}
		|A_{1, \epsilon}(s, t, S)| &= o(\epsilon^{1+2\gamma-\kappa})C_\phi^2\left(\frac{1}{2\epsilon^2 S}\right)^2 \int_s^{s+\epsilon^2 S} dr \int_0^{s-r+\epsilon^2 S}d\widetilde r,
		\\&=o(\epsilon^{1+2\gamma-\kappa}),
		\end{align*}
		which tends to zero as $\epsilon\searrow 0$ since $\kappa \in [0, 1+2\gamma]$. The last case to check is when $n_{x_1}(z)=1$, for which we control
		\begin{align*}
			A_{2, \epsilon}(s, t, S) &\coloneqq \left(\frac{1}{2\epsilon^2 S}\right)^2 \int_s^{s+\epsilon^2 S} dr \int_0^{s-r+\epsilon^2 S}d\widetilde r\epsilon^{2+2\gamma-2\kappa} \sum_{z \in \Z_\epsilon}\sum_{z' \in \Z_\epsilon} \phi^\epsilon_{r, t}(\epsilon z)\phi^\epsilon_{r+\widetilde{r}, t}(\epsilon z') 
			\\&\quad\quad\quad\quad\quad\quad\quad\quad\quad\quad\quad\times\sum_{\substack{x_1\in \X_\epsilon\\ n_{x_1}=2 \\ n_{x_1}(z)=1}}\G^\epsilon_{\widetilde r}(x^{2, z'}, x_1)\E[V^\epsilon(x^{2, z}, X_{r}^\epsilon; \rho_r^\epsilon)V^\epsilon(x_1, X_{r}^\epsilon; \rho_r^\epsilon)].
		\end{align*}
		We write $x_1=x^{1, z}+x_1'$ for some configuration $x_1' \in \X_\epsilon$ with a single particle not located at $z$. By part three of Lemma \ref{lem:V_fcn_properties_LLN}, we obtain 
		\[
		V^\epsilon(x^{2, z}, X_{r}^\epsilon; \rho_r^\epsilon)V^\epsilon(x^{1, z}, X_{r}^\epsilon ;\rho_r^\epsilon) = \sum_{j=0}^3 c_j^\epsilon V^\epsilon(x^{j, z}, X_r^\epsilon;\rho_r^\epsilon).
		\]
		The first point in Lemma \ref{lem:V_fcn_properties_LLN} then implies 
		\begin{align*}
			V^\epsilon(x^{2, z}, X_{r}^\epsilon; \rho_r^\epsilon)V^\epsilon(x_1, X_{r}^\epsilon ;\rho_r^\epsilon)&=V^\epsilon(x_1', X_{r}^\epsilon;\rho_r^\epsilon)\sum_{j=0}^3 c_j^\epsilon V^\epsilon(x^{j, z}, X_{r}^\epsilon;\rho_r^\epsilon)
			\\&= \sum_{j=0}^3 c_j^\epsilon V^\epsilon(x_1'\cup x^{j, z}, X_{r}^\epsilon;\rho_r^\epsilon).
		\end{align*}
		By Theorem \ref{thm:QLLN}, we have $\E[V^\epsilon(x_1'\cup x^{j, z}, X_{r}^\epsilon;\rho_r^\epsilon)]=O(\epsilon^{1+\kappa})$ uniformly in $z \in \Z_\epsilon$, $j \in \{0, 1, 2, 3\}$, $r \in [0, T]$ since the configuration $x_1'\cup x^{j, z}$ contains at least one particle. We note using the decomposition $x_1=x^{1, z}+x_1'$ and the Green function estimate \eqref{ineq:basic_heat_estimate} that 
		\[
		\G^\epsilon_{\widetilde r}(x^{2, z'}, x_1) = G^\epsilon_{\widetilde r}(z', z)\G^\epsilon_{\widetilde r}(x^{1, z'}, x_1') \leq C\G^\epsilon_{\widetilde r}(x^{1, z'}, x_1') \epsilon/ \sqrt{\widetilde r},
		\]
		for some $C=C(T)>0$. Hence, using that $\sum_{x_1}\G^\epsilon_{\widetilde r}(x^{1, z'}, x_1')<1$ and that $\kappa \in [0, 1+2\gamma]$, we bound the integrand of $A_{2, \epsilon}(s, t, S)$ as follows:
		\begin{align*}
			&\Bigg|\epsilon^{2+2\gamma-2\kappa} \sum_{z \in \Z_\epsilon}\sum_{z' \in \Z_\epsilon} \phi^\epsilon_{r, t}(\epsilon z)\phi^\epsilon_{r+\widetilde{r}, t}(\epsilon z') \sum_{\substack{x_1\in \X_\epsilon\\ n_{x_1}=2 \\ n_{x_1}(z)=1}}\G^\epsilon_{\widetilde r}(x^{2, z'}, x_1)\E[V^\epsilon(x^{2, z}, X_{r}^\epsilon; \rho_r^\epsilon)V^\epsilon(x_1, X_{r}^\epsilon; \rho_r^\epsilon)]\Bigg|
			\\&\leq O(\epsilon^{2+2\gamma-\kappa}) C_\phi^2 /\sqrt{\widetilde{r}}
			\\&=O(\epsilon)/\sqrt{\widetilde{r}},
		\end{align*}
		for all $\epsilon \in (0, 1)$ and $\widetilde{r} \in [0, t-r+\epsilon^2 S]$. This estimate a similar calculation as for $A_{0, \epsilon}(s, t, S)$ above show that $|A_{2, \epsilon}(s, t, S)|=O(S^{-1/2})$ uniformly in $\epsilon\in (0, 1)$ which vanishes when $S\to \infty$. Thus \eqref{eq:A_eps_lim} holds, which concludes the proof of the Proposition. 	
	\end{proof}
	
	In the proof of the central limit theorem, we will control time integrals of $F^\epsilon_r$ using the following corollary to the Boltzmann-Gibbs principle. 
	
	\begin{cor}\label{cor:BGP_control}
		Fix $0\leq \theta\leq \overline{\theta}$ and let $t \in [0, T+\overline{\theta}]$ and $s \in [0, T]$ with $0\leq t-s\leq \theta$. Take a generalised test function $\phi^\epsilon$ as in Proposition \ref{prop:BGP}. Then, there exist constants $c_i=c_i(T, C_\phi)>0$, $i=1, 2$, such that
		\begin{equation}\label{lem:bgp_cor_bd}
		\E\left[\sup_{\tau \in [s, t]}\left(\int_{s}^{\tau} [F_r^\epsilon(\phi^\epsilon_{r, t})-Y_r^\epsilon(\widetilde{\rho}_r^{\:\epsilon} \phi^\epsilon_{r, t})]dr\right)^2\right]\leq c_1 \theta^2+c_2\epsilon^{4+2\gamma-2\kappa},
		\end{equation}
		for all $\epsilon \in (0, 1)$. Additionally, for any $\phi \in C^\infty(\S^1)$, there exist $c_i=c_i(T, \phi)>0$, $i =1, 2$ such that  
		\begin{equation}\label{lem:bgp_cor_bd_sup}
			\E\left[\sup_{\substack{0\leq s \leq t\leq T+\overline{\theta} \\ t-s\leq \theta}}\left(\int_{s}^{t} [F_r^\epsilon(\phi)-Y_r^\epsilon(\widetilde{\rho}_r^{\:\epsilon} \phi)]dr\right)^2\right]\leq c_1 \theta^2+c_2\epsilon^{2+2\gamma-2\kappa},
		\end{equation}
		for all $\epsilon \in (0, 1)$.
	\end{cor}
	\begin{proof}
		We use Proposition \ref{prop:BGP} and the assumption that $0\leq t-s\leq \theta$ to control the expectation on the left-hand side of \eqref{lem:bgp_cor_bd}. We partition the time integral over intervals of length $\epsilon^2$, plus an additional shorter one:
		\begin{align*}
			&\int_{s}^{\tau} [F_r^\epsilon(\phi^\epsilon_{r, t})-Y_r^\epsilon(\widetilde{\rho}_r^{\:\epsilon} \phi^\epsilon_{r, t})]dr 
			\\&= \sum_{k=1}^{\lfloor\epsilon^{-2}(\tau-s)\rfloor}\int_{s+\epsilon^2 (k-1)}^{s+\epsilon^2 k}[F_r^\epsilon(\phi^\epsilon_{r, t})-Y_r^\epsilon(\widetilde{\rho}_r^{\:\epsilon}\phi^\epsilon_{r, t})]dr
			\\&+\int_{s+\epsilon^2\lfloor\epsilon^{-2}(\tau-s)\rfloor}^{\tau}[F_r^\epsilon(\phi^\epsilon_{r, t})-Y_r^\epsilon(\widetilde{\rho}_r^{\:\epsilon} \phi^\epsilon_{r, t})]dr.
			\\&=\epsilon^2\lfloor\epsilon^{-2}(\tau-s)\rfloor\frac{1}{\lfloor\epsilon^{-2}(\tau-s)\rfloor}\sum_{k=1}^{\lfloor\epsilon^{-2}(\tau-s)\rfloor}\frac{1}{\epsilon^2}\int_{s+\epsilon^2 (k-1)}^{s+\epsilon^2 k}[F_r^\epsilon(\phi^\epsilon_{r, t})-Y_r^\epsilon(\widetilde{\rho}_r^{\:\epsilon}\phi^\epsilon_{r, t})]dr
			\\&+\int_{s+\epsilon^2\lfloor\epsilon^{-2}(\tau-s)\rfloor}^{\tau}[F_r^\epsilon(\phi^\epsilon_{r, t})-Y_r^\epsilon(\rho_r^{\:\epsilon} \phi^\epsilon_{r, t})]dr.
		\end{align*}
		By the inequality $(a+b)^2\leq 2(a^2+b^2)$ for all $a, b \in \R$ and then Jensen's inequality on the first term, we obtain 
		\begin{align}
			&\left(\int_{s}^{\tau} [F_r^\epsilon(\phi^\epsilon_{r, t})-Y_r^\epsilon(\widetilde{\rho}_r^{\:\epsilon} \phi^\epsilon_{r, t})]dr\right)^2
			\\&\notag\leq 2\epsilon^4\lfloor \epsilon^{-2}(t-s)\rfloor\sum_{k=1}^{\lfloor\epsilon^{-2}(\tau-s)\rfloor}\left(\frac{1}{\epsilon^2}\int_{s+\epsilon^2 (k-1)}^{s+\epsilon^2 k}[F_r^\epsilon(\phi^\epsilon_{r, t})-Y_r^\epsilon(\widetilde{\rho}_r^{\:\epsilon}\phi^\epsilon_{r, t})]dr\right)^2
			\\&+2\left(\int_{s+\epsilon^2\lfloor\epsilon^{-2}(\tau-s)\rfloor}^{\tau}[F_r^\epsilon(\phi^\epsilon_{r, t})-Y_r^\epsilon(\widetilde{\rho}_r^{\:\epsilon} \phi^\epsilon_{r, t})]dr\right)^2.\label{eq:decomp_BGP_application}
		\end{align}
		Taking the supremum over $\tau \in [s, t]$ and using that $0\leq t-s\leq \theta$, we find
		\begin{align}
			&\E\left[\sup_{\tau\in [s, t]}\left(\int_{s}^{\tau} [F_r^\epsilon(\phi^\epsilon_{r, t})-Y_r^\epsilon(\widetilde{\rho}_r^{\:\epsilon} \phi^\epsilon_{r, t})]dr\right)^2\right]
			\\&\notag\leq 2\epsilon^2\theta \sum_{k=1}^{\lfloor\epsilon^{-2}\theta\rfloor}\E\left[\left(\frac{1}{\epsilon^2}\int_{s+\epsilon^2 (k-1)}^{s+\epsilon^2 k}[F_r^\epsilon(\phi^\epsilon_{r, t})-Y_r^\epsilon(\widetilde{\rho}_r^{\:\epsilon}\phi^\epsilon_{r, t})]dr\right)^2\right]
			\\&+2\E\left[\sup_{\tau\in [s, t]}\left(\int_{s+\epsilon^2\lfloor\epsilon^{-2}(\tau-s)\rfloor}^{\tau}[F_r^\epsilon(\phi^\epsilon_{r, t})-Y_r^\epsilon(\widetilde{\rho}_r^{\:\epsilon} \phi^\epsilon_{r, t})]dr\right)^2\right].\label{eq:decomp_BGP_application}
		\end{align}
		We first control the expectation of the first term on the right-hand side. We have 
		\begin{align*}
			&2\epsilon^2\theta \sum_{k=1}^{\lfloor\epsilon^{-2}\theta\rfloor}\E\left[\left(\frac{1}{\epsilon^2}\int_{s+\epsilon^2 (k-1)}^{s+\epsilon^2 k}[F_r^\epsilon(\phi^\epsilon_{r, t})-Y_r^\epsilon(\widetilde{\rho}_r^{\:\epsilon}\phi^\epsilon_{r, t})]dr\right)^2\right]
			\\&\leq 2\theta^2\max_{k \in \{1, \cdots, \lfloor \epsilon^{-2} \theta\rfloor\}}\E\left[\left(\frac{1}{\epsilon^2}\int_{s+\epsilon^2 (k-1)}^{s+\epsilon^2 k}[F_r^\epsilon(\phi^\epsilon_{r, t})-Y_r^\epsilon(\widetilde{\rho}_r^{\:\epsilon}\phi^\epsilon_{r, t})]dr\right)^2\right]
		\end{align*}
		By Proposition \ref{prop:BGP}, the right-hand side is bounded by a constant $c_1=c_1(T)>0$ uniformly over $\epsilon \in (0, 1)$, $k \in \{1, \cdots, \lfloor \epsilon^{-2} \theta\rfloor\}$, and $s\in [0, T+\theta]$. So
		\[
		2\epsilon^2\theta \sum_{k=1}^{\lfloor\epsilon^{-2}\theta\rfloor}\E\left[\left(\frac{1}{\epsilon^2}\int_{s+\epsilon^2 (k-1)}^{s+\epsilon^2 k}[F_r^\epsilon(\phi^\epsilon_{r, t})-Y_r^\epsilon(\widetilde{\rho}_r^{\:\epsilon}\phi^\epsilon_{r, t})]dr\right)^2\right]\leq c_1 \theta^2.
		\]
		We now bound the expectation of the second term on the right-hand side of \eqref{eq:decomp_BGP_application}. Recall from \eqref{eq:f_minus_y} that 
		\[
		F_r^\epsilon(\phi^\epsilon_{r, t})-Y_r^\epsilon(\widetilde{\rho}_r^{\:\epsilon} \phi^\epsilon_{r, t}) = \frac{\epsilon^{1+\gamma-\kappa}}{2}\sum_{z \in \Z_\epsilon} \big\{V^\epsilon(x^{2, z}, X_r^\epsilon;\rho_r^{\:\epsilon})-\E[V^\epsilon(x^{2, z}, X_r^\epsilon;\widetilde{\rho}_r^{\:\epsilon})]\big\}\phi^\epsilon_{r, t}(\epsilon z).
		\]
		We introduce the notation $\tau'\coloneqq s+\epsilon^2\lfloor\epsilon^{-2}(\tau-s)\rfloor$ and make the change of variables $r'=\epsilon^{-2}(r-\tau')$ to obtain
		\begin{align*}
			&\int_{\tau'}^{\tau}[F_r^\epsilon(\phi^\epsilon_{r, t})-Y_r^\epsilon(\widetilde{\rho}^{\:\epsilon}_r \phi^\epsilon_{r, t})]dr
			\\&=\epsilon^{2+\gamma-\kappa}\int_0^{\epsilon^{-2}(\tau-\tau')} \frac{\epsilon}{2}\sum_{z \in \Z_\epsilon}\big\{V^\epsilon(x^{2, z}, X_{\tau'+\epsilon^2 r'}^\epsilon;\widetilde{\rho}_{\tau'+\epsilon^2 r'}^{\:\epsilon})
			\\&\quad\quad\quad\quad\quad\quad\quad\quad\quad\quad\quad\quad\quad\quad\quad\quad\quad-\E[V^\epsilon(x^{2, z}, X_{\tau'+\epsilon^2 r'}^\epsilon;\widetilde{\rho}_{\tau'+\epsilon^2 r'}^{\:\epsilon})]\big\}\phi^\epsilon_{\tau'+\epsilon^2 r', u}(\epsilon z)dr'.
		\end{align*}
		We remark that $\epsilon^{-2}(\tau-\tau')\leq 1$. Therefore, by Jensen's inequality twice, assumption \eqref{ineq:assump_time_phi} on the test function, the fact that the second moment of $V^\epsilon$ is uniformly bounded by Corollary \ref{cor:SLBP_p_moments}, and part two of Lemma \ref{lem:FKPP_approx_properties}, we obtain for some ${c_2=c_2(T, C_\phi)>0}$ the bound 
		\[
		\E\left[\sup_{\tau\in [s, t]}\left(\int_{s+\epsilon^2\lfloor\epsilon^{-2}(\tau-s)\rfloor}^{\tau}[F_r^\epsilon(\phi^\epsilon_{r, t})-Y_r^\epsilon(\widetilde{\rho}^{\:\epsilon}_r \phi^\epsilon_{r, t})]dr\right)^2\right]\leq c_2 \epsilon^{4+2\gamma-2\kappa}.
		\]
		This shows \eqref{lem:bgp_cor_bd}. For \eqref{lem:bgp_cor_bd_sup}, we use the decomposition
		\begin{equation}\label{eq:sup_cas_decomp}
		\begin{aligned}
			\int_{s}^{t} [F_r^\epsilon(\phi)-Y_r^\epsilon(\widetilde{\rho}_r^{\:\epsilon} \phi)]dr& = \sum_{k=\lceil \epsilon^{-2}s\rceil+1}^{\lfloor \epsilon^{-2}t\rfloor} \int_{\epsilon^2(k-1)}^{\epsilon^2 k} [F_r^\epsilon(\phi)-Y_r^\epsilon(\widetilde{\rho}_r^{\:\epsilon} \phi)]dr
			\\&+\int_s^{\epsilon^2\left(\lceil \epsilon^{-2}s\rceil+1\right)} [F_r^\epsilon(\phi)-Y_r^\epsilon(\widetilde{\rho}_r^{\:\epsilon} \phi)] dr
			\\&+\int_{\epsilon^2\lfloor \epsilon^{-2} t\rfloor}^t [F_r^\epsilon(\phi)-Y_r^\epsilon(\widetilde{\rho}_r^{\:\epsilon} \phi)]dr.
		\end{aligned}
		\end{equation}
		Then if $\lceil \epsilon^{-2}s\rceil+1\leq \lfloor \epsilon^{-2}t\rfloor$, we have
		\begin{align*}
			&\E\left[ \sup_{\substack{0\leq s\leq t\leq T+\overline{\theta} \\ t-s\leq \theta}}\left(\int_s^t [F^\epsilon_r(\phi)-Y^\epsilon_r(\widetilde{\:\epsilon}_r \phi)]dr\right)^2\right]
			\\&\quad\quad\quad\quad\quad\leq 9\E\left[ \sup_{\substack{0\leq s\leq t\leq T+\overline{\theta} \\ t-s\leq \theta}} \left(\sum_{k=\lceil \epsilon^{-2}s\rceil+1}^{\lfloor \epsilon^{-2}t\rfloor}\int_{\epsilon^2(k-1)}^{\epsilon^2 k} [F_r^\epsilon(\phi)-Y_r^\epsilon(\widetilde{\rho}_r^{\:\epsilon} \phi)]dr\right)^2\right]
			\\&\quad\quad\quad\quad\quad+ 9\E\left[\sup_{\substack{0\leq s\leq t\leq T+\overline{\theta} \\ t-s\leq \theta}}\left(\int_s^{\epsilon^2\left(\lceil \epsilon^{-2}s\rceil+1\right)} [F_r^\epsilon(\phi)-Y_r^\epsilon(\widetilde{\rho}_r^{\:\epsilon} \phi)] dr\right)^2\right]
			\\&\quad\quad\quad\quad\quad+ 9\E\left[\sup_{\substack{0\leq s\leq t\leq T+\overline{\theta} \\ t-s\leq \theta}}\left(\int_{\epsilon^2\lfloor \epsilon^{-2} t\rfloor}^t [F_r^\epsilon(\phi)-Y_r^\epsilon(\widetilde{\rho}_r^{\:\epsilon} \phi)]dr\right)^2\right].
		\end{align*}
		We control each term in turn. For the first term, by Jensen's inequality and the fact that $\lceil \epsilon^{-2}s\rceil+1\leq \lfloor \epsilon^{-2}t\rfloor$, we have
		\begin{align*}
			&\E\left[ \sup_{\substack{0\leq s\leq t\leq T+\overline{\theta} \\ t-s\leq \theta}} \left(\sum_{k=\lceil \epsilon^{-2}s\rceil+1}^{\lfloor \epsilon^{-2}t\rfloor}\int_{\epsilon^2(k-1)}^{\epsilon^2 k} [F_r^\epsilon(\phi)-Y_r^\epsilon(\widetilde{\rho}_r^{\:\epsilon} \phi)]dr\right)^2\right]
			\\&\leq \E\left[ \sup_{\substack{0\leq s\leq t\leq T+\overline{\theta} \\ t-s\leq \theta}} \left\{\left(\lfloor \epsilon^{-2}t\rfloor-\lfloor \epsilon^{-2}s\rfloor\right) \sum_{k=\lceil \epsilon^{-2}s\rceil+1}^{\lfloor \epsilon^{-2}t\rfloor} \left(\int_{\epsilon^2(k-1)}^{\epsilon^2 k} [F_r^\epsilon(\phi)-Y_r^\epsilon(\widetilde{\rho}_r^{\:\epsilon} \phi)]dr\right)^2\right\}\right]
			\\&\leq \epsilon^2\theta\sum_{k=0}^{\lfloor \epsilon^{-2} (T+\overline{\theta})\rfloor}\E\left[\left(\epsilon^{-2}\int_{\epsilon^2(k-1)}^{\epsilon^2 k} [F_r^\epsilon(\phi)-Y_r^\epsilon(\widetilde{\rho}_r^{\:\epsilon} \phi)]dr\right)^2\right].
		\end{align*}
		By Proposition \ref{prop:BGP}, the expectation is bounded by a constant $C>0$ depending on $T+\overline{\theta}$ and $\|\phi\|_\infty$, uniformly in $k \in \{0, \cdots, \lfloor \epsilon^{-2} T\rfloor\}$ and $\epsilon \in (0, 1)$. Thus
		\[
		\E\left[ \sup_{\substack{0\leq s\leq t\leq T+\overline{\theta} \\ t-s\leq \theta}} \left\{\left(\lfloor \epsilon^{-2}t\rfloor-\lfloor \epsilon^{-2}s\rfloor\right) \sum_{k=\lceil \epsilon^{-2}s\rceil+1}^{\lfloor \epsilon^{-2}t\rfloor} \left(\int_{\epsilon^2(k-1)}^{\epsilon^2 k} [F_r^\epsilon(\phi)-Y_r^\epsilon(\widetilde{\rho}_r^{\:\epsilon} \phi)]dr\right)^2\right\}\right]\leq \theta (T+\overline{\theta})C.
		\]
		The last two terms are similar so we only show the details for the first one. By Jensen's inequality, we have 
		\begin{align*}
			&\E\left[\sup_{\substack{0\leq s\leq t\leq T+\overline{\theta} \\ t-s\leq \theta}}\left(\int_s^{\epsilon^2\left(\lceil \epsilon^{-2}s\rceil+1\right)} [F_r^\epsilon(\phi)-Y_r^\epsilon(\widetilde{\rho}_r^{\:\epsilon} \phi)] dr\right)^2\right]
			\\&\leq \E\left[\sup_{\substack{0\leq s\leq t\leq T+\overline{\theta} \\ t-s\leq \theta}}\epsilon^2\int_s^{\epsilon^2\left(\lceil \epsilon^{-2}s\rceil+1\right)} [F_r^\epsilon(\phi)-Y_r^\epsilon(\widetilde{\rho}_r^{\:\epsilon} \phi)]^2 dr\right]
			\\&\leq \E\left[\epsilon^4 \sup_{r \in [0, T+\overline{\theta}+1]} [F_r^\epsilon(\phi)-Y^\epsilon_r(\widetilde{\rho}^{\:\epsilon}_r \phi)]^2\right]. 
		\end{align*}
		By \eqref{eq:f_minus_y}, we have 
		\begin{align*}
			\epsilon^4[F_r^\epsilon(\phi)-Y^\epsilon_r(\widetilde{\rho}^{\:\epsilon}_t \phi)]^2&=\frac{\epsilon^{6+2\gamma-2\kappa}}{4} \left(\sum_{z \in \Z_\epsilon} X_r^\epsilon(z)^2 \phi(\epsilon z)-\sum_{z \in \Z_\epsilon}X_r^\epsilon(z)(1+\widetilde{\rho}_r^{\:\epsilon}(\epsilon z))\phi(\epsilon z)\right)^2.
		\end{align*}
		By part two of Lemma \ref{lem:FKPP_approx_properties}, there is a constant $C=C(T+\overline{\theta}+1)>0$ such that $|\widetilde{\rho}_r^{\:\epsilon}(\epsilon z)|\leq C$ uniformly in $r\in [0, T+\overline{\theta}+1]$. Therefore, using the elementary inequalities $(a+b)^2 \leq 2(a^2+b^2)$ and $\sum_{i=1}^n a_i^2 \leq \left(\sum_{i=1}^n a_i\right)^2$ for $a_i\geq 0$, we obtain
		\begin{align*}
			\epsilon^4[F_r^\epsilon(\phi)-Y^\epsilon_r(\widetilde{\rho}^{\:\epsilon}_t \phi)]^2&\leq \|\phi\|_\infty\frac{\epsilon^{6+2\gamma-2\kappa}}{2}\left[\bigg(\sum_{z \in \Z_\epsilon} X_r^\epsilon(z)^2\bigg)^2+\bigg( \sum_{z \in \Z_\epsilon}X_r^\epsilon(z)\bigg)^2(1+C)\right]
			\\&\leq \|\phi\|_\infty\frac{\epsilon^{2+2\gamma-2\kappa}}{2}(2+C)\bigg(\epsilon\sum_{z \in \Z_\epsilon} X_r^\epsilon(z)\bigg)^4.
		\end{align*}
		The process $\left(\epsilon \sum_{z \in \Z_\epsilon} X_t^\epsilon(z)\right)_{t \geq 0}$ is stochastically dominated by $\overline{X}_t^\epsilon\coloneqq \epsilon^{1+\kappa} \overline{N}_t^\epsilon$, where $(\overline{N}_t^\epsilon)_{t\geq 0}$ is a pure birth process with transitions $\overline{N}\mapsto \overline{N}+\ell$ at rate $\overline{N}p_\ell^\epsilon$ and started from ${\overline{N}^\epsilon_0=\sum_{z \in \Z_\epsilon} N_0^\epsilon(z)}$ where $\rN^\epsilon$ is the re-scaled SLBP (Definition \ref{def:rescaled_slbp}). One easily computes the semimartingale decomposition of $\overline{X}^\epsilon$:
		\[
		\overline{X}^\epsilon_t=\overline{X}^\epsilon_0 + \mu_\epsilon \int_0^t \overline{X}^\epsilon_s ds + M_t^\epsilon,
		\] 
		for some martingale $(M^\epsilon_t)_{t\geq 0}$. Taking the supremum over $t \in [0, T]$, the fourth power, and then expectations, we obtain
		\begin{align*}
			\E\left[\sup_{t \in [0, T]}\left(\overline{X}_t^\epsilon\right)^4\right] \leq 3^4\E\left[\left(\overline{X}_0^\epsilon\right)^4\right] +3^4\mu_\epsilon^q T^{3}\int_0^T \E\left[\sup_{r \in [0, s]}\left(\overline{X}_s^\epsilon\right)^4\right]ds + 3^4\E\left[\sup_{t\in [0, T]}(M_t^\epsilon)^4\right].
		\end{align*} 
		By Assumption \ref{assump:technical_X_0}.a, the first of the three terms on the right-hand side is bounded by a constant uniformly in $\epsilon$. By \cite[page 37, second inequality of second display]{LLP2004}, we have
		\[
		\E\left[\sup_{t\in [0, T]}(M_t^\epsilon)^4\right] \leq c\left(\E\left[\langle M_\cdot^\epsilon\rangle_T^2\right]+\E\left[\sup_{t \in [0, T]} |M_t^\epsilon-M_{t-}^\epsilon|^4\right]\right),
		\]
		for some constant $c>0$, where $\langle M^\epsilon_\cdot\rangle_T$ the predictable quadratic variation of $M^\epsilon$ is given by
		\[
		\langle M^\epsilon\rangle_T = \epsilon^{1+\kappa}(\sigma_\epsilon^2+\mu_\epsilon^2)\int_0^T \overline{X}_s^\epsilon ds. 
		\]
		By Jensen's inequality and $x^2\leq 1+x^4$, 
		\begin{align*}
		\E\left[\langle M_\cdot^\epsilon\rangle_T^2\right]&\leq \epsilon^{2+2\kappa}(\sigma_\epsilon^2+\mu_\epsilon^2)^2T\int_0^T \E\left[\sup_{r \in [0, s]}(\overline{X}_r)^2\right] ds 
		\\&\leq \epsilon^{2+2\kappa}(\sigma_\epsilon^2+\mu_\epsilon^2)^2\left(T^2+T\int_0^T \E\left[\sup_{r \in [0, s]}(\overline{X}_s)^4\right] ds\right).
		\end{align*}
		Moreover, we observe that $|M_t^\epsilon-M_{t-}^\epsilon|=\overline{X}_t^\epsilon-\overline{X}_{t-}^\epsilon=\epsilon^{1+\kappa}L(\epsilon)\leq \epsilon$ almost surely. It follows from Gr\"onwall's inequality that 
		\[
		C'\coloneqq \sup_{\epsilon \in (0, 1)}\E\left[\sup_{t \in [0, T]}\left(\overline{X}_t^\epsilon\right)^4\right]<\infty.
		\]
		Hence, 
		\[
		\E\left[\epsilon^4 \sup_{r \in [0, T+\overline{\theta}+1]} [F_r^\epsilon(\phi)-Y^\epsilon_r(\widetilde{\rho}^{\:\epsilon}_r \phi)]^2\right]\leq \|\phi\|_\infty\frac{\epsilon^{2+2\gamma-2\kappa}}{2}(2+C)C'. 
		\]
		This proves \eqref{lem:bgp_cor_bd_sup} assuming $\lceil \epsilon^{-2}s\rceil+1\leq \lfloor \epsilon^{-2}t\rfloor$. Otherwise, we have $0\leq t-s\leq 2\epsilon^\kappa$, and the method just used applies directly to the left-hand side of \eqref{eq:sup_cas_decomp}, implying \eqref{lem:bgp_cor_bd_sup}. This concludes the proof of the corollary. 
	\end{proof}
	
	\subsection{Moments and tightness}\label{sec:tight_fluct}
	
	The aim of this section is to show tightness of the SLBP fluctuations process introduced in Definition \ref{def:fluct_field_def}. 
	\begin{prop}\label{prop:Brownian_fluct_tight}
		The family $\{\rY^\epsilon:\epsilon \in (0, 1)\}$ is tight in $\cD([0, T], C^\infty(\S^1)')$. 
	\end{prop}
	\noindent By Mitoma's Theorem (Theorem 4.1 in \cite{Mitoma1983}), Proposition \ref{prop:Brownian_fluct_tight} follows from tightness of $\{Y^\epsilon(\phi):\epsilon \in (0, 1)\}$ for each $\phi \in C^\infty(\S^1)$. By Aldous' Criterion (Theorem 16.10 and equation (16.32) in \cite{B1999}), it suffices to show the following two conditions are satisfied for each $\phi \in C^\infty(\S^1)$.
	\begin{enumerate}
		\item[1.] For all $t \in [0, T]$, $\{Y^\epsilon_t(\phi):\epsilon \in (0, 1)\}$ is tight in $\R$.
		\item[2.] Fix any family ${\{\tau_\epsilon:\epsilon \in (0, 1)\}}$ of $[0, T]$-valued stopping times and a family ${\{\theta_\epsilon: \epsilon\in (0, 1)\}\subset [0, \infty)}$ with $\lim_{\epsilon \searrow 0}\theta_\epsilon=0$. Then, for any $\eta>0$, we have 
		\begin{equation}\label{cond:Ald_2_fluct}
		\limsup_{\epsilon \searrow 0}\P\left(|Y^\epsilon_{\tau_\epsilon +\theta_\epsilon}(\phi)-Y^\epsilon_{\tau_\epsilon}(\phi)|>\eta\right)=0.
		\end{equation}
	\end{enumerate}
	
	The most important tool in the proof is Corollary~\ref{cor:BGP_control} of the previous section which allows us to control the non-linear terms in the generator of the non-equilibrium fluctuations $\rY^\epsilon$. To apply this corollary in our proof of tightness of $\{\rY^\epsilon:\epsilon\in (0, 1)\}$, we need to introduce a term of the form $Y^\epsilon_r(\widetilde{\rho}^{\:\epsilon}_r\phi)$, for some $\phi\in C^\infty(\S^1)$, which is not originally present in the generator of $\rY^\epsilon$. To this end, given $\phi \in C^\infty(\S^1)$ possibly dependent on $\epsilon$, we study a martingale problem corresponding to $(Y^\epsilon_s(\phi^\epsilon_{s,t}))_{s \in [0, t]}$ for a specific test function of the form 
	\[
	\phi^\epsilon :\{(s,t) \in [0, \infty)^2:0\leq s\leq t\}\times \S^1\to \R, \quad ((s, t), z)\mapsto \phi^\epsilon_{s, t}(z),
	\]
	defined as follows. Recall that $\pi : \R \to \S^1$ denotes the canonical projection map, and define $\S^{1, \epsilon}\coloneqq \pi([0, 1]\cap K_\epsilon^{-1} \Z)$. We define $\phi^\epsilon$ as the unique solution to the backward terminal value problem
	\begin{equation}\label{eq:backward_test_eq}
		\begin{cases}
			\partial_s \phi^\epsilon_{s, t}(z)+\frac{1}{2}\Delta^\epsilon \phi^\epsilon_{s, t}(z)(\mu_\epsilon-\widetilde{\rho}^{\:\epsilon}_s(z))\phi^\epsilon_{s, t}(z) = 0, & s\in [0, t), \quad z \in \S^{1, \epsilon},
			\\\phi_{t, t}^\epsilon(z)=\phi(z), &z \in \S^{1, \epsilon},
		\end{cases}
	\end{equation}
	where the discrete Laplacian $\Delta^\epsilon$ is defined on all functions $\psi:\S^1\to \R$ by 
	\[
	\Delta^\epsilon  \psi(z)\coloneqq \epsilon^{-2}\left(\psi(z+K_\epsilon^{-1})+\psi(z-K_\epsilon^{-1})-2\psi(z)\right), \quad z \in \S^1.
	\] 
	We then extend the definition of $\phi^\epsilon$ to all of $\S^1$ by choosing an arbitrary $C^3$ interpolation. In Appendix \ref{append:time_dep_test_fcns}, we show existence and uniqueness of a solution which is continuous in $s$ and $t$, and we prove a priori uniform bounds on the solution. We also show convergence to a limiting problem. In the following lemma, we introduce a useful martingale.
	\begin{lem}\label{lem:fluct_mg_pb}
		Recall the definition of the non-linear fluctuations field $F^\epsilon$ in \eqref{def:non_lin_fluct}. For any $t \geq 0$, the process $(M^\epsilon_s(\phi_{\cdot, t}^\epsilon))_{s \in [0,t)}$ defined by 
		\begin{equation}\label{eq:useful_mg_fluct}
		M^\epsilon_s(\phi_{\cdot, t}^\epsilon)\coloneqq Y^\epsilon_s(\phi_{s, t}^\epsilon) - Y^\epsilon_0(\phi^\epsilon_{0, t})-\int_0^s \left(Y_r^\epsilon(\widetilde{\rho}^{\:\epsilon}_r \phi^\epsilon_{r, t})-F_r^\epsilon(\phi^\epsilon_{r, t})\right)dr,
		\end{equation}
		for $s \in [0, t)$, is a martingale with predictable quadratic variation 
		\begin{equation}\label{eq:fluct_quad_var_0}
		\langle M^\epsilon_\cdot(\phi_{\cdot, t}^\epsilon)\rangle_s = \int_0^s \Gamma_r^\epsilon(\phi_{r, t}^\epsilon)dr,
		\end{equation}
		where 
		\begin{equation}\label{eq:fluct_quad_var}
		\begin{aligned}
		\Gamma_r^\epsilon(\phi_{r, t}^\epsilon)&=\frac{\epsilon}{2}\sum_{z \in \Z_\epsilon}\phi_{r, t}^\epsilon(\epsilon z)^2 \left[X_r^\epsilon(z)(X_r^\epsilon(z)-\epsilon^\kappa)+X_r^\epsilon(z)(\sigma_\epsilon^2+\mu_\epsilon^2)\right]
		\\&+\frac{\epsilon}{2} \sum_{z \in \Z_\epsilon}X_s^\epsilon(z)\bigg[\left(\nabla^{\epsilon, -}\phi_{r, t}^\epsilon(\epsilon z)\right)^2+\left(\nabla^{\epsilon, +}\phi_{r, t}^\epsilon(\epsilon z)\right)^2\bigg],
		\end{aligned}
		\end{equation}
		where $\nabla^{\epsilon, \pm}$ are the first order discrete differentials defined in Lemma \ref{lem:fluct_L_RW_laplacian}.
	\end{lem}
	\begin{proof}
		For any $t \in [0, T]$, the process $\rM^\epsilon(\phi_{\cdot, t}^\epsilon)$ defined by 
		\[
		M^\epsilon_s(\phi_{\cdot, t}^\epsilon)\coloneqq Y^\epsilon_s(\phi_{s, t}^\epsilon) - Y^\epsilon_0(\phi^\epsilon_{0, t})-\int_0^s (\cL^\epsilon Y^\epsilon_r(\phi^\epsilon_{r, t})+\partial_s Y^\epsilon_r(\phi^\epsilon_{\cdot, t})(r))dr, \quad s \in [0, t),
		\]
		where $\cL^\epsilon$ is the generator of $\rY^\epsilon$ computed in \eqref{eq:Y_dyn_generic}, is a martingale. Let $\cG^\epsilon=\cG^\epsilon_{RW}+\cG^\epsilon_{BC}$ denote the generator of the rescaled SLBP from \eqref{def:slbp_gen}-\eqref{def:updated_config}. By \eqref{eq:Y_dyn_generic} and \eqref{eq:time_deriv_part_of_L}, we obtain 
		\begin{align*}
			\cL^\epsilon Y_r^\epsilon(\phi_{r,t}^\epsilon) &= \cL^\epsilon_{RW} Y_r^\epsilon(\phi^\epsilon_{r,t})-\epsilon^{\gamma-\kappa}\E[\cG^\epsilon_{RW} X_r^\epsilon(\phi^\epsilon_{r,t})]
			\\&+\cL^\epsilon_{BC} Y_r^\epsilon(\phi^\epsilon_{r,t})-\epsilon^{\gamma-\kappa}\E[\cG^\epsilon_{BC} X_r^\epsilon(\phi^\epsilon_{r,t})].
		\end{align*}
		Using that $\phi^\epsilon$ solves \eqref{eq:backward_test_eq}, we compute 
		\begin{align*}
		\partial_s Y^\epsilon_r(\phi^\epsilon_{\cdot, t})(r) &=Y_r^\epsilon((\widetilde{\rho}^{\:\epsilon}_r -\mu_\epsilon)\phi^\epsilon_{r, t})-\frac{1}{2}Y_r^\epsilon(\Delta^\epsilon \phi^\epsilon_{r, t}).
		\end{align*}
		It follows from Lemma \ref{lem:fluct_L_RW_laplacian} with the linear test function $f(y)=y$, $y \in \R$, that 
		\begin{equation}\label{eq:time_deriv_Y}
		\partial_s Y^\epsilon_r(\phi^\epsilon_{\cdot, t})(r)=Y_r^\epsilon((\widetilde{\rho}^{\:\epsilon}_r -\mu_\epsilon) \phi^\epsilon_{r, t})-\cL^\epsilon_{RW} Y^\epsilon_r(\phi^\epsilon_{r, t})+\epsilon^{\gamma-\kappa}\E[\cG^\epsilon_{RW} X_r^\epsilon(\phi^{\epsilon}_{r,t})].
		\end{equation}
		Consequently, we obtain
		\begin{align}
		\cL^\epsilon Y^\epsilon_r(\phi^\epsilon_{r, t})+\partial_sY^\epsilon_r(\phi^\epsilon_{\cdot, t})(r) &\notag= \cL^\epsilon_{BC} Y_r^\epsilon(\phi^{ \epsilon}_{r,t})-\epsilon^{\gamma-\kappa}\E[\cG^\epsilon_{BC} X_r^\epsilon(\widetilde{\phi}^{\:\epsilon}_{r,t})]+ Y_r^\epsilon((\widetilde{\rho}^{\:\epsilon}_r -\mu_\epsilon) \phi^\epsilon_{r, t})
		\\&\notag= Y_r^\epsilon(\widetilde{\rho}^{\:\epsilon}_r \phi^\epsilon_{r, t})-F_r^\epsilon(\phi^\epsilon_{r, t}),\label{eq:linear_test_gen}
		\end{align}
		where in the second step, we have applied Lemma \ref{lem:fluct_Brownian_BC_gen} with the linear test function. This proves that \eqref{eq:useful_mg_fluct} is a martingale. Moreover, one can check that its predictable quadratic variation is
		\[
		\langle M^\epsilon_\cdot(\phi_{\cdot, t}^\epsilon)\rangle_s =\int_0^s \Gamma_r^\epsilon(\phi_{r, t}^\epsilon) dr,
		\]
		where
		\begin{align*}
			\Gamma_r^\epsilon(\phi_{r, t}^\epsilon)&\coloneqq\cL^\epsilon(Y^\epsilon_r(\phi_{r, t}^\epsilon)^2)+\partial_s (Y^\epsilon_r(\phi_{\cdot, t}^\epsilon)(r))^2
			\\&\quad\quad\quad\quad\quad\quad\quad\quad-2Y_r^\epsilon(\phi_{r, t}^\epsilon)\Big[\cL^\epsilon Y^\epsilon_r(\phi_{r, t}^\epsilon)+\partial_r (Y^\epsilon_r(\phi_{r, t}^\epsilon)(r))\Big],
		\end{align*}
		for all $r \in [0, t]$. We simplify the first two terms on the right-hand side of this expression using \eqref{eq:Y_dyn_generic}, \eqref{eq:time_deriv_part_of_L}, \eqref{eq:time_deriv_Y} and Lemma \ref{lem:fluct_Brownian_BC_gen}:
		\begin{align*}
			&\cL^\epsilon(Y^\epsilon_r(\phi_{r, t}^\epsilon)^2)+\partial_s (Y^\epsilon_r(\phi_{\cdot, t}^\epsilon)(r))^2
			\\&=\cL^\epsilon_{RW}(Y_r^\epsilon(\phi_{r, t}^\epsilon)^2)-2\epsilon^{\gamma-\kappa}Y_r^\epsilon(\phi_{r, t}^\epsilon)\E[\cG^\epsilon_{RW} X_r^\epsilon(\phi^{\epsilon}_{r,t})]-2Y_r^\epsilon(\phi_{r, t}^\epsilon)Y_r^\epsilon(\Delta^\epsilon\phi_{r, t}^\epsilon)
			\\&+2Y_r^\epsilon(\phi_{r, t}^\epsilon)\Big[Y_r^\epsilon(\widetilde{\rho}^{\:\epsilon}_r \phi^\epsilon_{r, t})-F_r^\epsilon(\phi^\epsilon_{r, t})\Big]
			\\&+ \frac{\epsilon}{2}\sum_{z \in \Z_\epsilon}\phi_{r, t}^\epsilon(\epsilon z)^2 \left[X_r^\epsilon(z)(X_r^\epsilon(z)-\epsilon^\kappa)+2X_r^\epsilon(z)(\sigma_\epsilon^2+\mu_\epsilon^2)\right].
		\end{align*}
		By Lemma \ref{lem:fluct_L_RW_laplacian} and the assumption that $1+2\gamma-\kappa=0$, we have 
		\begin{align*}
			&\cL^\epsilon_{RW}(Y_r^\epsilon(\phi_{r, t}^\epsilon)^2)-2\epsilon^{\gamma-\kappa}Y_r^\epsilon(\phi_{r, t}^\epsilon)\E[\cG^\epsilon_{RW} X_r^\epsilon(\phi^{\epsilon}_{r,t})]-2Y_r^\epsilon(\phi_{r, t}^\epsilon)Y_r^\epsilon(\Delta^\epsilon\phi_{r, t}^\epsilon)
			\\&=\frac{\epsilon}{2} \sum_{z \in \Z_\epsilon}X_s^\epsilon(z)\bigg[\left(\nabla^{\epsilon, -}\phi_{r, t}^\epsilon(\epsilon z)\right)^2 +\left(\nabla^{\epsilon, +}\phi_{r, t}^\epsilon(\epsilon z)\right)^2\bigg].
		\end{align*}
		This concludes the proof of the lemma.
	\end{proof}
	The next lemma gives us control on the second moment of the fluctuations. 
	\begin{lem}\label{lem:finite_fluct_moment}
		There exists $c>0$ depending on $T$, $\phi$, and the uniform bound of Lemma \ref{lem:a_priori_est} such that
		\[
		\E\left[Y^\epsilon_t(\phi)^2\right]<c, \quad \E\left[Y^\epsilon_t(\Delta^\epsilon\phi)^2\right]<c, \quad \E\left[Y^\epsilon_t(\widetilde{\rho}^{\:\epsilon}_t\phi)^2\right]<c,
		\]
		for all $t \in [0, T]$ and $\epsilon \in (0, 1)$.
	\end{lem}
	\begin{proof}
		We first control $\E[Y^\epsilon_t(\phi)^2]$. By Lemma \ref{lem:fluct_mg_pb}, we have the semimartingale decomposition
		\[
		Y^\epsilon_s(\phi_{s, t}^\epsilon)= Y^\epsilon_0(\phi^\epsilon_{0, t})+\int_0^s \left(Y_r^\epsilon(\widetilde{\rho}^{\:\epsilon}_r \phi^\epsilon_{r, t})-F_r^\epsilon(\phi^\epsilon_{r, t})\right)dr+M^\epsilon_s(\phi_{\cdot, t}^\epsilon),
		\]
		for any $0\leq s\leq t \leq T$. We let $s=t$, take the square on both sides, take expectations, and use the inequality $(\sum_{i=1}^n a_i)^2\leq n\sum_{i=1}^n a_i^2$:
		\begin{align*}
			\E\left[Y^\epsilon_s(\phi_{s, t}^\epsilon)^2\right]&\leq 3\Bigg( \E\left[Y^\epsilon_0(\phi^{\:\epsilon}_{0, t})^2\right]+\E\left[\left(\int_0^t \left[ F_r^\epsilon(\phi^\epsilon_{r, t})-Y_r^\epsilon(\widetilde{\rho}^{\:\epsilon}_r \phi^\epsilon_{r, t})\right]dr\right)^2\right]+\E\left[M^\epsilon_t(\phi^\epsilon_{\cdot, t})^2\right]\Bigg).
		\end{align*}
		Denote these three terms $A_1$, $A_2$, and $A_3$, respectively. We bound them in turn. Recall from Definition \ref{def:general_init_data} that $\E[X_0^\epsilon(z)] = \rho_0^\epsilon(z)$ for all $z \in \Z_\epsilon$. Thus by Definition \ref{def:fluct_field_def}, \eqref{eq:V_fcn_2_pcs_not_diag} the formula for order two $V$-functions with particles on distinct sites, and Definition \ref{def:scaled_v_fcns}, we have 
		\begin{align*}
			A_1&= \epsilon^{2+2\gamma-2\kappa}\sum_{z\in \Z_\epsilon}\sum_{z' \in \Z_\epsilon}\E[(X^\epsilon_0(z)-\rho_0^\epsilon(z))(X^\epsilon_0(z')-\rho_0^\epsilon(z'))] \phi^\epsilon_{0, t}(\epsilon z)\phi^\epsilon_{0, t}(\epsilon z')
			\\&=\epsilon^{2+2\gamma-2\kappa}\sum_{z\in \Z_\epsilon}\E[(X^\epsilon_0(z)-\rho_0^\epsilon(z))^2] \phi_{0, t}^\epsilon(\epsilon z)^2
			\\&+\epsilon^{2+2\gamma-2\kappa}\sum_{z\not=z' \in  \Z_\epsilon}v_0^\epsilon(x^{1, z}+x^{1, z'}|\nu^\epsilon) \phi^\epsilon_{0, t}(\epsilon z)\phi_{0, t}^\epsilon(\epsilon z').
		\end{align*} 
		By Definition \ref{def:general_init_data}, there exists $c_0>0$ such that 
		\begin{equation}\label{ineq:bd_on_rho}
			\sup_{\epsilon \in (0, 1)}\max_{z \in \Z_\epsilon}|\rho_0^\epsilon(z)|\leq c_0.
		\end{equation} 
		Thus, by the explicit computation of a $V$-function of order two in \eqref{eq:V_fcn_2_pcs_diag}, Assumption \ref{assump:technical_X_0}.a, and $\kappa \in [0, 3/8)$, we have 
		\begin{align*}
			\E[(X^\epsilon_0(z)-\rho_0^\epsilon(z))^2]\leq \epsilon^{\kappa}\rho_0^\epsilon(z)+c\epsilon^{1+\kappa}\leq c_1\epsilon^{\kappa}.
		\end{align*}
		uniformly in $z \in \Z_\epsilon$, for some $c>0$ and $c_1\coloneqq \max(c, c_0)$. By Assumption \ref{assump:technical_X_0}.a, there is $c_2>0$ such that $v_0^\epsilon(x^{1, z}+x^{1, z'}|\nu^\epsilon)\leq c_2\epsilon^{1+\kappa}$ uniformly in $z, z' \in \Z_\epsilon$, where $x^{1, z}=\epsilon^\kappa\mathbbm{1}_{\{z\}}$. By Lemma \ref{lem:a_priori_est}, there exists $C=C(T)>0$ such that
		\begin{equation}\label{ineq:a_priori_est}
			|\phi_{s, t}^\epsilon(\epsilon z)|< C,
		\end{equation}
		uniformly over $0\leq s\leq t\leq T$, $z \in \Z_\epsilon$ and $\epsilon \in (0, 1)$. Thus overall for $A_1$, we have 
		\begin{align}
		A_1&\notag\leq \epsilon^{1+2\gamma-\kappa}c_1\epsilon \sum_{z \in \Z_\epsilon}\phi_{0, t}^\epsilon(\epsilon z)^2+c_2\epsilon^{1+2\gamma-\kappa}C^2
		\\&\leq (c_1+c_2)C^2,\label{ineq:for_id_lim_too}
		\end{align}
		where we have used that $1+2\gamma-\kappa=0$ in the second step. Next, we apply \eqref{lem:bgp_cor_bd} of Corollary \ref{cor:BGP_control} to control $A_2$. We obtain for some $c_1, c_2>0$ depending only on $T$, the constant $c_0$ above, and $C$ of \eqref{ineq:a_priori_est} that 
		\[
		A_2 \leq c_1 T^2+c_2\epsilon^{4+2\gamma-2\kappa}.
		\]
		It remains to control the martingale term $A_3$. Since for each $t\geq 0$, ${(M_s^\epsilon(\phi^\epsilon_{\cdot, t})^2-\langle M_\cdot^\epsilon(\phi^\epsilon_{\cdot, t})\rangle_s)_{s\in[0, t)}}$ is a martingale, we obtain 
		\begin{align}
			\E\left[M_t^\epsilon(\phi^\epsilon_{\cdot, t})^2\right]&\notag= \E\left[\langle M_\cdot^\epsilon(\phi^\epsilon_{\cdot, t})\rangle_t\right]
			\\&= \int_0^t\E\left[\Gamma^\epsilon_r(\phi^\epsilon_{r, t})\right]dr,
		\end{align}
		where we have used Fubini's theorem in the last step, and $\Gamma^\epsilon$ is defined in \eqref{eq:fluct_quad_var}. By Proposition~\ref{prop:SLBP_moments}, we obtain $A_3\leq c$ for some $c>0$ dependent on $T$ and $c_0$. Gathering the estimates, we have shown that $\E\left[Y^\epsilon_t(\phi)^2\right]\leq c_1$, for some $c_1=c_1(T, c_0, C)>0$. For $\E[Y^\epsilon_t(\Delta^\epsilon \phi)^2]$ and $\E\left[Y^\epsilon_t(\widetilde{\rho}^{\:\epsilon}_t\phi)^2\right]$, we consider the problem \eqref{eq:backward_test_eq} with terminal condition $\Delta^\epsilon \phi$ and $\widetilde{\rho}^{\:\epsilon}_t\phi$, respectively. Then the arguments just presented for the first bound carry over. This proves the lemma. 
	\end{proof}
	
	Using Markov's inequality and the fact that $\phi_{t, t}^\epsilon=\phi$, Lemma \ref{lem:finite_fluct_moment}, implies tightness of $\{Y^\epsilon_t(\phi):\epsilon \in (0, 1)\}$ in $\R$, for each $t\in [0, T]$. Next, we show the equicontinuity condition \eqref{cond:Ald_2_fluct}. Fix any $\phi \in C^\infty(\S^1)$. Let ${\{\tau_\epsilon:\epsilon \in (0, 1)\}}$ be a family of $[0, T]$-valued stopping times and let ${\{\theta_\epsilon: \epsilon\in (0, 1)\}\subset [0, \infty)}$ be such that $\lim_{\epsilon \searrow 0}\theta_\epsilon=0$. Note in particular that ${\{\theta_\epsilon: \epsilon\in (0, 1)\}\subset [0, \overline{\theta}]}$, where ${\overline{\theta}\coloneqq \sup_{\epsilon \in (0,1)}\theta_\epsilon \in [0, \infty)}$. By Markov's inequality, we have 
	\begin{equation}\label{ineq:Markov_fluct_tight}
		\P\left(|Y^\epsilon_{\tau_\epsilon +\theta_\epsilon}(\phi)-Y^\epsilon_{\tau_\epsilon}(\phi)|>\eta\right)\leq \eta^{-2}\E\left[\left(Y^\epsilon_{\tau_\epsilon +\theta_\epsilon}(\phi)-Y^\epsilon_{\tau_\epsilon}(\phi)\right)^2\right]
	\end{equation}
	Then, the following lemma implies \eqref{cond:Ald_2_fluct} the equicontinuity in probability.
	\begin{lem}\label{lem:second_aldous_Br_fluct}
		We have
		\[
			\limsup_{\epsilon \searrow 0}\E\left[\left(Y^\epsilon_{\tau_\epsilon +\theta_\epsilon}(\phi)-Y^\epsilon_{\tau_\epsilon}(\phi)\right)^2\right]=0.
		\]
	\end{lem}
	\begin{proof}
		By Lemmas \ref{lem:fluct_L_RW_laplacian} and \ref{lem:fluct_Brownian_BC_gen}, the semimartingale decomposition of $\rY^\epsilon$ is given by 
		\[
		Y_t^\epsilon(\phi) = Y_0^\epsilon(\phi)+\int_0^t \left(Y_r^\epsilon(\Delta^\epsilon \phi)-F_r^\epsilon(\phi)\right)dr + M_t^\epsilon(\phi), \quad t\geq 0,
		\]
		for some martingale $(M_t^\epsilon(\phi))_{t\geq 0}$. Thus, we have
		\begin{equation}\label{eq:fluct_decomp}
		\begin{aligned}
			Y_{\tau_\epsilon+\theta_\epsilon}^\epsilon(\phi)-Y_{\tau_\epsilon}^\epsilon(\phi)&= \int_{\tau_\epsilon}^{\tau_\epsilon+\theta_\epsilon}\left(Y^\epsilon_r(\Delta^\epsilon \phi)-Y^\epsilon_r(\widetilde{\rho}_r^{\:\epsilon}\phi)\right)dr
			\\&+\int_{\tau_\epsilon}^{\tau_\epsilon+\theta_\epsilon} \left(Y^\epsilon_r(\widetilde{\rho}_r^{\:\epsilon}\phi)-F_r^\epsilon(\phi)\right)dr
			\\&+M^\epsilon_{\tau_\epsilon+\theta_\epsilon}(\phi)-M_{\tau_\epsilon}^\epsilon(\phi),
		\end{aligned}
		\end{equation}
		Thus
		\begin{align*}
			&\E\left[\left(Y_{\tau_\epsilon+\theta_\epsilon}^\epsilon(\phi)-Y_{\tau_\epsilon}^\epsilon(\phi)\right)^2\right]
			\\&\leq 3\E\left[\sup_{\substack{0\leq s\leq t\leq T+\overline{\theta} \\ t-s\leq \theta_\epsilon}}\left(\int_s^t\left(Y^\epsilon_r(\Delta^\epsilon \phi)-Y^\epsilon_r(\widetilde{\rho}_r^{\:\epsilon}\phi)\right)dr\right)^2\right]
			\\&+ 3\E\left[\sup_{\substack{0\leq s\leq t\leq T+\overline{\theta} \\ t-s\leq \theta_\epsilon}}\left(\int_s^t\left(Y^\epsilon_r(\widetilde{\rho}_r^{\:\epsilon}\phi)-F_r^\epsilon(\phi)\right)dr\right)^2\right]
			\\&+3\E\left[\left(M^\epsilon_{\tau_\epsilon+\theta_\epsilon}(\phi)-M^\epsilon_{\tau_\epsilon}(\phi)\right)^2\right].
		\end{align*}
		We control the terms on the right-hand side in turn. Denote them $A_1$, $A_2$, and $A_3$, respectively. By Jensen's inequality, we have 
		\begin{align*}
			A_1&\leq \theta_\epsilon \int_0^{T+\overline{\theta}}\E\left[\left(Y^\epsilon_r(\Delta^\epsilon \phi)-Y^\epsilon_r(\widetilde{\rho}_r^{\:\epsilon}\phi)\right)^2\right]dr
			\\&\leq 3\theta_\epsilon \int_0^{T+\overline{\theta}}\left(\E\left[Y^\epsilon_r(\Delta^\epsilon \phi)^2\right]+\E\left[Y^\epsilon_r(\widetilde{\rho}_r^{\:\epsilon}\phi)^2\right]\right)dr
		\end{align*}
		By Lemma \ref{lem:finite_fluct_moment}, the first three expectations are bounded by a constant depending on $T$ and on $\phi$ and the uniform bound on $\widetilde{\rho}^{\:\epsilon}$ from part two of Lemma \ref{lem:FKPP_approx_properties}. By \eqref{lem:bgp_cor_bd_sup} of Corollary \ref{cor:BGP_control}, the next term, $A_2$, satisfies $A_2\leq c_1 \theta_\epsilon^2+c_2\epsilon^{2+2\gamma-2\kappa}$ for some $c_1=c_1(T, \phi), c_2=c_2(T, \phi)>0$. Finally, we control the martingale term $A_3$.
		By the It\^o isometry, we compute
		\begin{align*}
			\E\left[\left(M^\epsilon_{\tau_\epsilon+\theta_\epsilon}(\phi)-M^\epsilon_{\tau_\epsilon}(\phi)\right)^2\right]&= \E\left[\left(\int_{\tau_\epsilon}^{\tau_\epsilon+\theta_\epsilon} dM_r^\epsilon(\phi)\right)^2\right]
			\\&=\E\left[\int_{\tau_\epsilon}^{\tau_\epsilon+\theta_\epsilon} \Gamma^\epsilon_r(\phi) dr\right].
		\end{align*}
		Using the Cauchy-Schwarz and Jensen inequalities, we obtain
		\begin{align*}
			\E\left[\int_{\tau_\epsilon}^{\tau_\epsilon+\theta_\epsilon} \Gamma^\epsilon_r(\phi) dr\right]&\leq \E\left[\left(\int_{\tau_\epsilon}^{\tau_\epsilon+\theta_\epsilon} \Gamma^\epsilon_r(\phi) dr\right)^2\right]^{1/2}
			\\&\leq \theta_\epsilon^{1/2}\left(\int_{0}^{T+\overline{\theta}}\E\left[\Gamma^\epsilon_r(\phi)^2 \right]dr\right)^{1/2}
		\end{align*}
		It follows from \eqref{eq:fluct_quad_var}, Lemma \ref{lem:a_priori_est} and Proposition \ref{prop:SLBP_moments} that 
		\[
		\E\left[\left(M^\epsilon_{\tau_\epsilon+\theta_\epsilon}(\phi)-M^\epsilon_{\tau_\epsilon}(\phi)\right)^2\right] \leq c' \theta_\epsilon^{1/2},
		\]
		for some $c'$ depending on $T$ and \eqref{ineq:a_priori_est}. Thus for some $C>0$, we have $A_3\leq C \theta_\epsilon^{1/4}$.
	\end{proof}

	\begin{proof}[Proof of Proposition \ref{prop:Brownian_fluct_tight}]
		The proposition follows from Lemma~\ref{lem:finite_fluct_moment}, \eqref{ineq:Markov_fluct_tight}, Lemma~\ref{lem:second_aldous_Br_fluct}, and the proof outline given below Proposition \ref{prop:Brownian_fluct_tight}.
	\end{proof}

	\subsection{Identification of the limit}\label{sec:id_fluct_lim}
	
	By Proposition \ref{prop:Brownian_fluct_tight}, any subsequence of $(\rY^\epsilon)_{\epsilon \in (0, 1)}$ has a further subsequence (which we still denote $(\rY^\epsilon)_{\epsilon \in (0, 1)}$ for simplicity) which converges in distribution in $\cD([0, T], C^\infty(\S^1)')$ as $\epsilon \searrow \infty$ to a limit point $\widetilde{\rY} \in \cD([0, T], C^\infty(\S^1)')$. We establish the following proposition which implies that $\widetilde{\rY}$ is the solution to \eqref{eq:OU}. 
	\begin{prop}\label{prop:FDD_CLT}
		Let  $\rY =(Y_t)_{t\in [0, T]} \in \cD([0, T], C^\infty(\S^1)')$ denote the solution to the stochastic partial differential equation \eqref{eq:OU}. Then, for any $k \in \N$, $t_1, \cdots, t_k \in [0, T]$ and $\phi\in C^\infty(\S^1)$, we have 
		\begin{equation}\label{eq:fdd_conv}
		(\widetilde{Y}_{t_1}(\phi), \cdots, \widetilde{Y}_{t_k}(\phi))\overset{d}{=} (Y_{t_1}(\phi), \cdots, Y_{t_k}(\phi)).
		\end{equation}
	\end{prop}
	
	To prove this result, we adapt the approach in section 4.5 of \cite{FP2017}. By Lemma \ref{lem:fluct_mg_pb}, we have the semimartingale decomposition
	\begin{equation}\label{eq:smg_id_lim}
	Y^\epsilon_s(\phi_{s, t}^\epsilon) = Y^\epsilon_0(\phi^\epsilon_{0, t})+\int_0^s \left(Y_r^\epsilon(\widetilde{\rho}^{\:\epsilon}_r \phi^\epsilon_{r, t})-F_r^\epsilon(\phi^\epsilon_{r, t})\right)dr+M^\epsilon_s(\phi_{\cdot, t}^\epsilon), \quad t\geq 0,
	\end{equation}
	for some martingale $(M^\epsilon_s(\phi_{\cdot, t}^\epsilon))_{s \in [0, t)}$ with predictable quadratic variation
	\[
	\langle M^\epsilon_\cdot(\phi_{\cdot, t}^\epsilon)\rangle_s = \int_0^s \Gamma_r^\epsilon(\phi_{r, t}^\epsilon)dr,
	\]
	where $\Gamma_r^\epsilon(\phi_{r, t}^\epsilon)$ is given by \eqref{eq:fluct_quad_var}. On the right-hand side of \eqref{eq:smg_id_lim}, we will show that the first term converges in distribution to a centred Gaussian variable by Assumption \ref{assump:technical_X_0}.a, that the martingale part converges in the Skorohod topology on $\cD([0,t], \R)$ to a centred Gaussian martingale with time-$s$ variance 
	\[
	Q_s(\phi_{\cdot, t})\coloneqq \int_0^s\left(\langle (\phi_{r, t})^2,  [(\sigma^2+\mu^2)\rho_r+\rho_r^2/2]\rangle +\langle \left(\partial_z \phi_{r, t}\right)^2, \rho_r\rangle \right)dr,
	\]
	for a suitable limiting $\phi$, and that the integral term vanishes in $L^2$ by Proposition \ref{prop:BGP}. The convergence of the finite-dimensional distributions claimed in Proposition~\ref{prop:FDD_CLT} will then follow from standard arguments. In the following lemma, we prove the convergence of the martingale term. 
	
	\begin{lem}\label{lem:mg_term_conv}
		We have $(M_s^\epsilon(\phi_{\cdot, t}^\epsilon))_{s \in [0, t]}\overset{d}{\to} (M_s(\phi_{\cdot, t}))_{s \in [0, t]}$ in the Skorohod topology on $\cD([0, t], \R)$ as $\epsilon \searrow 0$, where $(M_s(\phi_{\cdot, t}))_{s \in [0, t)}$ is a centred Gaussian martingale with time-$s$ variance given by $Q_s(\phi_{\cdot, t})$, and $\phi_{\cdot, t}$ satisfies 
		\begin{equation}\label{eq:lim_test_eq_ID_lim}
			\begin{cases}
				\partial_s \phi_{s, t}(z)+\frac{1}{2}\partial_z^2 \phi_{s, t}(z)+(\mu-\rho_s(z))\phi_{s, t}(z) = 0, & s\in [0, t), \quad z \in \S^1,
				\\\phi_{t, t}(z)=\phi(z), &z \in \S^1.
			\end{cases}
		\end{equation}
	\end{lem}
	\begin{proof}
		Let $\Delta M_s(\phi_{\cdot, t}^\epsilon) \coloneqq M_s(\phi_{\cdot, t}^\epsilon)-M_{s-}(\phi_{\cdot, t}^\epsilon) $. By \cite[Theorem 3.11, Chapter VIII]{JS2002}, it suffices to show the following two points. 
		\begin{enumerate}
			\item[1.] Almost surely, we have $\sup_{\epsilon \in (0, 1)}\sup_{s \in [0, t)} |\Delta M^\epsilon_s(\phi_{\cdot, t}^\epsilon)|<\infty$ and ${\sup_{s \in [0, t)} |\Delta M^\epsilon_s(\phi_{\cdot, t}^\epsilon)|\to 0}$ in probability as $\epsilon\to 0$.
			\item[2.] For each $s \in [0, t)$, we have $\langle M^\epsilon(\phi_{\cdot, t}^\epsilon)\rangle_s\to Q_s(\phi_{\cdot, t})$ in probability as $\epsilon \searrow 0$. 
		\end{enumerate}
		Recall that the offspring distribution is truncated at $L(\epsilon) = o(\epsilon^{-\kappa})$. So we have almost surely that $\Delta M_s(\phi_{\cdot, t}^\epsilon) = Y_s^\epsilon(\phi_{s, t}^\epsilon)-Y_{s-}^\epsilon(\phi_{s-, t}^\epsilon)=O(L(\epsilon) \epsilon^{1+\gamma-\kappa})=o(1)$ as $\epsilon \searrow 0$, uniformly in $s \in [0, t)$. This proves the first point. To prove the second point, we observe using \eqref{eq:fluct_quad_var_0}, \eqref{eq:fluct_quad_var} and ${Y^\epsilon_r = \epsilon^{\kappa-\gamma}(X_r^\epsilon -\E[X_r^\epsilon])}$ that $\langle M^\epsilon_\cdot(\phi_{\cdot, t}^\epsilon)\rangle_s = \int_0^s \Gamma_r^\epsilon(\phi_{r, t}^\epsilon)dr$, where
		\begin{align*}
			\Gamma^\epsilon_r(\phi^\epsilon_{r, t}) &= \epsilon^{\kappa-\gamma}\left[F_r^\epsilon((\phi^\epsilon_{r, t})^2)-Y^\epsilon_r(\widetilde{\rho}^{\:\epsilon}_r(\phi^\epsilon_{r, t})^2)\right]
			+\frac{\epsilon}{2}\sum_{z \in \Z_\epsilon} \phi^\epsilon_{r, t}(\epsilon z)^2\E\left[X_r^\epsilon(z)(X_r^\epsilon(z)-\epsilon^\kappa)\right]
			\\&+\epsilon \sum_{z \in \Z_\epsilon}(X_r^\epsilon(z)-\E[X_r^\epsilon(z)])\widetilde{\rho}^{\:\epsilon}_r(\epsilon z)(\phi^\epsilon_{r, t}(\epsilon z))^2
			\\&+ X_r^\epsilon\left((\sigma_\epsilon^2+\mu_\epsilon^2)\phi^\epsilon_{r, t}(\epsilon \cdot)^2+\frac{1}{2}\left[\big(\nabla^{\epsilon, -}\phi^\epsilon_{r, t}(\epsilon \cdot)\big)^2+\big(\nabla^{\epsilon, +}\phi^\epsilon_{r, t}(\epsilon \cdot)\big)^2\right]\right).
		\end{align*}
		By Proposition~\ref{prop:BGP}, the time integral of first term vanishes in $L^2(\P)$ as $\epsilon \searrow 0$, and thus also in probability. By \eqref{eq:V_fcn_2_pcs_diag}, \eqref{eq:V_fcn_one_loc}, and Theorem~\ref{thm:QLLN}, we have 
		\begin{align*}
		\E\left[X_r^\epsilon(z)(X_r^\epsilon(z)-\epsilon^\kappa)\right] &= v^\epsilon_r(2\mathbbm{1}_{\{z\}}, \rho_r^\epsilon|\nu^\epsilon)+2\rho_r^\epsilon(z)\E\left[X_r^\epsilon(z)\right]-\rho_r^\epsilon(z)^2
		\\&= v^\epsilon_r(2\mathbbm{1}_{\{z\}}, \rho_r^\epsilon|\nu^\epsilon)+2\rho_r^\epsilon(z)\left(v^\epsilon_r(\mathbbm{1}_{\{z\}}, \rho_r^\epsilon|\nu^\epsilon)+\rho_r^\epsilon(z)\right)-\rho_r^\epsilon(z)^2
		\\&\overset{\epsilon \searrow 0}{\sim} \rho_r(K_\epsilon^{-1} z)^2.
		\end{align*}
		Therefore, using Lemma \ref{lem:conv_phi_eps}, we obtain the (deterministic) convergence
		\[
		\int_0^s \frac{\epsilon}{2}\sum_{z \in \Z_\epsilon} \phi^\epsilon_{r, t}(\epsilon z)^2\E\left[X_r^\epsilon(z)(X_r^\epsilon(z)-\epsilon^\kappa)\right]dr\overset{\epsilon \searrow 0}{\to} \int_0^s\langle (\phi_{r, t})^2,  \rho_r^2/2\rangle dr. 
		\]
		For the remaining two terms, we use the following observation. Suppose that $\psi^\epsilon_{r, t} \in C(\S^1)$ is continuous in $r$, and suppose that $\sup_{0\leq r\leq t \leq T}\|\psi^\epsilon_{r, t}-\psi_{r, t}\|_\infty\to 0$ as $\epsilon \searrow 0$ for some limiting $\psi_{r, t}$. Then, 
		\[
		\int_0^s X_r^\epsilon(\psi^\epsilon_{r, t}-\psi_{r, t})dr \overset{\epsilon \searrow 0}{\to}0,
		\]
		in probability. Indeed, by the Jensen and Cauchy-Schwarz inequalities, we have  
		\[
		\left|\int_0^s X_r^\epsilon(\psi^\epsilon_{r, t}-\psi_{r, t})dr\right| \leq \int_0^s dr\left(\epsilon \sum_{z \in \Z_\epsilon}X_r^\epsilon(z)^2\right)^{1/2}\left(\epsilon \sum_{z \in \Z_\epsilon}(\psi^\epsilon_{r, t}(\epsilon z)-\psi_{r, t}(\epsilon z))^2\right)^{1/2}.
		\]
		By the (concave) Jensen inequality and Corollary \ref{cor:SLBP_p_moments}, we have 
		\[
		\E\left[\left(\epsilon \sum_{z \in \Z_\epsilon}X_r^\epsilon(z)^2\right)^{1/2}\right]\leq \left(\epsilon \sum_{z \in \Z_\epsilon}\E\left[X_r^\epsilon(z)^2\right]\right)^{1/2}<\infty,
		\]
		uniformly in $\epsilon \in (0, 1)$ and $r\in [0, T]$.
		The claim then follows from Markov's inequality and Lemma \ref{lem:conv_phi_eps}. We conclude that 
		\[
		\int_0^s\epsilon \sum_{z \in \Z_\epsilon}(X_r^\epsilon(z)-\E[X_r^\epsilon(z)])\widetilde{\rho}^{\:\epsilon}_r(\epsilon z)(\phi^\epsilon_{r, t}(\epsilon z))^2dr \overset{\epsilon \searrow 0}{\to} 0,
		\]
		in probability, and that
		\begin{align*}
		&\int_0^sX_r^\epsilon\left((\sigma_\epsilon^2+\mu_\epsilon^2)\phi^\epsilon_{r, t}(\epsilon \cdot)^2+\frac{1}{2}\left[\big(\nabla^{\epsilon, -}\phi^\epsilon_{r, t}(\epsilon \cdot)\big)^2+\big(\nabla^{\epsilon, +}\phi^\epsilon_{r, t}(\epsilon \cdot)\big)^2\right]\right)dr
		\\&\overset{\epsilon \searrow 0}{\to}\int_0^s\left(\langle (\phi_{r, t})^2,  [(\sigma^2+\mu^2)\rho_r]\rangle +\langle \left(\partial_z \phi_{r, t}\right)^2, \rho_r\rangle \right)dr
		\end{align*}
		in probability.
	\end{proof}	
	\begin{proof}[Proof of Proposition \ref{prop:FDD_CLT}]
		Fix $t_1, \cdots, t_k \in [0, T]$, and apply \eqref{eq:smg_id_lim} to write 
		\begin{equation}\label{eq:pre_lim_fdd}
		\begin{aligned}
			&(Y_{t_1}^\epsilon(\phi)-Y_{0}^\epsilon(\phi_{0, t_1}), \cdots, Y_{t_k}^\epsilon(\phi)-Y_{0}^\epsilon(\phi_{0, t_k})) 
			\\&= (Y_{0}^\epsilon(\phi^\epsilon_{0, t_1}-\phi_{0, t_1}), \cdots, Y_{0}^\epsilon(\phi^\epsilon_{0, t_k}-\phi_{0, t_k})) 
			\\&+\left(\int_0^{t_1} \left(Y_r^\epsilon(\widetilde{\rho}^{\:\epsilon}_r \phi^\epsilon_{r, t_1})-F_r^\epsilon(\phi^\epsilon_{r, t_1})\right)dr, \cdots, \int_0^{t_k} \left(Y_r^\epsilon(\widetilde{\rho}^{\:\epsilon}_r \phi^\epsilon_{r, t_k})-F_r^\epsilon(\phi^\epsilon_{r, t_k})\right)dr\right)
			\\&+(M_{t_1}^\epsilon(\phi_{\cdot, t_1}^\epsilon), \cdots, M_{t_k}^\epsilon(\phi_{\cdot, t_k}^\epsilon)).
		\end{aligned}
		\end{equation}
		Re-using the arguments that lead to \eqref{ineq:for_id_lim_too}, we obtain that $Y_{0}^\epsilon(\phi^\epsilon_{0, t_i}-\phi_{0, t_i})\to 0$ in $L^2(\P)$ as $\epsilon \searrow 0$ for each $i \in \{1, \cdots, k\}$. Therefore, the first vector on the right-hand side of \eqref{eq:pre_lim_fdd} vanishes in probability as $\epsilon \searrow 0$. By Proposition \ref{prop:BGP}, we have 
		\[
		\sum_{i=1}^k\left(\int_0^{t_k} \left(Y_r^\epsilon(\widetilde{\rho}^{\:\epsilon}_r \phi^\epsilon_{r, t_k})-F_r^\epsilon(\phi^\epsilon_{r, t_k})\right)dr\right)^2\to 0,
		\]
		in probability as $\epsilon \to 0$, which shows that the second term on the right-hand side of \eqref{eq:pre_lim_fdd} vanishes in probability. By Lemma \ref{lem:mg_term_conv}, the martingale term converges in distribution in path space. In particular, we have 
		\begin{equation}\label{conv:mg_fdd}
		(M_{t_1}^\epsilon(\phi_{\cdot, t_1}^\epsilon), \cdots, M_{t_k}^\epsilon(\phi_{\cdot, t_k}^\epsilon)) \overset{d}{\to} (M_{t_1}(\phi_{\cdot, t_1}), \cdots, M_{t_k}(\phi_{\cdot, t_k}))
		\end{equation}
		as $\epsilon \searrow 0$, and where the limiting centred Gaussian variables $M_{t_i}(\phi_{\cdot, t_i})$ have variance given by $Q_s(\phi_{\cdot, t_i})$ for each coordinate $i \in \{1, \cdots, k\}$, and $\phi_{\cdot, t_i}$ solves \eqref{eq:lim_test_eq_ID_lim} with $t=t_i$. The full convergence in distribution of the right-hand side of \eqref{eq:pre_lim_fdd} to that of \eqref{conv:mg_fdd} follows from Slutsky's Theorem.
		
		By assumption, $\rY^\epsilon \overset{d}{\to}\widetilde{\rY}$ converges in $\cD([0, T], C^\infty(\S^1))$ as $\epsilon\searrow 0$. Since the jumps of $\rY^\epsilon$ are of magnitude $\epsilon^{1+\gamma}L(\epsilon) \leq \epsilon$, they vanish in the limit as $\epsilon \searrow 0$. By continuity of the map (see e.g. \cite[Example 12.1]{B1999}) 
		\[
		\cD([0, T], C^\infty(\S^1)')\to \R, \quad \rZ=(Z_t)_{t\in [0, T]}\mapsto \sup_{t \in [0, T]}|Z_t-Z_{t-}|,
		\]
		we have $\widetilde{\rY}\in C([0, T], C^\infty(\S^1)')$. Thus, the finite-dimensional distributions of $\rY^\epsilon$ converge in distribution to those of $\widetilde{\rY}$. Then, by the continuous mapping theorem, we have
		\[
		(Y_{t_1}^\epsilon(\phi)-Y_{0}^\epsilon(\phi_{0, t_1}), \cdots, Y_{t_k}^\epsilon(\phi)-Y_{0}^\epsilon(\phi_{0, t_k})) \overset{d}{\to}(\widetilde{Y}_{t_1}(\phi)-\widetilde{Y}_{0}(\phi_{0, t_1}), \cdots, \widetilde{Y}_{t_k}(\phi)-\widetilde{Y}_{0}(\phi_{0, t_k})),
		\]
		as $\epsilon \searrow 0$. Furthermore, by Assumption \ref{assump:technical_X_0}.a, Lindeberg's central limit theorem and the Cram\'er-Wold device, the initial data $\widetilde{Y}_0$ is a $C^\infty(\S^1)'$-valued centred Gaussian variable with variance $\E[\widetilde{Y}_0(\phi)^2]=\int_{\S^1}\phi(z)^2\rho_0(z)dz$ for all $\phi \in C^\infty(\S^1)$. Using the arguments for the uniqueness of the mild solution in \cite[Theorem 5.2]{W1986}, we obtain \eqref{eq:fdd_conv}. 
	\end{proof}
	
	\begin{proof}[Proof of Theorem \ref{thm:CLT_Gaussian}]
		The theorem follows from Propositions \ref{prop:Brownian_fluct_tight} and \ref{prop:FDD_CLT}, and \cite[Theorem 13.1]{B1999}.  
	\end{proof} 
	
	\section{Proof of convergence of the $v$-functions}\label{sec:QLLN_proof}
	
	The aim of this section is to prove Theorem \ref{thm:QLLN} by adapting the approach of \cite[Lemma 1]{Boldrig1992}. Let us briefly recall the notation and the statement of the result. We assume that the starting configuration $X_0^\epsilon$ of the rescaled SLBP satisfies Assumption \ref{assump:technical_X_0}. Recall from Definition \ref{def:scaled_v_fcns} the $v$-functions $v^\epsilon: [0, T]\times \X_\epsilon\to \R$ defined by 
	\begin{align*}
	v^\epsilon_t(x, \rho_t^\epsilon|\nu^\epsilon) &\coloneqq \E_{\nu^\epsilon}[V^\epsilon(x, X_t^\epsilon;\rho_t^\epsilon)], \quad x \in \X_\epsilon, \quad t\geq 0,
	\\V^\epsilon(x, X_t^\epsilon;\rho_t^\epsilon) &\coloneqq \sum_{x_0\cleq x} Q^\epsilon(x_0, X_t^\epsilon)(-1)^{n_x-n_{x_0}}\prod_{z \in x\backslash x_0}\rho_t^\epsilon(z), \quad x \in \X_\epsilon,
	\end{align*}
	and where $\rho^\epsilon$ is the discrete-space McKean representation \eqref{def:McKean_fcnl}. By Lemma \ref{lem:law_is_prod}, the $v$-functions  vanish if and only if $X_t^\epsilon$ is a product over $z \in \Z_\epsilon$ of independent Poisson distributions with intensities $\epsilon^{-\kappa}\rho_t^\epsilon(z)$ and rescaled by $\epsilon^\kappa$. Assume that the offspring distribution $\rP$ has a finite moment of order $p$, for some $p \in [1, 2]$. Suppose also that $\kappa \in [0, 3/8]$. If $p\in [1, 2]$ and $n>2$, we have
	\[
	\sup_{t \in [0, T]} \max_{\substack{x \in \X_\epsilon
			\\n_x=n}} |v^\epsilon_t(x|\nu^\epsilon)| = o(\epsilon^{1+\kappa}), \quad \textit{as $\epsilon\searrow 0$}.
	\]
	If $p\in [1, 2)$, then for all $n \in \{1, 2\}$, we have
	\[
	\sup_{t \in [0, T]} \max_{\substack{x \in \X_\epsilon
			\\n_x=n}} |v^\epsilon_t(x|\nu^\epsilon)| = o(\epsilon^{1+(p-1)\kappa}), \quad \textit{as $\epsilon\searrow 0$}.
	\]
	If $p=2$ and $n \in \{1, 2\}$, there exists $c=c(T, p, n)>0$ such that 
	\[
	\sup_{t \in [0, T]} \max_{\substack{x \in \X_\epsilon
			\\n_x=n}} |v^\epsilon_t(x|\nu^\epsilon)| \leq c\epsilon^{1+\kappa}, \quad \textit{for all $\epsilon\in (0, 1)$}.
	\] 
	
	The proof of Theorem \ref{thm:QLLN} consists of three main steps: short-term estimates of the $v$-functions, smoothing properties of the semi-discrete FKPP equation, and an application of the Markov property. In Section \ref{sec:short_time_est}, we prove a series of lemmas which allows us to control the rescaled SLBP over short time intervals of length $\epsilon^\beta$, for some $\beta>0$. We then show a smoothing property of the semi-discrete FKPP equation, namely that after a short amount of time, the solution started from a typical configuration of the rescaled SLBP stays close to that started from the smooth initial data $\rho_0^\epsilon$ of Definition \ref{def:general_init_data}. Finally, we tie these results together using the Markov property to suitably restart $\rX^\epsilon$ and obtain the control claimed in Theorem \ref{thm:QLLN}. 
	
	\subsection{Short-time control of the $v$-functions}\label{sec:short_time_est}
	Fix a time-horizon $T>0$ and we let $t_k^\epsilon = k\epsilon^\beta$, for $k \in \{0, 1, \cdots, \lfloor T/\epsilon^\beta\rfloor\}$. In this section, we show that there exists an event $\Omega_1$ of probability arbitrarily close to one, and such that if $\rX =(X_t)_{t\geq 0}\in \Omega_1$ is a path, then $\rX^\epsilon=(X^\epsilon_t)_{t\geq 0}$ started from $X_{t_k^\epsilon}$, for some $k$, remains close to the solution $\rho^\epsilon_\cdot(\cdot ;X_{t_k^\epsilon})$ to the semi-discrete FKPP equation \eqref{eq:interm_fkpp_diff_form} started from $X_{t_k^\epsilon}$, for small enough $t>0$. Since $\rX^\epsilon \in \Omega_1$ with high probability - see Lemma \ref{lem:rho_f_control_LLN} below - this result shows that the dynamics of typical trajectories of the rescaled SLBP are close to those of the semi-discrete FKPP equation \eqref{eq:interm_fkpp_diff_form} at least for small times. The lemma is phrased in terms of the $v$-functions, and its proof follows that of \cite[Proposition 4.1]{Boldrig1992}.
	
	\begin{lem}\label{lem:BC_small_fluct}
		Fix $\kappa\geq0$, $\beta\in (0, 2)$, $\xi \in (0, \beta/2)$, $n \in \N$, and $q>0$. Then, there exist an event $\Omega_1 = \Omega_1(\xi, \beta, \kappa, n, q)$ and a constant $c_1=c_1(T, n)>0$ with $\P(\rX^\epsilon \in \Omega_1)>1-c_1\epsilon^q$ and such that the following statements hold. Fix any path $\rX=(X_t)_{t\geq 0} \in \Omega_1(\xi)$. There exists $c=c(T, n)>0$ such that for all $\delta_0 \in (0, (1-\beta/2)/2)$, we have  
		\begin{equation}\label{ineq:control_v_eps_1}
			\sup_{t \in [\epsilon^\beta, 2\epsilon^\beta]}\max_{\substack{x \in \X_\epsilon
					\\n_x=n}} |v_t^\epsilon(x|X_{t_k^\epsilon})|<c \epsilon^{\delta_0 n},
		\end{equation}
		for all $k \in \{0, 1, \cdots, \lfloor T/\epsilon^\beta\rfloor\}$, and $\epsilon \in (0, 1)$. Moreover, $c(T, n')>c(T, n)$ for all natural numbers $n'>n$. Furthermore, there exists $c=c(T, n)>0$ such that
		\begin{equation}\label{ineq:control_v_eps_0}
			\sup_{t \in [0, \epsilon^\beta]}\max_{\substack{x \in \X_\epsilon
					\\n_x=n}} |v_t^\epsilon(x|\nu^\epsilon)|<c\epsilon^{\lfloor (n+1)/2\rfloor(1+\kappa)},
		\end{equation}
		for all $\epsilon \in (0, 1)$. We also have $c(T, n')>c(T, n)$ for all natural numbers $n'>n$.
	\end{lem}

	The proof of this result is divided into several lemmas on which we progressively establish short-time control on the $v$-functions. The first lemma controls, with high probability, two quantities of interest over a future time interval of length $\epsilon^\beta$. It quantifies the growth (in $\epsilon^{-1}$) of the solution to the semi-discrete FKPP equation \eqref{eq:interm_fkpp_diff_form} when started from a typical configuration of the rescaled SLBP configuration. It also bounds the expected growth rate of the polynomials $Q^\epsilon(x, X^\epsilon_{t_k^\epsilon})$ defined in \eqref{eq:interm_fkpp_diff_form}, for test configurations $x$ of a given size. The proof is very similar to that of \cite[Lemma 4.1]{Boldrig1992}. 
	
	\begin{lem}\label{lem:rho_f_control_LLN}
		Fix any $\xi \in (0, \beta)$, $n \in\N$, and $q>0$. Then, there exist an event $\Omega_1 = \Omega_1(\xi, \beta, \kappa, n, q)$ and a constant $c=c(T, n)>0$ such that $\P(\rX^\epsilon \in \Omega_1)>1-c\epsilon^q$, and for any $\rX =(X_t)_{t\geq 0}\in \Omega_1$, the following inequalities hold: for any $k \in \{0, \cdots, \lfloor T/\epsilon^\beta\rfloor\}$, there exists $c=c(T)>0$ such that 
		\begin{align}
			\label{ineq:rho_control}&\sup_{t\in [0, T]} \max_{z \in \Z_\epsilon} |\rho^\epsilon_t(z;X_{t_k^\epsilon})|<c\epsilon^{-\xi}
			\\&\sup_{t \in [0, \epsilon^\beta]}\max_{x \in \X_\epsilon: n_x= n} \E_{X_{t_k^\epsilon}}[Q^\epsilon(x, X_t^\epsilon)] < \epsilon^{-\xi n},\label{ineq:Q_eps_control}
		\end{align}
		where $\rho_t^\epsilon(z; X_{t_k^\epsilon})$, $t\geq 0$, $z \in \Z_\epsilon$ is the solution to \eqref{eq:interm_fkpp_diff_form} started from $X_{t_k^\epsilon}$. 
	\end{lem}
	\begin{proof}
		By Proposition \ref{prop:SLBP_moments}, there exists $c(n, T)>0$ such that
		\begin{equation}\label{eq:LLN_finite_mom_application}
			\sup_{0\leq t\leq T}\max_{\substack{x \in \X_\epsilon\\ n_x= n}} \E\left[Q^\epsilon(x, X_t^\epsilon)\right] \leq c(n, T).
		\end{equation}
		Let $\zeta \in (0, \xi)$ and $m \in \N$. By a union bound (recalling that there are $\lfloor \epsilon^{-1}\rfloor$ sites in $\Z_\epsilon$) and Markov's inequality, we have 
		\[
		\P\left(\max_{z \in \Z_\epsilon}X_{t_k^\epsilon}^\epsilon(z)>\epsilon^{-\zeta}\right)\leq \lfloor\epsilon^{-1}\rfloor\max_{z \in \Z_\epsilon}\E\left[X_{t_k^\epsilon}^\epsilon(z)^m\right]\epsilon^{\zeta m}.
		\]
		Observe that there exists $c_1=c_1(m)>0$ such that $n^m\leq c_1Q_m(n)$ for all natural numbers $n\geq m$. We consider $n=\epsilon^{-\kappa}X_{t_k^\epsilon}^\epsilon(z)$ and we partition on $n< m$ and $n\geq m$ to obtain
		\[
		X_{t_k^\epsilon}^\epsilon(z)^m \leq \epsilon^{\kappa m}m^m + c_1 Q^\epsilon_m(\epsilon^{-\kappa}X_{t_k^\epsilon}^\epsilon(z))\leq m^m+c_1 Q^\epsilon_m(\epsilon^{-\kappa}X_{t_k^\epsilon}^\epsilon(z)).
		\]
		Thus, by \eqref{eq:LLN_finite_mom_application}, we have
		\begin{align*}
		\max_{z \in \Z_\epsilon}\E\left[X_{t_k^\epsilon}^\epsilon(z)^m\right] &\leq m^m+c_1c(m, T).
		\end{align*}
		It follows that 
		\[
		\P\left(\max_{z \in \Z_\epsilon}X_{t_k^\epsilon}^\epsilon(z)>\epsilon^{-\zeta}\right)\leq c_2\epsilon^{\zeta m-1},
		\]
		uniformly in $k \in \{0, 1, \cdots, \lfloor T/\epsilon^\beta\rfloor\}$, and where $c_2(m, T) =m^m+c_1(m)c(m, T)$. Next, define the event
		\[
		\overline{\Omega} = \left\{\rX \in \cD([0, T], \X_\epsilon): \max_{k \in \{0, 1, \cdots, \lfloor T/\epsilon^\beta\rfloor\}}\max_{z \in \Z_\epsilon}X_{t^\epsilon_k}(z)\leq \epsilon^{-\zeta}\right\}.
		\]
		We have 
		\[
		\overline{\Omega}^c= \bigcup_{k=0}^{\lfloor T/\epsilon^\beta\rfloor} \left\{\max_{ z\in \Z_\epsilon} X_{t_k^\epsilon}(z)> \epsilon^{-\zeta}\right\}.
		\]		
		Hence, choosing $m$ sufficiently large for $\zeta m > 2+\beta+q$ to hold, we obtain by a union bound that
		\[
		\P\left(\rX^\epsilon \in\overline{\Omega}^c\right) \leq \sum_{k=0}^{\lfloor T/\epsilon^\beta\rfloor} \P\left(\max_{z \in \Z_\epsilon}X^\epsilon_{t_k^\epsilon}(z)>\epsilon^{-\zeta}\right) = T c_2(m, T)\epsilon^{\zeta m-2-\beta} < T c_2(m, T) \epsilon^q,
		\]
		for all $\epsilon \in (0, 1)$. Consequently, 
		\[
		\P\left(\rX^\epsilon \in\overline{\Omega}\right) = 1-\P\left(\rX^\epsilon \in\overline{\Omega}^c\right)\geq 1-T c_2(m, T)\epsilon^q.
		\]
		On $\overline{\Omega}$, the control of $\rho^\epsilon$ claimed in \eqref{ineq:rho_control} follows from part two of Lemma \ref{lem:FKPP_approx_properties}. It remains to control ${\E_{X_{t_k^\epsilon}}[Q^\epsilon(x, X_t^\epsilon)]}$. By Lemma \ref{lem:mg_pb_LLN}, for any $t\in [0, \epsilon^\beta]$, we have 
		\[
		\E_{X_{t_k^\epsilon}}[Q^\epsilon(x, X_t^\epsilon)] = \sum_{x'} \G^\epsilon_{t}(x, x')Q^\epsilon(x', X_{t_k^\epsilon})+\int_0^t \sum_{x'} \G^\epsilon_{t-s}(x, x')\E_{X_{t_k^\epsilon}}[\cG^\epsilon_{BC}Q^\epsilon(x', X_s^\epsilon)] ds,
		\]
		where $x'$ ranges over configurations $x'\in \X_\epsilon$ with $n_{x'}=n_x=n$. To bound the first term, we use that $n_{x'}=n$ and $\rX \in \overline{\Omega}$ to obtain
		\[
		\sum_{x'} \G^\epsilon_{t}(x, x')Q^\epsilon(x', X_{t_k^\epsilon})\leq \epsilon^{-\zeta n}\sum_{x'} \G^\epsilon_{t}(x, x') = \epsilon^{-\zeta n}. 	
		\] 
		where the last equality holds since $\sum_{x'} \G^\epsilon_{t}(x, x')=1$. In absolute value, the second term is upper bounded by 
		\[
		R^\epsilon(x) \coloneqq \int_0^{\epsilon^\beta} \sum_{x'} \G^\epsilon_{t-s}(x, x')\left|\E_{X_{t_k^\epsilon}}\left[\cG^\epsilon_{BC}Q^\epsilon(x', X^\epsilon_s)\right]\right| ds.
		\]
		Let $N \in \N$. It is easy to see using Corollary \ref{cor:SLBP_p_moments} that for some constant $c(n, N, T)>0$, we have 
		\[
		\max_{x\in \X_\epsilon:n_x=n}\E\left[R^\epsilon(x)^N\right]\leq c(n, N, T)\epsilon^{\beta N},
		\]
		for all $\epsilon\in (0, 1)$. Then, a union bound and Markov's inequality imply that
		\[
		\P\left(\max_{x\in \X_\epsilon:n_x=n} R^\epsilon(x)>1\right)\leq c(n, N, T)\epsilon^{\beta N}{\lfloor \epsilon^{-1}\rfloor \choose n} \leq c'(n, N, T) \epsilon^{\beta N-n}
		\]
		for some $c'(n, N, T)>0$. We introduce 
		\[
		\Omega_1 \coloneqq \overline{\Omega}\cap \left\{ \max_{\substack{x \in \X_\epsilon
			\\n_x=n}}R^\epsilon(x)\leq 1\right\}.
		\]
		By a union bound, we get
		\begin{align*}
			\P(\rX^\epsilon \in\Omega_1^c)&\leq \P\left(\rX^\epsilon \in\overline{\Omega}^c\right)+\P\left(\max_{x \in \X_\epsilon:n_x=n}R^\epsilon(x)> 1\right)
			\\&<Tc_1(m)c(m, T)\epsilon^q + c'(n, N, T)\epsilon^{\beta N-n}. 
		\end{align*}
		We may now choose $N$ large enough that $\beta N>n+q$, which leads to $\P(\Omega_1)>1-c\epsilon^q$, for some constant $c=c(n, T, m, N)>0$. Since $\zeta \in (0, \xi)$, this proves \eqref{ineq:Q_eps_control} as long as $\rX \in \Omega_1$. 
	\end{proof}
	
	In the following technical lemma, we find an upper bound on the $v$-functions by iterating the hierarchy of equations from Proposition \ref{prop:v_fcn_eq}, and by using the properties of paths which belong to the event $\Omega_1$ of Lemma \ref{lem:rho_f_control_LLN}. 
	
	\begin{lem}\label{lem:v_fcn_bd}
		Consider $\Omega_1=\Omega_1(\xi, \beta, \kappa, n, q)$ the event identified in Lemma \ref{lem:rho_f_control_LLN}. Fix a path $\rX=(X_t)_{t\geq 0} \in \Omega_1$, and introduce the space $\bT_i(t) = \{(s_{i-1}, \cdots, s_1, s_0): 0\leq s_{i-1}\leq \cdots\leq s_1\leq s_0=t\}$ of ordered times. Let
		$m:\Z\to \N$ be given by
		\[
		m(h)\coloneqq\begin{cases}1&h=1\\\max(2, -h)&h\leq 0\end{cases},
		\] 
		and for each $\epsilon \in (0, 1)$ and $t\geq 0$, define maps $a^\epsilon: \bT_2(t)\times \X_\epsilon \times \Z_\epsilon \times \Z\times \R^{\X_\epsilon} \to \R$ by 
		\[
		a^\epsilon_{s_1, t}(x, z_1, h_1)f \coloneqq \sum_{\substack{x_1 \in \X_\epsilon \\ n_{x_1}=n_x}} \mathbbm{1}_{n_{x_1}(z_1)\geq m(h_1)}\G^\epsilon_{t-s_1}(x, x_1)f\left(x_1^{(z_1, h_1)}\right), 
		\]
		for all $(s_1, t) \in \bT_2(t)$, $x \in\X_\epsilon$, $z_1 \in \Z_\epsilon$, $h_1 \in \Z$, $f \in \R^{\X_\epsilon}$, where $x_1^{(z_1, h_1)}$ denotes the configuration $x_1$ with $h_1$ particles added at $z_1$ if $h_1\geq 0$, and $-h_1$ particles removed from $z_1$ if $h_1<0$, and $\rho^\epsilon_\cdot(\cdot ;X_{t_k^\epsilon})$ is the solution to the semi-discrete FKPP equation \eqref{eq:interm_fkpp_diff_form} started from $X_{t_k^\epsilon}$ which is constructed in \eqref{def:McKean_fcnl}. We also introduce for each $\epsilon \in (0, 1)$, $t\geq 0$, and $i\geq 0$, maps ${\Gamma^\epsilon:\bT_{i+1}(t) \times \X_\epsilon^2 \times \Z_\epsilon^i \times \Z^i\to \R}$ given by 
		\begin{align}
		&\label{def:Gamma_eps}\Gamma^\epsilon_{t, x, x_{i+1}}(s_1, \cdots, s_i; z_1, \cdots, z_i; h_1, \cdots, h_i) 
		\\&\notag\quad\quad\quad\quad\quad\quad \coloneqq a_{s_1, t}^\epsilon(x, z_1, h_1)a_{s_2, s_1}^\epsilon(\cdot, z_2, h_2)\cdots a_{s_i, s_{i-1}}^\epsilon(\cdot, z_i, h_i)\G^\epsilon_{s_i}(\cdot, x_{i+1})\mathbbm{1}_{n_\cdot=n_{x_{i+1}}},
		\end{align}
		for all $(s_i, \cdots, s_1, s_0=t)\in \bT_{i+1}(t)$, $x, x_{i+1}\in \X_\epsilon$, $z_1, \cdots, z_i \in \Z_\epsilon$, and $h_1, \cdots, h_i \in \Z$. Then, for any $x \in \X_\epsilon$ with $n_x=n \in \N$ and for any $N \in \N$, there exist constants $c_1=c_1(n, N), c_2=c_2(n)>0$ such that
		\begin{equation}
		\begin{aligned}
		|v_t^\epsilon(x, \rho_t^\epsilon|X_{t_k^\epsilon})|&\label{ineq:v_control}\leq c_1\epsilon^{\left[\beta-\xi(1+n/(N+1))\right](N+1)}+ C_0^\epsilon(t, x) 
		\\&+c_2 \epsilon^{-\xi N(n+N)} \sum_{i=1}^N \sum_{\substack{x_{i+1}\in \X_\epsilon \\ z_1, \cdots z_i \in \Z_\epsilon \\ h_1, \cdots, h_i \in \Z}} \epsilon^{\kappa\lceil n_{x_{i+1}}/2\rceil-\xi \lfloor n_{x_{i+1}}/2\rfloor } 
		\\&\quad\quad\quad\quad\quad\times\int_0^t ds_1 \int_0^{s_1}ds_2 \cdots \int_0^{s_{i-1}}ds_i\Gamma^\epsilon_{t, x, x_{i+1}}(s_1, \cdots, s_i; z_1, \cdots, z_i; h_1, \cdots, h_i),
		\end{aligned}
		\end{equation}
		for all $t\geq 0$ and $\epsilon \in (0, 1)$. 
	\end{lem}
	\begin{proof}
		For each $\epsilon \in (0, 1)$, we define maps $A^\epsilon: \bT_2(t)\times \X_\epsilon \times \R^{\X_\epsilon} \to \R$ by 
		\[
		A^\epsilon_{s, t}(x)f \coloneqq \sum_{\substack{x_1 \in \X_\epsilon \\ n_{x_1}=n_x}} \G^\epsilon_{t-s}(x, x_1)\sum_{z_1 \in \supp(x_1)}\sum_{h_1=-n_{x_1}(z_1)}^1 c_h^\epsilon(x_1(z_1), \rho_{s_1}^\epsilon(\cdot; X_{t_k^\epsilon}))f\left(x_1^{(z_1, h_1)}\right), 
		\]
		for all $(s, t) \in \bT_2(t)$, $x \in\X_\epsilon$, $f \in \R^{\X_\epsilon}$, where $c_h^\epsilon$ are the coefficients derived in Proposition \ref{prop:v_fcn_eq}, $x_1^{(z_1, h_1)}$ denotes the configuration $x_1$ with $h_1$ particles added at $z_1$ if $h_1\geq 0$, and $-h_1$ particles removed from $z_1$ if $h_1<0$, and $\rho^\epsilon_\cdot(\cdot ;X_{t_k^\epsilon})$ is the solution to the semi-discrete FKPP equation \eqref{eq:interm_fkpp_diff_form} started from $X_{t_k^\epsilon}$ which is constructed in \eqref{def:McKean_fcnl}. We also introduce for each $\epsilon \in (0, 1)$ maps $C_i^\epsilon: [0, \infty)\times \X_\epsilon\to \R$, $i\geq 0$, defined by 
		\begin{equation}\label{def:C_0}
		C_0^\epsilon(t, x) \coloneqq \sum_{\substack{x_1 \in \X_\epsilon \\ n_{x_1}=n_x}} \G^\epsilon_t(x, x_1)V^\epsilon(x_1, X_{t_k^\epsilon};X_{t_k^\epsilon}), \quad t\geq 0, \quad x \in \X_\epsilon,
		\end{equation}
		and 
		\[
		C_i^\epsilon(t, x) \coloneqq \int_0^t ds_1 \int_0^{s_1}ds_2 \cdots \int_0^{s_{i-1}}ds_i A_{s_1, t}^\epsilon(x)A_{s_2, s_1}^\epsilon(\cdot)\cdots A_{s_i, s_{i-1}}^\epsilon(\cdot)C_0^\epsilon(s_i, \cdot), \quad i\geq 1,
		\]
		with $s_0=t$. Let $N \in \N$ and define for each $\epsilon \in (0, 1)$ the map $R_{N+1}^\epsilon:[0, \infty)\times \X_\epsilon \to \R$ by
		\[
		R_{N+1}^\epsilon(t, x) \coloneqq \int_0^t ds_1 \int_0^{s_1}ds_2 \cdots \int_0^{s_N}ds_{N+1} A_{s_1, t}^\epsilon(x)A_{s_2, s_1}^\epsilon(\cdot)\cdots A_{s_{N+1}, s_{N}}^\epsilon(\cdot) v_{s_{N+1}}^\epsilon(\cdot, \rho_{s_{N+1}}|X_{t_k^\epsilon}). 
		\]
		By Proposition \ref{prop:v_fcn_eq}, the $v$-functions solve the following hierarchy of equations
		\begin{align}
			&v_t^\epsilon(x, \rho_t^\epsilon(\cdot\;;X_{t_k^\epsilon})|X_{t_k^\epsilon}) = \sum_{\substack{x_1\in \X_\epsilon\\ n_{x_1}=n_x}}\G^\epsilon_t(x, x_1)V^\epsilon(x_1, X_{t_k^\epsilon}; X_{t_k^\epsilon}) \notag
			\\&\quad\quad\quad\quad\quad\quad\quad\quad\quad+ \int_0^t ds_1 \sum_{\substack{x_1\in \X_\epsilon\\ n_{x_1}=n_x}} \G^\epsilon_{t-s_1}(x, x_1)\label{eq:BBGKY_before_iteration}
			\\&\quad\quad\quad\quad\quad\quad\quad\quad\quad\times\sum_{z_1 \in \supp(x_1)}\sum_{h_1=-|x_1(z_1)|}^1 c_h^\epsilon(x_1(z_1), \rho_{s_1}^\epsilon(z_1; X_{t_k^\epsilon}))v_{s_1}^\epsilon(x_1^{(z_1, h_1)}, \rho_{s_1}^\epsilon(\cdot; X_{t_k^\epsilon})|X_{t_k^\epsilon}),\notag
		\end{align}
		for all $x \in \X_\epsilon$, $t\geq 0$, and $\epsilon \in (0, 1)$. The same equation is satisfied by $v_{s_1}^\epsilon$. Thus, by iterating \eqref{eq:BBGKY_before_iteration} $N$ times, we may write
		\[
		v_t^\epsilon(x, \rho_t^\epsilon|X_{t_k^\epsilon}) = \sum_{i=0}^N C_i^\epsilon(t, x) + R_{N+1}^\epsilon(t, x).
		\]
		Let $n \in \N$. We claim that there exists $c_1(n, N)>0$ such that
		\[
		\max_{\substack{x \in \X_\epsilon
			\\n_x=n}}|R_N^\epsilon(t, x)|< c_1\epsilon^{[\beta-\xi(1+2n/N)]N}.
		\]
		Since $\rX \in \Omega_1$ and $n_{x_N}+h_N\leq n+N$, we apply Lemma \ref{lem:rho_f_control_LLN} to obtain 
		\begin{align*}
			&|v_{s_N}^\epsilon(x_N^{(z_N, h_N)}, \rho^\epsilon_{s_N}(\cdot; X_{t_k^\epsilon})| X_{t_k^\epsilon})|
			\\&=\Bigg|\sum_{x'\cleq x_N^{(z_N, h_N)}} \E_{X_{t_k^\epsilon}}[Q^\epsilon(x', X_{s_N}^\epsilon)](-1)^{n_{x_N}+h_N-n_{x'}} \prod_{z \in x_N^{(z_N, h_N)}\backslash x'} \rho_{s_N}^\epsilon(z)\Bigg|
			\\&\leq \sum_{x'\cleq x_N^{(z_N, h_N)}} \E_{X_{t_k^\epsilon}}[Q^\epsilon(x', X_{s_N}^\epsilon)] \epsilon^{-\xi (n_{x_N}+h_N)}
			\\&\leq 2^{n+N} \epsilon^{-\xi(1+ 2n/N)N},
		\end{align*}
		which proves the claim with $c_1(n, N)=2^{n+N}$. We interchange the order of summation in each $A_{s_i, s_{i-1}}^\epsilon(\cdot)$ as follows:
		\begin{align*}
			&\sum_{\substack{x_i \in \X_\epsilon \\ n_{x_i} = n_{x_{i-1}}+h_{i-1}}} \G^\epsilon_{s_{i-1}-s_i}(x_{i-1}^{(z_{i-1}, h_{i-1})}, x_i) \sum_{z_i \in \supp(x_i)} \sum_{h_i=-n_{x_i}(z_i)}^1 c_{h_i}^\epsilon(x_i(z_i), \rho_{s_i}^\epsilon(\cdot; X_{t_k^\epsilon}))
			\\&=\sum_{z_i \in \Z_\epsilon}\sum_{h_i \in \Z}\sum_{\substack{x_i \in \X_\epsilon \\ n_{x_i} = n_{x_{i-1}}+h_{i-1}}}\G^\epsilon_{s_{i-1}-s_i}(x_{i-1}^{(z_{i-1}, h_{i-1})}, x_i)\mathbbm{1}_{n_{x_i}(z_i)\geq m(h_i)} c^\epsilon_{h_i}(x_i(z_i), \rho_{s_i}^\epsilon(\cdot;X_{t_k^\epsilon})).
		\end{align*}
		The threshold in the indicator function ensures that births can only occur at locations where there is at least one existing particle, and deaths can only occur at locations with two or more particles; the same constraints apply in the dynamics of the SLBP. Observe that for each $j \in \{0, 1, \cdots, i\}$, we have $n_{x_j}(z_j)\leq n+j$ since at most one particle is added to the test configuration during each interaction. Then, by Proposition \ref{prop:v_fcn_eq} and Lemma \ref{lem:rho_f_control_LLN}, we have for some constant $c_0(n)>0$
		\begin{equation}\label{ineq:control_coeffs_t_k}
		c^\epsilon_{h_j}(x_j(z_j), \rho_{s_j}^\epsilon(\cdot;X_{t_k^\epsilon}))\leq c_0(n) \epsilon^{-\xi (n_{x_j}(z_j)+1)}\leq c_0(n)\epsilon^{-\xi (n+j+1)},
		\end{equation}
		for each $j \in \{1, \cdots, i\}$. Hence the integrand of $C_i^\epsilon(t, x)$ satisfies
		\begin{align*}
			&A_{s_1, t}^\epsilon(x)A_{s_2, s_1}^\epsilon(\cdot)\cdots A_{s_i, s_{i-1}}^\epsilon(\cdot)C_0^\epsilon(s_i, \cdot)
			\\&\leq c_0(n) \epsilon^{-\xi i(n+i)} a_{s_1, t}^\epsilon(x, z_1, h_1)a_{s_2, s_1}^\epsilon(\cdot, z_2, h_2)\cdots a_{s_i, s_{i-1}}^\epsilon(\cdot, z_i, h_i)C_0^\epsilon(s_i, \cdot),
		\end{align*}
		for some other constant $c_0(n)>0$. At most one particle is added to the test configuration during each interaction, so $n_{x_{i+1}}\leq n+i$. The second point of Lemma \ref{lem:V_fcn_properties_LLN} then implies that 
		\begin{align}
			|V^\epsilon(x_{i+1}, X_{t_k^\epsilon};X_{t_k^\epsilon})|\notag&< c'(n)  \epsilon^{\kappa \lceil n_{x_{i+1}}/2\rceil}\left( \max_{z \in \supp(x_{i+1})}X_{t_k^\epsilon}(z)\right)^{\lfloor n_{x_{i+1}}/2\rfloor}
			\\&<c'(n)\epsilon^{\kappa\lceil n_{x_{i+1}}/2\rceil-\xi \lfloor n_{x_{i+1}}/2\rfloor },\label{ineq:init_V_bd_to_improve}
		\end{align}
		for some $c'(n)>0$. Combining the calculations above shows \eqref{ineq:v_control} with $c_2(n)=c_0(n)c'(n)$. 
	\end{proof}
	
	A key step in the proof of Lemma \ref{lem:BC_small_fluct} is to show that the term involving $\Gamma^\epsilon$ in \eqref{ineq:v_control} is small with respect to $\epsilon$. The method used to show the analogous bound (4.13) in Appendix B of \cite{Boldrig1992} applies to our setting with the following added difficulty. The (deterministic) dynamics of the $v$-functions computed in Proposition \ref{prop:v_fcn_eq} can result in the killing of the entire local population in a single interaction. This is directly caused by the unboundedness of the offspring distribution $\rP$. By contrast, in the binary branching case in \cite{Boldrig1992}, it is not possible to kill the entire local population in the dynamics of the $v$-functions as long as the local population size is larger than a certain fixed threshold. In the main step of the argument, one simplifies the problem of controlling $\Gamma^\epsilon$ by first bounding the trajectories of particles which are eventually killed. We resolve the added difficulty by carefully recording the impact of this first step on the initial configuration. Finally, there is also one aspect of the method which is simpler in our setting, namely the fact that our rates of birth and competition are linear and quadratic, respectively, as opposed to falling factorials of arbitrary degrees. This simplifies Lemma \ref{lem:vertex_type} below.
	
	\begin{lem}\label{lem:Gamma_control}
		Assuming that $\beta \in (0, 2)$, there exists $c=c(T, n)>0$ such that 
		\begin{equation}\label{ineq:Gamma_eps_control}
			\int_0^{2\epsilon^{\beta}} ds_1 \int_0^{s_1}ds_2 \cdots \int_0^{s_{i-1}}ds_i \sum_{\substack{x_{i+1}\in \X_\epsilon \\ z_1, \cdots z_i \in \Z_\epsilon}} \Gamma^\epsilon_{t, x, x_{i+1}}(s_1, \cdots, s_i; z_1, \cdots, z_i; h_1, \cdots, h_i)\leq c \epsilon^{\beta i + (1-\beta/2)n/2},
		\end{equation}
		for all $x \in \X_\epsilon$ such that $n_x=n$, $h_1, \cdots, h_i \in \Z$, and $\epsilon \in (0, 1)$. 
	\end{lem}
	Before proving this lemma, we follow Appendix B of \cite{Boldrig1992} and introduce a graphical representation of $\Gamma^\epsilon$. 
	
	The function $\Gamma^\epsilon$ defined in \eqref{def:Gamma_eps} describes the (deterministic) evolution of a configuration $x\in \X_\epsilon$ of $n$ particles backwards in time from $t>0$ to time zero. The particles undergo successive periods of diffusive motion, separated by instantaneous reactions in which at some location $z$ at some time $s$, a number of particles $h$ is added (if $h\in \{0, 1\}$) or removed (if $h< 0$). During the interactions, particles can only be added at $z$ if there were at least  particle present at time $s+$, and particles can only be removed if there were at least two particles at time $s+$. At most one particle is added locally, while the entire local population may be removed. 
	
	We represent $\Gamma^\epsilon$ as a directed graph with vertices at $(s_i, z_i)$, where the $s_i$ are interaction times (with $s_0=t$ the initial time since we work backwards in time), and the $z_i$ are the respective locations of the interactions. If there are multiple particles at the same location $z \in \Z_\epsilon$ at the initial time $t>0$, they each get assigned their own vertex, still denoted $(t, z)$. This is to account for the fact that their subsequent diffusive motion is independent. For each pair of vertices $(s_i, z_i)$ and $(s_j, z_j)$, $s_i>s_j$, we add an edge $\{(s_i,z_i), (s_j, z_j)\}$, directed backwards in time, and corresponding to the factor $G^\epsilon_{s_i-s_j}(z_i, z_j)$ (i.e. to the random motion of a given particle from $z_i$ to $z_j$ in time $s_i-s_j$). At a given reaction time $s_i+<t$ which is not the initial time, particles which are about to interact non-trivially at $z_i$ meet at the same vertex $(s_i, z_i)$, and interact according to $h_i$. In the following lemma, we identify as in \cite{Boldrig1992} four types of vertices which are sufficient to describe any such interaction. 
	
	\begin{lem}\label{lem:vertex_type}
		Given a vertex $(s, z)$, we denote by $m$ the number of incoming edges, and by $p$ the number of outgoing edges. We can encode $\Gamma^\epsilon$ of \eqref{def:Gamma_eps} as a (directed) graph by the above construction using only the following four types of vertices :
		\begin{enumerate}
			\item[T1.] No interaction: $h=0$, so $m=1$, $p=1$;
			\item[T2.] Birth of one particle: $h=1$, so $m=1$, $p=2$;
			\item[T3.] Death of one particle: $h=-1$, so $m=2$, $p=1$;
			\item[T4.] Death of $-h\geq 2$ particles: $h\leq -2$, so $m=-h$, $p=0$.
		\end{enumerate}
	\end{lem}
	\begin{proof}
		We note that in the dynamics described by $\Gamma^\epsilon$ and generally by \eqref{eq:BBGKY_before_iteration}, any local sub-colony of particles may vanish during an interaction as long as there are two particles present before the interaction. If there are $m\geq 2$ incoming edges at vertex $(s, z)$ and $h \in \{-1, \cdots, -m\}$, then we can describe the event that $|h|$ of the $m$ particles vanish in terms of the vertices of type T1-T4. Indeed, if $m\geq 2$ and $h\not=-1$ when $m=2$, then we decompose the interaction into a T4 vertex where $|h|$ of the $m$ particles meet and vanish, and $m-|h|$ T1 vertices, where particles just perform diffusion and interact trivially. When $m=2$ and $h=-1$, this construction is not allowed since it would require a vertex in which a particle self-annihilates. Instead, we simply use a T3 vertex. If $m=1$, trivial interactions ($h=0$) are captured by T1 vertices, and births, which in $\Gamma^\epsilon$ can at most add one particle locally ($h=1$) are directly captured by T2 vertices. If $m\geq 2$ and $h=0$, we simply let $m$ particles have their own T1 vertex at $(s, z)$. If $m\geq 2$ and $h=1$, then we let $m-1$ particles have their own vertex at at $(s, z)$, while one particle has a T2 vertex at $(s, z)$. 
	\end{proof}
	
	\begin{proof}[Proof of Lemma \ref{lem:Gamma_control}]
		We may assume without loss of generality that the undirected version of the graph just constructed is connected. Indeed, if there are more than one connected components, the integrand on the left-hand side of \eqref{ineq:Gamma_eps_control} splits as a product over the connect components. 
		
		We prove the estimate \eqref{ineq:Gamma_eps_control} using the following two steps as in \cite{Boldrig1992}. We first use basic properties of the Green's function (the semigroup property, $\sum_{x'}\G^\epsilon_t(x, x')=1$, and the estimate \eqref{ineq:basic_heat_estimate}) to prune the graph until we have reduced the sub-graph which excludes the final (time-$0$) vertices to a binary tree. This has the effect of greatly simplifying the expression for $\Gamma^\epsilon$, and leaves us with the task of showing \eqref{ineq:Gamma_eps_control} when only T2 vertices, i.e. terms of the form $G^\epsilon_{s_k-s_i}(x_i, x_k)G^\epsilon_{s_k-s_j}(x_j, x_k)$, are present in its graphical representation. While carrying out this pruning, we keep track of the maximum number of initial (time-$2\epsilon^\beta$) particles which are removed, and of the resulting contributions to an upper bound on the left-hand side of \eqref{ineq:Gamma_eps_control}. In a second step, we control the remaining binary tree using \eqref{ineq:basic_heat_estimate} and the key observation that all final vertices have at least two incoming edges, since otherwise $V^\epsilon(x_{i+1}, X_{t_k^\epsilon};X_{t_k^\epsilon})=0$ by Lemma \ref{lem:V_fcn_properties_LLN}.

		Before starting the pruning procedure, we let $I_4$ denote the number of particles in the initial configuration $x$ which die at a $T4$ vertex before reaching time zero. We also denote by $T_i$ the number of T$i$ vertices for $i\in \{1, \cdots, 4\}$. 
		
		We first remove the vertices of type T1 using the semigroup property. Indeed, if $(s_k, z_k)$ is of type T1 with incoming edge $\{(s_i, z_i), (s_k, z_k)\}$ and outgoing edge $\{(s_k, z_k), (s_j, z_j)\}$, then
		\[
		\sum_{z_k\in \X_\epsilon}G^\epsilon_{s_i-s_k}(z_i, z_k)G^\epsilon_{s_k-s_j}(z_k, z_j) = G^\epsilon_{s_i-s_j}(z_i, z_j), 
		\]
		so we remove $(s_k,z_k)$, $\{(s_i, z_i), (s_k, z_k)\}$, and $\{(s_k, z_k), (s_j, z_j)\}$ from the graph and we add an edge $\{(s_i, z_i), (s_j, z_j)\}$. This step does not cause the removal of any initial particle. Next, we collapse the T3 vertices into T1 vertices, starting with the smallest times, and using the Green' function estimate \eqref{ineq:basic_heat_estimate}. Each time the estimate \eqref{ineq:basic_heat_estimate} is applied to remove an edge, we need to worry about the vertex at the other end of the edge. If it is an initial vertex, we remove it from the initial configuration $x$. If it is a vertex of type T2 , then it only has two out-edges (we are in a special case of Appendix B of \cite{Boldrig1992} since our birth rates are linear in the local population size), and it therefore becomes a T1 vertex. If it is a T3 vertex, it becomes a T4 vertex. By Lemma \ref{lem:vertex_type}, T3 vertices have two incoming edges, so the new T4 vertex also has two incoming edges. For future reference, denote by $T_4^+$ the total number of additional T4 vertices created in this way. Since each T3 vertex has two incoming edges, we have overall that the number $I_3\coloneqq c_3 n_x$, for some $c_3 \in [0, 1]$, of initial vertices that are removed by the above procedure satisfies $I_3 \leq \lfloor (n_x-I_4)/2\rfloor$. 
		
		Next, we control all the T4 vertices (the $T_4$ original ones, as well as the $T_4^+$ newly created ones), again in order of increasing time coordinate. By Lemma \ref{lem:vertex_type}, a T4 vertex always has at least two incoming edges, so that $2T_4\leq I_4$. At a given T4 vertex with $k$ incoming edges, we apply \eqref{ineq:basic_heat_estimate} $k-1$ times to remove all but one of the edges, and then we use $\sum_{z'\in \Z_\epsilon}\G^\epsilon_s(z, z')=1$ to control the remaining edge. The other end of the removed edges is handled as above, and we apply the semigroup property to control any remaining T1 vertices. Overall, if $I_4 =c_4n_x$, for some $c_4 \in [0, 1]$, the number of particles removed from the initial configuration $x$ by controlling the $T_4$ original T4 vertices is $I_4-T_4=c_4n_x-T_4\geq c_4n_x/2$. We have also controlled $T_4^+\leq \lfloor (n_x-I_4)/2\rfloor-I_3$ new T4 vertices created. Each of these has exactly two incoming edges, hence if $I_4^+\coloneqq c_4^+ n_x$, some $c_4^+\in[0, 1]$ denotes the number of initial particles removed in the process, then $I_4^+\leq T_4^+\leq \lfloor (n_x-I_4)/2\rfloor-I_3$. 
		
		Let us summarise what we have shown so far. Let $K\coloneqq T_1+T_3+T_4+T_4^+\in\{0, 1, \cdots, i\}$ be the number of vertices of $\Gamma^\epsilon$ that we have just pruned. The above algorithm constructs a sub-configuration $x'$ of $x$ and a simplified object 
		\[
		\overline{\Gamma}^\epsilon=\overline{\Gamma}^\epsilon_{t, x', x_{i+1}}(s_1', \cdots,  s_k'; z_1', \cdots, z_k';\underbrace{1, \cdots, 1}_{k\textnormal{-times}}),
		\]
		 where $k\in \{0, 1, \cdots, i-K\}$ is the number of remaining vertices, which are all of type T2, and if $k=0$ the particles diffuse without interacting. The (directed) graph representation of $\overline{\Gamma}_\epsilon$ for all times larger than zero  is a binary tree. Let $I\coloneqq I_3+I_4^++I_4-T_4\geq (c_3+c_4^++c_4/2)n_x$ denote the number of initial particles removed while reducing $\Gamma^\epsilon$ to $\overline{\Gamma}^\epsilon$. Using the bounds on $I_3$, $I_4^+$, and $I_4$ above, we compute 
		 \begin{equation}\label{ineq:n_x_prime_lb}
		 	n_{x'} = n_x-I\geq (1-c_4)n_x/2+T_4.
		 \end{equation}
		 Let $J\geq I$ be the total number of factors of the form \eqref{ineq:basic_heat_estimate} picked up to bound T3 and T4 vertices. Then we have also shown that there exists is $c_0=c_0(T, n)>0$ such that
		\begin{align}
			&\notag\sum_{\substack{x_{i+1}\in \X_\epsilon \\ z_1, \cdots z_i \in \Z_\epsilon}} \Gamma^\epsilon_{t, x, x_{i+1}}(s_1, \cdots, s_i; z_1, \cdots, z_i; h_1, \cdots, h_i)
			\\&\leq c_0\epsilon^{J} \prod_{j=1}^{J} (s_{i_j-r_j}-s_{i_j})^{-1/2} \sum_{\substack{x_{i+1}\in \X_\epsilon \\ z_1', \cdots z_k' \in \Z_\epsilon}} \overline{\Gamma}^\epsilon_{t, x', x_{i+1}}(s_1', \cdots,  s_k'; z_1', \cdots, z_k';1, \cdots, 1),\label{ineq:new_gamma_eps}
		\end{align}
		where $s_{i_j}$, $j\in \{1, \cdots, J\}$ are the time coordinates of the T3 and T4 vertices that we have just removed, and $s_{i_j-r_j}$ are the time coordinates of their incoming edges, with $1\leq r_j\leq i_j\leq i$, and $\epsilon \in (0, 1)$.
		
		We now implement the second step of the proof, in which we find an upper bound for the right-hand side of \eqref{ineq:new_gamma_eps} using \eqref{ineq:basic_heat_estimate} and the fact if the final configuration $x_{i+1}$ is not zero, then each site has at least two incoming edges. We may assume without loss of generality that $\overline{\Gamma}^\epsilon$ is connected. We say that a vertex is in the bulk (of the graph) if it is neither an initial nor a final vertex. We distinguish three types of edges in the graphical representation of $\overline{\Gamma}^\epsilon$.
		\begin{enumerate}
			\item[E1.] Edge with no interaction: $\{(t, z), (0, z')\}$, for some $z \in x'$ and $z' \in x_{i+1}$.
			\item[E2.] Edge from initial to bulk: $\{(t, z), (s_i, z_i)\}$, for some $s_i\in (0, t)$, $z_i \in \X_\epsilon$ and $z \in \X_\epsilon$.
			\item[E3.] Edge from bulk to the final configuration: $\{(s_i, z_i), (0, z')\}$, for some $s_i\in(0, t)$, $z_i \in \Z_\epsilon$, and $z' \in x_{i+1}$.
		\end{enumerate}
		Since $\overline{\Gamma}^\epsilon$ is connected, either it is comprised entirely of a single final vertex with only type E1 incoming edges, or it has no final vertex with only type E1 incoming edges. In the first case, by \eqref{ineq:n_x_prime_lb}, there are $n_{x'}\geq n_x-I$ such edges, so the worst bound is obtained when $x_{i+1}$ consists of $\lfloor(n_{x'}+1)/2\rfloor$ sites each with two incoming E1 edges. For each final site, if the two initial particles are at $z_1, z_2 \in \Z_\epsilon$ and two E1 edges carry them to their final location $z \in \Z_\epsilon$, we have by \eqref{ineq:basic_heat_estimate}
		\[
		\sum_{z \in \Z_\epsilon} G^\epsilon_t(z_1, z)G^\epsilon_t(z_2, z) \leq C \epsilon/\sqrt{2\epsilon^\beta}\leq C\epsilon^{1-\beta/2},
		\]
		where $C=C(T)>0$. Overall, in this case, we have 
		\[
		\sum_{\substack{x_{i+1}\in \X_\epsilon\\z_1', \cdots, z_k'\in \Z_\epsilon}} \overline{\Gamma}^\epsilon_{t, x', x_{i+1}}(s_1', \cdots,  s_k'; z_1', \cdots, z_k';1, \cdots, 1)\leq C \epsilon^{(1-\beta/2)\lfloor(n_{x'}+1)/2\rfloor},
		\]
		for some $C=C(T)>0$. Combined with \eqref{ineq:new_gamma_eps}, we obtain
		\begin{equation}\label{ineq:case_one_gamma}
		\begin{aligned}
			\sum_{\substack{x_{i+1}\in \X_\epsilon \\ z_1, \cdots z_i \in \Z_\epsilon}} \Gamma^\epsilon_{t, x, x_{i+1}}(s_1, \cdots, s_i; z_1, \cdots, z_i; h_1, \cdots, h_i)
			\\\leq c_0\epsilon^{J+(1-\beta/2)\lfloor(n_{x'}+1)/2\rfloor} \prod_{j=1}^{J} (s_{i_j-r_j}-s_{i_j})^{-1/2},
		\end{aligned}
		\end{equation}
		uniformly in $h_1, \cdots, h_i$, and where by construction the times in the first product do not intersect with the ones in the second product. 
		
		Suppose now that there are no final vertices whose incoming edges are all of type E1. Since any final vertex has at least two incoming edges, each one has at least one E3 edge. We use that edge to bound the summation over the final location by one. The other edges may be of type E1 or E3. We bound all of them using \eqref{ineq:basic_heat_estimate}: an E1 edge is controlled by $C\epsilon^{1-\beta/2}$ as above, while an E3 edge is controlled by $C\epsilon / \sqrt{s_{l_q}'}$, if $s_{l_q}' >0$ is the time of its parent vertex in the bulk, where $l_q \in \{1, \cdots, k\}$, $q \in \{1, \cdots, \overline{k}\}$, for some $\overline{k}\in \{0, 1, \cdots, k\}$ representing the number of E3 edges we control in this way. Let $\overline{K} \in \{0, 1, \cdots, n_{x'}\}$ denote the number of E1 edges in $\overline{\Gamma}_\epsilon$. By construction, we have $\overline{k}+\overline{K}\geq \lfloor(n_{x'}+1)/2\rfloor$. Therefore, there exists $c_1=c_1(T, n)>0$ such that
		\[
			\sum_{\substack{x_{i+1}\in \X_\epsilon\\z_1', \cdots, z_k'\in \Z_\epsilon}} \overline{\Gamma}^\epsilon_{t, x', x_{i+1}}(s_1', \cdots,  s_k'; z_1', \cdots, z_k';1, \cdots, 1)\leq c_1 \epsilon^{(1-\beta/2)\overline{K}}\epsilon^{\overline{k}}\prod_{q=1}^{\overline{k}}(s_{l_q}')^{-1/2}.
		\]
		With \eqref{ineq:new_gamma_eps}, this implies that 
		\begin{equation}\label{ineq:case_two_gamma}
		\begin{aligned}
			\sum_{\substack{x_{i+1}\in \X_\epsilon \\ z_1, \cdots z_i \in \Z_\epsilon}} \Gamma^\epsilon_{t, x, x_{i+1}}(s_1, \cdots, s_i; z_1, \cdots, z_i; h_1, \cdots, h_i)
			\\\leq c_0\epsilon^{J+\overline{k}+(1-\beta/2)\overline{K}} \prod_{j=1}^{J} (s_{i_j-r_j}-s_{i_j})^{-1/2}\prod_{q=1}^{\overline{k}}(s_{l_q}')^{-1/2},
		\end{aligned}
		\end{equation}
		uniformly in $h_1, \cdots, h_i$, and where by construction the times in the first product do not intersect with the ones in the second product. Our next step is to integrate \eqref{ineq:case_one_gamma} and \eqref{ineq:case_two_gamma} with the integration bounds specified in \eqref{ineq:Gamma_eps_control} in the statement of the lemma. We may extend all the upper integration bounds to $2\epsilon^\beta$ by working with $|s_{i_j-r_j}-s_{i_j}|^{-1/2}$ on the right-hand side in either case. Then performing the integration of the second product in \eqref{ineq:case_two_gamma} gives a bound $2^{3\overline{k}/2}\epsilon^{\overline{k}\beta/2}$. Similarly, integrating over the first product in both \eqref{ineq:case_one_gamma} and \eqref{ineq:case_two_gamma}, starting with the $s_{i_j}$, we get a contribution of $4^{J}\epsilon^{J\beta/2}$. For \eqref{ineq:case_one_gamma}, we are left with $i-J$ integrals on $(0, 2\epsilon^\beta)$ with constant integrand, leading to a contribution of $2^{i-J}\epsilon^{(i-J)\beta}$, and to 
		\begin{align*}
				&\int_0^{2\epsilon^{\beta}} ds_1 \int_0^{s_1}ds_2 \cdots \int_0^{s_{i-1}}ds_i \sum_{\substack{x_{i+1}\in \X_\epsilon \\ z_1, \cdots z_i \in \Z_\epsilon}} \Gamma^\epsilon_{t, x, x_{i+1}}(s_1, \cdots, s_i; z_1, \cdots, z_i; h_1, \cdots, h_i)
				\\&\leq c \epsilon^{\beta i+(J+\lfloor(n_{x'}+1)/2\rfloor)(1-\beta/2)},
		\end{align*}
		for some $c=c(T, n)>0$. For \eqref{ineq:case_two_gamma}, we are left with ${i-J-\overline{k}}$ integrals on $(0, 2\epsilon^\beta)$ with constant integrand, which contribute $2^{i-J-\overline{k}}\epsilon^{(i-J-\overline{k})\beta}$. Hence, for some $c=c(T, n)>0$, we have 
		\begin{align*}
			&\int_0^{2\epsilon^{\beta}} ds_1 \int_0^{s_1}ds_2 \cdots \int_0^{s_{i-1}}ds_i \sum_{\substack{x_{i+1}\in \X_\epsilon \\ z_1, \cdots z_i \in \Z_\epsilon}} \Gamma^\epsilon_{t, x, x_{i+1}}(s_1, \cdots, s_i; z_1, \cdots, z_i; h_1, \cdots, h_i)
			\\&\leq c \epsilon^{J+\overline{k}+(1-\beta/2)\overline{K}+\beta(i-J/2-\overline{k}/2)}=c\epsilon^{\beta i+(J+\overline{K}+\overline{k})(1-\beta/2)}\leq c \epsilon^{\beta i + (J+\lfloor(n_{x'}+1)/2\rfloor)(1-\beta/2)}.
		\end{align*}
		It remains to recall \eqref{ineq:n_x_prime_lb}, and the fact that $J\geq I\geq 0$. Then, 
		\[
		J+\lfloor(n_{x'}+1)/2\rfloor\geq I+\frac{n_{x'}}{2}\geq \frac{n_x}{2}+\frac{I}{2}\geq \frac{n_x}{2},
		\]
		and assuming that $\beta \in (0, 2)$, this concludes the proof of \eqref{ineq:Gamma_eps_control}.
		\end{proof}
	
	We can now prove Lemma \ref{lem:BC_small_fluct} starting from the bound established in \eqref{ineq:v_control} of Lemma \ref{lem:v_fcn_bd}. 
	
	\begin{proof}[Proof of Lemma \ref{lem:BC_small_fluct}] By Lemmas \ref{lem:v_fcn_bd} and \ref{lem:Gamma_control}, assuming $\beta \in (0, 2)$ and $\xi \in (0, \beta/2)$, there exist $c_1=c_1(n, N)>0$ and $c_2=c_2(n, N)>0$ such that 
	\[
	|v_t^\epsilon(x, \rho_t^\epsilon|X_{t_k^\epsilon})|\leq c_1 \epsilon^{[\beta-\xi(1+n/(N+1))](N+1)}+C_0^\epsilon(t, x)+c_2\epsilon^{(1-\beta/2)n/2}\sum_{i=1}^N\epsilon^{\beta i},
	\]
	for all $x \in \X_\epsilon$ with $n_x=n$, $t \in [\epsilon^\beta, 2\epsilon^\beta]$. We bound $C_0^\epsilon(t, x)$, defined in \eqref{def:C_0}, using Lemma \ref{lem:V_fcn_properties_LLN}. If $x_1 \in \X_\epsilon$ with $n_{x_1}=n_x$ has a location $z \in \Z_\epsilon$ with a single particle, then we split the $V$-function as a product and we see that the factor corresponding to $z$ vanishes. Thus, if it contributes non-trivially, $x_1$ necessarily has at least two particles at every location on its support. Using \eqref{ineq:basic_heat_estimate} and $t>\epsilon^\beta$, the term $\G^\epsilon_{\epsilon^\beta}(x, x_1)$ is dominated by a product of $\lfloor (n_x+1)/2\rfloor\geq n/2$ factors of $C \epsilon^{1-\beta/2}$ for some constant $C=C(T)>0$, for all $\epsilon \in (0, 1)$. On the other hand, by part two of Lemma \ref{lem:V_fcn_properties_LLN} and the fact that $\rX \in \Omega_1$, we have
	\[
	|V^\epsilon(x_1, X_{t_k^\epsilon}; X_{t_k^\epsilon})|\leq \epsilon^{\kappa\lceil n/2\rceil-\xi\lfloor n/2\rfloor},
	\]
	for some absolute constant $c>0$. Since $\xi \in (0, \beta/2)$ is arbitrarily small and $\kappa \geq 0$, it follows that for some $c=c(T)>0$, we have for any $\delta_0 \in (0, (1-\beta/2)/2)$ that
	\begin{equation}\label{ineq:first_C_0_bd}
	C_0^\epsilon(t, x)\leq c\epsilon^{\delta_0 n},
	\end{equation}
	Assume now that $\xi <\beta/2$, and choose $N=N(\beta, n)$ sufficiently large that $\beta[1-(1+n/(N+1))/2](N+1)-\beta n/4\geq (1-\beta/2)n/2$. Along with the observation that $\sum_{i=0}^N \epsilon^{\beta i} = (1-\epsilon^{\beta (N+1)})/(1-\epsilon^{\beta N})<2$, we obtain for some $C=C(T)>0$ the bound
	\[
	|v^\epsilon_t(x, \rho_t^\epsilon(\cdot\;;X_{t_k^\epsilon})|X_{t_k^\epsilon})|\leq C\epsilon^{(1-\beta/2)n/2},
	\]
	for all $x \in \X_\epsilon$ with $n_x=n$, $t \in [\epsilon^\beta, 2\epsilon^\beta]$, and $\epsilon \in (0, 1)$, which concludes the proof of \eqref{ineq:control_v_eps_1}. To show \eqref{ineq:control_v_eps_0}, we repeat the arguments above with the following modifications. By Proposition \ref{prop:v_fcn_eq}, we have 
	\begin{align}
		v_t^\epsilon(x, \rho_t^\epsilon|\nu^\epsilon) &\notag= \sum_{\substack{x_1\in \X_\epsilon\\ n_{x_1}=n_x}}\G^\epsilon_t(x, x_1)v_0^\epsilon(x_1, \rho_0^\epsilon| \nu^\epsilon)
		\\&+ \int_0^t ds_1 \sum_{\substack{x_1\in \X_\epsilon\\ n_{x_1}=n_x}} \G^\epsilon_{t-s_1}(x, x_1)\label{eq:BBGKY_before_iteration_0}
		\\&\notag\quad\quad\quad\quad\times\sum_{z_1 \in \supp(x_1)}\sum_{h_1=-n_{x_1}(z_1)}^1 c_h^\epsilon(x_1(z_1), \rho_{s_1}^\epsilon(z_1))v_{s_1}^\epsilon(x_1^{(z_1, h_1)}, \rho_{s_1}^\epsilon|\nu^\epsilon),\notag
	\end{align}
	for all $t \in [0, \epsilon^\beta]$ and $x \in \X_\epsilon$ with $n_x=n$.
	Thus, in the proof of Lemma \ref{lem:v_fcn_bd}, the initial term $C_0^\epsilon$ is now
	\[
	C_0^\epsilon(t, x)=\sum_{\substack{x_1\in \X_\epsilon\\ n_{x_1}=n_x}}\G^\epsilon_t(x, x_1)v_0^\epsilon(x_1, \rho_0^\epsilon| \nu^\epsilon).
	\]
	By Assumption \ref{assump:technical_X_0}.a, it satisfies the inequality $|C_0^\epsilon(t, x)|\leq c_{1, n}\epsilon^{\lfloor (n+1)/2\rfloor(1+\kappa)}$ uniformly in $x \in \X_\epsilon$ with $n_x=n$, for some constant $c_{1, n}>0$. Still in the proof of Lemma \ref{lem:v_fcn_bd}, the new coefficients $c_{h_j}^\epsilon$ in the iteration of \eqref{eq:BBGKY_before_iteration_0} are controlled using \eqref{ineq:coef_bd} and part two of Lemma \ref{lem:FKPP_approx_properties}. Indeed, there is an absolute constant $C>1$ such that $\max_{z\in \Z_\epsilon}|\rho_0^\epsilon(z)|<C$, and there exists $c=c(n)>0$ such that
	\begin{equation}\label{ineq:new_control_coeffs}
	|c_h^\epsilon(x_j(z_j), \rho_{s_j}^\epsilon(z_j))|\leq c \left(\max_{z\in \Z_\epsilon}|\rho_0^\epsilon(z)|+1\right)^{n_{x_1}(z_1)+1}\leq c C^{n+j+1},
	\end{equation}
	where we have used in the last inequality that $n_{x_j}\leq n+j$. By Proposition \ref{prop:SLBP_moments}, the integrand of the new remainder $|R_{N+1}^\epsilon(t, x)|$ is bounded by a constant, uniformly in $\epsilon$, $t$, $s_1, \cdots, s_{N+1}$, and $x$. Thus we obtain the uniform bound $R_{N+1}^\epsilon(t, x)\leq c\epsilon^{\beta(N+1)}$ for some constant $c=c(T, n)>0$. The function $\Gamma^\epsilon$ is unchanged and can therefore be controlled by Lemma \ref{lem:Gamma_control}.  The power of $\epsilon$ which multiplies $\Gamma^\epsilon$ in \eqref{ineq:v_control} changes to $\epsilon^{\lfloor (n+1)/2\rfloor(1+\kappa)}$ since the bounds \eqref{ineq:control_coeffs_t_k} and \eqref{ineq:init_V_bd_to_improve} are replaced with \eqref{ineq:new_control_coeffs} and \hyperref[assump:technical_X_0]{(A2)}.
	Overall, we obtain for some $c_2=c_2(T, n)>0$ the bound
	\begin{align*}
		|v_t^\epsilon(x, \rho_t^\epsilon|\nu^\epsilon)|&\leq c_1\epsilon^{\beta(N+1)}+c_2\epsilon^{\lfloor (n+1)/2\rfloor(1+\kappa)}\sum_{i=0}^N \epsilon^{\beta i}
		\\&\leq c_1\epsilon^{\beta(N+1)}+2c_2\epsilon^{\lfloor (n+1)/2\rfloor(1+\kappa)}.
	\end{align*}
	For $N=N(\delta, n)$ large enough that $\beta(N+1)>\lfloor (n+1)/2\rfloor(1+\kappa)$, this shows \eqref{ineq:control_v_eps_0}.
	\end{proof}
	
	\subsection{A smoothing argument}
	
	To establish Theorem \ref{thm:QLLN}, which quantifies in terms of $\epsilon$ the error in the approximation $X_t^\epsilon(z) \approx \rho_t^\epsilon(z)$, $z \in \Z_\epsilon$ and $t\in [0, T]$, we need one more result. We have shown in Lemma \ref{lem:BC_small_fluct} that with high probability, we have control on the error in 
	\[
	X_t^\epsilon(z)\approx \rho_{t-t_k^\epsilon}^\epsilon(z;X_{t_k^\epsilon}^\epsilon),
	\]
	for all $z\in \Z_\epsilon$, $t \in [t_{k+1}^\epsilon, t_{k+2}^\epsilon]$, and $\epsilon \in (0, 1)$. In particular, the dynamics of the rescaled SLBP are close to those of the semi-discrete FKPP equation \eqref{eq:interm_fkpp_diff_form} for short-times. By iterating the last displayed expression, we have the heuristic
	\begin{align*}
		X_{t_1^\epsilon}^\epsilon(z)&\approx \rho_{t_1^\epsilon}^\epsilon(z;X_0^\epsilon)
		\\X_{t_2^\epsilon}^\epsilon(z)&\approx \rho_{t_1^\epsilon}^\epsilon(z;\rho_{t_1^\epsilon}^\epsilon(\cdot;X_0^\epsilon))
		\\&\vdots
		\\ X_{t_k^\epsilon}^\epsilon(z) &\approx \underbrace{\rho_{t_1^\epsilon}^\epsilon(z;\rho_{t_1^\epsilon}^\epsilon(\cdot;\cdots\rho_{t_1^\epsilon}^\epsilon(\cdot;X_0^\epsilon)))}_{k-times}\approx\rho^\epsilon_{t_k^\epsilon}(z).
	\end{align*}
	We now show that with high probability, we have error bounds on
	\[
	\rho^\epsilon_{t-t_k^\epsilon}(z; X^\epsilon_{t_k^\epsilon}) \approx \rho_t^\epsilon(z), \quad z \in \Z_\epsilon, \quad t\in[t_{k+1}^\epsilon, T].
	\]
	At time $t_k^\epsilon$, the left-hand side is $X_{t_k^\epsilon}^\epsilon$, the time-$t_k^\epsilon$ configuration of a rescaled SLBP, while the right-hand side is $\rho_{t_k^\epsilon}^\epsilon$. This result shows that if $X_{t_k^\epsilon}$ is close to $\rho_{t_k^\epsilon}^\epsilon$, then the semi-discrete FKPP at any later time and started from either one takes on approximately the same value. This is a very similar statement to the continuity of solution or flow maps of well-posed differential equations. More precisely, we have the following statement, whose proof is based on that of Proposition 4.2 in \cite{Boldrig1992}.  
	
	\begin{lem}\label{lem:diffusion_helps_local_equil}
		Assume that $\kappa\geq 0$, $\beta \in (0, 6/15)$, and $\xi \in (0, \beta/2)$. Let $\delta_*\coloneqq \min(3/8+\beta/16, 1/2-\beta/4)$. For any $\delta_1 \in (0, \delta_*-\beta)$ and $q>0$, there exist an event $\Omega_2=\Omega_2(\xi, \beta, \kappa, n, q)\subset \Omega_1$ with $\P(\rX^\epsilon \in \Omega_2)>1-c\epsilon^q$ for some $c=c(T, p)>0$ and a constant $\epsilon_0=\epsilon_0(T, \beta)\in (0, 1)$ such that
		\begin{equation}\label{eq:claim_diffusion_helps}
		\P\left(\sup_{t \in [t_{k+1}^\epsilon, T]}\max_{z \in \Z_\epsilon} |\rho^\epsilon_{t-t_k^\epsilon}(z;X^\epsilon_{t_k^\epsilon})-\rho_t^\epsilon(z)|<\epsilon^{\delta_1}\Big| \rX^\epsilon \in \Omega_2\right)=1,
		\end{equation}
		for all $k\in \{0, 1, \cdots, \lfloor T/\epsilon^\beta\rfloor\}$, for all $\epsilon \in (0, \epsilon_0)$.
	\end{lem} 
	\begin{proof}
		For $t\geq0$ and $z \in \Z_\epsilon$, we define  
		\[
		\Delta_0^\epsilon(t, z)= \rho_t^\epsilon(z; X^\epsilon_0)-\rho_t^\epsilon(z), 
		\]
		and 
		\[
		\hat X^\epsilon_t(z)=X_t^\epsilon(z)-\E[X^\epsilon_0(z)].
		\] 
		Since $\rho^\epsilon$ solves \eqref{eq:interm_fkpp_diff_form}, we have 
		\begin{align*}
			\Delta_0^\epsilon(t, z) &= \sum_{z_1 \in \Z_\epsilon}G^\epsilon_t(z, z_1)\hat X_0^\epsilon(z_1)
			\\& +\int_0^t ds \sum_{z_1 \in \Z_\epsilon}G^\epsilon_{t-s}(z, z_1)[(\mu_\epsilon-\rho_s^\epsilon(z))\Delta_0^\epsilon(s, z_1)-\frac{1}{2}(\Delta_0^\epsilon(s, z_1))^2].
		\end{align*}
		We begin by finding a bound for the absolute value of the integral on the right-hand side. Let  
		\[
		m_0=\sup_{\epsilon \in (0, 1)}\max_{z \in \Z_\epsilon}\E[X_0^\epsilon(z)],
		\]
		which is finite by Definition \ref{def:general_init_data}. By part two of Lemma \ref{lem:FKPP_approx_properties}, there is $C=C(T)>0$ such that $\rho_t^\epsilon(z)\leq (m_0+1)C$ for all $t\geq 0$, $z \in \Z_\epsilon$, and $\epsilon \in (0, 1)$. Suppose that $\rX^\epsilon \in \Omega_1$. Then using the triangle inequality, $(a+b)^2\leq 2(a^2+b^2)$ for all $a, b \in \R$, and the fact transition probabilities sum to one, we obtain
		\begin{align*}
		&\left|\sum_{z_1 \in \Z_\epsilon}G^\epsilon_{t-s}(z, z_1)[(\mu_\epsilon-\rho_s^\epsilon(z))\Delta_0^\epsilon(s, z_1)-\frac{1}{2}(\Delta_0^\epsilon(s, z_1))^2]\right|
		\\&\leq (\mu_\epsilon+(m_0+1)C)(\epsilon^{-\xi}+(m_0+1)C)+\frac{1}{2}(\epsilon^{-2\xi}+(m_0+1)^2C^2),
		\end{align*}
		for all $s \in [0, t)$, $z \in \Z_\epsilon$ and $\epsilon \in (0, 1)$. Let $c_0=\mu_\epsilon(1+C)+(1+C^2)/2$, so that the right-hand side of the last displayed expression is dominated by $c_0\epsilon^{-2\xi} (m_0+1)^2$, since $m_0>0$. We are interested in $t\geq t_1^\epsilon=\epsilon^\beta$ but let us first assume that $t\geq \underline{t}^\epsilon\coloneqq \epsilon^{(1-\beta/2)/2}$. Take $t=\underline{t}^\epsilon$. We have 
		\begin{equation}\label{ineq:integr_F_0_bound}
			\left|\int_0^{\underline{t}^\epsilon} ds \sum_{y \in \Z_\epsilon}G^\epsilon_{t-s}(x, y)[(\mu-\rho_s^\epsilon(x))\Delta_0^\epsilon(s, y)-\frac{1}{2}(\Delta_0^\epsilon(s, y))^2]\right|\leq c_0\epsilon^{((1-\beta/2)/2-2\xi}(m_0+1)^2.
		\end{equation}
		Next, we estimate the term
		\[
		F_0^\epsilon(z) \coloneqq \sum_{z_1 \in \Z_\epsilon}G_{\underline{t}^\epsilon}^\epsilon(z, z_1)\hat X_0^\epsilon(z_1).  
		\]
		We will apply Chebyshev's inequality to show that with high probability, we can control $F_0^\epsilon(z)$ uniformly in $z$. Observe that 
		\[
		\E\left[ (F_0^\epsilon(z))^N\right]=\sum_{z_1, \cdots, z_N\in \Z_\epsilon} \prod_{i=1}^N G_{\underline{t}^\epsilon}^\epsilon(z, z_i)\E\left[\prod_{j =1}^N \hat X^\epsilon_0(z_j)\right].
		\]
		Given a function $\rho:\Z_\epsilon\to [0, \infty)$, define $\tilde X_t^\epsilon=X_t^\epsilon-\rho$, for all $t\geq 0$. For $x \in \X_\epsilon$ a configuration, we let $\hat x \in \X_\epsilon$ be defined by $\hat x(z)=x(z)\mathbbm{1}_{n_x(z)\geq 2}$. By \eqref{eq:prod_formula_V} in Lemma \ref{lem:V_fcn_properties_LLN}, we have 
		\begin{equation}
			\prod_{z \in x} \tilde X_0^\epsilon(z) = \sum_{x_1\cleq \hat x} d^\epsilon(\hat x, x_1;\rho)V^\epsilon(x\backslash x_1, X^\epsilon_0;\rho),
		\end{equation}
		for all $x \in \X_\epsilon$, for some coefficients $d^\epsilon(\hat x, x_1;\rho)$ satisfying
		\begin{equation}\label{ineq:bd_d_eps}
		d^\epsilon(\hat x, 0;\rho)=1, \quad  |d^\epsilon(\hat x, x_1;\rho)|\leq c\left(\max_{z \in \Z_\epsilon} \rho(z)\right)^{n_{x_1}},
		\end{equation}
		for some constant $c=c(n_{\hat x}, n_{x_1})>0$. We apply this formula with the function $\rho(z)=\E[X_0^\epsilon(z)]=\rho_0^\epsilon(z)$ where the second equality is an assumption on our initial data - see Definition \ref{def:general_init_data}. We obtain
		\begin{align*}
		\E[(F_0^\epsilon(z))^N] &= \sum_{\substack{x \in \X_\epsilon \\ n_x=N}} \prod_{z' \in x} G_{\underline{t}^\epsilon}^\epsilon(z, z')\sum_{x_1\cleq \hat x} d^\epsilon(\hat x, x_1;\rho_0^\epsilon)\E\left[V^\epsilon(x\backslash x_1, X^\epsilon_0;\rho_0^\epsilon)\right]
		\\&=\sum_{\substack{x \in \X_\epsilon \\ n_x=N}} \prod_{z' \in x} G_{\underline{t}^\epsilon}^\epsilon(z, z')\sum_{x_1\cleq \hat x} d^\epsilon(\hat x, x_1;\rho_0^\epsilon)v_0^\epsilon(x\backslash x_1, \rho_0^\epsilon|\nu^\epsilon).
		\end{align*}
		We partition the first summation over the size of $\hat x$ and split the product of transition probabilities over $\hat x$ and $x \backslash \hat x$:
		\[
		\E[(F_0^\epsilon(z))^{2N}] = \sum_{r=0}^{2N}\sum_{\substack{x \in \X_\epsilon \\ n_x=N \\ n_{\hat x}=r}} \sum_{x_1\cleq \hat x}d^\epsilon(\hat x, x_1;\rho_0^\epsilon)\left(\prod_{z' \in x \backslash \hat x} G_{\underline{t}^\epsilon}^\epsilon(z, z')\right)\left(\prod_{z' \in \hat x} G_{\underline{t}^\epsilon}^\epsilon(z, z')\right)v_0^\epsilon(x\backslash x_1, \rho_0^\epsilon|\nu^\epsilon);
		\]
		the even exponent $2N$ ensures positivity of the integrand on the left-hand side. By \eqref{ineq:basic_heat_estimate}, there exists $C=C(T)>0$ such that
		\[
		G_t^\epsilon(z, z')\leq C\epsilon/\sqrt{t},
		\]
		for all $z, z' \in \Z_\epsilon$, $t \in (0, T]$, and $\epsilon \in (0, 1)$. By definition, each factor in the product over $z' \in \hat x$ is repeated at least twice, meaning that we can extract at least $\lfloor (r+1)/2\rfloor$ factors of $C\epsilon/\sqrt{\underline{t}^\epsilon}$ to obtain an upper bound on that product. Using the bound on $d^\epsilon$ from \eqref{ineq:bd_d_eps}, and the uniform bound $\rho_t^\epsilon(z)\leq (m_0+1)C(T)$, we obtain 
		\[
		|d^\epsilon(x, x_1;\rho_0^\epsilon)|\leq c C(T)^{2N} (m_0+1)^{2N},
		\]
		for some absolute constant $c>0$, uniformly in $x$ and $x_1\cleq x$. By Assumption \ref{assump:technical_X_0}.a and the fact that $(1-\beta/2)/2\in (2/5, 1/2)$, there exist $c_{1, N}>0$ such that
		\begin{equation}\label{ineq:crude_lb}
		v^\epsilon_0(x\backslash x_1, \rho_0^\epsilon|\nu^\epsilon) \leq c_{1, N}\epsilon^{\lfloor(2N-n_{x_1}+1)/2\rfloor(1+\kappa)}< c_{1, N}\epsilon^{(2N-n_{x_1})(1-\beta/2)/2}\leq c_{1, N}\epsilon^{(2N-r)\delta_0},
		\end{equation}
		for all $\epsilon \in (0, 1)$, $x_1\cleq \hat x$, and $\delta_0 \in (0, (1-\beta/2)/2)$ (see Remark \ref{rmk:crude_lb} below). The bounds just discussed imply that 
		\begin{align*}
				\E[(F_0^\epsilon(z))^{2N}] &\leq c C(T)^{2N} (m_0+1)^{2N}c_{1, n} \sum_{r=0}^{2N} \epsilon^{r(1-(1-\beta/2)/4)/2} \epsilon^{(2N-r)\delta_0}.
		\end{align*}
		The exponents simplify as follows
		\begin{align*}
		r(1-(1-\beta/2)/4)/2+(2N-r)\delta_0 &= r(3/8+\beta/16)+(2N-r)\delta_0
		\\&\geq 2N \delta_*.
		\end{align*}
		where $\delta_*= \min(3/8+\beta/16, \delta_0)$. Hence, for some $C=C(T, N)>0$, we have 
		\begin{equation}\label{ineq:F_0_final_bound}
		\E[(F_0^\epsilon(z))^{2N}]\leq (m_0+1)^{2N}C\epsilon^{2N\delta_*}
		\end{equation}
		uniformly in $z \in \Z_\epsilon$. By a union bound and Chebyshev's inequality, we obtain for any $\zeta>0$ that
		\begin{align*}
			\P\left(\max_{z \in \Z_\epsilon}|F_0^\epsilon(z)|>\epsilon^{\delta_*-\zeta}\right)&\leq\sum_{z \in \Z_\epsilon}\P(|F_0^\epsilon(z)|>\epsilon^{\delta_*-\zeta})
			\\&\leq C (m_0+1)^2 \epsilon^{2N\zeta}.
		\end{align*}
		We take $N>0$ large enough that $2N\zeta>q$. This implies that
		\[
		\P\left(\max_{z \in \Z_\epsilon}|F_0^\epsilon(z)|>\epsilon^{\delta_*-\zeta}\right) < C (m_0+1)^2 \epsilon^q.
		\]
		Let $\delta_*'\in (0, \delta_*-\zeta)$. It follows from the bounds above that, on the event $\Omega_{2, 0}=\{\max_{z \in \Z_\epsilon}|F_0^\epsilon(z)|< \epsilon^{\delta_*-\zeta}\}$ intersected with $\Omega_1$, we have 
		\[
		\max_{z \in \Z_\epsilon}|\Delta^\epsilon_0(\underline{t}^\epsilon, z)|<\epsilon^{\delta_*'}.
		\]
		Since $t\mapsto \Delta_0^\epsilon(t, z)$ is continuous and bounded in absolute value by one up to time 
		\[
		\tau_0^\epsilon\coloneqq \sup\{t\geq \underline{t}^\epsilon:\max_{z \in \Z_\epsilon}|\Delta_0^\epsilon(t, z)|<1\},
		\]
		then it is bounded from above, pointwise, by the solution to
		\[
		y(t) = \epsilon^{\delta_*'}+M_0\int_{s_0}^t y(s)ds, \quad t\in [\underline{t}^\epsilon, \tau_0^\epsilon),
		\]
		where $M_0=M_0(T, p)=C(m_0+ 1)^2$, for some constant $C=C(T, p)>0$. So $y(t)=\epsilon^{\delta_*'}e^{M_0(t-\underline{t}^\epsilon)}$, and for $\epsilon>0$ small enough, we obtain 
		\[
		y(\tau^\epsilon_0)=1\iff \tau_0^\epsilon\geq \underline{t}^\epsilon+\delta_*'\ln(\epsilon^{-1})/M_0.
		\]
		The lower bound on the right-hand side is larger than our time horizon $T$ if $\epsilon \in (0, e^{-M_0 T /\delta_*'})$. Thus for $\epsilon$ in that range, we have $\tau_0^\epsilon>T$, so that
		\[
		\max_{z \in \Z_\epsilon} |\Delta^\epsilon_0(t, z)|\leq \epsilon^{\delta_*'}e^{M_0 T},
		\]
		for all $t \in [\underline{t}^\epsilon, T]$. By the triangle inequality, if $m_1^\epsilon\coloneqq m_0+\epsilon^{\delta_*'}e^{M_0 T}$, this implies
		\begin{equation}\label{ineq:m_1}
		\max_{z \in  \Z_\epsilon} \rho_t^\epsilon(z;X_{t_0^\epsilon}^\epsilon)\leq \max_{z \in \Z_\epsilon} |\rho_t^\epsilon(z)|+\max_{z \in \Z_\epsilon} |\Delta^\epsilon_0(t, z)|\leq m_1^\epsilon,
		\end{equation}
		for all $t \in [\underline{t}^\epsilon, T]$, where $m_1^\epsilon=m_0+\epsilon^{\delta_*'}e^{M_0 T}$. Next, we consider 
		\[
		\Delta^\epsilon_k(t, z) = \rho_{t-t_k^\epsilon}^\epsilon(z;X_{t_k^\epsilon}^\epsilon)-\rho^\epsilon_{t-t_{k-1}^\epsilon}(z;X_{t_{k-1}^\epsilon}^\epsilon),\quad k\geq 1.
		\]
		For $k\geq 1$, define 
		\[
		F^\epsilon_k(z) \coloneqq \sum_{z' \in \Z_\epsilon} G^\epsilon_{\underline{t}^\epsilon}(z, z')\Delta_k^\epsilon(t_k^\epsilon, z'),
		\]
		and $\Omega_{2, k} \coloneqq \{\max_{z \in \Z_\epsilon}|F^\epsilon_k(z)|<\epsilon^{\delta_*-\zeta}\}$.
		The goal is to show that for each $k\in \{1, \cdots, \lfloor T/\epsilon^\beta\rfloor\}$, we have on $\Omega_{2, k}\cap \Omega_1$ the bounds
		\begin{equation}\label{ineq:delta_k_goal_1}
		\max_{z \in \Z_\epsilon} |\Delta_k^\epsilon(t, z)|<\epsilon^{\delta_*'}e^{M^\epsilon_k T}, 
		\end{equation}
		and
		\begin{equation}\label{ineq:delta_k_goal_2}
		\max_{z \in  \Z_\epsilon} \rho_t^\epsilon(z;X_{t_k^\epsilon}^\epsilon)\leq m_{k+1}^\epsilon,
		\end{equation}
		for all $t \in [t_k^\epsilon+\underline{t}^\epsilon, T]$, where $m_{k+1}^\epsilon=m_k^\epsilon+\epsilon^{\delta_*'}e^{M_k^\epsilon T}$, with $M_k^\epsilon = C(m_k^\epsilon+1)^{2N}$, for some constant $C=C(T)$. We use induction on $k$, with both the base case $k=1$ and the induction step following the approach used to control $\Delta_0^\epsilon(t, z)$ above. The main difference is that we replace the time-$0$ control on the $v$-functions provided by our Assumption \ref{assump:technical_X_0}.a with an application of Lemma \ref{lem:BC_small_fluct}. Let $t\coloneqq t_1^\epsilon+\underline{t}^\epsilon$. As before, we have 
		\begin{align*}
			\Delta_1^\epsilon(t, z) &= \sum_{z' \in \Z_\epsilon}G^\epsilon_{t-t_1^\epsilon}(z, z')\Delta_1^\epsilon(t_1, z')
			\\& +\int_{t_1^\epsilon}^t ds \sum_{z' \in \Z_\epsilon}G^\epsilon_{t-s}(z, z')[(\mu-\rho_{s-t_0^\epsilon}^\epsilon(z;X_{t_0^\epsilon}^\epsilon))\Delta_1^\epsilon(s, z')-\frac{1}{2}(\Delta_1^\epsilon(s, z'))^2].
		\end{align*}
		Using \eqref{ineq:m_1}, $t-t_1^\epsilon=\underline{t}^\epsilon$ and $\rX^\epsilon\in \Omega_1$, the integral is bounded by $M_1^\epsilon\epsilon^{(1-\beta/2)/2-2\xi}$, where $M^\epsilon_1=C(m_1^\epsilon +1)^2$ for some $C=C(T)>0$. Define 
		\[
		F_1^\epsilon(z)\coloneqq \sum_{z' \in \Z_\epsilon}G_{\underline{t}^\epsilon}^\epsilon(z, z') \Delta_1^\epsilon(t_1^\epsilon, z').
		\]
		By similar calculations as above, we obtain
		\begin{align*}
		&\E_{X_{t_0^\epsilon}^\epsilon}[(F_1^\epsilon(z))^{2N}]
		\\&=\sum_{r=0}^{2N}\sum_{\substack{x \in \X_\epsilon \\ n_x=N \\ n_{\hat x}=r}} \sum_{x_1\cleq \hat x}d^\epsilon(\hat x, x_1;\rho_0^\epsilon(\cdot\;;X_{t_0^\epsilon}^\epsilon))\left(\prod_{z' \in x \backslash \hat x} G_{t-t_1^\epsilon}^\epsilon(z, z')\right)\left(\prod_{z' \in \hat x} G_{t-t_1^\epsilon}^\epsilon(z, z')\right)v_{t_1^\epsilon}^\epsilon(x\backslash x_1, \rho_{t_1^\epsilon}^\epsilon(\cdot\;;X_{t_0^\epsilon}^\epsilon)|\nu^\epsilon).
		\end{align*}
		By Lemma \ref{lem:BC_small_fluct}, there exists $c=c(T, N)>0$ such that
		\[
		v_{t_1^\epsilon}^\epsilon(x\backslash x_1, \rho_{t_1^\epsilon}^\epsilon(\cdot\;;X_{t_0^\epsilon}^\epsilon)|\nu^\epsilon)\leq c\epsilon^{(2N-n_{x_1})\delta_0}< c\epsilon^{(2N-r)\delta_*},
		\]
		uniformly in $x$ and $x_1\cleq \hat x$. By \eqref{ineq:bd_d_eps} and \eqref{ineq:m_1}, we have 
		\[
		d(x, x_1;\rho_{t_1^\epsilon}^\epsilon(\cdot;X_{t_0^\epsilon}^\epsilon))\leq C(m_1^\epsilon +1)^2,
		\] 
		for some $C=C(T)>0$. We control the transition probabilities as we did for $\Delta_0^\epsilon(t, z)$ above.
		Hence, the event $\Omega_{2, 1}=\{\max_{z \in \Z_\epsilon}|F_1^\epsilon(z)|<\epsilon^{\delta_*-\zeta}\}$ satisfies 
		\[
		\P(\rX^\epsilon \in \Omega_{2, 1}^c)\leq (m_1^\epsilon+1)^{2N} C\epsilon^{2N\zeta}\leq (m_1^\epsilon+1)^{2N} C\epsilon^q
		\]
		for some for some $C=C(T, N)>0$. Moreover, on $\Omega_{2, 1}\cap \Omega_1$, we get
		\[
		\max_{z \in \Z_\epsilon} |\Delta_1^\epsilon(t, z)|<\epsilon^{\delta_*'}e^{M^\epsilon_1 T}, 
		\]
		for all $t \in [t_1^\epsilon+\underline{t}^\epsilon, T]$.  As in \eqref{ineq:m_1}, we have by the triangle inequality that
		\[
		\max_{z \in  \Z_\epsilon} \rho_t^\epsilon(z;X_{t_1^\epsilon}^\epsilon)\leq m_2^\epsilon,
		\]
		for all $t \in [\underline{t}^\epsilon, T]$, where $m_2^\epsilon=m_1^\epsilon+\epsilon^{\delta_*'}e^{M_1^\epsilon T}$. The method used to show the $k=1$ case applies to prove the induction step. Thus, in particular, \eqref{ineq:delta_k_goal_1} and \eqref{ineq:delta_k_goal_2} hold for all $k \in \{1, \cdots, \lfloor T/\epsilon^\beta\rfloor\}$. To bound the right-hand side of \eqref{ineq:delta_k_goal_1} by $c\epsilon^{\delta_*'}$ for some $c=c(T, N)>0$ independent of $k$ or $\epsilon$, we must check that $M_k^\epsilon$ remains bounded uniformly in $\epsilon$ as $k$ becomes large. It is easily seen that the arguments at the end of the proof of Proposition 4.2 of \cite{Boldrig1992} apply since $\delta_*>\beta$, hence $\delta_\star'>\beta$. Define 
		\[
		\Omega_2 = \Omega_1\cap\bigcap_{k\in \{0, \cdots, \lfloor T/\epsilon^\beta\rfloor\}}\Omega_{2, k}.
		\]
		Then, $\P(\rX^\epsilon \in \Omega_2)>1-c\epsilon^q$ for some $c=c(T, N)$, and if $\delta_1 \in (0, \delta_\star'-\beta)$ and $\rX^\epsilon \in \Omega_2$, we obtain the bound
		\[
		|\rho_{t-t_k^\epsilon}^\epsilon(z;X^\epsilon_{t_k^\epsilon})-\rho_t^\epsilon(z)|\leq \sum_{j=0}^k |\Delta^\epsilon_j(t, z)|\leq c_1 T e^{c_2 T} \epsilon^{\delta_\star'-\beta}<c_3 \epsilon^{\delta_1}, 
		\]
		for some $c_1=c_1(T, N), c_2(T, N), c_3(T, N)>0$, for all $k\in \{0, \cdots, \lfloor T/\epsilon^\beta\rfloor\}$ and $z \in \Z_\epsilon$. Since $0<\delta_\star'<\delta_\star-\zeta$, and $\zeta>0$ is arbitrary, this proves the lemma. 
	\end{proof}
	\begin{remark}\label{rmk:crude_lb}
		Note that in the inequality \eqref{ineq:crude_lb}, we have used the crude lower bound $1+\kappa\geq 1$. The reason is that in choosing $\delta_*$ in the proof of Lemma \ref{lem:diffusion_helps_local_equil}, we are forced to consider the worst case of \eqref{ineq:control_v_eps_1} and \eqref{ineq:control_v_eps_0}. This is however not an issue for the proof of Theorem \ref{thm:QLLN} below. Indeed, following the method of \cite{Boldrig1992} in the proof of their Lemma 1, we will first take $n$ large enough that the control claimed in Theorem \ref{thm:QLLN} holds, and then we will use the hierarchy of equations for the $v$-functions and Assumption \ref{assump:technical_X_0}.a to prove the result for lower order $v$-functions - see \eqref{ineq:v_fcn_ineq} below. We note here that one crucial difference between the proofs of Theorem \ref{thm:QLLN} and Lemma 1 of \cite{Boldrig1992} which allows us to obtain a stronger rate of convergence in terms of the moments of the offpsring distribution and the competition exponent $\kappa$ is that the coefficients of the hierarchy of equations for low order $v$-functions are small in $\epsilon^\kappa$. See Proposition \ref{prop:v_fcn_eq}. 
	\end{remark}
	\begin{proof}[Proof of Theorem \ref{thm:QLLN}]
		This proof follows that of Lemma 1 in \cite{Boldrig1992}. Let $k \in \{0, \cdots, \lfloor T/\epsilon^\beta\rfloor\}$, and $t \in [t_{k+1}^\epsilon, t_{k+2}^\epsilon]$. Consider a test configuration $x\in \X_\epsilon$ with $n_x=n$ for some $n \in \N$. By Definition \ref{def:V_fcn}, we have 
		\begin{align*}
			V^\epsilon(x, X_t^\epsilon;\rho_t^\epsilon) \coloneqq \sum_{x_0\cleq x} Q^\epsilon(x_0, X_t^\epsilon)(-1)^{n_x-n_{x_0}}\prod_{z \in x\backslash x_0}\rho_t^\epsilon(z).
		\end{align*}
		Since 
		\begin{align*}
			\prod_{z \in x\backslash x_0}\rho_t^\epsilon(z) &= \prod_{z \in x\backslash x_0} (\rho_t^\epsilon(z)-\rho_{t-t_k^\epsilon}^\epsilon(z;X_{t_k^\epsilon}^\epsilon)+\rho_{t-t_k^\epsilon}^\epsilon(z;X_{t_k^\epsilon}^\epsilon))
			\\&=\sum_{x_1\cleq x\backslash x_0}\prod_{z \in x\backslash(x_0\cup x_1)} \rho_{t-t_k^\epsilon}^\epsilon(z;X_{t_k^\epsilon}^\epsilon)\prod_{z \in x_1}(\rho_t^\epsilon(z)-\rho^\epsilon_{t-t_k^\epsilon}(z;X_{t_k^\epsilon}^\epsilon))<c,
		\end{align*}
		we can write
		\begin{align*}
			V^\epsilon(x, X_t^\epsilon;\rho_t^\epsilon)&=\sum_{x_0\cleq x}Q^\epsilon(x_0, X_t^\epsilon)(-1)^{x\backslash x_0}\sum_{x_1\cleq x\backslash x_0}\prod_{z \in x\backslash (x_0\cup x_1)}\rho_{t-t_k^\epsilon}^\epsilon(z;X_{t_k^\epsilon}^\epsilon)
			\\&\quad\quad\quad\quad\quad\quad\quad\quad\quad\quad\quad\quad\quad\quad\quad\quad\quad\quad\quad\quad\quad\times \prod_{z \in x_1}(\rho_t^\epsilon(z)-\rho^\epsilon_{t-t_k^\epsilon}(z;X_{t_k^\epsilon}^\epsilon)).
		\end{align*}
		We interchange the order of summation to obtain 
		\begin{align*}
			V^\epsilon(x, X_t^\epsilon;\rho_t^\epsilon)&=\sum_{x_1\cleq x} (-1)^{n_{x_1}}\left(\sum_{x_0\cleq x\backslash x_1} Q^\epsilon(x_0, X_t^\epsilon)(-1)^{n_x-n_{x_0}-n_{x_1}}\prod_{z \in x\backslash (x_0\cup x_1)}\rho_{t-t_k^\epsilon}^\epsilon(z; X_{t_k^\epsilon}^\epsilon)\right)
			\\&\quad\quad\quad\quad\quad\quad\quad\quad\quad\quad\quad\quad\quad\quad\quad\quad\quad\quad\quad\quad\quad\times \prod_{z \in x_1}(\rho_t^\epsilon(z)-\rho^\epsilon_{t-t_k^\epsilon}(z;X_{t_k^\epsilon}^\epsilon))
			\\&=\sum_{x_1 \cleq x}(-1)^{n_{x_1}} V^\epsilon(x\backslash x_1, X_t^\epsilon, \rho_{t-t_k^\epsilon}^\epsilon(\cdot;X_{t_k^\epsilon}^\epsilon))\prod_{z \in x_1}(\rho_t^\epsilon(z)-\rho^\epsilon_{t-t_k^\epsilon}(z;X_{t_k^\epsilon}^\epsilon)).
		\end{align*}
		Suppose that $\rX^\epsilon \in \Omega_2$ and choose $r$, $\beta$, $\xi$, and $\delta$ as in Lemma \ref{lem:diffusion_helps_local_equil}. Then, using the Markov property, Lemma \ref{lem:BC_small_fluct}, and Lemma \ref{lem:diffusion_helps_local_equil}, we get almost surely, 
		\begin{align}
			&\notag\left|\E\left[V^\epsilon(Z, X_t;\rho_t^\epsilon)\Big| X_{t_k^\epsilon}^\epsilon\right]\right| 
			\\&\notag=\left|\sum_{x_1 \cleq x}(-1)^{n_{x_1}} \E_{X_{t_k^\epsilon}^\epsilon}\left[V^\epsilon(x\backslash x_1, X_{t-t_k^\epsilon}^\epsilon, \rho_{t-t_k^\epsilon}^\epsilon(\cdot;X_{t_k^\epsilon}^\epsilon))\right]\prod_{z \in x_1}(\rho_t^\epsilon(z)-\rho^\epsilon_{t-t_k^\epsilon}(z;X_{t_k^\epsilon}^\epsilon))\right| 
			\\&\notag\leq \sum_{x_1 \cleq x} |v_{t-t_k^\epsilon}^\epsilon(x\backslash x_1|X_{t_k^\epsilon}^\epsilon)|\left|\prod_{z \in x_1}(\rho_t^\epsilon(z)-\rho^\epsilon_{t-t_k^\epsilon}(z;X_{t_k^\epsilon}^\epsilon))\right|
			\\&\leq C\sum_{x_1 \cleq x}\epsilon^{(n-n_{x_1})\delta_0}\epsilon^{\delta_1 n_{x_1}}=C\left(\epsilon^{\delta_0}+\epsilon^{\delta_1}\right)^n<c \epsilon^{\delta_1 n},\label{ineq:on_omega_2}
		\end{align}
		for some $C=C(T, n)>0$, and where $c=c(T, n)\coloneqq 2^nC$. By the tower law of expectations and the triangle inequality, we have 
		\begin{align*}
			&|\E[V^\epsilon(x, X_t^\epsilon;\rho_t^\epsilon)]|\\&\leq\left|\E\left[\E\left[V^\epsilon(x, X_t^\epsilon;\rho_t^\epsilon)\Big|X^\epsilon_{t_k^\epsilon}\right]\mathbbm{1}_{\rX^\epsilon \in \Omega_2}\right]\right| +\left|\E\left[\E\left[V^\epsilon(x, X_t^\epsilon;\rho_t^\epsilon)\Big|X^\epsilon_{t_k^\epsilon}\right]\mathbbm{1}_{\rX^\epsilon \in \Omega_2^c}\right]\right|.
		\end{align*}
		We control the first term using \eqref{ineq:on_omega_2}:
		\[
			\left|\E\left[\E\left[V^\epsilon(x, X_t^\epsilon;\rho_t^\epsilon)\Big|X^\epsilon_{t_k^\epsilon}\right]\mathbbm{1}_{\rX^\epsilon \in \Omega_2}\right]\right|<c \epsilon^{\delta_1 n}.
		\]
		For the second term, we apply the Cauchy-Schwarz and Jensen inequalities, yielding the bound
		\begin{align*}
		\left|\E\left[\E\left[V^\epsilon(x, X_t^\epsilon;\rho_t^\epsilon)\Big|X^\epsilon_{t_k^\epsilon}\right]\mathbbm{1}_{\rX^\epsilon \in \Omega_2^c}\right]\right|\leq \E\left[V^\epsilon(x, X_t^\epsilon;\rho_t^\epsilon)^2\right]^{1/2}\P(\rX^\epsilon\in \Omega_2^c)^{1/2}.
		\end{align*}
		By another application of Jensen's inequality, we compute 
		\begin{align*}
		\E[V^\epsilon(x, X_t^\epsilon;\rho_t^\epsilon)^2]&= 2^{2n}\E\left[\left(2^{-n}\sum_{x_1\cleq x}(-1)^{n_x-n_{x_1}}Q^\epsilon(x_1, X_t^\epsilon)\prod_{z \in x\backslash x_1}\rho_t^\epsilon(z)\right)^2\right]
		\\&\leq 2^n \sum_{x_1\cleq x} \E\left[Q^\epsilon(x_1, X_t^\epsilon)^2\right] \max_{z\in \Z_\epsilon}\rho_t^\epsilon(z)^{2(n-n_{x_1})}
		\\&\leq 4^n \left(\max_{\substack{x_1\in \X_\epsilon \\ n_{x_1}\leq n}} \E\left[Q^\epsilon(x_1, X_t^\epsilon)^2\right]\right)\times \left(\max_{i\in \{0, \cdots, n\}}\max_{z\in \Z_\epsilon}\rho_t^\epsilon(z)^{2(n-i)}\right).
		\end{align*}
		Using Corollary \ref{cor:SLBP_p_moments}, we have $\max_{\substack{x_1\in \X_\epsilon \\ n_{x_1}\leq n}} \E\left[Q^\epsilon(x_1, X_t^\epsilon)^2\right]\leq C_1$ uniformly in $\epsilon$, for some constant $C_1=C_1(T, n)>0$. By part three of Lemma \ref{lem:FKPP_approx_properties}, there exists $C_2=C_2(T, n)>0$ such that $\max_{i\in \{0, \cdots, n\}}\max_{z\in \Z_\epsilon}\rho_t^\epsilon(z)^{2(n-i)}\leq C_2$. Combined with Lemma~\ref{lem:diffusion_helps_local_equil}, we obtain the existence of a constant $C=C(T, n, p)>0$ such that 
		\[
		\left|\E\left[\E\left[V^\epsilon(x, X_t^\epsilon;\rho_t^\epsilon)\Big|X^\epsilon_{t_k^\epsilon}\right]\mathbbm{1}_{\rX^\epsilon \in \Omega_2^c}\right]\right|\leq C\epsilon^{q/2}. 
		\]
		By choosing $q>2\delta_1 n$ in Lemma \ref{lem:diffusion_helps_local_equil} and gathering the above estimates, we obtain, for some $c=c(T, n)>0$, the bound
		\begin{equation}\label{ineq:eps_control_v_fcn}
			|v_t^\epsilon(x|\nu^\epsilon)|<c \epsilon^{\delta_1 n},
		\end{equation}
		for all configurations $x\in \X_\epsilon$ with $n_x=n$, and $t \in [t_1^\epsilon, T]$. If $t \in[0, t_1^\epsilon]$, by \eqref{ineq:control_v_eps_0} Lemma \ref{lem:BC_small_fluct}, there exists $c=c(T, n)>0$ such that
		\[
			\sup_{t \in [0, t_1^\epsilon]}\max_{\substack{x \in \X_\epsilon
			\\n_x=n}} |v_t^\epsilon(x|\nu^\epsilon)|<c\epsilon^{\delta_1 n(1+\kappa)}.
		\]
		Define $n_0 \coloneqq (1+\kappa)\lceil 2/\delta_1\rceil$. From Lemma \ref{lem:diffusion_helps_local_equil}, we have $\delta_1<1/2$, so that $n_0\geq 4(1+\lceil\kappa\rceil)=8$. If $n\geq n_0$, we have
		\begin{equation}\label{eq:little_o_v_fcn}
		\sup_{t \in [0, T]}\max_{\substack{x \in \X_\epsilon
			\\n_x=n}}v^\epsilon_t(x|\nu^\epsilon)=o(\epsilon^{1+\kappa}),
		\end{equation}
		for some $c=c(T, n)>0$, as $\epsilon \searrow 0$. Suppose now that $n=n_0-1$. By \eqref{eq:v_fcn_eq} in Proposition \ref{prop:v_fcn_eq}, Assumption \ref{assump:technical_X_0}.a, and the triangle inequality, we obtain 
		\begin{equation}
			\begin{aligned}\label{ineq:v_fcn_ineq}
			|v_t^\epsilon(x|\nu^\epsilon)| &\leq c\epsilon^{\lfloor(n+1)/2\rfloor(1+\kappa)}
			\\&+ \left|\int_0^t ds \sum_{\substack{x_1\in \X_\epsilon\\ n_{x_1}=n_x}} \G^\epsilon_{t-s}(x, x_1)\sum_{z \in \supp(x_1)}\sum_{h=-n_{x_1}(z)}^1 c_h^\epsilon(n_{x_1}(z), \rho_s^\epsilon(z))v_s^\epsilon(x_1^{(z, h)}|\nu^\epsilon)\right|,
			\end{aligned}
		\end{equation}
		for all $x \in \X_\epsilon$ with $n_x=n$, and $t\in [0, T]$, for some $c=c(T, n)>0$. By \eqref{ineq:coef_bd} in Proposition \ref{prop:v_fcn_eq}, part two of Lemma \ref{lem:FKPP_approx_properties}, the $c_h^\epsilon$ are bounded in absolute value, uniformly in all their parameters. If $h=1$ in the inner-most summation, then $x_1^{(z, h)}$ is comprised of $n_0$ particles, and \eqref{eq:little_o_v_fcn} shows that the integral is $o(\epsilon^{1+\kappa})$. If $h=-1$, then by Proposition \ref{prop:v_fcn_eq}, there are at least two particles locally, and using \eqref{ineq:basic_heat_estimate}, we obtain a contribution of $O(\sqrt{t}\epsilon^{1+\delta_1(n_0-1)})$. Since 
		\[
		1+\delta_1(n_0-1)\geq 2(1+\kappa)-\delta_1\geq 3/2+2\kappa,
		\]
		we have $O(\sqrt{t}\epsilon^{1+\delta_1(n_0-1)})=o(\epsilon^{1+\kappa})$ as $\epsilon \searrow 0$. For $h\in \{-2, -3, -4\}$, we get using the same method $O(\sqrt{t}\epsilon^{1+(n_0+h)\delta_1})=o(\epsilon^{1+\kappa})$ as $\epsilon\searrow 0$. In fact, for any $h\leq -3$ where there are at least three particles locally, using \eqref{ineq:basic_heat_estimate} and the fact that $\kappa \in [0, 3/8]$, we obtain that the integral is 
		\[
		O\left(\epsilon^{1+2\kappa}+\epsilon^2\int_{\epsilon^{1+2\kappa}}^t s^{-1}ds\right)=o(\epsilon^{1+\kappa}),
		\]
		as $\epsilon \searrow 0$. By Gr\"onwall's inequality, we get 
		\[
		\sup_{t \in [0, T]} \max_{\substack{x \in \X_\epsilon
				\\n_x=n_0-1}} |v^\epsilon_t(x|\nu^\epsilon)| = o(\epsilon^{1+\kappa}).
		\]
		By taking $\beta \in (0, 6/15)$ sufficiently small, we can always get $\delta_1$ to approximate $3/8$ from below arbitrarily well. In particular, we require $\delta_1>\kappa$. Then if $\kappa\in [0, 3/8]$, we may repeat the arguments above starting from $n=n_0-2$ and down to $n=3$, and obtain for all $n\geq 3$ that
		\[
		\sup_{t \in [0, T]} \max_{\substack{x \in \X_\epsilon
				\\n_x=n}} |v^\epsilon_t(x|\nu^\epsilon)| = o(\epsilon^{1+\kappa}).
		\]
		as $\epsilon\searrow 0$ with $\epsilon\in (0, 1)$. Next, consider $n=2$. Let $p\in [1, 2]$ be such that the offspring distribution $\rP$ has a finite moment of order $p$. If $h=1$ in \eqref{ineq:v_fcn_ineq}, we are back in the case $n=3$, and we obtain a contribution of $o(\epsilon^{1+\kappa})$. Let now $h=-1$. By Proposition \ref{prop:v_fcn_eq}, if $m=1$ and $h\leq -1$, we have $c_h^\epsilon(m, u)=0$. Therefore, we may assume that $x_1=2\epsilon^\kappa\mathbbm{1}_{\{z'\}}$ for some $z' \in \Z_\epsilon$. In that case, if $x$ is comprised of two particles at some locations $z_1, z_2 \in \Z_\epsilon$, then by \eqref{ineq:basic_heat_estimate}, we get
		\begin{equation}\label{ineq:G_prod_conv_corr}
		\G^\epsilon_{t-s}(x, x_1)\leq CG^\epsilon_{t-s}(z_1, z')\epsilon/\sqrt{t-s},
		\end{equation}
		for some $C=C(T)>0$, and where $\epsilon \in (0, 1)$. By \eqref{ineq:eps_control_v_fcn} with $n=1$, we get $v_s^\epsilon(x_1^{(z', -1)}|\nu^\epsilon)=O(\epsilon^{\delta_1})=o(\epsilon^\kappa)$. Now by \eqref{ineq:coeff_bd_2_pcs_1_removed} of Proposition \ref{prop:v_fcn_eq}, and part two of Lemma \ref{lem:FKPP_approx_properties}, we obtain $c^\epsilon_{-1}(2, \rho_s^\epsilon(z))=O(\epsilon^\kappa)+o(\epsilon^{(p-1)\kappa})$ uniformly in $s$ and $z$. Thus using that $\sum_{z' \in \Z_\epsilon}G^\epsilon_{t-s}(z_1, z')=1$ and $\int_0^tds/{\sqrt{t-s}}=2\sqrt{t}$, the integral of the $h=-1$ summand in \eqref{ineq:v_fcn_ineq} contributes $o(\epsilon^{1+(p-1)\kappa})$ as $\epsilon \searrow 0$. Finally, we consider the case $h=-2$, where $x_1=2\epsilon^\kappa\mathbbm{1}_{\{z'\}}$ for some $z' \in \Z_\epsilon$ and both particles are removed. We begin by recalling that the $v$-function on the test configuration $x\equiv 0$ is equal to 1, by Definitions \eqref{def:scaled_v_fcns} and \eqref{def:V_fcn}. We again apply \eqref{ineq:G_prod_conv_corr} which gives us a factor of $\epsilon/\sqrt{t-s}$. Additionally, by \eqref{ineq:coeff_bd_2_pcs_removed} of Proposition \ref{prop:v_fcn_eq}, and part two of Lemma \ref{lem:FKPP_approx_properties}, the coefficient $c_{-2}^\epsilon(2, \rho_s^\epsilon(z))$ is $O(\epsilon^\kappa)+o(\epsilon^{(p-1)\kappa})$ uniformly in $s$ and $z$. Thus, overall the contribution of the $h=-2$ case is $O(\epsilon^{1+\kappa})+o(\epsilon^{1+(p-1)\kappa})$ with a constant depending on $T$, $n=2$, and $p$. By Gr\"onwall's inequality, we obtain 
		\[
		\sup_{t \in [0, T]}\max_{x \in \X_\epsilon: n_x=2} |v_t^\epsilon(x|\nu^\epsilon)|\leq C\left[O(\epsilon^{1+\kappa})+o(\epsilon^{1+(p-1)\kappa})\right],
		\]
		for some $C=C(T, n, p)>0$, as $\epsilon \searrow0$. It follows that 
		\[
		\sup_{t \in [0, T]}\max_{x \in \X_\epsilon: n_x=2} |v_t^\epsilon(x|\nu^\epsilon)| \leq
		\begin{cases}
			o(\epsilon^{1+(p-1)\kappa})&p\in (1, 2),\\
			C\epsilon^{1+\kappa}&p=2.
		\end{cases}
		\]
		If now $n=1$, we consider \eqref{ineq:v_fcn_ineq} with $n_x=1$. The summand $h=1$ in \eqref{ineq:v_fcn_ineq} contributes the right-hand side of the last displayed expression, and by Proposition \ref{prop:v_fcn_eq}, $c_h^\epsilon(1, u)=0$ for all $h\leq -1$. Hence by Gr\"onwall's inequality, we obtain 
		\[
		\sup_{t \in [0, T]}\max_{x \in \X_\epsilon: n_x=1} |v_t^\epsilon(x|\nu^\epsilon)|\leq \begin{cases}
			o(\epsilon^{1+(p-1)\kappa})&p\in (1, 2),\\
			C\epsilon^{1+\kappa}&p=2,
		\end{cases}
		\]
		for some $C=C(T, n, p)>0$, and for $\epsilon \in (0, 1)$. This proves Theorem \ref{thm:QLLN}. 
	\end{proof}
	
	\section{Acknowledgements}
	
	I thank Julien Berestycki for supervising this project. I am also grateful to Alison Etheridge, Sylvie M\'el\'eard, Christina Goldschmidt, Lorenzo Zambotti, and Cyril Labb\'e for helpful discussions in early stages of this work. I thank the PIMS Summer School and the Quebec Analysis and Related Fields seminar for opportunities to present this work. This publication is based on work supported by the EPSRC Centre for Doctoral
	Training in Mathematics of Random Systems: Analysis, Modelling and Simulation (EP/S023925/1).
	
	\bibliographystyle{apalike}
	\bibliography{CLT}

	\appendix
	
	\section{Infinitesimal generator for the fluctuations}
	
	In this section, we compute the generator of the fluctuations process $\rY^\epsilon$. We start with the diffusion component of the generator, and then move to the birth-competition part.

	We focus on test functionals of the form 
	\[
	C^\infty(\S^1)'\to \R, \quad Y\mapsto f(\langle Y, \phi\rangle), \quad f \in C^2(\R).
	\] 
	We recall that $\rN^\epsilon=\{(N^\epsilon_t(z))_{z \in \Z_\epsilon}:t\geq 0\}$ is the SLBP in the weak competition regime, and that $\rX^\epsilon=\epsilon^\kappa \rN^\epsilon$. We view $\rX^\epsilon$ as an element of $\cD([0, T], C^\infty(\S^1)')$ via the definition
	\begin{equation}\label{def_eq:slbp_as_dist}
	X^\epsilon_t(\phi)=\epsilon \sum_{z \in \Z_\epsilon} X_t^\epsilon(z)\phi(\epsilon z), \quad \phi \in C^\infty(\S^1)'
	\end{equation}
	We consider the process $\rY^\epsilon=(\{Y^\epsilon_t(z):z \in \Z_\epsilon\})_{t\geq 0}$, where 
	\[
	Y^\epsilon(z)=\epsilon^{\gamma-\kappa}(X_t^\epsilon(z)-\E[X_t^\epsilon(z)]). 
	\]
	We also view it as an element of $\cD([0, T], C^\infty(\S^1)')$ by letting 
	\begin{equation}\label{def_eq:fluct_as_dist}
	Y_t^\epsilon(\phi)=\epsilon \sum_{z \in \Z_\epsilon} Y_t^\epsilon(z)\phi(\epsilon z).
	\end{equation}
	After a jump of size $\ell\in \{-1, 1, 2, \cdots\}$ at a location $z \in \Z_\epsilon$ at time $t-$, we have by definition
	\[
	Y_t^\epsilon(z) = Y_{t-}^\epsilon(z)+\ell\epsilon^\gamma \delta_{z_0 z}, \quad z \in \Z_\epsilon,
	\]
	where $\delta_{z_0z}$ is the Kronecker delta. Then, 
	\[
	Y_t^\epsilon(\phi) = Y_{t-}^\epsilon(\phi)+\ell\epsilon^{\gamma+1}\phi(\epsilon z), \quad \phi \in C^\infty(\S^1). 
	\]
	The dynamics of the SLBP fluctuations process $\rY^\epsilon$ have the following generator. Given a time $s \geq 0$, we have
	\begin{align}
		\cL^\epsilon f(Y_s^\epsilon(\phi))&\notag=\lim_{t\searrow 0} (t-s)^{-1}\left[\E\left[f(Y^\epsilon_t(\phi))\big|Y_s^\epsilon(\phi)\right]-f(Y_s^\epsilon(\phi))\right]
		\\&=\notag\lim_{t\searrow s} (t-s)^{-1}\left[\E\left[f(\epsilon^{\gamma-\kappa}(X_t^\epsilon(\phi)-\E[X_t^\epsilon(\phi)]))\right]-f(\epsilon^{\gamma-\kappa}(X_s^\epsilon(\phi)-\E[X_t^\epsilon(\phi)]))\right]
		\\&\notag+\lim_{t\searrow s} (t-s)^{-1}\left[f(\epsilon^{\gamma-\kappa}(X_s^\epsilon(\phi)-\E[X_t^\epsilon(\phi)]))-f(\epsilon^{\gamma-\kappa}(X_s^\epsilon(\phi)-\E[X_s^\epsilon(\phi)]))\right].
	\end{align}
	The first term equals $\cL^\epsilon_{RW}f(Y_s^\epsilon(\phi))+\cL^\epsilon_{BC}f(Y_s^\epsilon(\phi))$, where 
	\begin{align}\label{eq:fluct_motion_def}
		\cL^\epsilon_{RW}f(Y_s^\epsilon(\phi))&\notag\coloneqq \frac{1}{2}\sum_{z \in \Z_\epsilon}\epsilon^{-2} N_s^\epsilon(z)\Big\{f(Y_s^\epsilon(\phi)+\epsilon^{\gamma+1}[\phi(\epsilon(z+1))-\phi(\epsilon z)])
		\\&\quad\quad\quad\quad\quad\quad\quad\quad+f(Y_s^\epsilon(\phi)+\epsilon^{\gamma+1}[\phi(\epsilon(z-1))-\phi(\epsilon z)])-2f(Y_s^\epsilon(\phi))\Big\},
	\end{align}
	is the generator of the spatial motion, and
	\begin{equation} 
	\begin{aligned}\label{eq:fluct_BC_def}
		\cL^\epsilon_{BC}f(Y_s^\epsilon(\phi))&\coloneqq\sum_{z \in \Z_\epsilon} N_s^\epsilon(z)\sum_{\ell\geq 1} p_\ell^\epsilon\Big \{f(Y_s^\epsilon(\phi)+\epsilon^{\gamma+1} k\phi(\epsilon z))-f(Y_s^\epsilon(\phi))\Big\}
		\\&+\sum_{z \in \Z_\epsilon} \epsilon^\kappa \frac{N_s^\epsilon(z)(N_s^\epsilon(z)-1)}{2} \Big\{ f(Y_s^\epsilon(\phi)-\epsilon^{\gamma+1}\phi(\epsilon z))-f(Y_s^\epsilon(\phi))\Big\},
	\end{aligned}
	\end{equation}
	is the generator of the birth-competition dynamics, while the second term is easily computed by the chain rule. We obtain
	\begin{equation}
		\cL^\epsilon f(Y_s^\epsilon(\phi))\label{eq:Y_dyn_generic} = \cL^\epsilon_{RW}f(Y_s^\epsilon(\phi))+\cL^\epsilon_{BC}f(Y_s^\epsilon(\phi))-\epsilon^{\gamma-\kappa}\partial_t \E[X_t^\epsilon(\phi)]\Big|_{t=s} f'(Y_s^\epsilon(\phi)).
	\end{equation}
	Since 
	\[
	X_t^\epsilon(\phi)-X_0^\epsilon(\phi)-\int_0^t \cG^\epsilon X_s^\epsilon(\phi) ds,\quad t\geq 0,
	\]
	is a martingale, we obtain
	\begin{equation}
	\partial_t\E[X_t^\epsilon(\phi)]\Big|_{t=s} = \E[\cG^\epsilon X_s^\epsilon(\phi)]=\E[\cG^\epsilon_{RW} X_s^\epsilon(\phi)]+\E[\cG^\epsilon_{BC} X_s^\epsilon(\phi)],\label{eq:time_deriv_part_of_L}
	\end{equation}
	where $\cG^\epsilon=\cG^\epsilon_{RW}+\cG^\epsilon_{BC}$ is defined in \eqref{def:slbp_gen}-\eqref{eq:GBC}.
	
	In the next lemma, we simplify the expression for $\cL^\epsilon_{RW}f(Y^\epsilon_s(\phi))$ and relate it to the Laplacian.
	\begin{lem}\label{lem:fluct_L_RW_laplacian}
		Let $f \in C^2(\R)$ and $\phi \in C^\infty(\S^1)$. Then, there exist 
		\[
		\xi^{\pm}_z=\xi^{\pm}(\epsilon, z, \phi) \in (0, \epsilon^{\gamma+1}(\phi(\epsilon(z\pm 1))-\phi(\epsilon z))), \quad z \in \Z_\epsilon,
		\] 
		such that 
		\begin{align*}
			\cL^\epsilon_{RW}f(Y_s^\epsilon(\phi))&=\epsilon^{\gamma-\kappa}\E[\cG^\epsilon_{RW} X_s^\epsilon(\phi)]f'(Y_s^\epsilon(\phi))+f'(Y_s^\epsilon(\phi))\frac{\epsilon}{2}\sum_{z \in \Z_\epsilon}\Delta^\epsilon \phi(\epsilon z)Y_s^\epsilon(z)
			\\&+ \frac{\epsilon}{4} \sum_{z \in \Z_\epsilon}\epsilon^{1+2\gamma-\kappa}X_s^\epsilon(z)\bigg[\left(\nabla^{\epsilon, -}\phi(\epsilon z)\right)^2 f''(Y_s^\epsilon(\phi)+\xi^-_z)
			\\&\quad\quad\quad\quad\quad\quad\quad\quad\quad+\left(\nabla^{\epsilon, +}\phi(\epsilon z)\right)^2f''(Y_s^\epsilon(\phi)+\xi^+_z)\bigg].
		\end{align*}
		where $\Delta^\epsilon \phi(\epsilon z)\coloneqq\epsilon^{-2}\left(\phi(\epsilon (z+1))+\phi(\epsilon (z-1))-2\phi(\epsilon z)\right)$ and $\nabla^{\epsilon, \pm}\phi(\epsilon z)\coloneqq \epsilon^{-1}\left(\phi(\epsilon (z\pm1))-\phi(\epsilon z)\right)$.
	\end{lem}
	\begin{proof}
		By Taylor's Theorem applied to the test function $f$, there exist 
		\[
		\xi_z^{\pm}=\xi^\pm(\epsilon, z, \phi) \in (0, \epsilon^{\gamma+1}[\phi(\epsilon(z\pm 1))-\phi(\epsilon z)]),
		\]
		such that 
		\begin{align*}
			f(Y_s^\epsilon(\phi)+\epsilon^{\gamma+1}[\phi(\epsilon(z\pm1))-\phi(\epsilon z)])&=\epsilon^{\gamma+1}[\phi(\epsilon(z\pm1))-\phi(\epsilon z)]f'(Y_s^\epsilon(\phi))
			\\&+\frac{\epsilon^{2\gamma+2}}{2}[\phi(\epsilon(z\pm1))-\phi(\epsilon z)]^2f''(Y_s^\epsilon(\phi)+\xi_z^{\pm}).
		\end{align*}
		Plugging these expansions into the definition \eqref{eq:fluct_motion_def} of the generator $\cL^\epsilon_{RW}$, we obtain
		\begin{align*}
			\cL^\epsilon_{RW}f(Y_s^\epsilon(\phi))&=\frac{1}{2}\sum_{z \in \Z_\epsilon}\epsilon^{-2} N_s^\epsilon(z)\Big\{\epsilon^{\gamma+3}\Delta^\epsilon\phi(\epsilon z)f'(Y_s^\epsilon(\phi))
			\\&\quad\quad\quad\quad\quad\quad\quad\quad+\frac{\epsilon^{2\gamma+4}}{2}\left(\nabla^{\epsilon, -}\phi(\epsilon z)\right)^2f''(Y_s^\epsilon(\phi)+\xi_z^{-})
			\\&\quad\quad\quad\quad\quad\quad\quad\quad+\frac{\epsilon^{2\gamma+4}}{2}\left(\nabla^{\epsilon, +}\phi(\epsilon z)\right)^2f''(Y_s^\epsilon(\phi)+\xi_z^{+})\Big\}
		\end{align*}
		By definition, we have $N_s^\epsilon(z)=\epsilon^{-\kappa}X_s^\epsilon(z)=\epsilon^{-\gamma}Y_s^\epsilon(z)+\epsilon^{-\kappa}\E[X_s^\epsilon(z)]$. Hence
		\begin{align*}
			\cL^\epsilon_{RW}f(Y_s^\epsilon(\phi))&=\frac{\epsilon}{2}\sum_{z \in \Z_\epsilon}\left(Y_s^\epsilon(z)+\epsilon^{\gamma-\kappa}\E[X_s^\epsilon(z)]\right)\Delta^\epsilon\phi(\epsilon z)f'(Y_s^\epsilon(\phi))
			\\&+ \frac{\epsilon}{4} \sum_{z \in \Z_\epsilon}\epsilon^{1+2\gamma-\kappa}X_s^\epsilon(z)\bigg[\left(\nabla^{\epsilon, -}\phi(\epsilon z)\right)^2 f''(Y_s^\epsilon(\phi)+\xi^-_z)
			\\&\quad\quad\quad\quad\quad\quad\quad\quad\quad+\left(\nabla^{\epsilon, +}\phi(\epsilon z)\right)^2f''(Y_s^\epsilon(\phi)+\xi^+_z)\bigg]
		\end{align*}
		It follows from a straightforward calculation using \eqref{eq:GRW} that 
		\[
		\E[\cG^\epsilon_{RW} X_s^\epsilon(\phi)]=\frac{\epsilon}{2}\sum_{z \in \Z_\epsilon} \E[X_s^\epsilon(z)]\Delta^\epsilon \phi(\epsilon z).
		\]
		Using this expression to compute the difference $\cL^\epsilon_{RW}f(Y_s^\epsilon(\phi))-\epsilon^{\gamma-\kappa}\E[\cG^\epsilon_{RW} X_s^\epsilon(\phi)]f'(Y_s^\epsilon(\phi))$ proves the lemma. 
	\end{proof}
	In simplifying the expression \eqref{eq:fluct_BC_def} for the generator $\cL^\epsilon_{BC}f(Y^\epsilon_s(\phi))$ of the birth-competition fluctuations, we assume that the offspring distribution $\rP$ has a finite second moment. We denote by $\mu=\sum_{k\geq 1} kp_k$ and $\sigma^2=\sum_{k\geq 1}(k-\mu)^2p_k$ the mean and variance of the offspring distribution, respectively. We also introduce $\mu_\epsilon=\sum_{\ell\geq 1}\ell p_\ell^\epsilon$ and $\sigma^2_\epsilon=\sum_{\ell\geq 1}(\ell-\mu_\epsilon)^2p_\ell^\epsilon$ the mean and variance of the truncated offspring distribution.
	
	\begin{lem}\label{lem:fluct_Brownian_BC_gen}
		Let $f \in C^2(\R)$ be such that $\|f''\|_\infty<\infty$, and let $\phi \in C^\infty(\S^1)$. Then 
		\begin{align*}
			&\cL_{BC}^\epsilon f(Y_s^\epsilon(\phi))-\epsilon^{\gamma-\kappa}\E[\cG^\epsilon_{BC} X_s^\epsilon(\phi)]f'(Y^\epsilon_s(\phi))
			\\&= f'(Y_s^\epsilon(\phi))\epsilon\mu_\epsilon\sum_{z \in \Z_\epsilon}Y_s^\epsilon(z)\phi(\epsilon z)
			\\&-f'(Y_s^\epsilon(\phi))\frac{\epsilon^{1+\gamma-\kappa}}{2}\sum_{z \in \Z_\epsilon}\left[X_s^\epsilon(z)(X_s^\epsilon(z)-\epsilon^\kappa)-\E[X_s^\epsilon(z)(X_s^\epsilon(z)-\epsilon^\kappa)]\right]\phi(\epsilon z)
			\\&+\epsilon^{2+2\gamma-\kappa}\sum_{z \in \Z_\epsilon}\frac{\phi(\epsilon z)^2}{2}\left\{\frac{X_s^\epsilon(z)(X_s^\epsilon(z)-\epsilon^\kappa)}{2} f''(Y_s^\epsilon(\phi)+\xi^-_z)+X_s^\epsilon(z)\sum_{\ell\geq 1} \ell^2 p_\ell^\epsilon f''(Y^\epsilon_s(\phi)+\xi_{z, \ell}^+)\right\},
		\end{align*}
		for some 
		\begin{align*}
			\xi^-_z&=\xi^-(\epsilon, z, \phi) \in (-\epsilon^{\gamma+1}\phi(\epsilon z), 0),
			\\\xi^+_{z, \ell}&=\xi^+(\epsilon, z, \ell, \phi) \in (0, \epsilon^{\gamma+1} \ell\phi(\epsilon z)).
		\end{align*}
	\end{lem}
	\begin{proof} By Taylor's theorem, we have 
		\[
			f(Y_s^\epsilon(\phi)+\epsilon^{\gamma+1}\ell\phi(\epsilon z))-f(Y_s^\epsilon(\phi))= \epsilon^{\gamma+1}\ell\phi(\epsilon z)f'(Y_s^\epsilon(\phi))+\frac{\epsilon^{2\gamma+2}\ell^2\phi(\epsilon z)^2}{2} f''(Y_s^\epsilon(\phi)+\xi^+_{z, \ell}),
		\]
		for some $\xi^+_{z, \ell}=\xi_+(\epsilon, z, \ell, \phi) \in (0, \epsilon^{\gamma+1} \ell\phi(\epsilon z))$. Similarly, we have 
		\[
			f(Y_s^\epsilon(\phi)-\epsilon^{\gamma+1}\phi(\epsilon z))-f_s^\epsilon(Y_s^\epsilon(\phi))= -\epsilon^{\gamma+1}\phi(\epsilon z)f'(Y_s^\epsilon(\phi))+\frac{\epsilon^{2\gamma+2}\phi(\epsilon z)^2}{2} f''(Y_s^\epsilon(\phi)+\xi^-_z),
		\]
		for some $\xi_-(z)=\xi_-(\epsilon, z, \phi) \in (-\epsilon^{\gamma+1} \phi(\epsilon z), 0)$. Using $\rX^\epsilon=\epsilon^\kappa \rN^\epsilon$, we compute 
		\[
		\epsilon^\kappa N_s^\epsilon(z)(N_s^\epsilon(z)-1)=\epsilon^{-\kappa} X_s^\epsilon(z)(X_s^\epsilon(z)-\epsilon^\kappa),
		\]
		and hence
		\begin{align*}
			&\cL_{BC}^\epsilon f(Y_s^\epsilon(\phi))
			\\&= \sum_{z \in \Z_\epsilon}\epsilon^{-\kappa}X_s^\epsilon(z)\sum_{\ell\geq 1}p_\ell^\epsilon\Big\{\epsilon^{\gamma+1}\ell\phi(\epsilon z)f'(Y_s^\epsilon(\phi))+\frac{\epsilon^{2\gamma+2}\ell^2\phi(\epsilon z)^2}{2} f''(Y_s^\epsilon(\phi)+\xi^+_{z, \ell})\Big\}
			\\&+\sum_{z \in \Z_\epsilon}\epsilon^{-\kappa}\frac{X_s^\epsilon(z)(X_s^\epsilon(z)-\epsilon^\kappa)}{2}\Big\{-\epsilon^{\gamma+1}\phi(\epsilon z)f'(Y_s^\epsilon(\phi))+\frac{\epsilon^{2\gamma+2}\phi(\epsilon z)^2}{2} f''(Y_s^\epsilon(\phi)+\xi^-_z)\Big\}.
		\end{align*}
		Factoring out an $\epsilon^{\gamma+1}$, and using that by definition, $\sum_{\ell\geq 1}\ell p_\ell^\epsilon=\mu_\epsilon$ and 
		\[
		\epsilon^{-\kappa}X_s^\epsilon(z)=N_s^\epsilon(z)=\epsilon^{-\gamma}Y_s^\epsilon(z)+ \epsilon^{-\kappa}\E[X_s^\epsilon(z)],
		\]
		we obtain
		\begin{align*}
			&\cL_{BC}^\epsilon f(Y_s^\epsilon(\phi))
			\\&= \epsilon\mu_\epsilon\sum_{z \in \Z_\epsilon}(Y_s^\epsilon(z)+\epsilon^{\gamma-\kappa}\E[X_s^\epsilon(z)])\phi(\epsilon z)f'(Y_s^\epsilon(\phi))
			\\&-\frac{\epsilon^{1+\gamma-\kappa}}{2}\sum_{z \in \Z_\epsilon}X_s^\epsilon(z)(X_s^\epsilon(z)-\epsilon^\kappa)\phi(\epsilon z)f'(Y_s^\epsilon(\phi))
			\\&+\epsilon^{2+2\gamma-\kappa}\sum_{z \in \Z_\epsilon}\frac{\phi(\epsilon z)^2}{2}\left\{\frac{X_s^\epsilon(z)(X_s^\epsilon(z)-\epsilon^\kappa)}{2} f''(Y_s^\epsilon(\phi)+\xi^-_z)+X_s^\epsilon(z)\sum_{\ell\geq 1} \ell^2 p_\ell^\epsilon f''(Y^\epsilon_s(\phi)+\xi_{z, \ell}^+)\right\}.
		\end{align*}
		By \eqref{eq:GBC}, we have 
		\begin{align*}
		&\epsilon^{\gamma-\kappa}\E[\cG_{BC}^\epsilon X_s^\epsilon(\phi)]f'(Y_s^\epsilon(\phi)) 
		\\&= f'(Y_s^\epsilon(\phi)) \epsilon^{1+\gamma-\kappa} \sum_{z \in \Z_\epsilon} \phi(\epsilon z) \left(\mu_\epsilon \E\left[X_s^\epsilon(z)\right]-\frac{1}{2}\E\left[X_s^\epsilon(z)(X_s^\epsilon(z)-\epsilon^\kappa)\right]\right).
		\end{align*}
		Taking the difference of the last two displayed expressions proves the lemma. 
	\end{proof}

	\section{Properties of the $V$- and $v$-functions}\label{append:v_fcn_props}
		
	In this section, we derive the properties of the $V$- and $v$-functions introduced in Definitions \ref{def:V_fcn} and \ref{def:scaled_v_fcns}, respectively. The statements and proofs below are adapted from Appendix A in \cite{Boldrig1992} to our setting. For $k, n\in \N_0$, we define
	\[
	Q^\epsilon_k(n) \coloneqq 
	\begin{cases}
		\epsilon^{\kappa k} Q_k(n),& k\leq n
		\\ 1, & k=0 
		\\ 0, & k>n 	
	\end{cases},
	\]
	Then if $u \in \R$, $k, n \in \N_0$, we let 
	\begin{equation}\label{def:tilde_V_eps}
	\widetilde{V}^\epsilon_k(n, u) \coloneqq \sum_{\ell=0}^k{k \choose \ell} Q^\epsilon_\ell(n)u^{k-\ell}(-1)^{k-\ell}.
	\end{equation}
	In our setting, we care about $k, n \in \epsilon^\kappa \N_0$, for which we define
	\[
	V^\epsilon_k(n, u) \coloneqq \widetilde{V}^\epsilon_{\epsilon^{-\kappa}k}(\epsilon^{-\kappa}n, u).
	\]
	
	\begin{lem}\emph{(Properties of the $V$-functions)}\label{lem:V_fcn_properties_LLN}
		The following hold.
		\begin{enumerate}
			\item[1.] For any $\rho : \Z_\epsilon \to [0, \infty)$, we have 
			\[
			V^\epsilon(x, x'; \rho) = \prod_{z \in \supp(x)} V^\epsilon_{x(z)}(x'(z), \rho(z)), \quad x, x' \in \X_\epsilon. 
			\] 
			\item[2.] For any $x, x' \in \X_\epsilon$, we have $V^\epsilon(x, x';x')=0$ if $n_x(z)=1$ for some $z\in \Z_\epsilon$, and there exists an absolute constant $c>0$ such that
			\[
			|V^\epsilon(x, x';x')|< c \epsilon^{\kappa \lceil n_x/2\rceil}\left(\max_{z \in \supp(x)}x'(z)\right)^{\lfloor n_x/2\rfloor}.
			\]
			\item[3.] For any $k,m,q \in \N_0$ and $u \in \R$, there exist $c^\epsilon_r(k, m;u)\in \R$, $r \in \{0, 1, \cdots, k+m\}$ such that 
			\[
			V^\epsilon_{\epsilon^\kappa k}(\epsilon^\kappa q, u)V_{\epsilon^\kappa m}^\epsilon(\epsilon^\kappa q, u) = \sum_{r=0}^{k+m}c_r^\epsilon(k, m; u) V^\epsilon_{\epsilon^\kappa r}(\epsilon^\kappa q, u),
			\] 
			with
			\[
			\max_{r\in \{0, 1, \cdots, k+m\}}\left|c^\epsilon_r(k, m;u)\right|\leq c(k, m)(\max(|u|, 1))^{k+m}.
			\]
			uniformly in $\epsilon \in (0, 1)$. 
			\item[4.] Fix $x_0 \in \X_\epsilon$. Given a function $\rho:\Z_\epsilon\to [0, \infty)$, define $\widetilde{x}_0=x_0-\rho$. For any $x \in \X_\epsilon$, we introduce a configuration $\hat x \in \X_\epsilon$ given by $\hat x(z)\coloneqq x(z)\mathbbm{1}_{n_x(z)\geq 2}$ for all $z \in \Z_\epsilon$. Then, we have 
			\begin{equation}\label{eq:prod_formula_V}
				\prod_{z \in x} \widetilde{x}_0(z) = \sum_{x_1\cleq \hat x} d^\epsilon(\hat x, x_1;\rho)V^\epsilon(x\backslash x_1, x_0;\rho),
			\end{equation}
			for all $x \in \X_\epsilon$, for some coefficients $d^\epsilon(\hat x, x_1;\rho)$ satisfying
			\begin{equation}\label{ineq:control_d_eps}
			d^\epsilon(\hat x, 0;g)=1, \quad |d^\epsilon(\hat x, x_1;\rho)|\leq c\left(\max_{z \in \Z_\epsilon} \rho(z)\right)^{n_{x_1}},
			\end{equation}
			for some constant $c=c(n_{\hat x}, n_{x_1})>0$.
		\end{enumerate}
	\end{lem}
	
	\begin{proof}
		This proof is an adaptation of the proof of Proposition A.1 of \cite{Boldrig1992}. We prove the first property by induction. For a configuration $x$ with single site, the factorization property holds trivially. Assume that it holds for $k$ sites $z_1, \cdots, z_k \in \Z_\epsilon$ and consider a configuration $x = x_k \cup x(z_{k+1})$ where $x_k \in \X_\epsilon$ is supported on the first $k$ sites and $z_{k+1} \in \Z_\epsilon$ is an additional site. To prove the induction step, we partition 
		\[
		V^\epsilon(x, x';\rho) = \sum_{x''\cleq x}Q^\epsilon(x'', x')(-1)^{n_x-n_{x''}}\prod_{z \in x\backslash x''}\rho(z)
		\] 
		according to the number $i$ of elements of $x(z_{k+1})$ contained in the subset $x''$. This approach yields the expansion
		\begin{align*}
			V^\epsilon(x, x';\rho) &= \sum_{i=0}^{n_x(z_{k+1})} {{n_x(z_{k+1})} \choose i} \epsilon^{\kappa i} Q_i(\epsilon^{-\kappa}x(z_{k+1}))(-1)^{n_x(z_{k+1})-i}\rho(z_{k+1})^i
			\\&\quad\quad\quad\quad\quad\quad\quad\quad\quad\quad\quad\quad\quad\quad\quad\times \sum_{\substack{x''\cleq x\\n_{x''}(z_{k+1})=i}}Q^\epsilon(x_k'', x'_k)(-1)^{n_{x_k}- n_{x_k''}}\prod_{z \in x_k \backslash x_k''}\rho(z).
		\end{align*}
		The second summation does not depend on $i$, and is simply $V^\epsilon(x_k, x_k'; \rho)$, which by the induction hypothesis splits as a product of $V$-functions over $z \in \supp(x_k)$. The first summation above is by definition $V^\epsilon(x'(z_{k+1}); \rho(z_{k+1}))$. This proves that
		\[
		V^\epsilon(x, x';\rho) = V^\epsilon(x'(z_{k+1}); \rho(z_{k+1})) \prod_{z \in \supp(x_k)} V^\epsilon_{x_k(z)}(x_k'(z), \rho(z)).
		\]
		Hence property 1. holds. Consider claim 2.. Suppose that $x \in \X_\epsilon$ satisfies $x(z)=1$ for some $z \in \Z_\epsilon$. Since $V^\epsilon(x, x';x')$ splits as a product over $\supp(x)$ and $V^\epsilon(x(z), x'(z);\rho(z)) = x'(z)-\rho(z)$, it follows that $V^\epsilon(x, x';x')=0$. The inequality claimed in 2. follows from the following inequality: for any $k, n \in \N_0$ with $k\leq n$, we have
		\begin{equation}\label{ineq:V_eps_LLN}
			|\widetilde{V}^\epsilon_k(n, n)| \leq c \epsilon^{\kappa \lceil k/2\rceil} n^{\lfloor k/2\rfloor},
		\end{equation}
		for some $c=c(k)>0$. We can write 
		\[
		Q_\ell(n) = \sum_{h=0}^{\ell-1}d_h(\ell)n^{\ell-h}, \quad \ell\leq n, \quad \ell, n \in \N_0, 
		\]
		for some coefficients $d_h(\ell)$ with $d_0\equiv 1$, and where $Q_0(n)=1$ as an empty product. Setting $d_h(\ell)=0$ for $h\geq \ell$, we obtain, for all $m, n \in \N_0$ and $u \in \R$, that
		\begin{align*}
			\widetilde{V}^\epsilon_k(n, n) &= \sum_{\ell=0}^{k} n^{k-\ell}(-1)^{k-\ell}{k \choose \ell}\epsilon^{\kappa \ell}\sum_{h=0}^{\ell -1}d_h(\ell)n^{\ell-h}
			\\&=\sum_{h=0}^{k-1}\left(\sum_{\ell=0}^k(-1)^{k-\ell}{k \choose \ell} d_h(\ell)\right) \epsilon^{\kappa h}n^{k-h}. 
		\end{align*}
		The coefficient $d_h(\ell)$ is a polynomial in $\ell$ of degree at most $2h$. Indeed, suppose that $d_h(\ell)$ had degree $2h+1$. Then, since $Q_\ell(n)$ is a falling factorial, the degrees of $d_h(\ell)$ and of $n^{\ell-h}$ need to add up to $\ell$, i.e. we require that $\ell=(2h+1)+\ell-h$. This does not hold since $\ell-(2h+1)<\ell-h$ because $-h-1<0$ for all $h\geq 0$. Thus, we can write
		\[
		d_h(\ell) = \sum_{r=0}^{2h} c_r(h) Q_r(\ell), 
		\]
		for some coefficients $c_r(h)$. Then,
		\begin{align*}
			\widetilde{V}^\epsilon_k(n, n)&= \sum_{h=0}^{k-1}\left(\sum_{r=0}^{2h} c_r(h) \sum_{\ell=0}^kQ_r(\ell)(-1)^{k-\ell} {k \choose \ell} \right)\epsilon^{\kappa h}n^{k-h}
			\\& = \sum_{h=0}^{k-1}\left(\sum_{r=0}^{2h} c_r(h) \left(\frac{d}{dx}\right)^r (x-1)^{k}\Bigg|_{x=1} \right)\epsilon^{\kappa h}n^{k-h}.
		\end{align*}
		If $2h<k$, then the derivatives are zero, so the coefficient is zero. Thus, $h\geq \lceil k/2\rceil$ and $k-h \leq k-\lceil k/2\rceil\leq \lfloor k/2\rfloor$, meaning that there exists $c=c(k)>0$ independent of $\epsilon$ such that \eqref{ineq:V_eps_LLN} holds. Using claim 1. to split the $V^\epsilon(x, x';x')$ as a product, and then applying \eqref{ineq:V_eps_LLN} to each factor proves claim 2. Next, we establish claim 3. It is shown in the proof of Proposition A.1 of \cite{Boldrig1992} that for all $k, m, q \in \N$, we have 
		\begin{equation}\label{eq:product_Q}
		Q_k(q)Q_m(q) = \sum_{j=0}^{\min(k, m)}{k\choose j}{m\choose j} j! Q_{k+m-j}(q),
		\end{equation}
		with ${q \choose k}=0$ for $k>q$. Thus, 
		\begin{align*}
			\widetilde V_k^\epsilon(q, u)\widetilde V_m^\epsilon(q, u)&=\sum_{j_1=0}^k \sum_{j_2=0}^m {k\choose j_1}{m \choose j_2}(-u)^{k+m-j_1-j_2} Q_{j_1}^\epsilon(q)Q_{j_2}^\epsilon(q)
			\\& = \sum_{j_1=0}^k \sum_{j_2=0}^m {k\choose j_1}{m \choose j_2}(-u)^{k+m-j_1-j_2}\epsilon^{\kappa(j_1+j_2)}\sum_{\ell=0}^{\min(j_1, j_2)}{j_1 \choose \ell}{j_2 \choose \ell} \ell! Q_{j_1+j_2-\ell}(q). 
		\end{align*}
		Let $\cL$ be the endomorphism of the vector space of polynomials of one variable $q$ which sends the basis of falling factorials $\{n\mapsto Q_k(q):k \in \N_0\}$, to the basis of monomials $\{z\mapsto q^k:k \in \N_0\}$, and is defined by $\cL(Q_k(q))=q^k$, for all $k \in \N_0$. It is clear that $\cL$ is an isomorphism of vector spaces. Moreover, for all $k, q \in \N_0$ and $u \in \R$, we have 
		\begin{align*}
			\cL(\widetilde V_k^\epsilon(q, u))& = \sum_{\ell=0}^k \cL(Q^\epsilon_\ell(q))u^{k-\ell}(-1)^{k-\ell} {k \choose \ell}
			\\&= \sum_{\ell=0}^k (\epsilon^\kappa q)^\ell u^{k-\ell}(-1)^{k-\ell}{k\choose \ell}
			\\&=(\epsilon^\kappa q - u)^k.
		\end{align*}
		Hence, 
		\begin{align*}
			&\cL(\widetilde V^\epsilon_k(q, u)\widetilde V^\epsilon_m(q, u))
			\\&= \sum_{j_1=0}^k \sum_{j_2=0}^m {k \choose j_1}{m \choose j_2} (-1)^{k+m -j_1-j_2}\epsilon^{\kappa(j_1+j_2)}\\&\quad\quad\quad\quad\quad\quad\quad\quad\quad\quad\quad\quad\quad\times\sum_{\ell=0}^{\min(j_1, j_2)}\sum_{r=0}^{j_1+j_2-\ell}{j_1 \choose \ell}{j_2 \choose \ell} \ell!{{j_1+j_2-\ell} \choose r} u^{k+m-\ell-r}(\epsilon^\kappa q-u)^r. 
		\end{align*}
		Applying $\cL^{-1}$ and pulling the summation of $r$ to the front, we can write
		\begin{align*}
			\widetilde V^\epsilon_k(q, u)\widetilde V^\epsilon_m(q, u)&=\sum_{r=0}^{k+m} c_r^\epsilon(k, m;u)\widetilde V^\epsilon_r(q, u),
		\end{align*}
		where 
		\[
		c_r^\epsilon(k, m;u)=\sum_{j_1=0}^k \sum_{j_2=0}^m {k \choose j_1}{m \choose j_2} (-1)^{k+m -j_1-j_2}\epsilon^{\kappa(j_1+j_2)}\sum_{\ell=0}^{\min(j_1, j_2)}{j_1 \choose \ell}{j_2 \choose \ell} \ell!{{j_1+j_2-\ell} \choose r} u^{k+m-\ell-r}.
		\]
		Finally, for $k, m, q, r \in \N_0$ and $u \in\R$, we compute
		\begin{align*}
			V^\epsilon_{\epsilon^\kappa k}(\epsilon^\kappa q, u) V^\epsilon_{\epsilon^\kappa m}(\epsilon^\kappa q, u) &= \widetilde V^\epsilon_k(q, u)\widetilde V^\epsilon_m(q, u)
			\\&=\sum_{r=0}^{k+m}c_r^\epsilon(k, m;u)\widetilde{V}^\epsilon_r(q, u)
			\\&=\sum_{r=0}^{k+m}c_r^\epsilon(k, m; u) V^\epsilon_{\epsilon^\kappa r}(\epsilon^\kappa q, u). 
		\end{align*}
		The bound claimed on $c_r^\epsilon(k,m;u)$ is readily inferred from its definition above. This proves claim 3.. Lastly, for $x_0 \in \X_\epsilon$, we prove \eqref{eq:prod_formula_V} of claim 4 by induction on the number of distinct sites in $\hat x$. First observe that
		\[
			\prod_{z \in x}\widetilde{x}_0(z) = \prod_{z' \in x\backslash \hat x}\widetilde{x}_0(z')\prod_{z \in \hat x}\widetilde{x}_0(z).
		\]
		Using Definition \ref{def:V_fcn} and claim 1, we compute $\prod_{z' \in x\backslash \hat x}\widetilde{x}_0(z') = V^\epsilon(x\backslash \hat x, x_0;\rho)$. Thus 
		\[
		\prod_{z \in x}\widetilde{x}_0(z)=V^\epsilon(x\backslash \hat x, x_0;\rho)\prod_{z \in \hat x}\widetilde{x}_0(z).
		\]
		If 
		\begin{equation}\label{eq:to_prove_claim_4}
		\prod_{z \in \hat x}\widetilde{x}_0(z)=\sum_{x_1\cleq \hat x} d^\epsilon(\hat x, x_1;\rho)V^\epsilon(\hat x\backslash x_1, x_0;\rho),
		\end{equation}
		then another application of claim 1 shows that \eqref{eq:prod_formula_V}. Therefore, to show claim 4, it suffices to prove \eqref{eq:to_prove_claim_4} by induction on $m$ the number of distinct sites in $\hat x$. If $m=0$, we have on the left-hand side the empty product, which equals one. On the right-hand side, the $V$-function is also one by definition, and we set $d^\epsilon(0, 0;\rho)=1$. We also do the case $m=1$ as it will be useful in the induction step. Denote the only site in $\hat x$ by $z \in \Z_\epsilon$. We will define the coefficients $d^\epsilon(x, x_1;\rho)$ in the decomposition claimed in \eqref{eq:to_prove_claim_4} where the left-hand side equals $\widetilde{x}_0(z)^{n_{\hat x}}$. By Definition \ref{def:V_fcn} and \eqref{eq:stirling_ff}, we compute
		\begin{align*}
		V^\epsilon(\hat x\backslash x_1, x_0;\rho) &= \sum_{i=0}^{n_{\hat x}-n_{x_1}}{{n_{\hat x}-n_{x_1}} \choose i} (-\rho(z))^{n_{\hat x}-n_{x_1}-i}Q^\epsilon_i(n_{x_0}(z))
		\\&=\sum_{j=0}^{n_{\hat x}-n_{x_1}}\left( \sum_{i=j}^{n_{\hat x}-n_{x_1}}{{n_{\hat x}-n_{x_1}} \choose i} (-\rho(z))^{n_{\hat x}-n_{x_1}-i}s(i, j)\right)x_0(z)^j,
		\end{align*}
		for all $x_1 \cleq \hat x$. Since $x_0(z)^j=\sum_{k=0}^j{j \choose k}\rho(z)^{j-k}\widetilde{x}_0(z)^k$, we obtain, after interchanging the order of summation, a polynomial in $\widetilde{x}_0(z)$ of leading order $n_{\hat x}-n_{x_1}$ with coefficient one:
		\[
		V^\epsilon(\hat x\backslash x_1, x_0;\rho)=\sum_{k=0}^{n_{\hat x}-n_{x_1}}c_k(n_{\hat x}, n_{x_1}; \rho)\widetilde{x}_0(z)^k,
		\]
		where 
		\begin{equation}\label{eq:claim_4_interm_coeffs}
		c_k(n_{\hat x}, n_{x_1}; \rho)= \sum_{j=k}^{n_{\hat x}-n_{x_1}}{j \choose k}\rho(z)^{j-k}\sum_{i=j}^{n_{\hat x}-n_{x_1}}{{n_{\hat x}-n_{x_1}} \choose i} (-\rho(z))^{n_{\hat x}-n_{x_1}-i}s(i, j).
		\end{equation}
		For $n_{x_1}=0$, we set $d^\epsilon(\hat x, 0;\rho)=1$. Then the leading term already contributes $\widetilde{x}_0(z)^{n_{\hat x}}$, and we need to choose the $d^\epsilon(\hat x, x_1;\rho)$, for non-empty $x_1\cleq \hat x$, so as to cancel all lower order terms on the right-hand side of \eqref{eq:to_prove_claim_4}. This is achieved by taking for all non-empty $x_1\cleq \hat x$ with $n_{x_1}=k$, some $k\geq 1$, the coefficients
		\begin{equation}\label{eq:claim_4_coeffs}
		d^\epsilon(x, x_1;\rho) \coloneqq -c_{n_{\hat x}-k}(n_{\hat x}, 0;\rho)-\sum_{\ell=1}^{k-1}d^\epsilon(\hat x, \ell\epsilon^\kappa\mathbbm{1}_{\{z\}};\rho)c_{n_{\hat x}-k}(n_{\hat x}-\ell, 0;\rho),
		\end{equation}
		This proves the $m=1$ case. Assume now that $m\geq 1$ and that the decomposition \eqref{eq:to_prove_claim_4} holds. Suppose now that $\hat x$ is supported on $m+1$ distinct sites $z_1, \cdots, z_{m+1} \in \Z_\epsilon$. Given a configuration $x\in \X_\epsilon$ and a site $z \in \Z_\epsilon$, we define two configurations $\xnot{z}, \xonly{z}\in \X_\epsilon$ with $x=\xnot{z}+\xonly{z}$, as follows:
		\begin{equation}\label{def:x_not_only}
		\xnot{z}(z') \coloneqq x(z')\mathbbm{1}_{\Z_\epsilon\backslash \{z\}}(z')
		\quad\quad\quad\textnormal{and} \quad\quad\quad \xonly{z}(z') \coloneqq 
		x(z)\mathbbm{1}_{\{z\}}(z').
		\end{equation}
		By the induction hypothesis, we have 
		\begin{align*}
			\prod_{z \in \hat x}\widetilde{x}_0(z)&=\widetilde{x}_0(z_{m+1})^{n_{\hat x}(z_{m+1})}\prod_{i=1}^m\widetilde{x}_0(z_i)^{n_{\hat x}(z_i)}
			\\&=\widetilde{x}_0(z_{m+1})^{n_{\hat x}(z_{m+1})}\sum_{x_1\cleq \xnottwo{\hat x}{z_{m+1}}} d^\epsilon(\xnottwo{\hat x}{z_{m+1}}, x_1;\rho)V^\epsilon(\xnottwo{\hat x}{z_{m+1}}\backslash x_1, x_0;\rho),
		\end{align*}
		By the $m=1$ case, we can write 
		\[
		\widetilde{x}_0(z_{m+1})^{n_{\hat x}(z_{m+1})}=\sum_{x_2\cleq \xonlytwo{\hat x}{z_{m+1}}} d^\epsilon(\xonlytwo{\hat x}{z_{m+1}}, x_2;\rho)V^\epsilon(\xonlytwo{\hat x}{z_{m+1}}\backslash x_2, x_0;\rho).
		\]
		It follows from claim 1 that
		\begin{align*}
			\prod_{z \in \hat x}\widetilde{x}_0(z)&=\sum_{x_1\cleq \xnottwo{\hat x}{z_{m+1}}}\sum_{x_2\cleq \xonlytwo{\hat x}{z_{m+1}}}d^\epsilon(\xnottwo{\hat x}{z_{m+1}}, x_1;\rho)d^\epsilon(\xonlytwo{\hat x}{z_{m+1}}, x_2;\rho)
			\\&\quad\quad\quad\quad\quad\quad\quad\quad\quad\quad\times V^\epsilon(\xnottwo{\hat x}{z_{m+1}}\backslash x_1, x_0;\rho)V^\epsilon(\xonlytwo{\hat x}{z_{m+1}}\backslash x_2, x_0;\rho)
			\\&=\sum_{x_1 \cleq \hat x}d^\epsilon(x, x_1;\rho)V^\epsilon(\hat x\backslash x_1, x_0;\rho),
		\end{align*}
		where 
		\[
		d^\epsilon(x, x_1;\rho)\coloneqq d^\epsilon(\xnottwo{\hat x}{z_{m+1}}, \xnottwo{x_1}{z_{m+1}};\rho)d^\epsilon(\xonlytwo{\hat x}{z_{m+1}}, \xonlytwo{x_1}{z_{m+1}};\rho).
		\]
		This concludes the proof of \eqref{eq:prod_formula_V}. The inequality \eqref{ineq:control_d_eps} readily follows from \eqref{eq:claim_4_interm_coeffs} and \eqref{eq:claim_4_coeffs}. The lemma is proved. 
	\end{proof}
	
	In the next lemma, we gather some technical calculations on generators. Recall from \eqref{def:disc_Laplacian_config} and \eqref{def:disc_Laplacian_sites} the definition of the discrete Laplacian acting on functions $F:\X_\epsilon\to \R$ and $f:\Z_\epsilon\to \R$, respectively. Recall also the generator $\cG^\epsilon=\cG^\epsilon_{RW}+\cG^\epsilon_{BC}$ of $\rX^\epsilon$ which we defined in \eqref{def:slbp_gen}.
	\begin{lem}\label{lem:localisation}
		The following hold for all $x, x_0 \in \X_\epsilon$ and $\rho:\Z_\epsilon\to \R$.
		\begin{enumerate}
			\item[1.] The generator $\cG_{RW}^\epsilon$ of the spatial motion satisfies 
			\begin{equation}\label{eq:GRW_is_laplacian}
			\cG_{RW}^\epsilon Q^\epsilon(x, \cdot)(x_0)=\frac{1}{2}\Delta^\epsilon Q^\epsilon(\cdot, x_0)(x)=\frac{1}{2}\sum_{z \in x}\Delta^\epsilon_z Q^\epsilon(\cdot, x_0)(x).
			\end{equation}
			\item[2.] The Laplacian of the $V$-functions is given by 
			\begin{equation}
			\begin{aligned}\label{eq:Laplacian_V}
			\Delta^\epsilon V^\epsilon(\cdot, x_0;\rho)(x)&=\sum_{x_1\cleq x}(-1)^{n_x-n_{x_1}}\Bigg\{\Big(\sum_{z \in x_1}\Delta^\epsilon_z Q^\epsilon(\cdot, x_0)(x_1)\Big)\prod_{z' \in x\backslash x_1}\rho(z')
			\\&\quad\quad\quad\quad\quad\quad\quad\quad\quad\quad +Q^\epsilon(x_1, x_0)\sum_{z \in x\backslash x_1} \Delta^\epsilon_z \rho_t^\epsilon(z)\prod_{\substack{z' \in x\backslash x_1 \\ z'\not=z}}\rho(z')\Bigg\}.
			\end{aligned}
			\end{equation}
			\item[3.] The birth-competition generator $\cG^\epsilon_{BC}$ splits into its local components:
			\begin{equation}
				\cG^\epsilon_{BC}Q^\epsilon(x, \cdot)(x_0)=\sum_{z \in \supp(x)}Q^\epsilon(\xnot{z}, x_0)\cG^\epsilon_{BC}Q^\epsilon(\xonly{z}, \cdot)(x_0),\label{eq:Q_is_local}
			\end{equation}
			where $\xnot{z}, \xonly{z} \in \X_\epsilon$ are defined in \eqref{def:x_not_only}. 
			\item[4.] We consider for all $x, x_0 \in \X_\epsilon$ and $\rho:\Z_\epsilon\to \R$ the quantity
			\begin{equation}
				\begin{aligned}\label{def:GBC_tilde}
					\widetilde{\cG}_{BC}^\epsilon V^\epsilon(x, \cdot; \rho)(x_0)&\coloneqq \sum_{x_1\cleq x}(-1)^{n_x-n_{x_1}}\Bigg\{\Big(\cG_{BC}^\epsilon Q^\epsilon(x_1, x_0)\Big)\prod_{z' \in x\backslash x_1}\rho(z')
					\\&\quad\quad\quad\quad\quad +Q^\epsilon(x_1, x_0)\sum_{z \in x\backslash x_1} \rho_t^\epsilon(z)(\mu_\epsilon-\rho_t^\epsilon(z)/2)\prod_{\substack{z' \in x\backslash x_1 \\ z'\not=z}}\rho(z')\Bigg\}.
				\end{aligned}
			\end{equation}
			Then
			\begin{equation}
				\widetilde{\cG}^\epsilon_{BC}V^\epsilon(x, \cdot; \rho)(x_0)=\sum_{z \in \supp(x)}V^\epsilon(\xnot{z}, x_0; \rho)\widetilde{\cG}^\epsilon_{BC}V^\epsilon(\xonly{z}, \cdot;\rho)(x_0),\label{eq:V_is_local}
			\end{equation}
		\end{enumerate}	
	\end{lem}
	\begin{proof}
		We begin by proving \eqref{eq:GRW_is_laplacian}. By \eqref{eq:GRW}, we have 
		\[
		\cG^\epsilon_{RW}Q^\epsilon(x, \cdot)(x_0) = \sum_{z \in \supp(x)}n_{x_0}(z)\left\{Q^\epsilon(x, x_0^{z, z+1})+Q^\epsilon(x, x_0^{z, z-1})-2Q^\epsilon(x, x_0)\right\},
		\]
		where the updated configurations $x_0^{z, z\pm1}$ are defined in \eqref{def:updated_config}. We compute the contributions at $z$ and $z\pm 1$ by noticing that
		\[
		n_{x_0}(z)Q^\epsilon_{n_x(z)}(n_{x_0}(z)-1) = \epsilon^{-\kappa}Q^\epsilon_{n_x(z)+1}(n_{x_0}(z))=(n_{x_0}(z)-n_x(z))Q^\epsilon_{n_x(z)}(n_{x_0}(z)),
		\]
		and using that $n_{x_0}(z\pm1)+1 = (n_{x_0}(z\pm1)-n_x(z\pm1)+1) +n_x(z\pm1)$,  
		\[
		Q^\epsilon_{n_x(z\pm1)}(n_{x_0}(z\pm1)+1) = Q^\epsilon_{n_x(z\pm1)}(n_{x_0}(z\pm1))+n_x(z\pm1)Q^\epsilon_{n_x(z\pm1)-1}(n_{x_0}(z\pm1)).
		\]
		It follows that
		\begin{align*}
			&n_{x_0}(z)Q^\epsilon(x, x_0^{z, z\pm1}) 
			\\&= \Big[(n_{x_0}(z)-n_x(z))Q^\epsilon_{n_x(z)}(n_{x_0}(z))Q^\epsilon_{n_x(z\pm1)}(n_{x_0}(z\pm1))
			\\&+n_x(z\pm1)Q^\epsilon_{n_x(z)+1}(n_{x_0}(z))Q^\epsilon_{n_x(z\pm1)-1}(n_{x_0}(z\pm1))\Big] \prod_{z' \in \supp(x)\backslash\{z, z\pm 1\}} Q^\epsilon_{n_x(z')}(n_{x_0}(z')). 
		\end{align*}
		Using this computation, we obtain 
		\begin{align*}
			\notag\cG^\epsilon_{RW}Q^\epsilon(x, \cdot)(x_0)&=\frac{1}{2}\sum_{z \in \supp(x)}[-2n_x(z)Q^\epsilon(x, x_0)+n_x(z+1)Q^\epsilon(x, x_0^{z+1, z})+n_x(z-1)Q^\epsilon(x^{z-1, z}, x_0)] 
			\\&=\frac{1}{2}\sum_{z \in \supp(x)}n_x(z)[Q^\epsilon(x^{z, z+1}, x_0)+Q^\epsilon(x^{z, z-1}, x_0)-2Q^\epsilon(x, x_0)]
			\\&=\frac{1}{2}\sum_{z \in x} \Delta^\epsilon_z Q^\epsilon(\cdot, x_0)(x)=\Delta^\epsilon Q^\epsilon(\cdot, x_0)(x).
		\end{align*}
		This proves the first claim. To show \eqref{eq:Laplacian_V}, we first recall the notation \eqref{def:sum_prod_config_iteration}, and observe that for all $x\in \X_\epsilon$ and $x_1 \cleq x$, 
		\begin{align*}
		\sum_{z \in x}\Delta^\epsilon_z \left\{(-1)^{n_x-n_{x_1}}Q^\epsilon(\cdot, x_0)(x_1)\prod_{z' \in x\backslash x_1}\rho(z')\right\}&=(-1)^{n_x-n_{x_1}}\Bigg\{\sum_{z \in x_1}\left(\Delta^\epsilon_z Q^\epsilon(\cdot, x_0)(x_1)\right)\prod_{z' \in x\backslash x_1}\rho(z')
		\\&+\sum_{z \in x\backslash x_1}Q^\epsilon(x_1, x_0)\Delta^\epsilon_z \rho(z) \prod_{\substack{z' \in x\backslash x_1 \\ z'\not= z}}\rho(z')\Bigg\},
		\end{align*}
		where we have used that $n_x-n_{x_1}$ is invariant under spatial motion. Thus \eqref{eq:Laplacian_V} follows if we can show that the Laplacian $\Delta^\epsilon_z$ commutes with the summation over $x_1 \cleq x$. We have 
		\begin{align}
			\Delta^\epsilon_z \sum_{x_1 \cleq x}(-1)^{n_x-n_{x_1}}Q^\epsilon(\cdot, x_0)(x_1)\prod_{z' \in x\backslash x_1}\rho(z')&\notag=\sum_{x_1 \cleq x^{z, z+1}}(-1)^{n_x-n_{x_1}}Q^\epsilon(\cdot, x_0)(x_1)\prod_{z' \in x^{z, z+1}\backslash x_1}\rho(z')
			\\&\notag+\sum_{x_1 \cleq x^{z, z-1}}(-1)^{n_x-n_{x_1}}Q^\epsilon(\cdot, x_0)(x_1)\prod_{z' \in x^{z, z-1}\backslash x_1}\rho(z')
			\\&-2\sum_{x_1 \cleq x}(-1)^{n_x-n_{x_1}}Q^\epsilon(\cdot, x_0)(x_1)\prod_{z' \in x\backslash x_1}\rho(z').\label{eq:lapl_commutes}
		\end{align}
		We make the changes of variables $x_1\leftarrow x_1^{z, z+1}$ and $x_1 \leftarrow x_1^{z, z-1}$ in the first and second summations, respectively. This implies
		\begin{align*}
			&\sum_{x_1 \cleq x^{z, z\pm 1}}(-1)^{n_x-n_{x_1}}Q^\epsilon(\cdot, x_0)(x_1)\prod_{z' \in x^{z, z\pm1}\backslash x_1}\rho(z')
			\\&=\sum_{x_1 \cleq x}(-1)^{n_x-n_{x_1}}Q^\epsilon(\cdot, x_0)(x_1^{z, z\pm1})\prod_{z' \in x^{z, z\pm1}\backslash x_1^{z, z\pm1}}\rho(z').
		\end{align*}
		We note here that if $z=z_j^{x_1}$ for some $j \in \{1, \cdots, n_{x_1}\}$, then also $z=z^x_{k}$ for some $k \in \{1, \cdots, n_x\}$ since $x_1\cleq x$. It follows that $x^{z, z\pm1}\backslash x_1^{z, z\pm1}$ in this case. If on the other hand $z\not =z_j^{x_1}$ for any $j \in \{1, \cdots, n_{x_1}\}$, then $x_1^{z, z\pm 1}=x_1$. Thus going back to \eqref{eq:lapl_commutes}, we obtain 
		\begin{align*}
			\Delta^\epsilon_z \sum_{x_1 \cleq x}(-1)^{n_x-n_{x_1}}Q^\epsilon(\cdot, x_0)(x_1)\prod_{z' \in x\backslash x_1}\rho(z')=\sum_{x_1 \cleq x}\Delta^\epsilon_z\left\{(-1)^{n_x-n_{x_1}}Q^\epsilon(\cdot, x_0)(x_1)\prod_{z' \in x\backslash x_1}\rho(z')\right\}.
		\end{align*}
		Summing this expression over $z \in x$ using \eqref{def:sum_prod_config_iteration}, we recover \eqref{eq:Laplacian_V}, which concludes the proof of part two of the lemma. Equation \eqref{eq:Q_is_local} follows immediately from \eqref{eq:GBC} with $f(x) = Q^\epsilon(x, x_0)$ by noticing that 
		\[
		Q^\epsilon(x^{z, \ell}, x_0)-Q^\epsilon(x, x_0) = Q^\epsilon(\xnot{z}, x_0) \times\left[Q^\epsilon((\xonly{z})^{z, \ell}, x_0)-Q^\epsilon(\xonly{z}, x_0)\right],
		\]
		for all $\ell \in \N\cup \{-1\}$. It remains to show \eqref{eq:V_is_local}. By first applying \eqref{eq:Q_is_local}, and then interchanging the order of summation in \eqref{def:GBC_tilde} using the fact that $\cG^\epsilon_{BC}Q^\epsilon(\cdot, x_0)(\xonlytwo{x_1}{z})=0$ if $x_1(z)=0$, we obtain 
		\begin{align*}
			\widetilde{\cG}_{BC}^\epsilon V^\epsilon(x, \cdot; \rho)(x_0)&= \sum_{x_1\cleq x}(-1)^{n_x-n_{x_1}}\Bigg\{\Big(\sum_{z \in \supp(x_1)}Q^\epsilon(\xnottwo{x_1}{z}, x_0)\cG^\epsilon_{BC}Q^\epsilon(\cdot, x_0)(\xonlytwo{x_1}{z})\Big)
			\\&\quad\quad\quad\quad\quad\quad\quad \times \prod_{z' \in x\backslash x_1}\rho(z')+Q^\epsilon(x_1, x_0)\sum_{z \in x\backslash x_1} \rho_t^\epsilon(z)(\mu_\epsilon-\rho_t^\epsilon(z)/2)\prod_{\substack{z' \in x\backslash x_1 \\ z'\not=z}}\rho(z')\Bigg\}
			\\&=\sum_{z \in \supp(x)}\sum_{x_1\cleq x}(-1)^{n_x-n_{x_1}}\Bigg\{Q^\epsilon(\xnottwo{x_1}{z}, x_0)\cG^\epsilon_{BC}Q^\epsilon(\cdot, x_0)(\xonlytwo{x_1}{z}) \prod_{z' \in x\backslash x_1}\rho(z')
			\\&\quad\quad\quad\quad\quad\quad\quad+Q^\epsilon(x_1, x_0)n_{x\backslash x_1}(z) \rho_t^\epsilon(z)(\mu_\epsilon-\rho_t^\epsilon(z)/2)\prod_{\substack{z' \in x\backslash x_1 \\ z'\not=z}}\rho(z')\Bigg\}.
		\end{align*}
		We have the factorisation
		\[
		Q^\epsilon(x_1, x_0)=Q^\epsilon(\xnottwo{x_1}{z}, x_0)Q^\epsilon(\xonlytwo{x_1}{z}, x_0);
		\]
		so letting $x_2=\xnottwo{x_1}{z}$ and $x_3=\xonlytwo{x_1}{z}$, we get
		\begin{align*}
			\widetilde{\cG}_{BC}^\epsilon V^\epsilon(x, \cdot; \rho)(x_0)&=\sum_{z \in \supp(x)}\sum_{x_2\cleq \xnot{z}}(-1)^{n_{\xnot{z}}-n_{x_2}}Q^\epsilon(x_2, x_0)\prod_{z' \in \xnot{z}\backslash x_2}\rho_t^\epsilon(z')
			\\&\times \sum_{x_3 \cleq \xonly{z}}(-1)^{n_{\xonly{z}}-n_{x_3}}\Bigg\{\cG^\epsilon Q^\epsilon(\cdot, x_0)(x_3)\prod_{z'' \in \xonly{z}\backslash x_3}\rho^\epsilon_t(z'')
			\\&\quad\quad\quad\quad+n_{\xonly{z}\backslash x_3}(z)Q^\epsilon(x_3, x_0)\rho_t^\epsilon(z)(\mu_\epsilon-\rho_t^\epsilon(z)/2)\prod_{z''\in \xonly{z}\backslash x_3}\rho_t^\epsilon(z'')\Bigg\}
			\\&=\sum_{z \in \supp(x)}V^\epsilon(\xnot{z}, x_0; \rho)\widetilde{\cG}^\epsilon_{BC}V^\epsilon(\xonly{z}, \cdot;\rho)(x_0).
		\end{align*}
		This proves \eqref{eq:V_is_local} and the lemma. 
	\end{proof}
	
	We now apply the above properties to derive the semi-discrete heat equation \eqref{eq:v_fcn_eq} of Proposition \ref{prop:v_fcn_eq} for the $v$-functions. 

	\begin{proof}[Proof of Proposition \ref{prop:v_fcn_eq}]
		We begin by showing that 
		\begin{equation}\label{eq:prop_3_goal}
		\frac{d}{dt}v_t^\epsilon(x, \rho_t^\epsilon |\nu^\epsilon)=\frac{1}{2}\Delta^\epsilon v_t^\epsilon(x, \rho_t^\epsilon |\nu^\epsilon)+\sum_{z\in \supp(x)}\sum_{h=-n_x(z)}^1 c_h^\epsilon(x(z), \rho_t^\epsilon(z)) v_t^\epsilon(x^{(z, h)}, \rho_t^\epsilon|\nu^\epsilon),
		\end{equation}
		for all $x \in \X_\epsilon$, $t\geq 0$ and $\epsilon \in (0, 1)$. Recalling Definitions \ref{def:V_fcn} and \ref{def:scaled_v_fcns}, we observe that 
		\begin{align*}
		v_t^\epsilon(x, \rho_t^\epsilon|\nu^\epsilon) &= \E\left[V^\epsilon(x, X_t^\epsilon;\rho_t^\epsilon)\right]
		\\&= \E\left[\sum_{x_1\cleq x}(-1)^{n_x- n_{x_1}}Q^\epsilon(x_1, X^\epsilon_t)\prod_{z \in x\backslash x_1}\rho_t^\epsilon(z)\right]
		\\&= \sum_{x_1\cleq x}(-1)^{n_x- n_{x_1}} \E\left[Q^\epsilon(x_1, X^\epsilon_t)\right]\prod_{z \in x\backslash x_1}\rho_t^\epsilon(z).
		\end{align*}
		Thus, by the product rule
		\begin{align}
			\frac{d}{dt}v_t^\epsilon(x, \rho_t^\epsilon |\nu^\epsilon) &\notag= \sum_{x_1\cleq x}(-1)^{n_x- n_{x_1}} \Bigg\{ \left(\prod_{z \in x\backslash x_1}\rho_t^\epsilon(z)\right)\frac{d}{dt} \E\left[Q^\epsilon(x_1, X^\epsilon_t)\right]
			\\&\quad\quad\quad\quad\quad\quad\quad\quad\quad\quad+\E\left[Q^\epsilon(x_1, X^\epsilon_t)\right]\sum_{z \in x\backslash x_1}\partial_t\rho_t^\epsilon(z)\prod_{\substack{z' \in x\backslash x_1 \\ z'\not=z}}\rho_t^\epsilon(z')\Bigg\}.\label{eq:v_fcn_lin_expect}
		\end{align}
		Since $\rX^\epsilon$ is Markov, we have that for each $x \in \X_\epsilon$, the process
		\[
		Q^\epsilon(x, X_t^\epsilon)-Q^\epsilon(x, X_0^\epsilon)-\int_0^t \cG^\epsilon Q^\epsilon(x, X_s^\epsilon)ds, \quad t\geq 0,
		\]
		is  a martingale, where $\cG^\epsilon=\cG^\epsilon_{RW}+\cG^\epsilon_{BC}$, defined in \eqref{def:slbp_gen}, is the generator of $\rX^\epsilon$. Taking expectations and then differentiating in time, we get 
		\[
		\frac{d}{dt}\E\left[Q^\epsilon(x, X_t^\epsilon)\right] = \E\left[\cG^\epsilon_{RW}Q^\epsilon(x, X_t^\epsilon)\right] + \E\left[\cG^\epsilon_{BC}Q^\epsilon(x, X_t^\epsilon)\right].
		\]
		Moreover, by \eqref{eq:interm_fkpp_diff_form}, we have $\partial_t \rho_t^\epsilon(z) = \frac{1}{2}\Delta^\epsilon_z \rho_t^\epsilon(z)+\rho_t^\epsilon(z)(\mu_\epsilon-\rho_t^\epsilon(z)/2)$. Thus we can decompose the expression \eqref{eq:v_fcn_lin_expect} as 
		\begin{equation}\label{eq:dv_dt_decomp}
		\frac{d}{dt}v_t^\epsilon(x, \rho_t^\epsilon |\nu^\epsilon) = \E\left[\widetilde{\cG}_{RW}^\epsilon V^\epsilon(x, X^\epsilon_t; \rho_t^\epsilon)\right]+\E\left[\widetilde{\cG}_{BC}^\epsilon V^\epsilon(x, X^\epsilon_t; \rho_t^\epsilon)\right].
		\end{equation}
		where for any $x, x_0 \in \X_\epsilon$ and $\rho:\Z_\epsilon \to \R$,
		\begin{align*}
			\widetilde{\cG}_{RW}^\epsilon V^\epsilon(x, x_0; \rho)&\coloneqq \sum_{x_1\cleq x}(-1)^{n_{x}-n_{x_1}}\Bigg\{\Big(\cG_{RW}^\epsilon Q^\epsilon(x_1, x_0)\Big)\prod_{z' \in x\backslash x_1}\rho(z')
			\\&\quad\quad\quad\quad\quad\quad\quad\quad\quad\quad\quad\quad +Q^\epsilon(x_1, x_0)\sum_{z \in x\backslash x_1} \frac{1}{2}\Delta^\epsilon_z \rho(z)\prod_{\substack{z' \in x\backslash x_1 \\ z'\not=z}}\rho(z')\Bigg\},
		\end{align*}
		and $\widetilde{\cG}_{BC}^\epsilon V^\epsilon(x, x_0; \rho)$ is defined in \eqref{def:GBC_tilde}. We compute these two expressions in turn. By part one of Lemma \ref{lem:localisation}, we see that 
		\[
		\cG_{RW}^\epsilon Q^\epsilon(x_1, \cdot)(x_0)=\frac{1}{2}\Delta^\epsilon Q^\epsilon(\cdot, x_0)(x_1)= \frac{1}{2}\sum_{z \in x_1}\Delta^\epsilon_z Q^\epsilon(\cdot, x_0)(x_1),
		\]
		where $\Delta^\epsilon$ and $\Delta^\epsilon_z$ are defined above in \eqref{def:disc_Laplacian_config}. Thus, using part two of Lemma \ref{lem:localisation}, we obtain 
		\[
		\widetilde{\cG}_{RW}^\epsilon V^\epsilon(x, \cdot\;; \rho_t^\epsilon)(x_0)=\frac{1}{2}\Delta^\epsilon V^\epsilon(\cdot, x_0;\rho_t^\epsilon)(x) = \frac{1}{2}\Delta^\epsilon V^\epsilon(\cdot, x_0;\rho_t^\epsilon)(x),
		\]
		and 
		\begin{equation}\label{eq:exp_G_RW}
		\E\left[\widetilde{\cG}_{RW}^\epsilon V^\epsilon(x, X^\epsilon_t; \rho_t^\epsilon)\right] = \frac{1}{2}\Delta^\epsilon v_t^\epsilon(x, \rho_t^\epsilon |\nu^\epsilon) =\frac{1}{2}\sum_{z \in x} \Delta^\epsilon_z v_t^\epsilon(x, \rho_t^\epsilon |\nu^\epsilon).
		\end{equation}
		The contribution of the birth-competition term of \eqref{eq:dv_dt_decomp} is given in the following lemma.
		\begin{lem}\label{lem:GBC_contrib}
			There exist coefficients $c_h^\epsilon(m, u) \in \R$, for all $m \in \{1, \cdots, n\}$,  $u \in \R$, $h \in \{-m, \cdots, 1\}$, $\epsilon\in (0, 1)$ such that 
			\begin{equation}\label{eq:exp_G_BC}
				\E\left[\widetilde{\cG}_{BC}^\epsilon V^\epsilon(x, X^\epsilon_t; \rho_t^\epsilon)\right]=\sum_{z\in \supp(x)}\sum_{h=-n_x(z)}^1 c_h^\epsilon(n_x(z), \rho_t^\epsilon(z)) v_t^\epsilon(x^{(z, h)}, \rho_t^\epsilon|\nu^\epsilon),
			\end{equation}
			for all $x \in \X_\epsilon$, $t\geq 0$, and $\epsilon \in (0, 1)$. Furthermore, if we fix $n \in \N$, there exists a constant $c=c(n)>0$ independent of $\epsilon$ such that
			\[
				|c^\epsilon_h(m, u)|\leq c \max(1, |u|^{m+1}), \quad m \in \{1, \cdots, n\}, \quad u \in \R, \quad h \in \{-m, \cdots, 1\}, \quad \epsilon\in (0, 1).
			\]
			In the particular case where $m=2$ and $h=-2$ and the offspring distribution $\rP$ has a moment of order $p\in[1, 2]$, there also exists $c>0$ such that 
			\[
			|c_{-2}^\epsilon(2, u)| \leq c \epsilon^\kappa\max(1, |u|^3)+o(\epsilon^{(p-1)\kappa})|u|,
			\]
			as $\epsilon \searrow 0$. For $m=2$ and $h=-1$, we have
			\[
			c_{-1}^\epsilon(2, u)=(2\mu_\epsilon+1)\epsilon^\kappa+o(\epsilon^{(p-1)\kappa}),
			\]
			as $\epsilon \searrow 0$. Finally, if $m=1$ and $h\leq -1$, we have $c_h^\epsilon(1, u)=0$.
		\end{lem}
		\noindent We first conclude the proof of Proposition \ref{prop:v_fcn_eq}, and then we prove Lemma \ref{lem:GBC_contrib}. Plugging \eqref{eq:exp_G_RW} and \eqref{eq:exp_G_BC} in \eqref{eq:dv_dt_decomp}, we obtain \eqref{eq:prop_3_goal}.
		To obtain the integrated form \eqref{eq:v_fcn_eq} of this equation which we claimed in the statement of Proposition \ref{prop:v_fcn_eq}, it suffices to compute the derivative of the map
		\[
		(0, t)\times \X_\epsilon \to \R, \quad (s, x)\mapsto{\sum_{\substack{x_1\in \X_\epsilon \\ n_{x_1}=n_x}} \G^\epsilon_{t-s}(x, x_1)v_s^\epsilon(x_1, \rho_s^\epsilon|\nu^\epsilon)}
		\]
		with respect to $s$, and then to integrate the result over $s \in (0, t)$. We have 
		\begin{align*}
			\frac{d}{ds} \sum_{x_1} \G^\epsilon_{t-s}(x, x_1)v_s^\epsilon(x_1, \rho_s^\epsilon|\nu^\epsilon) &= -\frac{1}{2}\sum_{x_1} \Delta^\epsilon \G^\epsilon_{t-s}(x, x_1)v_s^\epsilon(x_1, \rho_s^\epsilon|\nu^\epsilon)
			\\&+\sum_{x_1} \G^\epsilon_{t-s}(x, x_1)\frac{1}{2}\Delta^\epsilon v_s^\epsilon(x_1, \rho_s^\epsilon|\nu^\epsilon)
			\\&+ \sum_{x_1} \G^\epsilon_{t-s}(x, x_1)\sum_{z \in \supp(x_1)}\sum_{h=-n_{x_1}(z)}^1 c_h^\epsilon(x_1(z), \rho_s^\epsilon(z))v_s^\epsilon(x_1^{(z, h)}, \rho_s^\epsilon|\nu^\epsilon).
		\end{align*}
		By a simple change of variable, the first two terms on the right-hand side cancel out. Integrating over $s \in (0, t)$ establishes \eqref{eq:v_fcn_eq} and the proposition.
	\end{proof}
	
	\noindent It only remains to prove Lemma \ref{lem:GBC_contrib}.
	
	\begin{proof}[Proof of Lemma \ref{lem:GBC_contrib}]
		By \eqref{eq:V_is_local} in Lemma \ref{lem:localisation}, we have
		\begin{equation}\label{eq:reduce_pb_via_localisation}
		\widetilde{\cG}^\epsilon_{BC}V^\epsilon(x, x_0; \rho_t^\epsilon)=\sum_{z \in \supp(x)}V^\epsilon(\xnot{z}, x_0; \rho_t^\epsilon)\widetilde{\cG}^\epsilon_{BC}V^\epsilon(\xonly{z}, x_0;\rho_t^\epsilon),
		\end{equation}
		for all $x, x_0 \in \X_\epsilon$. To show \eqref{eq:exp_G_BC}, we will compute explicitly the term $\widetilde{\cG}^\epsilon_{BC}V^\epsilon(\xonly{z}, x_0;\rho_t^\epsilon)$ on the right-hand side of the last displayed expression. It will be easier to work with generic notation which we introduce now. Suppose that $n_x(z)=n_{\xonly{z}}=m \in \N_0$, $n_{x_0}(z)= q \in \N_0$ and $\rho_t^\epsilon(z)=u \in \R$, and recall from \eqref{def:tilde_V_eps} the definition of $\widetilde{V}_m(q, u)$. We introduce 
		\begin{align}
			\widetilde{\cG}_{BC}^\epsilon\widetilde{V}^\epsilon_m(q, u) &\notag\coloneqq \sum_{k=0}^{m}{m \choose k}(-u)^{m-k} \Big\{ \cG_{BC}^\epsilon Q^\epsilon_k(q)+Q^\epsilon_k(q)(m-k)u^{m-k-1}u(\mu_\epsilon-u/2)\Big\}
			\\&\label{eq:gothic_G_V}= \sum_{k=0}^{m}{m \choose k}(-u)^{m-k} \cG_{BC}^\epsilon Q^\epsilon_k(q)-m \widetilde{V}^\epsilon_{m-1}(q, u)u(\mu_\epsilon-u/2),
		\end{align}
		where 
		\[
		\cG_{BC}^\epsilon Q^\epsilon_k(q)\coloneqq \epsilon^{k\kappa}q\sum_{\ell\geq 1}p_\ell^\epsilon\left[Q_k(q+\ell)-Q_k(q)\right]-k\epsilon^{(k+1)\kappa} \frac{q(q-1)}{2}Q_{k-1}(q-1).
		\]
		Then 
		\begin{equation}\label{eq:G_eps_simple_notation}
		\widetilde{\cG}^\epsilon_{BC}V^\epsilon(\xonly{z}, x_0;\rho_t^\epsilon)=\widetilde{\cG}_{BC}^\epsilon\widetilde{V}^\epsilon_m(q, u),
		\end{equation}
		and we aim to compute the right-hand side of \eqref{eq:gothic_G_V} explicitly. We begin by computing the contribution of $\cG_{BC}^\epsilon Q^\epsilon_k(q)$ to \eqref{eq:gothic_G_V}, starting with the birth part. By telescoping and interchanging the order of summation, we compute 
		\begin{align*}
			&q\sum_{\ell\geq 1}p_\ell^\epsilon\left[Q_k(q+\ell)-Q_k(q)\right]
			\\&=q \sum_{\ell=1}^{L(\epsilon)}\left(\sum_{j=\ell}^{L(\epsilon)}p_j^\epsilon\right)[Q_k(q+\ell)-Q_k(q+\ell-1)]
			\\&= Q_1(q) \sum_{\ell=1}^{L(\epsilon)}\left(\sum_{j=\ell}^{L(\epsilon)}p_j^\epsilon\right) kQ_{k-1}(q+\ell-1),
		\end{align*}
		where in the last step, we have written $q=Q_1(q)$. Applying the binomial expansion for falling factorials, we have
		\[
		Q_{k-1}(q+\ell-1) = \sum_{i=0}^{k-1}{k-1 \choose i}Q_i(\ell-1)Q_{k-1-i}(q).
		\]
		Thus 
		\[
		q\sum_{\ell\geq 1}p_\ell^\epsilon\left[Q_k(q+\ell)-Q_k(q)\right] = k \sum_{\ell=1}^{L(\epsilon)}\left(\sum_{j=\ell}^{L(\epsilon)}p_j^\epsilon\right) \sum_{i=0}^{k-1}{k-1 \choose i}Q_i(\ell-1)Q_1(q)Q_{k-1-i}(q).
		\]
		Using \eqref{eq:product_Q}, we write
		\[
		Q_1(q)Q_{k-1-i}(q) = \sum_{r=0}^{\min(1, k-1-i)} {1 \choose r}Q_r(k-1)Q_{k-i-r}(q).
		\]
		Then
		\begin{align*}
			q\sum_{\ell\geq 1}p_\ell^\epsilon\left[Q_k(q+\ell)-Q_k(q)\right] &= k \sum_{\ell=1}^{L(\epsilon)}\left(\sum_{j=\ell}^{L(\epsilon)}p_j^\epsilon\right) \sum_{i=0}^{k-1}{k-1 \choose i}Q_i(\ell-1)
			\\&\quad\quad\quad\quad\quad\quad\times\sum_{r=0}^{\min(1, k-1-i)} Q_r(k-1)Q_{k-i-r}(q).
		\end{align*}
		Of course, the inner-most summation is just ${Q_{k-i}(q)+(k-1)Q_{k-i-1}(q)}$, but keeping the summation notation will simplify the presentation later in the proof. We turn to the competition term of $\cG_{BC}^\epsilon Q^\epsilon_k(q)$, still with the end goal of computing \eqref{eq:gothic_G_V}. Using \eqref{eq:product_Q} and then $q-1= Q_1(q-1)$, we obtain
		\begin{align*}
			k \frac{q(q-1)}{2}Q_{k-1}(q-1)&=\frac{k}{2} q \sum_{j=0}^{\min(1, k-1)}{1 \choose j}Q_j(k-1)Q_{k-j}(q-1)
			\\&=\frac{k}{2} \sum_{j=0}^{\min(1, k-1)}Q_j(k-1)Q_{k+1-j}(q).
		\end{align*}
		We apply $\cL$, the change-of-basis map introduced in the proof of Lemma \ref{lem:V_fcn_properties_LLN}, to the birth and competition components of \eqref{eq:gothic_G_V} just computed. In both cases, we use that
		\[
		\cL(Q^\epsilon_k(q)) = (\epsilon^\kappa q)^k = \sum_{a=0}^k {k \choose a}u^a(\epsilon^\kappa q-u)^{k-a},
		\] 
		for all $k \in \N$. For births, with $m \in \N$ and $q \in \N_0$, we get
		\begin{align*}
			B&\coloneqq \cL\Big( \sum_{k=0}^m {m \choose k} (-u)^{m-k}k \sum_{\ell=1}^{L(\epsilon)}\left(\sum_{j=\ell}^{L(\epsilon)}p_j^\epsilon\right) \sum_{i=0}^{k-1}{k-1 \choose i}Q_i^\epsilon(\ell-1)\sum_{r=0}^{\min(1, k-1-i)} Q_r^\epsilon(k-1)Q_{k-i-r}^\epsilon(q)\Big)
			\\&= \sum_{k=0}^m {m \choose k} (-u)^{m-k}k \sum_{\ell=1}^{L(\epsilon)}\left(\sum_{j=\ell}^{L(\epsilon)}p_j^\epsilon\right) \sum_{i=0}^{k-1}{k-1 \choose i}Q_i^\epsilon(\ell-1)
			\\&\quad\quad\quad\quad\quad\quad\quad\quad\quad\quad\quad\quad\times\sum_{r=0}^{\min(1, k-1-i)} Q_r^\epsilon(k-1)\sum_{a=0}^{k-i-r} {k-i-r \choose a}u^a(\epsilon^\kappa q - u)^{k-i-r-a}
			\\& = \sum_{h=-m}^1 c_+^\epsilon(h; m, u)(\epsilon^\kappa q-u)^{m+h},
		\end{align*}
		where 
		\begin{align*}
			c_+^\epsilon(h; m, u) &\coloneqq \sum_{k=0}^m {m \choose k} (-u)^{m-k}k \sum_{\ell=1}^{L(\epsilon)}\left(\sum_{j=\ell}^{L(\epsilon)}p_j^\epsilon\right) \sum_{i=0}^{k-1}{k-1 \choose i}Q_i^\epsilon(\ell-1)
			\\&\quad\quad\quad\quad\quad\quad\quad\quad\quad\quad\quad\quad\times\sum_{r=0}^{\min(1, k-1-i)} {1 \choose r}Q_r^\epsilon(k-1)\sum_{a=0}^{k-i-r} {k-i-r \choose a}u^a\delta_{m+h, k-i-r-a},
		\end{align*}
		with $\delta_{i, j}$ denoting the Kronecker delta. We note that the summation over $\ell$ remains finite uniformly in $\epsilon$ since $\epsilon^\kappa L(\epsilon)=o(1)$ as $\epsilon \searrow 0$ by assumption, and the offspring distribution has finite mean. For the competition, for all $m \in \N$ and $q \in \N_0$, we have 
		\begin{align*}
			C&\coloneqq \cL\left(\sum_{k=0}^m {m \choose k} (-u)^{m-k}\frac{k}{2} \sum_{j=0}^{\min(1, k-1)}{1 \choose j}Q_j^\epsilon(k-1)Q_{k+1-j}^\epsilon(q)\right)
			\\&= \sum_{k=0}^m {m \choose k} (-u)^{m-k}\frac{k}{2} \sum_{j=0}^{\min(1, k-1)}{1 \choose j}Q_j^\epsilon(k-1)\sum_{a=0}^{k+1-j}{k+1-j \choose a}u^a(\epsilon^\kappa q-u)^{k+1-j-a}
			\\& = \sum_{h=0}^1 c_-^\epsilon(h; m, u)(\epsilon^\kappa q-u)^{m+h},
		\end{align*}
		where
		\begin{align*}
			c_-^\epsilon(h; m, u) &= \sum_{k=0}^m {m \choose k} (-u)^{m-k}\frac{k}{2} \sum_{j=0}^{\min(1, k-1)}Q_j^\epsilon(k-1)\sum_{a=0}^{k+1-j}{k+1-j \choose a}u^a\delta_{m+h, k+1-j-a}.
		\end{align*}
		Recall that $\cL(\widetilde{V}^\epsilon_k(q, u)) = (\epsilon^\kappa q-u)^{k}$ for all $k, q \in \N_0$. Hence, applying the inverse map $\cL^{-1}$ to $B-C$, we calculate \eqref{eq:gothic_G_V}:
		\begin{equation}\label{eq:goth_V_expression}
			\widetilde{\cG}_{BC}^\epsilon \widetilde{V}^\epsilon_m(q, u) = \sum_{h=-m}^1c^\epsilon_{h}(m, u) \widetilde{V}_{m+h}^\epsilon(q, u),
		\end{equation}
		where
		\begin{equation}\label{def:BBGKY_coef}
			c^\epsilon_{h}(m, u) = c_+^\epsilon(h; m, u)-c_-^\epsilon(h; m, u)-\delta_{h, -1}m u(\mu_\epsilon-u/2).
		\end{equation}
		Combined with \eqref{eq:reduce_pb_via_localisation}, \eqref{eq:G_eps_simple_notation}, and point 1 of Lemma \ref{lem:V_fcn_properties_LLN}, this implies \eqref{eq:exp_G_BC}. It remains to verify the bounds claimed in the statement of Proposition \ref{prop:v_fcn_eq} on the coefficients $c_h^\epsilon$ defined in \eqref{def:BBGKY_coef}. Fix $n \in \N$. It is clear from \eqref{def:BBGKY_coef} that, for some $c=c(n)>0$, we have $|c^\epsilon_h(m, u)|\leq c \max(1, |u|^{m+1})$ for all $m \in \{1, \cdots, n\}$, $u\in \R$, and $h \in \{-m, \cdots, 1\}$. Suppose now that $m=2$ and $h=-2$. We compute the coefficients \eqref{def:BBGKY_coef} explicitly and get
		\[
		c^\epsilon_{-2}(2, u) =2\epsilon^\kappa \left[1+u\times \bigg(\mu_\epsilon+\sum_{\ell=1}^{L(\epsilon)}(\ell-1) \sum_{j=\ell}^{L(\epsilon)} p_j^\epsilon\bigg)+3u^3/2\right].
		\]
		Observe that by interchanging the order of summation,
		\[
		\sum_{\ell=1}^{L(\epsilon)}(\ell-1) \sum_{j=\ell}^{L(\epsilon)} p_j^\epsilon = \sum_{j=1}^{L(\epsilon)} p_j^\epsilon\sum_{\ell=1}^j (\ell-1) =\frac{1}{2}\sum_{j=1}^{L(\epsilon)}j(j-1)p_j^\epsilon.
		\]
		Hence for $p \in [1, 2]$, 
		\[
		\epsilon^\kappa \sum_{\ell=1}^{L(\epsilon)}(\ell-1) \sum_{j=\ell}^{L(\epsilon)} p_j^\epsilon \leq \epsilon^{(p-1)\kappa} (\epsilon^\kappa L(\epsilon))^{2-p} \sum_{j=1}^{L(\epsilon)} j^p p_j^\epsilon. 
		\]
		If the offspring distribution has a moment of order $p$, then $\sup_{\epsilon \in (0, 1)}\sum_{j=1}^{L(\epsilon)} j^p p_j^\epsilon<\infty$. Moreover, by assumption, we have $(\epsilon^\kappa L(\epsilon))^{2-p} = o(1)$ as $\epsilon \searrow 0$. So
		\[
		\epsilon^{\kappa(p-1)} (\epsilon^\kappa L(\epsilon))^{2-p} \sum_{j=1}^{L(\epsilon)} j^p p_j^\epsilon = o(\epsilon^{(p-1)\kappa}),
		\]
		and there exists $c>0$ such that 
		\[
		|c_{-2}^\epsilon(2, u)| \leq c \epsilon^\kappa\max(1, |u|^3)+o(\epsilon^{(p-1)\kappa})|u|,
		\]
		as $\epsilon \searrow 0$. Similarly, if $m=2$ and $h=-1$, we compute from \eqref{def:BBGKY_coef} that
		\[
		c_{-1}^\epsilon(2,u) = \epsilon^\kappa\left[2\mu_\epsilon+1+2\sum_{\ell=1}^{L(\epsilon)}(\ell-1) \sum_{j=\ell}^{L(\epsilon)} p_j^\epsilon\right]=(2\mu_\epsilon+1)\epsilon^\kappa+o(\epsilon^{(p-1)\kappa}).
		\]
		Finally, if $m=1$ and $h\leq-1$, it is immediate from \eqref{def:BBGKY_coef} that $c_h^\epsilon(1, u)=0$ for all $u$. 
	\end{proof}

	\section{Auxiliary lemmas}\label{append:aux_lems}
	
	In this appendix, we gather technical calculations on Green's function, convergence rates of certain Riemann sums, the preservation of the martingale property under limits and changes of variables, and stable processes. 
	
	\subsection{Estimates for the semi-discrete Green's function}\label{app:Green_fcn_props}
	
	In this section, we gather auxiliary technical calculations involving the Green's function $G^\epsilon_t(z_1, z_2)$ of the diffusively rescaled spatial motion on the  discrete circle $\Z_\epsilon$. Given $t\geq 0$ and $z_1, z_2\in \Z_\epsilon$, $G^\epsilon_t(z_1, z_2)$ is the probability that a particle performing a simple symmetric random walk on $\Z_\epsilon$ and located at $z_1$ at time $0$ is found at $z_2$ at time $\epsilon^{-2}t$. Thus, 
	\begin{equation}\label{eq:Green_eq}
		\begin{cases}
			\partial_t G^\epsilon_t(z_1, z_2) = \frac{1}{2}\Delta_{z_1}^\epsilon G^\epsilon_t(z_1, z_2), & z_1, z_2 \in \Z_\epsilon, \quad t>0,\\
			G_0^\epsilon(z_1, z_2) = \delta_{z_1, z_2}, & z_1, z_2 \in \Z_\epsilon,
		\end{cases}
	\end{equation}
	where we define $\Delta_z^\epsilon f(z) = \epsilon^{-2}(f(z+1)+f(z-1)-2f(z))$ if $f:\Z_\epsilon\to \R$ or $f:\Z\to \R$ by abuse of notation.
	We begin by proving some estimates on $G^\epsilon$. 
	
	\begin{lem}\emph{(Green function estimates)}\label{lem:Green_estimates}
		Fix a time horizon $T>0$. The following estimates hold.
		\begin{enumerate}
			\item[1.] There exists $C=C(T)>0$ such that 
			\begin{equation}\label{ineq:basic_heat_estimate}
			G^\epsilon_t(z_1, z_2) \leq C \min(1, \epsilon t^{-1/2}),
			\end{equation}
			for all $z_1, z_2 \in \Z_\epsilon$, $t \in (0, T]$, and $\epsilon \in (0, 1)$.
			\item[2.] There exists $C=C(T)>0$ such that 
			\begin{equation}\label{ineq:heat_time_holder_cts}
			\left|G_t^\epsilon(z_1, z_2)-G_s^\epsilon(z_1, z_2)\right| \leq C\epsilon\left[s^{-a-1/2}(t-s)^a +s^{-1/2}\Big((t-s)+(t-s)^{1/2}\Big)\right],
			\end{equation}
			for all $z_1, z_2 \in \Z_\epsilon$, $0< s \leq t\leq T$, $a \in [0, 1/2)$, and $\epsilon \in (0, 1)$.
			\item[3.] There exists a constant $C>0$ such that
			\begin{equation}\label{ineq:exp_control_G}
				\sup_{t\geq 0}\sup_{\epsilon\in (0, 1)}\sum_{z_2 \in \Z_\epsilon} G_t^\epsilon(z_1, z_2) e^{|z_2-z_1|/\max(t^{1/2}, \epsilon)}<C,
			\end{equation}
			for all $z_1 \in \Z_\epsilon$.
			\item[4.] There exists $C=C(T)>0$ such that 
			\begin{equation}\label{ineq:G_exp_bd}
			G^\epsilon_t(0, z)\leq C\min(1, \epsilon t^{-1/2})e^{-z\min(1, \epsilon t^{-1/2})},
			\end{equation}
			for all $t \in [0, T]$, $z \in \Z_\epsilon$, and $\epsilon \in (0, 1)$. 
		\end{enumerate}
	\end{lem}
	\begin{proof}
		Consider $\overline{G}_t^\epsilon(z)$ the probability that a continuous-time simple symmetric random walk on $\Z$ is at site $z$ at time $\epsilon^{-2}t$. Then $\overline{G}^\epsilon$ solves 
		\[
		\begin{cases}
				\partial_t \overline{G}_t^\epsilon(z) = \frac{1}{2}\Delta^\epsilon_z \overline{G}_t^\epsilon(z)& z \in \Z, \quad t\geq 0,\\
				\overline{G}_0^\epsilon(z) = \delta_{z, 0}& z \in \Z.
		\end{cases}
		\]
		Recall that $K_\epsilon=|\Z_\epsilon|=\lfloor\epsilon^{-1}\rfloor$. The semi-discrete heat kernel $G^\epsilon$ on $\Z_\epsilon$ is the $K_\epsilon$-periodisation of $\overline{G}^\epsilon$, that is
		\begin{equation}\label{eq:G_G_bar}
		G^\epsilon(\pi^\epsilon(z), \pi^\epsilon(z')) = \sum_{k \in \Z} \overline{G}^\epsilon(z'-z+kK_\epsilon), \quad z, z' \in \Z.
		\end{equation}
		where $\pi^\epsilon:\Z\to \Z_\epsilon$, $\pi^\epsilon(z) = z \mod K_\epsilon$, is the canonical projection map. Recall from the equation displayed below (2.12) in \cite{LL2010} that the Fourier inversion formula implies
		\[
		\overline{G}_t^\epsilon(z) = \frac{1}{2\pi} \int_{-\pi}^\pi e^{-iz\theta} e^{-\epsilon^{-2}t(1-\cos(\theta))}d\theta, \quad z \in \Z, \quad t\geq 0.
		\]
		Following the proof of Proposition A.1 in \cite{DT2015}, we use this expression to derive another integral representation of $\overline{G}^\epsilon$ from which the estimates \eqref{ineq:basic_heat_estimate} and \eqref{ineq:heat_time_holder_cts} follow. Let $\xi=e^{i\theta}$, so $d\xi=i\xi d\theta$, and 
		\[
		\overline{G}_t^\epsilon(z)= \frac{1}{2\pi i}\oint_{\Gamma_1}\xi^{-z-1}e^{-\epsilon^{-2}t\psi(\xi)}d\xi,
		\]
		where $\Gamma_1=\{\xi \in \C:|\xi|=1\}$, and $\psi(\xi)=1-(e^{i\theta}+e^{-i\theta})/2=1-(\xi+\xi^{-1})/2$. The map $\xi \mapsto \xi^{-z-1}$ is holomorphic on $\C\backslash \{0\}$, so we may continuously deform the contour to ${\Gamma_2=\{\xi\in \C: |\xi|=1+\delta\}}$, where we define $\delta=e^{\operatorname{sign}(z)\min(1, \epsilon t^{-1/2})}-1 >-1$. We then reverse the original change of variables by defining $\theta \in [-\pi, \pi]$ by $(1+\delta)e^{i\theta}=\xi$. Thus $d\xi=i\xi d\theta$, and 
		\begin{align}
		\overline{G}_t^\epsilon(z) &=\notag \frac{1}{2\pi} \int_{-\pi}^\pi (1+\delta)^{-z} e^{-iz \theta} e^{-\epsilon^{-2}t\psi((1+\delta)e^{i\theta})} d\theta
		\\&= \frac{1}{2\pi} \int_{-\pi}^\pi e^{-|z|\min(1, \epsilon t^{-1/2})} e^{-iz \theta} e^{-\epsilon^{-2}t\psi((1+\delta)e^{i\theta})} d\theta.\label{eq:integr_rep_G}
		\end{align}
		It is easy to compute
		\[
		\psi((1+\delta)e^{i\theta})=1-\cos(\theta)-\frac{\delta^2}{2(1+\delta)}-i\frac{\delta}{1+\delta}\sin(\theta). 
		\] 
		Using the Taylor expansion of $\cos(\theta)$ up to order one and two, we obtain 
		\[
		{\theta^2/2-\theta^4/24\leq 1-\cos(\theta)\leq \theta^2/2}, \quad \theta \in [-\pi, \pi].
		\]
		In particular, we have 
		\begin{equation}\label{ineq:cos_bd}
		\frac{\theta^2}{12}\leq \theta^2\left(\frac{1}{2}-\frac{\pi^2}{24}\right) \leq 1-\cos(\theta)\leq \frac{\theta^2}{2}, \quad \theta \in [-\pi, \pi].
		\end{equation}
		Fix any $a\in [0, 1]$. For all $y \in [-1, 0)$, we have $|1-e^{-y}|=|y|\left|\sum_{k=0}^\infty (-1)^k\frac{y^k}{(k+1)!}\right|$ and ${\sum_{k=0}^\infty (-y)^k/(k+1)!<e}$, so $|1-e^{-y}|<e|y|<e|y|^a$. For $y \in [0, 1]$, we have by a Taylor expansion of order two $1-e^{-y}=y-\xi^2/2<y<y^a$, for some $\xi \in (0, y)$. Finally if $y>1$, then $|1-e^{-y}|\leq 1-e^{-1}<1<y^a$. Thus
		\begin{equation}\label{ineq:exp_bd}
		|1-e^{-y}|\leq e\min(1, |y|^a), \quad y \geq -1.
		\end{equation}
		It follows that $\epsilon^{-2}t\delta^2 \leq e$. Since also $1+\delta=e^{\operatorname{sign}(z)\min(1, \epsilon t^{-1/2})}$, we find the upper bound 
		\[
		|e^{-\epsilon^{-2}t\psi((1+\delta)e^{i\theta})}|\leq e^{-\epsilon^{-2}t \theta^2/12} e^{e^{-\operatorname{sign}(z)\min(1, \epsilon t^{-1/2})}/2} \leq e^{e/2}e^{-\epsilon^{-2}t \theta^2/12}.
		\]
		By combining the above estimates and applying the triangle inequality, we obtain
		\[
			\overline{G}^\epsilon_t(z)=|\overline{G}^\epsilon_t(z)|\leq e^{e/2} e^{-|z|\min(1,\epsilon t^{-1/2})} \int_{-\pi}^\pi e^{-\epsilon^{-2}t \theta^2/12}d\theta.
		\]
		We write 
		\[
		\epsilon^{-2}t=\epsilon^{-2}t-\min(1, \epsilon t^{-1/2})^{-2}+\min(1, \epsilon t^{-1/2})^{-2}.
		\]
		For $t\geq\epsilon^2$, we have $\epsilon^{-2}t-\min(1, \epsilon t^{-1/2})^{-2}=0$, while for $t<\epsilon^2$, we get a lower bound 
		\[
		\epsilon^{-2}t-\min(1, \epsilon t^{-1/2})^{-2}\geq -1.
		\]
		Thus
		\begin{align*}
			\overline{G}^\epsilon_t(z)& \leq C e^{-|z|\min(1, \epsilon t^{-1/2})} \min(1, \epsilon t^{-1/2})\int_{-\pi}^{\pi} \min(1, \epsilon t^{-1/2})^{-1}6^{-1/2}e^{-\frac{\theta^2}{12\min(1, \epsilon t^{-1/2})^2}}d\theta
			\\&\leq Ce^{-|z|\min(1, \epsilon t^{-1/2})} \min(1, \epsilon t^{-1/2}),
		\end{align*}
		where $C=e^{1+e/2}\sqrt{6}$. Given $z_1, z_2 \in \Z_\epsilon$, let $\widetilde{z}_i\in \Z\cap [0, K_\epsilon-1]$ be such that $\pi^\epsilon(\widetilde z_i)=z_i$ for $i=1, 2$. Then, by \eqref{eq:G_G_bar}, we obtain 
		\begin{equation}\label{ineq:circle_green_exp_bd}
		G_t^\epsilon(z_1, z_2)\leq C \min(1, \epsilon t^{-1/2})\sum_{k \in \Z} e^{-|\widetilde{z}_2-\widetilde{z}_1+kK_\epsilon|\min(1, \epsilon t^{-1/2})}.
		\end{equation}
	    To bound the summation, observe that $(\widetilde z_2-\widetilde z_1)/K_\epsilon \in (-1, 1)$. Hence 
	    \begin{align*}
	    e^{-|\widetilde{z}_2-\widetilde{z}_1+kK_\epsilon|\min(1, \epsilon t^{-1/2})}&=e^{-|(\widetilde{z}_2-\widetilde{z}_1)/K_\epsilon+k|K_\epsilon\min(1, \epsilon t^{-1/2})}
	    \\&\leq e^{-|\widetilde{z}_2-\widetilde{z}_1|\min(1, \epsilon t^{-1/2})}\begin{cases}1&k=0,\\e^{-(|k|-1)K_\epsilon\min(1, \epsilon t^{-1/2})}&|k|\geq 1.\end{cases}
	    \end{align*}
	    and we find 
	    \begin{equation}\label{ineq:z_to_circle_1}
	    \sum_{k \in \Z} e^{-|\widetilde{z}_2-\widetilde{z}_1+kK_\epsilon|\min(1, \epsilon t^{-1/2})}\leq e^{-|\widetilde{z}_2-\widetilde{z}_1|\min(1, \epsilon t^{-1/2})}\left(1 + 2 \sum_{k\geq 1} e^{-(k-1)K_\epsilon \min(1, \epsilon t^{-1/2})}\right).
	    \end{equation}
	    Since $K_\epsilon\epsilon t^{-1/2}\geq T^{-1/2}/2>0$ for all $\epsilon \in (0, 1)$ and $t \in [0, T]$, it follows that
	    \begin{equation}\label{ineq:z_to_circle_2}
	    \sup_{\epsilon \in (0, 1)}\sup_{t \in [0, T]}\sum_{k\geq 1} e^{-(k-1)K_\epsilon \min(1, \epsilon t^{-1/2})} < C',
	    \end{equation}
	    for some $C'=C'(T)>0$. Gathering the above estimates, we have shown \eqref{ineq:basic_heat_estimate} with constant $C\times (3+2C')$ which is independent of $z_1$, $z_2$, $\epsilon$, $t$ and $T$. This proves \eqref{ineq:basic_heat_estimate}. Next, we use \eqref{eq:integr_rep_G}	to show \eqref{ineq:heat_time_holder_cts}. We have 
	    \begin{align*}
	    	G_t^\epsilon(z_1, z_2)-G_s^\epsilon(z_1, z_2)&= G_t^\epsilon(\pi^\epsilon(\widetilde{z}_1), \pi^\epsilon(\widetilde{z}_2))-G_s^\epsilon(\pi^\epsilon(\widetilde{z}_1), \pi^\epsilon(\widetilde{z}_2))
	    	\\&=\sum_{k \in \Z}[\overline{G}_t^\epsilon(\widetilde{z}_2-\widetilde{z}_1+k K_\epsilon)-\overline{G}_s^\epsilon(\widetilde{z}_2-\widetilde{z}_1+k K_\epsilon)].
	    \end{align*}
	    Therefore, we need to establish the required control on the right-hand side. We define $z=\widetilde{z}_2-\widetilde{z}_1$. By \eqref{eq:integr_rep_G} with $\delta=e^{\operatorname{sign}(z+kK_\epsilon)\epsilon\min(1, e^{-3/2}T^{-1/2})}-1$, we have
	    \begin{align}
	    &\notag\overline{G}_t^\epsilon(z+k K_\epsilon)-\overline{G}_s^\epsilon(z+k K_\epsilon)
	    \\&=\frac{1}{2\pi}e^{-|z+k K_\epsilon|\epsilon\min(1, e^{-3/2}T^{-1/2})}\int_{-\pi}^\pi e^{-iz\theta} e^{-\epsilon^{-2}s \psi((1+\delta)e^{i\theta})} (e^{-\epsilon^{-2}(t-s)\psi((1+\delta)e^{i\theta})}-1)d\theta.\label{eq:diff_G_int}
	    \end{align}
	    We make the change of variable $u=\epsilon^{-1}s^{1/2}\theta$, which gives
	    \begin{align}
	    	&\notag\overline{G}_t^\epsilon(z+k K_\epsilon)-\overline{G}_s^\epsilon(z+k K_\epsilon)
	    	\\&=\frac{\epsilon}{2\pi\sqrt{s}}e^{-|z+k K_\epsilon|\epsilon\min(1, e^{-3/2}T^{-1/2})}\int_{-\pi\epsilon^{-1}s^{1/2}}^{\pi\epsilon^{-1}s^{1/2}} e^{-iz\epsilon s^{-1/2}u} e^{-\epsilon^{-2}s \psi((1+\delta)e^{i\epsilon s^{-1/2}u})} \label{eq:diff_G_int}
	    	\\&\quad\quad\quad\quad\quad\quad\quad\quad\quad\quad\quad\quad\quad\quad\quad\quad\quad\quad\quad\times(e^{-\epsilon^{-2}(t-s)\psi((1+\delta)e^{i\epsilon s^{-1/2}u})}-1)du.\notag
	    \end{align}
	    Write $\psi((1+\delta)e^{i\epsilon s^{-1/2}u})=A_u+iB_u$, where $A_u=1-\cos(\epsilon s^{-1/2}u)-\delta^2/(1+\delta)$ and $B_u=-\delta \sin(\epsilon s^{-1/2}u)/(1+\delta)$ are real numbers. Then, a simple calculation implies that 
	    \[
	    \left|e^{-\epsilon^{-2}(t-s)(A_u+iB_u)}-1\right|\leq \left|e^{-\epsilon^{-2}(t-s)A_u}-1\right|+e^{-\epsilon^{-2}(t-s)A_u}\left|e^{-\epsilon^{-2}(t-s)B_u i}-1\right|.
	    \]
	    The first term satisfies
	    \[
	    \left|e^{-\epsilon^{-2}(t-s)A_u}-1\right|\leq e^{\epsilon^{-2}(t-s)\delta^2/(1+\delta)}\left(e^{-\epsilon^{-2}(t-s)(1-\cos(\epsilon s^{-1/2}u))}-1\right)+e^{\epsilon^{-2}(t-s)\delta^2/(1+\delta)}-1.
	    \]
	    By definition of $\delta$ and using \eqref{ineq:exp_bd} with $a=1$, we have $\epsilon^{-2}(t-s)\delta^2 \leq (t-s)e^{-1}T^{-2}\leq e^{-1}$. Using next \eqref{ineq:exp_bd} twice for $a \in (0, 1/2)$, \eqref{ineq:cos_bd}, and the observation that $(1+\delta)^{-1}\leq e$, we find
	    \[
	    \left|e^{-\epsilon^{-2}(t-s)A_u}-1\right|\leq e^2(t-s)^a s^{-a}|u|^{2a}+(t-s)T^{-1}e^{-1}.
	    \]
	    We control the second term using the facts that $|\sin(\epsilon s^{-1/2}u)|\leq |\epsilon s^{-1/2} u|$ and $\delta \leq e\epsilon$. Indeed, these observations imply that $|B_u|\leq e\epsilon^2 |u|$, and hence
	    \[
	    \left|e^{-\epsilon^{-2}(t-s)B_u i}-1\right|\leq \int_0^{\epsilon^{-2}(t-s)|B_u|}|ie^{-iy}|dy\leq \epsilon^{-2}(t-s)|B_u|\leq (t-s)s^{-1/2}e|u|.
	    \]
	    Gathering the last estimates shows that
	    \begin{align*}
	    \left|e^{-\epsilon^{-2}(t-s)(A_u+iB_u)}-1\right|&\leq e^2(t-s)^a s^{-a}|u|^{2a}+(t-s)(T^{-1}+e^2|u|e^{-(t-s)s^{-1}u^2/12}s^{-1/2})e.
	    \end{align*}
	    We denote the right-hand side by $h_{\epsilon, a}(s, t, u)$, and we apply the last few bounds to control the absolute value of \eqref{eq:diff_G_int}. We get
	    \begin{align*}
	    &|\overline{G}_t^\epsilon(z+k K_\epsilon)-\overline{G}_s^\epsilon(z+k K_\epsilon)|
	    \\&\leq e^{-|z+k K_\epsilon|\epsilon\min(1, e^{-3/2}T^{-1/2})}\frac{\epsilon}{2\pi\sqrt{s}}\int_{-\pi \epsilon^{-1}s^{1/2}}^{\pi \epsilon^{-1}s^{1/2}} e^{-u^2/12} h_{\epsilon, a}(s, t, u)du,
	    \end{align*}
	    for a $\in (0, 1/2)$. The integral can be controlled by recognising a Gaussian density in the integrand as follows:
	    \begin{align*}
	    	&\frac{\epsilon}{2\pi\sqrt{s}}\int_{-\pi \epsilon^{-1}s^{1/2}}^{\pi \epsilon^{-1}s^{1/2}} e^{-u^2/12} h_{\epsilon, a}(s, t, u)du
	    	\\&\leq \epsilon e^2(t-s)^as^{-a-1/2}\int_{-\pi \epsilon^{-1}s^{1/2}}^{\pi \epsilon^{-1}s^{1/2}} |u|^{2a}e^{-u^2/12} du+ \epsilon (t-s)s^{-1/2}T^{-1}e \int_{-\pi \epsilon^{-1}s^{1/2}}^{\pi \epsilon^{-1}s^{1/2}} e^{-u^2/12} du
	    	\\&+ \epsilon e^3(t-s) \int_{-\pi \epsilon^{-1}s^{1/2}}^{\pi \epsilon^{-1}s^{1/2}} |u|e^{-u^2/12}s^{-1}e^{-(t-s)s^{-1}u^2/12} du.
	    \end{align*}
	    The first two integrals on the right-hand side are easily controlled by a Gaussian integral:
	    \begin{align*}
	    \max\left(e^2\int_{-\pi \epsilon^{-1}s^{1/2}}^{\pi \epsilon^{-1}s^{1/2}} |u|^{2a}e^{-u^2/12} du, T^{-1}e \int_{-\pi \epsilon^{-1}s^{1/2}}^{\pi \epsilon^{-1}s^{1/2}} e^{-u^2/12} du\right)\leq C,
	    \end{align*}
	    where
	    \[
	    C=C(T)=(e+T^{-1})e\int_\R e^{-u^2/12}\max(1, |u|^{2a})du\in (0, \infty).
	    \]
	    For the last integral, we simplify the exponents and use even symmetry to show that 
	    \begin{align*}
	    	\int_{-\pi \epsilon^{-1}s^{1/2}}^{\pi \epsilon^{-1}s^{1/2}} |u|e^{-u^2/12}s^{-1}e^{-(t-s)s^{-1}u^2/12} du&\leq 2 \int_0^{\pi \epsilon^{-1}s^{1/2}} s^{-1}u e^{-u^2 t s^{-1}/12}du 
	    	\\&=6t^{-1}(1-e^{-\pi^2 \epsilon^{-2}t/12}).
	    \end{align*}
	    We note that 
	    \[
	    6t^{-1}(1-e^{-\pi^2 \epsilon^{-2}t/12})\leq 6(s/t)^{1/2}(ts)^{-1/2}\leq 6 ((t-s)s+s^2)^{-1/2}\leq 6(t-s)^{-1/2}s^{-1/2}.
	    \]
	    since $s/t\leq 1$. Therefore, 
	    \[
	    e^3(t-s)\int_{-\pi \epsilon^{-1}s^{1/2}}^{\pi \epsilon^{-1}s^{1/2}} |u|e^{-u^2/12}s^{-1}e^{-(t-s)s^{-1}u^2/12} du\leq 6e^3 (t-s)^{1/2}s^{-1/2}.
	    \]
	    We have shown that there exists a constant $C'=C'(T)>0$ such that
	    \begin{align*}
	    |\overline{G}_t^\epsilon(z+k K_\epsilon)-\overline{G}_s^\epsilon(z+k K_\epsilon)|&\leq C'\epsilon e^{-|z+k K_\epsilon|\epsilon\min(1, e^{-3/2}T^{-1/2})}
	    \\&\times \left[(t-s)^a s^{-a-1/2}+s^{-1/2}\Big((t-s)+(t-s)^{1/2}\Big)\right].
	    \end{align*}
	  	To obtain \eqref{ineq:heat_time_holder_cts}, it only remains to show that
	  	\begin{equation}\label{ineq:double_sum_control}
	  	\sup_{\epsilon \in (0, 1)} \sup_{z \in \Z_\epsilon}\sum_{k \in \Z}e^{-|z+k K_\epsilon|\epsilon\min(1, e^{-3/2}T^{-1/2})}<\infty.
	  	\end{equation}
	  	We write $|z+k K_\epsilon|=|z/K_\epsilon+k|K_\epsilon$, and note that $z/K_\epsilon \in (-1, 1)$. Thus, we may partition the summation over $k \in \Z$ as we did above in the proof of \eqref{ineq:basic_heat_estimate} into the cases $k=0$, $k\leq -1$ and $k\geq 1$ to obtain the bound
	  	\[
	  		\sum_{k \in \Z}e^{-|z+k K_\epsilon|\epsilon\min(1, e^{-3/2}T^{-1/2})}\leq 1+2\sum_{k\geq 1}e^{-(k-1)\epsilon K_\epsilon \min(1, e^{-3/2}T^{-1/2})}.
	  	\]
	  	Since $K_\epsilon=\lfloor \epsilon^{-1}\rfloor$, we have $\epsilon K_\epsilon=1+o(1)$. So
	  	\[
	  	\sum_{k \in \Z}e^{-|z+k K_\epsilon|\epsilon\min(1, e^{-3/2}T^{-1/2})}\leq 1+2e^2 \sum_{k\geq 0}e^{-(k-1)\min(1, e^{-3/2}T^{-1/2})}.
	  	\]
	  	The right-hand side is finite and only depends on $T$, which proves \eqref{ineq:double_sum_control}, and concludes the proof of \eqref{ineq:heat_time_holder_cts}. Next, we show \eqref{ineq:exp_control_G} by following an argument in Lemma 27 of \cite{EthLab2014} and in Proposition A.1 of \cite{DT2015}. Let $(S_\ell)_{\ell\geq 0}$ be a discrete-time simple symmetric random walk on $\Z_\epsilon$ started from $S_0=\pi^\epsilon(0)$. Let $R$ be a Rademacher random variable. Then 
	  	\begin{align*}
	  		\sum_{z_2 \in \Z_\epsilon} G_t^\epsilon(z_1, z_2) e^{|z_2-z_1|/\max(t^{1/2}, \epsilon)}&=e^{-\epsilon^{-2}t}\sum_{\ell=0}^\infty \frac{(\epsilon^{-2}t)^\ell}{\ell!}   \E\left[\exp\left(\frac{S_\ell}{\max(t^{1/2}, \epsilon)}\right)\right]
	  		\\&=e^{-\epsilon^{-2}t}\sum_{\ell=0}^\infty \frac{(\epsilon^{-2}t)^\ell}{\ell!}   \E\left[\exp\left(\frac{\epsilon R}{\max(t^{1/2}, \epsilon)}\right)\right]^\ell
	  		\\&=\exp\left(\epsilon^{-2}t\left[\frac{1}{2}\left(e^{-\epsilon/\max(t^{1/2}, \epsilon)}+e^{\epsilon/\max(t^{1/2}, \epsilon)}\right)-1\right]\right).
	  	\end{align*}
	  	By a Taylor expansion of order two, we obtain that there exists $\xi \in (0, \epsilon/\max(t^{1/2}, \epsilon))$ such that 
	  	\[
	  	e^{-\epsilon/\max(t^{1/2}, \epsilon)}+e^{\epsilon/\max(t^{1/2}, \epsilon)}=\frac{\epsilon^2}{2\max(t^{1/2}, \epsilon)^2}\left(e^{-\xi}+e^\xi\right).
	  	\]
	  	Since for a fixed $t\geq 0$, we have $\epsilon/\max(t^{1/2}, \epsilon)\in [0, 1]$ for all $\epsilon \in (0, 1)$, the exponentials on the right-hand side are bounded by $1+e$. Moreover, using also that $e<3$, we get 
	  	\[
	  	\sum_{z_2 \in \Z_\epsilon} G_t^\epsilon(z_1, z_2) e^{|z_2-z_1|/\max(t^{1/2}, \epsilon)}\leq \exp\left(\epsilon^{-2}t\left[\frac{\epsilon^2}{\max(t^{1/2}, \epsilon)^2}\frac{1+e}{4}-1\right]\right)\leq e^{(e-3)/4},
	  	\]
	  	as required. The fourth and last claim \eqref{ineq:G_exp_bd} follows from \eqref{ineq:circle_green_exp_bd}-\eqref{ineq:z_to_circle_2}.
	\end{proof}
	
	\subsection{Approximation to the FKPP equation}\label{append:disc_FKPP}
	In this section, we prove Lemma \ref{lem:FKPP_approx_properties} concerning the properties of the discrete-space McKean representation. Before the proof, we briefly introduce the McKean representation of the solution and some notation.
	
	Consider a diffusively rescaled dyadic branching random walk $\rZ^\epsilon = (Z^\epsilon_t)_{t\geq 0}$ on $\Z_\epsilon$ with branching rate $\mu_\epsilon$, where $\mu_\epsilon$ is the mean of the offspring distribution $\rP^\epsilon$. Thus, in between branching events, particles perform independent simple symmetric random walks on $\Z_\epsilon$ at rate $\epsilon^{-2}$. Each particle independently branches at rate $\mu_\epsilon$, and when it does, it splits into two independent particles. We denote by $\cN_t^\epsilon$ the set of particles alive at time $t$, and by $Z^\epsilon_u(t)$ the location of particle $u\in \cN_t^\epsilon$ on the lattice $\Z_\epsilon$. The time-$t$ configuration of the branching random walk is given by 
	\[
	Z_t^\epsilon = (Z_u^\epsilon(t): u \in \cN^\epsilon),
	\]
	and we assume that $Z_0^\epsilon$ consists of a single particle at $0 \in \Z_\epsilon$.  
	Consider a family of functions $g_\epsilon:\Z_\epsilon\to [0, 1]$, for $\epsilon \in (0, 1)$, and introduce by analogy with the McKean representation of \cite[Section 2]{M1975} the function
	\begin{equation}\label{def:McKean_fcnl}
		\rho_t^\epsilon(z) = 2\mu_\epsilon\left(1-\E\left[\prod_{u \in \cN^\epsilon_t}g_\epsilon\left(z+Z_u^\epsilon(t)\right)\right]\right), \quad t\geq 0, \quad z \in \Z_\epsilon. 
	\end{equation}
	It follows from the next lemma that the unique solution to \eqref{eq:interm_fkpp_diff_form} has the representation \eqref{def:McKean_fcnl}.
	
	\begin{lem}
		The function $\rho^\epsilon$ defined by \eqref{def:McKean_fcnl} solves the Cauchy problem \eqref{eq:interm_fkpp_diff_form}.
	\end{lem}
	\begin{proof}
		The first claim is a classical one-step calculation for branching processes as in Section 2 of \cite{M1975}. Define 
		\[
		\overline{\rho}_t^{\:\epsilon}(z) = \E\left[\prod_{u \in \cN_t^\epsilon} g_\epsilon\left(z+Z_u^\epsilon(t)\right)\right], \quad t\geq 0, \quad z \in \Z_\epsilon,
		\]
		and let $\tau$ denote the first branching time of the initial particle. Given $t>0$, we partition the sample space on $\{\tau>t\}$ and $\{\tau\leq t\}$. Thus, 
		\begin{align*}
			\overline{\rho}_t^\epsilon(z) &= \P(\tau>t)\E\left[\prod_{u \in \cN_t^\epsilon} g_\epsilon\left(z+Z_u^\epsilon(t)\right)\bigg| \tau>t\right]
			\\&\quad\quad\quad\quad\quad\quad\quad\quad\quad\quad\quad\quad+ \int_0^t \P(\tau \in ds) \E\left[\prod_{u \in \cN_t^\epsilon} g_\epsilon\left(z+Z_u^\epsilon(t)\right)\bigg| \tau\in ds\right]
			\\&=e^{-\mu_\epsilon t}\sum_{z' \in \Z_\epsilon} G_t^\epsilon(z, z')g_\epsilon(z')+\int_0^t ds \mu_\epsilon e^{-\mu_\epsilon s}\sum_{z' \in \Z_\epsilon}G_s^\epsilon(z, z')\E\left[\prod_{u \in \cN_{t-s}^\epsilon} g_\epsilon(z'+Z_u^\epsilon(t-s))\right]^2.
			\\&=e^{-\mu_\epsilon t}\sum_{z' \in \Z_\epsilon} G_t^\epsilon(z, z')g_\epsilon(z')+\int_0^t ds \mu_\epsilon e^{-\mu_\epsilon (t-s)}\sum_{z' \in \Z_\epsilon}G_{t-s}^\epsilon(z, z')\overline{\rho}_s^\epsilon(z')^2,
		\end{align*} 
		where $G^\epsilon_t(z, z')$, $z, z' \in \Z_\epsilon$, $t\geq 0$, $\epsilon \in (0, 1)$ denotes the probability that a single particle performing a simple symmetric random walk moves between $z$ and $z'$ in time $\epsilon^{-2}t$. 
		Taking the time derivative of this expression using Leibniz's rule for differentiation under the integral sign, we obtain 
		\[
		\partial_t \overline{\rho}_t^\epsilon(z) = \frac{1}{2}\Delta^\epsilon_z \overline{\rho}_t^\epsilon(z)+\mu_\epsilon\overline{\rho}_t^\epsilon(z)(\overline{\rho}_t^\epsilon(z)-1), \quad z \in \Z_\epsilon, \quad t\geq 0. 
		\]
		Consequently, 
		\[
		\partial_t \rho_t^\epsilon(z) = -2\mu_\epsilon\partial_t \overline{\rho}_t^\epsilon(z)=\frac{1}{2}\Delta^\epsilon_z \rho_t^\epsilon(z)+\mu_\epsilon\overline{\rho}_t^\epsilon(z)\rho_t^\epsilon(z)=\frac{1}{2}\Delta^\epsilon_z \rho_t^\epsilon(z)+\rho_t^\epsilon(z)(\mu_\epsilon-\rho_t^\epsilon(z)/2).
		\] 
	\end{proof}
	
	\begin{proof}[Proof of Lemma \ref{lem:FKPP_approx_properties}]
		
		For the first claim, using \eqref{def:McKean_fcnl}, we have
		\[
		\widetilde{\rho}_t^{\:\epsilon}(z) = 2\mu_\epsilon\left(1-\E\left[\prod_{u \in \cN^\epsilon_t}g_\epsilon\left(K_\epsilon \left[z+K_\epsilon^{-1}Z_u^\epsilon(t)\right]\right)\right]\right), \quad t\geq 0, \quad z \in \pi([0, 1]\cap K_\epsilon^{-1}\Z).
		\]
		It is easy to see using Donsker's invariance principle that our diffusively rescaled branching random walk $K_\epsilon^{-1} Z_t^\epsilon$ (where we recall that time is accelerated by $\epsilon^{-2}$ and $K_\epsilon = \lfloor\epsilon^{-1}\rfloor$) converges in distribution to $Z_t=(Z_u(t):u \in \cN_t)$, $t\geq 0$, a dyadic branching Brownian motion with branching rate $\mu$. Since $g_\epsilon(K_\epsilon \cdot )$ is $[0, 1]$-valued, and converges to some $g_0$ uniformly on $\S^1$, it follows that for each $t\in [0, T]$, we have $\widetilde{\rho}_t^{\:\epsilon}\to \rho_t$ uniformly on $\S^1$ where 
		\[
		\rho_t(z) = 2\mu\left(1-\E\left[\prod_{u \in \cN_t}g_0\left( z+Z_u(t)\right)\right]\right), \quad t\geq 0, \quad z \in \S^1.
		\]
		Since $[0, T]$ is compact, the uniform convergence $\widetilde{\rho}^{\:\epsilon} \to \rho$ on $[0, T]\times \S^1$ follows. Moreover, $\rho$ is the McKean representation (see Section 2 in \cite{M1975}), and hence solves the FKPP equation \eqref{eq:FKPP}. 
		
		The last claim follows immediately from the maximum principle for \eqref{eq:interm_fkpp_diff_form}. 
	\end{proof} 
	
	\subsection{Time-dependent test functions}\label{append:time_dep_test_fcns}
	
	The proof of the central limit theorem, Theorem \ref{thm:CLT_Gaussian}, requires us to test the fluctuations against a family $(\phi_{s,t}^\epsilon)_{s \in [0, t]}$, $\epsilon \in (0, 1)$, $t \in [0, T]$, of functions on $\S^1$ constructed as the solution to the problem \eqref{eq:backward_test_eq} which we briefly recall below. Let $\pi : \R \to \S^1$ denote the canonical projection map, and $\S^{1, \epsilon}\coloneqq \pi([0, 1]\cap K_\epsilon^{-1} \Z)$. We consider
	\begin{equation}\label{eq:recall_time_dep_test_fcns}
		\begin{cases}
			\partial_s \phi^\epsilon_{s, t}(z)+\frac{1}{2}\Delta^\epsilon \phi^\epsilon_{s, t}(z)+(\mu_\epsilon-\widetilde{\rho}^{\:\epsilon}_r(z))\phi^\epsilon_{s, t}(z) = 0, & s\in [0, t), \quad z \in \S^{1, \epsilon},
			\\\phi_{t, t}^\epsilon(z)=\phi(z), &z \in \S^{1, \epsilon},
		\end{cases}
	\end{equation}
	where the discrete Laplacian $\Delta^\epsilon$ is defined on all functions $\psi:\S^1\to \R$ by 
	\[
	\Delta^\epsilon  \psi(z)\coloneqq \epsilon^{-2}\left(\psi(z+K_\epsilon^{-1})+\psi(z-K_\epsilon^{-1})-2\psi(z)\right), \quad z \in \S^1.
	\] 
	We observe that for a given $z \in \S^{1, \epsilon}$, this problem is simply a linear system of ordinary differential equations with continuous coefficients. Additionally, by the simple change of variables $s\mapsto t-s$, we may turn the terminal condition into an initial condition. Thus, by the existence and uniqueness theorem for ordinary differential equations, there exists a unique solution $\phi^\epsilon_{\cdot, t}(z)$. We then define $\phi^\epsilon$ on all of $\S^1$ by choosing an arbitrary $C^3$ interpolation. Moreover, by \eqref{eq:mild_test} below, the solution is continuous in $s$. In this section, we derive some a priori estimates on the solution. 
	\begin{lem}\label{lem:a_priori_est}
		There exists $C=C(T, \phi)>0$ such that 
		\begin{equation}\label{ineq:unif_contr_phi_eps}
		\sup_{\epsilon \in (0, 1)}\sup_{0\leq s\leq t\leq T}\|\phi^\epsilon_{s, t}\|_\infty < C,
		\end{equation}
		where $\|\cdot \|_\infty$ is the supremum norm on $\S^1$. Additionally, there is $C=C(T, \phi)>0$ such that for $k\in\{1, 2, 3\}$, we have
		\begin{equation}\label{ineq:unif_contr_deriv_phi_eps}
		\sup_{\epsilon \in (0, 1)}\sup_{0\leq s\leq t\leq T}\|\partial_z^k\phi^\epsilon_{s, t}\|_\infty < C.
	\end{equation}
	\end{lem}
	\begin{proof}
		Given $z \in \S^{1, \epsilon}$, we write $K_\epsilon z$ for the corresponding element of $\Z_\epsilon$. Writing \eqref{eq:recall_time_dep_test_fcns} in integral form, we obtain 
		\begin{equation}\label{eq:mild_test}
			\phi_{s, t}^\epsilon(z)= \sum_{z' \in \Z_\epsilon}G^\epsilon_{t-s}(K_\epsilon z, z')\phi(\epsilon z')+\int_s^t \sum_{z' \in \Z_\epsilon}G_{r-s}^\epsilon(K_\epsilon z, z')\left(\mu_\epsilon-\widetilde{\rho}^{\:\epsilon}_r(\epsilon z')\right)\phi_{r,t}^\epsilon(\epsilon z') dr,
		\end{equation}
		for all $0\leq s\leq t\leq T$ and $z \in \S^1$. Here, given a fixed $z'\in \Z_\epsilon$, the Green's function $G^\epsilon_{t-s}(K_\epsilon z, z')$ of \eqref{eq:Green_eq} is defined for all $z \in \S^1$ using the $C^3$ interpolation. Let $c_0>0$ be the uniform bound on $\widetilde{\rho}$ introduced in \eqref{ineq:bd_on_rho} and let $c_1 \coloneqq (c_0+\sup_{\epsilon \in (0, 1)}\mu_\epsilon)$. Then by Jensen's inequality, 
		\begin{align*}
			|\phi_{s, t}^\epsilon(z)|&\leq \sum_{z' \in \Z_\epsilon}G^\epsilon_{t-s}(K_\epsilon z, z')|\phi(\epsilon z')|+(t-s)\int_s^t \sum_{z' \in \Z_\epsilon}G^\epsilon_{r-s}(K_\epsilon z, z')\left|\mu_\epsilon-\widetilde{\rho}^{\:\epsilon}_r(\epsilon z')\right||\phi_{r,t}^\epsilon(\epsilon z')| dr 
			\\&\leq \sum_{z' \in \Z_\epsilon}G^\epsilon_{t-s}(K_\epsilon z, z')|\phi(\epsilon z')|+Tc_1\int_s^t \sum_{z' \in \Z_\epsilon}G^\epsilon_{r-s}(K_\epsilon z, z')|\phi_{r,t}^\epsilon(\epsilon z')| dr.
		\end{align*}
		We use the Green's function estimate \eqref{ineq:basic_heat_estimate} to bound the integral as follows
		\[
		\int_s^t \sum_{z' \in \Z_\epsilon}G^\epsilon_{r-s}(K_\epsilon z, z')|\phi_{r,t}^\epsilon(\epsilon z')| dr\leq \int_s^t \frac{|\phi_{r,t}^\epsilon(\epsilon z')|}{\sqrt{r-s}} dr\leq \int_s^t \frac{\|\phi_{r,t}^\epsilon\|_\infty}{\sqrt{r-s}} dr,
		\]
		for all $z \in \S^1$, $s \in [0, t]$, and $\epsilon \in (0, 1)$. Moreover, we have 
		\[
		\sum_{z' \in \Z_\epsilon}G^\epsilon_{t-s}(K_\epsilon z, z')|\phi(\epsilon z')|\leq \|\phi\|_\infty^2,
		\]
		for all $z \in \S^1$ and $0\leq s\leq t\leq T$. Then the result follows from the following Gr\"onwall inequality, whose proof is adapted from \cite[Lemma 6 page 33]{Haraux2006}.
		\begin{lem}\label{lem:gen_Gron}
			Let $\psi:\{(s, t):0\leq s\leq t\leq T\}\to \R$ be a bounded non-negative function such that 
			\[
			\psi(s,t)\leq C_0+C_1\int_s^t \frac{\psi(r, t)}{\sqrt{r-s}}dr, \quad 0\leq s\leq t \leq T,
			\]
			for some $C_0=C_0(T), C_1=C_1(T)>0$. Then $\sup_{0\leq s\leq t\leq T}\psi(s, t)<\infty$. 
		\end{lem}
		\begin{proof}
			Following \cite[Lemma 6 page 33]{Haraux2006}, we reduce the problem to an application of the standard Gr\"onwall inequality. We partition the range of integration into $[s, s+\delta]$ and $[s+\delta, t]$, for a small $\delta>0$, and bound the corresponding integrals separately. We compute
			\begin{align*}
				\psi(s, t)&\leq C_0 + C_1 \int_{s+\delta}^t \frac{\psi(r, t)}{\sqrt{r-s}}dr+C_1\sup_{u \in [s, t]}\psi(u, t) \int_s^{s+\delta}\frac{dr}{\sqrt{r-s}}
				\\&\leq C_0+C_1 \delta^{-1/2}\int_s^t \psi(r, t)dr+C_1 \sqrt{\delta} \sup_{u \in [s, t]}\psi(u, t)
				\\&\leq C_0+C_1 \delta^{-1/2}\int_s^t \sup_{u \in [0, r]}\psi(u, t)dr+2C_1 \sqrt{\delta} \sup_{s\in [0, t]}\psi(s, t),
			\end{align*}
			for all $0\leq s\leq t \leq T$. Taking a supremum over $0\leq s\leq t\leq T$ and re-arranging the terms, we have 
			\[
			(1-2C_1\sqrt{\delta})\sup_{s\in [0, t]}\psi(s, t) \leq C_0+C_1\delta^{-1/2}\int_s^t \sup_{u \in [0, r]} \psi(u, t)dr.
			\]
			Choose $\delta\coloneqq (4C_1)^{-2}$. Then
			\[
			\sup_{s\in [0, t]}\psi(s, t)\leq 2C_0+8C_1^2\int_s^t \sup_{u \in [0, r]}\psi(u, t)dr,
			\]
			for all $t \in [0, T]$, and the lemma follows from Gr\"onwall's inequality. 
		\end{proof}
		\noindent We have shown above that 
		\[
		|\phi_{s, t}^\epsilon(z)|\leq \|\phi\|_\infty+Tc_1\int_s^t \frac{\|\phi_{r,t}^\epsilon\|_\infty}{\sqrt{r-s}} dr,
		\]
		for all $z \in \S^1$ and $0\leq s\leq t\leq T$. We take a supremum over $z$ and then apply Lemma \ref{lem:gen_Gron}. This proves \eqref{ineq:unif_contr_phi_eps}. To control the derivatives of $\phi^\epsilon$, we first write \eqref{eq:mild_test} as 
		\[
		\phi_{s, t}^\epsilon(z)= \sum_{z' \in \Z_\epsilon}G^\epsilon_{t-s}(0, z')\phi(\epsilon z'-z)+\int_s^t \sum_{z' \in \Z_\epsilon}G_{r-s}^\epsilon(0, z')\left(\mu_\epsilon-\widetilde{\rho}^{\:\epsilon}_r(\epsilon z'-z)\right)\phi_{r,t}^\epsilon(\epsilon z'-z) dr.
		\]
		Thus, 
		\begin{align*}
		\partial_z^k\phi_{s, t}^\epsilon(z)&= \sum_{z' \in \Z_\epsilon}G^\epsilon_{t-s}(0, z')\partial_z^k\phi(\epsilon z'-z)
		\\&+\int_s^t \sum_{z' \in \Z_\epsilon}G_{r-s}^\epsilon(0, z')\partial_z^k\left[\left(\mu_\epsilon-\widetilde{\rho}^{\:\epsilon}_r(\epsilon z'-z)\right)\phi_{r,t}^\epsilon(\epsilon z'-z)\right] dr.
		\end{align*}
		The control \eqref{ineq:unif_contr_deriv_phi_eps} follows from a simple inductive argument starting from $k=0$ with \eqref{ineq:unif_contr_phi_eps}, and using part three of Remark \ref{rmk:start_X_rho} and Gronw\"all's inequality. 
	\end{proof}
	
	In the following Lemma, we prove that $\phi^\epsilon$ converges uniformly to a function  $\phi_{\cdot, t}$ which is the unique solution to the problem
	\begin{equation}\label{eq:lim_phi_eq}
		\begin{cases}
			\partial_s \phi_{s, t}(z)+\frac{1}{2}\partial_z^2 \phi_{s, t}(z)+(\mu-\rho_s(z))\phi_{s, t}(z) = 0, & s\in [0, t), \quad z \in \S^1,
			\\\phi_{t, t}(z)=\phi(z), &z \in \S^1.
		\end{cases}
	\end{equation} 
	\begin{lem}\label{lem:conv_phi_eps}
		We have 
		\[
		\lim_{\epsilon \searrow 0} \sup_{0\leq s\leq t \leq T} \|\phi^\epsilon_{s, t}-\phi_{s, t}\|_\infty=0.
		\]
	\end{lem}
	\begin{proof}
		Let $G_t(z)$, where $t \in [0, T]$ and $z\in \S^1$ denote the continuum Green function which satisfies
		\[
		\begin{cases}
			\partial_t G_t(z) = \frac{1}{2}\partial_z^2 G_t(z), & z \in \S^1
			\\\lim_{t \searrow 0} G_t(z) = \delta_z(z'),& z \in \S^1.
		\end{cases}
		\]
		Denote $*$ the convolution over $\S^1$. We compute
		\[
		\partial_r(G_{r-s}*\phi^\epsilon_{r, t})(z) = \frac{1}{2}\partial_z^2 (G_{r-s}*\phi^\epsilon_{r, t})(z)+G_{r-s}*\left(-\frac{1}{2}\Delta^\epsilon\phi^\epsilon_{r, t}+(\mu_\epsilon-\widetilde{\rho}^{\:\epsilon}_r)\phi^\epsilon_{r, t}\right)(z).
		\]
		for all $r \in [s, t]$. We compute ${\partial_z^2(G_{r-s}*\phi^\epsilon_{r, t})=(G_{r-s}*\partial_z^2\phi^\epsilon_{r, t})}$. Then, integrating the last displayed expression yields
		\[
		\phi^\epsilon_{s, t}(z)= G_{t-s}*\phi(z)-\frac{1}{2}\int_s^t G_{r-s}*[(\partial_z^2-\Delta^\epsilon)\phi^\epsilon_{r, t}](z)dr-\int_s^t G_{r-s}*[(\mu_\epsilon-\widetilde{\rho}^{\:\epsilon}_r)\phi^\epsilon_{r, t}](z)dr.
		\]
		We also write the limiting problem \eqref{eq:lim_phi_eq} as 
		\[
		\phi_{s, t}(z)= G_{t-s}*\phi(z)-\int_s^t G_{r-s}*[(\mu-\rho_r)\phi_{r, t}](z)dr.
		\]
		Taking the difference of the last two expression, we obtain
		\begin{align*}
			\phi^\epsilon_{s, t}(z)-\phi_{s, t}(z)=&-\frac{1}{2}\int_s^t G_{r-s}*[(\partial_z^2-\Delta^\epsilon)\phi^\epsilon_{r, t}](z)dr
			\\-&\int_s^t G_{r-s}*[(\mu_\epsilon-\widetilde{\rho}^{\:\epsilon}_r)\phi^\epsilon_{r, t}-(\mu-\rho_r)\phi_{r, t}](z)dr
			\\=&-\frac{1}{2}\int_s^t G_{r-s}*[(\partial_z^2-\Delta^\epsilon)\phi^\epsilon_{r, t}](z)dr
			\\-&\int_s^t G_{r-s}*[(\mu_\epsilon-\mu-\widetilde{\rho}^{\:\epsilon}_r+\rho_r)\phi^\epsilon_{r, t}](z)dr
			\\-&\int_s^tG_{r-s}*[(\mu-\widetilde{\rho}_r)(\phi^\epsilon_{r, t}-\phi_{r, t})](z)dr.
		\end{align*}
		We control the three terms in turn. First observe using a Taylor expansion of order three and \eqref{ineq:unif_contr_deriv_phi_eps} in Lemma \ref{lem:a_priori_est} that $(\partial_z^2-\Delta^\epsilon)\phi^\epsilon_{r, t}(z)=O(\epsilon)$ uniformly in $r \in [0, t]$ and $z \in \S^1$. It follows that there exists $c=c(T, \phi)>0$ such that
		\[
		\left|-\frac{1}{2}\int_s^t G_{r-s}*[(\partial_z^2-\Delta^\epsilon)\phi^\epsilon_{r, t}](z)dr\right|\leq c \epsilon.
		\]
		To control the second term, we observe that 
		\[
		\left|\int_s^t G_{r-s}*[(\mu_\epsilon-\mu-\widetilde{\rho}^{\:\epsilon}_r+\rho_r)\phi^\epsilon_{r, t}](z)dr\right|\leq T\times\left(|\mu_\epsilon-\mu|+\sup_{r \in [0, t]}|\rho_r-\widetilde{\rho}^{\:\epsilon}_r|\right)\sup_{r \in [0, t]}\|\phi^\epsilon_{r, t}\|_\infty.
		\]
		Denote the right-hand side by $h_\epsilon$. By Lemma \ref{lem:a_priori_est} and part two of Lemma \ref{lem:FKPP_approx_properties}, we have $\lim_{\epsilon \searrow 0} h_\epsilon=0$. It follows that 
		\[
		\|\phi^\epsilon_{s, t}-\phi_{s, t}\|_\infty\leq c\epsilon + h_\epsilon + \sup_{r \in [0, T]}\|\mu-\widetilde{\rho}_r\|_\infty \int_s^t \|\phi^\epsilon_{r, t}-\phi_{r, t}\|_\infty dr.
		\]
		The result follows from Gr\"onwall's inequality.
	\end{proof}

	\subsection{Characterisation of independent Poisson random variables}\label{append:Poisson_aux_lem}
	
	\begin{lem}\label{lem:indep_Poisson_characterisation}
		Fix $n \in \N$ and let $X_1, \cdots, X_n$ be $\N_0$-valued random variables. If there exist $\lambda_1, \cdots, \lambda_n$ positive real numbers such that 
		\begin{equation}\label{eq:mom_poi}
		\E\left[\prod_{i=1}^n Q_{k_i}(X_i)\right] = \prod_{i=1}^n \lambda_i^{k_i}, \quad k_1, \cdots, k_n \in \N_0, 
		\end{equation}
		where $Q_k(n)$ is the falling factorial defined in \eqref{def:ff}, then, $X_1, \cdots, X_n$ are independent, and each $X_i$ is Poisson distributed with intensity $\lambda_i$. Conversely, if $X_1, \cdots, X_n$ are independent Poisson random variables with respective intensities $\lambda_1, \cdots, \lambda_n>0$, then \eqref{eq:mom_poi} holds. 
	\end{lem}
	\begin{proof}
		Let $t_1, \cdots, t_n \in \R$. Recall the identity $n^j=\sum_{\ell=0}^j\stirling{j}{\ell} Q_\ell(n)$, $n, j \in \N$, where $\stirling{j}{\ell}$, $\ell \in \{0, 1, \cdots, j\}$ denotes the Stirling numbers of the second kind. Using the Taylor series for the exponential function and the monotone convergence theorem, we compute
		\begin{align*}
			\E\left[\exp\left(\sum_{i=1}^n t_i X_i\right)\right]&= \E\left[\prod_{i=1}^n \left(\sum_{j=0}^\infty \frac{t_i^j X_i^j}{j!}\right)\right]
			\\&=\sum_{k_1=0}^\infty \cdots \sum_{k_n=0}^\infty\sum_{\ell_1=0}^{k_1} \cdots \sum_{\ell_n=0}^{k_n}\prod_{i=1}^n \frac{t_i^{k_i}}{k_i !}\stirling{k_i}{\ell_i} \E\left[\prod_{j=1}^n Q_{\ell_j}(X_j)\right]
			\\&=\prod_{i=1}^n\left(\sum_{k=0}^\infty\frac{t_i^{k}}{k!}\sum_{\ell=0}^k\stirling{k}{\ell}\lambda_i^{\ell}\right).
		\end{align*}
		By equation (4.2.16) in \cite{W1994}, and the fact that $\stirling{k}{0}=\mathbbm{1}_{\{0\}}(k)$, we have 
		\[
		\sum_{\ell=0}^k \stirling{k}{\ell} \lambda_i^\ell = e^{-\lambda_i}\sum_{r=0}^\infty \frac{r^k}{r!}\lambda_i^r, \quad i\in \{1, \cdots, n\}. 
		\]
		Thus,
		\begin{align*}
			\E\left[\exp\left(\sum_{i=1}^n t_i X_i\right)\right]&=\prod_{i=1}^n\left(\sum_{k=0}^\infty\frac{t_i^{k}}{k!}e^{-\lambda_i}\sum_{r=0}^\infty \frac{r^k}{r!}\lambda_i^r\right)
			\\&=\prod_{i=1}^n\left(e^{-\lambda_i}\sum_{r=0}^\infty \frac{\lambda_i^r}{r!} \sum_{k=0}^\infty\frac{(t_ir)^{k}}{k!}\right)
			\\&=\prod_{i=1}^n \left(e^{-\lambda_i} \sum_{r=0}^\infty \frac{\left(\lambda_ie^{t_i}\right)^r}{r!}\right)
			\\&= \prod_{i=1}^n e^{\lambda_i\left(e^{t_i}-1\right)}.
		\end{align*}
		The last expression is the moment generating function of a vector of $n$ independent Poisson random variables with respective intensities $\lambda_1, \cdots, \lambda_n$. Hence, the first direction of the lemma is proved. The converse statement follows from independence of the $X_1, \cdots, X_n$, and the observation that if $X$ is Poisson distributed with intensity $\lambda>0$, then $\E[Q_k(X)] = \lambda^k$, for all $k\geq 0$. 
	\end{proof}
	
	\begin{proof}[Proof of Lemma \ref{lem:law_is_prod}.]
		We begin with the forward implication. By Lemma \ref{lem:indep_Poisson_characterisation}, a product of Poisson distributions is uniquely determined by its falling factorial moments. Hence, if 
		\begin{equation}\label{eq:goal_product_Poi}
			\E\left[\prod_{z \in \supp(x)}Q_{n_x(z)}(\epsilon^{-\kappa}X_t^\epsilon(z))\right] = \prod_{z \in \supp(x)} (\epsilon^{-\kappa}\rho_t(z))^{n_{x}(z)}, \quad x \in \X_\epsilon,		\end{equation}
		then the law of $\epsilon^{-\kappa}X_t^\epsilon$ equals that of a product over $z\in \Z_\epsilon$ of Poisson distributions with intensity $\epsilon^{-\kappa}\rho_t(z)$. After scaling both sides by $\epsilon^\kappa$, this shows \eqref{eq:law_is_prod}. We prove \eqref{eq:goal_product_Poi} using (strong) induction on $n_{x}$ the size of $x$, and the fact that 
		\begin{equation}\label{eq:prod_Q_eps}
			Q^\epsilon(x, X_t^\epsilon) \prod_{z \in \supp(x)} \epsilon^{-\kappa n_x(z)}= \prod_{z \in \supp(x)}Q_{n_x(z)}(\epsilon^{-\kappa}X_t^\epsilon(z)), \quad x \in \X_\epsilon.
		\end{equation}
		
		If $n_{x}=1$, then $v_t^\epsilon(x;\rho_t|\nu^\epsilon)=0$ implies $\E[Q^\epsilon(x, X_t^\epsilon)]=\E[X_t^\epsilon(z)]=\rho_t(z)$ where $z \in \Z_\epsilon$ is the only location where $x$ has a particle. Multiplying both sides of the last equality by $\epsilon^{-\kappa}$ gives \eqref{eq:goal_product_Poi} for $n_x=1$. Let $n\geq 1$, and suppose that \eqref{eq:goal_product_Poi} holds for all $x \in \X_\epsilon$ with $n_{x}\leq n$. Consider a test configuration with $n_{x}=n+1$. By the induction hypothesis and \eqref{eq:prod_Q_eps}, we have
		\begin{align*}
			0=v_t^\epsilon(x, \rho_t|\nu^\epsilon)&= \sum_{x'\cleq x}\E[Q^\epsilon(x', X_t^\epsilon)](-1)^{n_x-n_{x'}}\prod_{z \in x\backslash x'} \rho_t(z)
			\\&= \E[Q^\epsilon(x, X_t^\epsilon)] + \sum_{x'\cleq x}\prod_{z' \in \supp(x')} \rho_t(z')^{n_{x'}(z')}(-1)^{n_x-n_{x'}}\prod_{z \in x\backslash x'} \rho_t(z).
		\end{align*}
		Since 
		\[
		\prod_{z' \in \supp(x')} \rho_t(z')^{n_{x'}(z')}\prod_{z \in x\backslash x'} \rho_t(z) = \prod_{z \in \supp(x)} \rho_t(z)^{n_x(z)},
		\]
		does not depend on $x'$, we obtain
		\begin{align*}
			0&=\E[Q^\epsilon(x, X_t^\epsilon)] + \left(\sum_{i=0}^n {n+1 \choose i} (-1)^{n+1-i}\right) \prod_{z \in \supp(x)}\rho_t(z)^{n_x(z)}
			\\&= \E[Q^\epsilon(x, X_t^\epsilon)]-\prod_{z \in \supp(x)}\rho_t(z)^{n_x(z)},
		\end{align*}
		where in the last step, we have used that $\sum_{i=0}^n {n+1 \choose i} (-1)^{n+1-i}=-1+\sum_{i=0}^{n+1} {n+1 \choose i} (-1)^{n+1-i}=-1$. After an application of \eqref{eq:prod_Q_eps}, we get \eqref{eq:goal_product_Poi} for $n_x=n+1$. This proves \eqref{eq:goal_product_Poi} and also \eqref{eq:law_is_prod} as outlined at the beginning of the proof. For the converse, if \eqref{eq:law_is_prod} holds, then \eqref{eq:goal_product_Poi} is also true. This easily implies that $v^\epsilon(x, \rho_t|\nu^\epsilon)=0$ for all $x \in \X_\epsilon$.
	\end{proof}
	
\end{document}